\documentclass[11pt,reqno]{amsart}
\usepackage[utf8]{inputenc}
\usepackage[T1]{fontenc}
\usepackage{lmodern}
\usepackage{amsfonts,amsthm,amsmath,amssymb,thmtools,mathtools,stmaryrd}
\usepackage{graphicx}
\usepackage{comment} 
\usepackage{float}
\usepackage{enumerate}
\usepackage{hyperref}
\usepackage[dvipsnames]{xcolor}
\usepackage{tikz-cd}
\usepackage[margin=1in]{geometry}
\usepackage[
    maxbibnames=99,
    backend=biber,
    style=alphabetic,
    giveninits=true
]{biblatex}
\DeclareFieldFormat{pages}{#1}
\renewbibmacro{in:}{%
  \ifentrytype{article}
    {}
    {\bibstring{in}%
     \printunit{\intitlepunct}}}
\DeclareFieldFormat
[article,inbook,incollection,inproceedings,patent,thesis,unpublished]
  {title}{\mkbibemph{#1}}
\DeclareFieldFormat{journaltitle}{#1\isdot}
\DeclareFieldFormat[article]{volume}{\mkbibbold{#1}}
\DeclareFieldFormat[article]{number}{\bibstring{number}\addnbspace #1}

\renewbibmacro*{journal+issuetitle}{%
  \usebibmacro{journal}%
  \setunit*{\addspace}%
  \iffieldundef{series}
    {}
    {\newunit
     \printfield{series}%
     \setunit{\addspace}}%
  \printfield{volume}%
  \setunit{\addspace}%
  \usebibmacro{issue+date}%
  \setunit{\addcomma\space}%
  \printfield{number}%
  \setunit{\addcolon\space}%
  \usebibmacro{issue}%
  \setunit{\addcomma\space}%
  \printfield{eid}
  \newunit}

\newtheoremstyle{mystyle}
  {}
  {0}
  {\itshape}
  {}
  {\bfseries}
  {.}
  { }
  {\thmname{#1}\thmnumber{ #2}\thmnote{ (#3)}}
\theoremstyle{mystyle}
\newtheorem{Thm}{Theorem}[section]
\newtheorem{Lem}[Thm]{Lemma}
\newtheorem{Cor}[Thm]{Corollary}
\newtheorem{Prop}[Thm]{Proposition}

\theoremstyle{definition}
\newtheorem{Def}[Thm]{Definition}
\newtheorem{Ex}[Thm]{Example}

\theoremstyle{remark}
\newtheorem{Rmk}[Thm]{Remark}

\newcommand{\R}{\mathbb R}
\newcommand{\Z}{\mathbb Z}
\newcommand{\Q}{\mathbb Q}
\let\C\nothing
\newcommand{\C}{\mathbb C}
\newcommand{\Diff}{\mathrm{Diff}}
\newcommand{\gl}{\mathfrak{gl}}

\newcommand{\cat}[1]{\mathchoice
  {\ensuremath{\mbox{\bfseries {\upshape {#1}}}}}
  {\ensuremath{\mbox{\bfseries {\upshape {#1}}}}}
  {\scalebox{.7}{\ensuremath{\mbox{\bfseries {\upshape {#1}}}}}}
  {\scalebox{.5}{\ensuremath{\mbox{\bfseries {\upshape {#1}}}}}}%
  }

\newcommand{\TanWeb}[1]{\mathbf{Webs}_{#1}}
\newcommand{\TanWebsing}[1]{\TanWeb{#1}^{\mathrm{sing}}}
\newcommand{\Foams}[1]{\mathbf{Foams}_{#1}}

\newcommand{\BNfunc}[1]{ \left\llbracket #1 \right\rrbracket }

\newcommand{\Kb}{\cat{K}^b} 
\newcommand{\HChb}{\cat{HCh}^b} 

\newcommand{\listsymbol}{--}
\newcommand{\Khdot}{x}

\title{Khovanov skein lasagna modules with $1$-dimensional inputs}

\author{Qiuyu Ren}
\address{Department of Mathematics, University of California, Berkeley, Berkeley, CA 94720, USA}
\email{qiuyu\_ren@berkeley.edu}

\author{Ian Sullivan}
\address{Department of Mathematics, University of California, Davis, One Shields Avenue, Davis, CA 95616, USA}
\email{iasullivan@ucdavis.edu}

\author{Paul Wedrich}
\address{Fachbereich Mathematik, Universit\"at Hamburg, 
Bundesstra{\ss}e 55, 
20146 Hamburg, Germany
\href{https://paul.wedrich.at/}{paul.wedrich.at}}
\email{paul.wedrich@uni-hamburg.de}

\author{Michael Willis}
\address{Department of Mathematics, Texas A\&M University, College Station, TX 77840, USA}
\email{msw188@tamu.edu}

\author{Melissa Zhang}
\address{Department of Mathematics, University of California, Davis, One Shields Avenue, Davis, CA 95616, USA}
\email{mlzhang@ucdavis.edu}

\begin{document}

\begin{abstract}
We construct a variant of Khovanov skein lasagna modules, which takes the Khovanov homology in connected sums of $S^1\times S^2$ defined by Rozansky and Willis as the input link homology. To carry out the construction, we prove functoriality of Rozansky-Willis's homology for cobordisms in a class of $4$-manifolds that we call $4$-dimensional relative $1$-handlebody complements, by using, as a bypass, an isomorphism proved in Sullivan--Zhang \cite{sullivan2024kirby} relating the Rozansky-Willis homology and the classical Khovanov skein lasagna module of links on the boundary of $D^2\times S^2$. Along the way, we also present new results on diffeomorphism groups, on Gluck twists for Khovanov skein lasagna modules, and on the functoriality of $\mathfrak{gl}_2$ foams.
\end{abstract}

\maketitle

\section{Introduction}
In \cite{morrison2022invariants}, Morrison--Walker--Wedrich defined a package of smooth $4$-manifold invariants, including the so-called Khovanov skein lasagna modules. Roughly, for every compact oriented $4$-manifold $X$ and framed oriented link $L\subset\partial X$, the Khovanov skein lasagna module of $(X,L)$, denoted $\mathcal S_0^2(X;L)$, is the $R$-module generated by properly embedded framed oriented surfaces $\Sigma$ (called skeins) in $X\backslash B$ for some $B$, with Khovanov decorations on the inputs, modulo certain relations. Here $B$ is the tubular neighborhood of an embedded finite $0$-dimensional CW complex in the interior of $X$ (i.e. a disjoint union of finitely many open $4$-balls), and $R$ is a fixed commutative coefficient ring (usually suppressed from the notation). The Khovanov decoration on such a skein $\Sigma$ consists of labels $v_i\in KhR_2(\Sigma\cap\partial B_i)$, one for each component $B_i$ of $B$, where $KhR_2$ denotes the Khovanov homology \cite{khovanov2000categorification} of links in $S^3$ over $R$, suitably renormalized.

Explicit formulas for computing Khovanov skein lasagna modules in terms of handle decompositions are available \cite{manolescu2022skein,manolescu2023skein}, with the caveat that formulas concerning new $1$- and $2$-handle attachments usually involve an infinite colimit, and are therefore impractical to carry out in general. Nevertheless, in the absence of $1$-handles, interesting explicit calculations have been made using these formulas \cite{sullivan2024kirby,ren2024khovanov}. Notably, Ren--Willis's calculation \cite{ren2024khovanov} shows that Khovanov skein lasagna modules can detect exotic $4$-manifolds.

In the presence of $1$-handles, however, the formula for $\mathcal S_0^2(X;L)$ \cite{manolescu2023skein} is computationally complex,
rendering explicit computations virtually impossible except in the simplest cases. The purpose of this paper is to define a variant of Khovanov skein lasagna modules over the rationals, denoted $\bar{\mathcal S}_0^2(X;L)$, that removes this complexity. Roughly, the construction is the same as the usual Khovanov skein lasagna modules, except that skeins will live in the complement of tubular neighborhoods of embedded finite $1$-dimensional CW complexes in the interior of $X$, and the decorations will be elements in a suitable renormalization of the Khovanov homology for links in connected sums of $S^1\times S^2$'s defined by Rozansky \cite{rozansky2010categorification} and Willis \cite{willis2021khovanov}. We state our main result informally as follows.

Throughout the rest of our paper, unless stated otherwise, the base ring will be the field of rational numbers $\Q$; we henceforth suppress it from the notation.

\begin{Thm}[Definition~\ref{def:S02_bar}]\label{thm:main}
There is a well-defined invariant $\bar{\mathcal S}_0^2(X;L)$ for pairs $(X,L)$ of compact oriented smooth $4$-manifold $X$ and framed oriented link $L\subset\partial X$, which is a $\Q$-vector space graded by $(\tfrac12\Z)^2\times H_2(X,L;\Z/2)$. When $X=\natural^n(S^1\times B^3)$, the invariant $\bar{\mathcal S}_0^2(X;L)$ is canonically isomorphic to $\widetilde{KhR}_2^-(L)$, a suitable renormalization of the Rozansky-Willis homology of the framed oriented link $L\subset\#^n(S^1\times S^2)$ over $\Q$.
\end{Thm}

The main difficulty of defining $\bar{\mathcal S}_0^2(X;L)$ concerns the extension of Rozansky-Willis homology to link cobordisms embedded in a special class of $4$-manifolds, called $4$-dimensional relative $1$-handlebody complements, in a functorial way. The precise statement is formulated in Theorem~\ref{thm:KhR_2-}, which is the main theorem of this paper. We state one very special case of this, namely the functoriality of $\widetilde{KhR}_2^-$ for cobordisms in $I\times\#^m(S^1\times S^2)$.

\begin{Thm}[Theorem~\ref{thm:concrete_functoriality}]\label{thm:concrete_functoriality_intro}
For any link cobordism $\Sigma\subset I\times\#^m(S^1\times S^2)$ between framed oriented links $L_0,L_1\subset\#^m(S^1\times S^2)$, there is an induced $\Q$-linear map $\widetilde{KhR}_2^-(\Sigma)\colon\widetilde{KhR}_2^-(L_0)\to\widetilde{KhR}_2^-(L_1)$ of bidegree $(0,-\chi(\Sigma))$, such that the assignment $\Sigma\mapsto\widetilde{KhR}_2^-(\Sigma)$ is functorial.
\end{Thm}

We mention three new ingredients of various flavors required for proving Theorem~\ref{thm:main}, which might be of independent interest.

\begin{Thm}[Theorem~\ref{thm:diff_D2S2_rel_boundary}]\label{thm:diff_D2S2_rel_boundary_intro}
Let $k\ge1$, $m_1,\cdots,m_k\ge0$. The diffeomorphism group of $D_{std}:=\#_{i=1}^k\natural^{m_i}(D^2\times S^2)$ rel boundary fits into an exact sequence $$1\to\Diff_{\partial,loc}(D_{std})\to\Diff_\partial(D_{std})\to\Z^{m(m-1)/2}\times(\Z/2)^{k-1}\to1.$$ Here, $\Diff_{\partial,loc}(D_{std})$ denotes the subgroup consisting of diffeomorphisms isotopic rel boundary to one supported in a local $4$-ball, and $m=\sum_{i=1}^km_i$.
\end{Thm}
When $D_{std}=D^2\times S^2$, this is a version of Gabai's $4$-dimensional lightbulb theorem \cite[Corollary~1.7]{gabai20204}. We deduce Theorem~\ref{thm:diff_D2S2_rel_boundary_intro} as a consequence of Gabai's result. The cokernel of $\Diff_{\partial,loc}(D_{std})\to\Diff_\partial(D_{std})$ in Theorem~\ref{thm:diff_D2S2_rel_boundary_intro} is generated by Dehn twists along embedded $3$-spheres, as well as implanted barbell diffeomorphisms defined by Budney--Gabai \cite{budney2019knotted}. We rediscovered the barbell diffeomorphism during this work and will present an alternative description of it in the proof of Theorem~\ref{thm:diff_D2S2_rel_boundary}. Since $D_{std}$ embeds into $B^4$, $\pi_0(\Diff_\partial(B^4))\xrightarrow{\cong}\pi_0(\Diff_{\partial,loc}(D_{std}))$ via a local embedding $B^4\subset D_{std}$. By comparing to the work of Orson--Powell \cite[Theorem~A(1)]{orson2025mapping}, the cokernel can also be identified with the topological mapping class group.

\begin{Cor}
The natural map $$\Diff_\partial(D_{std})/\Diff_{\partial,loc}(D_{std})\cong\pi_0(\Diff_\partial(D_{std}))/\pi_0(\Diff_{\partial}(B^4))\to\pi_0(\mathrm{Homeo}_\partial(D_{std}))$$ is an isomorphism.\qed
\end{Cor}
In other words, modulo diffeomorphisms contained in $B^4$, $D_{std}=\#_{i=1}^k\natural^{m_i}(D^2\times S^2)$ does not admit exotic diffeomorphisms. This is in contrast to the existence of such exotic diffeomorphisms on many $4$-manifolds (including contractible ones) detected using gauge theory; see \cite{ruberman1998obstruction,kronheimer2021dehn,konno2023exotic,qiu2024exotic} and references therein.

\begin{Thm}[Theorem~\ref{thm:gluck}]\label{thm:gluck_intro}
Gluck twists induce isomorphisms on Khovanov skein lasagna modules (over $\Q$).
\end{Thm}

In particular, Khovanov skein lasagna modules over $\Q$ cannot detect exotica arising from Gluck twists (to our knowledge, no such phenomenon has ever been detected on compact orientable $4$-manifolds). This was already hinted at in Ren--Willis \cite[Section~6.10]{ren2024khovanov}. See Theorem~\ref{thm:gluck} for a more precise statement as well as some formal properties enjoyed by the induced map. Insensitivity of Khovanov skein lasagna modules to Gluck twists as a consequence of \cite{sullivan2024kirby} was also observed by Krushkal--Wedrich, following a different approach independent of Theorem~\ref{thm:gluck}; this may appear elsewhere.

\begin{Thm}[Theorem~\ref{thm:gl2_webs_functorial}]\label{thm:gl2_webs_functorial_intro}
The universal Khovanov-Rozansky $\gl_2$ homology for $\gl_2$ webs in $S^3$ and $\gl_2$ foams in $I\times S^3$ between them is functorial (on the chain level up to chain homotopy, over $\Z[E_1,E_2]$). Moreover, it can be extended to singular $\gl_2$ foams, and the induced map by such a foam is independent of the embedding of the interior of $2$-labeled faces.
\end{Thm}

See Appendix~\ref{sec:sign} for the precise setup we are using in Theorem~\ref{thm:gl2_webs_functorial_intro}, in particular our definition of singular $\gl_2$ foams. It suffices to say here that we are allowing transverse double points between $1$-,$2$- or $2$-,$2$-labeled faces, as well as singular points on faces that introduce framing changes. The functoriality of $\gl_2$ homology for links and link cobordisms in $S^3$ was proved by Morrison--Walker--Wedrich \cite{morrison2022invariants}, following the work of Jacobsson \cite{jacobsson2004invariant}, Blanchet \cite{blanchet2010oriented}, and others. The functoriality of $\gl_2$ homology for $\gl_2$ webs in $\R^3$ and $\gl_2$ foams between them was proved by Queffelec \cite{queffelec2022gl2}. Theorem~\ref{thm:gl2_webs_functorial_intro} is a simultaneous generalization of these results.\medskip

Our proof of the functoriality of $\widetilde{KhR}_2^-$ is rather unconventional. Sullivan--Zhang \cite{sullivan2024kirby} proved that the Rozansky-Willis homology can be recovered from the Khovanov skein lasagna module of $\natural^m(D^2\times S^2)$ with boundary links in an appropriate sense (see Theorem~\ref{thm:SZ}). To check the functoriality of Rozansky-Willis homology in the strong sense we need, instead of checking all movie moves (``second order'' moves) relating sequences of elementary cobordisms (the authors are unaware of how to obtain a complete set of movie moves in our setup), we employ Sullivan--Zhang's result as a bypass. As lasagna gluing operations are manifestly functorial, this alternative viewpoint significantly simplifies the task, allowing us to perform checks only on the elementary cobordisms (``first order'' moves) themselves.

After some preliminaries in Section~\ref{sec:prelim}, we are able to give a more comprehensive overview of this paper in Section~\ref{sec:outline}. In Section~\ref{sec:input_homology}, we state the precise functoriality statement for $\widetilde{KhR}_2^-$ (Theorem~\ref{thm:KhR_2-}). In Section~\ref{sec:def_S02_bar}, we define the $1$-dimensional-input skein lasagna module (Definition~\ref{def:S02_bar}). In Section~\ref{sec:upside_down}, we give an overview of the proof of Theorem~\ref{thm:KhR_2-}. We refer readers to Section~\ref{sec:rest} for a discussion of the remaining sections of the paper. To define a Lee version of $1$-dimensional-input skein lasagna modules, some adjustments to our current proof are required. We hope to investigate the Lee version, as well as computational aspects of the $1$-dimensional-input Khovanov skein lasagna modules, in future work.

Throughout the main text of this paper, we use the simpler Bar-Natan formalism for Khovanov homology as opposed to the $\gl_2$ webs and foams formalism. Consequently, some constructions, definitions, and arguments will carry sign ambiguities at various stages. We resolve the sign issues in Appendix~\ref{sec:sign}.

\subsection*{TQFT context} 
Skein lasagna modules appear in \cite{morrison2022invariants} as the $4$-dimensional layer of an extended topological quantum field theory (TQFT) that is determined locally, i.e. based on 0-dimensional inputs, by a link homology theory, e.g. Khovanov homology, see \cite{wedrich2025link} for a recent survey. In this somewhat speculative section we aim to situate our Khovanov skein lasagna modules with $1$-dimensional inputs in the TQFT landscape by comparing it with related constructions.

The \emph{blob complex} was developed by Morrison--Walker \cite{morrison2012blob} as one possible extension of skein lasagna modules to an invariant that, in principle, supports computations via skein exact triangles. The underlying idea is to replace the skein module, a quotient of a space of (decorated) skeins modulo a subspace of relations, by a resolution: the space of skeins in degree zero, linear combinations of basic relations between skeins in degree one, relations between basic relations in degree two, and so on. Taking the zeroth homology of this blob complex recovers the skein module. In the context of Khovanov homology, we note that the blob complex is $\Z^3$-graded by blob degree and the homological and quantum gradings of Khovanov homology.

It is expected that a \emph{chain-level refinement} of Khovanov skein lasagna modules could be constructed from a conjectural fully homotopy-coherent version of Khovanov chain complexes \cite[Conjecture~4.1]{wedrich2025link}. In this case, the blob complex is simply a homotopy colimit \cite[Section~7]{morrison2012blob}, the blob degree and the internal homological degree get collapsed into a single grading, and to compute the invariant of a $4$-manifold, one only takes homology once. Although this invariant has not yet been constructed, it is possible to predict structural properties \cite[Section~4.7]{manolescu2023skein}: the resulting homology should appear on the $E_\infty$ page of a spectral sequence approximated by the blob homology of Khovanov homology on the $E_2$ page. On $4$-dimensional $0$-handlebodies it should recover Khovanov homology, and on $1$-handlebodies Rozansky-Willis homology.

Our \emph{Khovanov skein lasagna modules with $1$-dimensional inputs} interpolate between Khovanov skein lasagna modules and the homology of the desired chain level refinement in the sense that they agree with the latter on $4$-dimensional $1$-handlebodies by Theorem~\ref{thm:main}, but treat handles of index $\geq 2$ skein-theoretically. In particular, they are locally finite-dimensional and algorithmically computable in any finite range of degrees on $1$-handlebodies.

As a special feature of our construction in the setting of $\gl_2$ link homology, we can allow as boundary conditions for $4$-manifolds framed oriented links with $2$-divisible fundamental class, i.e. not necessarily null-homologous. We speculate that this is possible due to the nontriviality of the sylleptic center of the underlying braided monoidal $2$-category; see Section~\ref{sec:sylleptic}. On the other hand, we also define a version of skein lasagna modules with 1-dimensional inputs which requires null-homologous boundary conditions and allows integral gradings; see Remark~\ref{rmk:bounding}.

\section*{Acknowledgements}
We thank Ian Agol, William Ballinger, Richard Bamler, Anna Beliakova, Christian Blanchet, Eugene Gorsky, Matthew Hogancamp, Danica Kosanovi\'c, Robert Lipshitz, Patrick Orson, Hoel Queffelec, Daniel Ruberman, and Kevin Walker for helpful discussions. Special thanks to Seungwon Kim for many topological inputs.

Parts of this work were completed during 
the \textit{Annual Meeting of the Simons Collaboration on New Structures in Low-Dimensional Topology} in New York City, USA;
the \textit{Summer School on Modern Tools in Low-Dimensional Topology} and \textit{Conference on Modern Developments in Low-Dimensional Topology} at ICTP in Trieste, Italy;
the \textit{Categorification in Low Dimensional Topology} workshop and conference at Ruhr University Bochum in Bochum, Germany;
and the \textit{New Perspectives on Skein Modules} workshop at CIRM in Marseille, France.
We are grateful for the hospitality of these institutions and for the opportunity for us to collaborate in-person at these venues.

\section*{Funding}
QR was partially supported by the Simons Investigator Award 376200. PW acknowledges support from the Deutsche Forschungsgemeinschaft (DFG, German Research Foundation) under Germany's Excellence Strategy - EXC 2121 ``Quantum Universe'' - 390833306 and the Collaborative Research Center - SFB 1624 ``Higher structures, moduli spaces and integrability'' - 506632645.

\section{Preliminaries}\label{sec:prelim}
Throughout this paper, unless stated otherwise, we work over the base ring $\Q$ and suppress it from the notation. Homology groups of spaces are always taken to have integral coefficients unless otherwise indicated. The Khovanov(-Rozansky) $\gl_2$ homology of the unknot is standardly identified as a unital algebra with $\Q[\Khdot]/(\Khdot^2)$ and the Lee homology of the unknot with $\Q[\Khdot]/(\Khdot^2-1)$.

\subsection{Khovanov skein lasagna modules}\label{sec:S02}
We recall the definition of Khovanov skein lasagna modules constructed in \cite{morrison2022invariants} following the concise account in \cite[Section~2]{ren2024khovanov}. We do not attempt to be comprehensive; see \cite{wedrich2025link} for a recent survey. The purpose of this section is to parallel the upcoming definition for the $1$-dimensional-input skein lasagna modules in Section~\ref{sec:def_S02_bar}.

Let $X$ be a compact oriented $4$-manifold and $L\subset\partial X$ be a framed oriented link. A \textit{skein} in $X$ rel $L$ is a properly embedded framed oriented surface $\Sigma\subset X\backslash int(B)$ with $\partial\Sigma\cap\partial X=L$ (with framing and orientation), where $B\subset int(X)$ is a finite disjoint union of unparametrized $4$-balls, called the \textit{input balls} of $\Sigma$. The \textit{input links} of $\Sigma$ are the framed oriented links $(-\partial\Sigma)\cap\partial B_i$ in $\partial B_i$, where $B_i$ runs over the connected components of $B$, and the negative sign denotes orientation-reversal. A \textit{lasagna filling} of $(X,L)$ is a pair $(\Sigma,v)$ where $\Sigma$ is a skein with some input balls $B=\sqcup_iB_i$, and $v\in\otimes_iKhR_2((-\partial\Sigma)\cap\partial B_i)$, where $KhR_2$ is the Khovanov-Rozansky $\gl_2$ homology, defined for framed oriented links in $S^3$.

\begin{Def}\label{def:S02}
The \textit{Khovanov skein lasagna module} of $(X,L)$ is the $\Q$-vector space
$$\mathcal S_0^2(X;L):=\Q\{\text{lasagna fillings of }(X,L)\}/\sim,$$ where $\sim$ is the equivalence relation generated by
\begin{itemize}
\item Isotopy of the skein rel boundary;
\item Linearity in the decoration: $(\Sigma,v)+\lambda(\Sigma,w)\sim(\Sigma,v+\lambda w)$, $\lambda\in\Q$;
\item Enclosement relation: Let $(\Sigma,v)$ be a lasagna filling with input balls $B$, and let $B'\subset int(X)$ be a finite disjoint union of unparametrized $4$-balls in $int(X)$ that contains $B$ in its interior. Then $(\Sigma,v)\sim(\Sigma\backslash int(B'),KhR_2(\Sigma\cap(B'\backslash int(B)))(v))$, where $\Sigma\backslash int(B')$ is regarded as a skein with input balls $B'$.
\end{itemize}
The \textit{Khovanov skein lasagna module} of $X$ is $\mathcal S_0^2(X):=\mathcal S_0^2(X;\emptyset)$.
\end{Def}
In the enclosement relation, $B'\backslash int(B)$ is a disjoint union of $4$-balls with some finite number of input holes, and $\Sigma\cap(B'\backslash int(B))$ is a cobordism in $B'\backslash int(B)$ between the input links of $\Sigma$ and the input links of 
$\Sigma\backslash int(B')$. 
The induced map on $KhR_2$ is defined on each connected component of $B'$ by first tubing the input holes together along paths disjoint from $\Sigma$, and then using the usual induced maps for cobordisms in $I\times S^3$. The induced map is independent of the choice of the tubing paths.

Although Definition~\ref{def:S02} does not include a relation allowing input balls to move, this is implied by the other relations.

If $(\Sigma,v)$ is a lasagna filling of $(X,L)$ where $v$ is homogeneous of degree $(h,q)$, then the \textit{tridegree} of $(\Sigma,v)$ is defined to be $(h,q-\chi(\Sigma),[\Sigma])\in\Z^2\times H_2^L(X)$, where $H_2^L(X)$ is the preimage of $[L]$ under the connecting homomorphism $H_2(X,L)\to H_1(L)$, which is an $H_2(X)$-torsor. This descends to a trigrading on $\mathcal S_0^2(X;L)$. The three gradings, usually denoted by $(h,q,\alpha)$, are called \textit{homological grading}, \textit{quantum grading}, \textit{skein grading}, respectively.

In the case where $X=B^4$, for every $L\subset S^3$, we have a bigrading-preserving isomorphism $$KhR_2(L)\xrightarrow{\cong}\mathcal S_0^2(B^4;L),\ v\mapsto[(I\times L,v)],$$ where $I\times L$ denotes a standard product skein with an input ball $\tfrac12B^4$.
Note that the skein grading on $\mathcal S_0^2(B^4;L)$ is trivial.

We remark that the classical Khovanov-Rozansky $\gl_2$ homology is defined for links in concretely parametrized $S^3$ and cobordisms in concrete $I\times S^3$. In the definitions above, however, we have not distinguished between unparametrized $S^3$, $B^4$, $I\times S^3$, etc., with parametrized ones. It is a nontrivial fact, proved carefully in \cite[Section~4.2]{morrison2022invariants}, that one may be ambiguous in the statements and forget the parametrization issue. Later in this paper, we will devote considerable effort to removing the reparametrization ambiguity in our setup, where the topology of the inputs is significantly more nontrivial.

\subsubsection{Khovanov skein lasagna modules of \texorpdfstring{$\#\natural(D^2\times S^2)$}{D2*S2} and \texorpdfstring{$I\times\sqcup(\#(S^1\times S^2))$}{I*S1*S2}}\label{sec:D2S2}
We study the Khovanov skein lasagna module for the following two examples, which will be useful for our construction. This section is not required for reading Section~\ref{sec:outline}.

Throughout this section, fix $k\ge0$, $m_1,\cdots,m_k\ge0$, and write $D_{std}:=\#_{i=1}^k\natural^{m_i}(D^2\times S^2)$.

\begin{Ex}\label{ex:D2S2}
The skein lasagna module of $D^2\times S^2$ was studied in \cite[Theorem~1.2]{manolescu2022skein}. Using the connect/boundary sum formula \cite[Theorem~1.4,Corollary~7.3]{manolescu2022skein}, we see that $$\mathcal S_0^2(D_{std})=\Q[A_{i,j,0}^\pm,A_{i,j,1}\colon1\le i\le k,1\le j\le m_i]$$ where $A_{i,j,0}$ has tridegree $(0,0,e_{i,j})$ and $A_{i,j,1}$ has tridegree $(0,-2,e_{i,j})$. Here $e_{i,j}$ is the second homology class of $D_{std}$ represented by the $j$-th core sphere of the $i$-th connected summand. The generator $A_{i,j,0}$ (resp. $A_{i,j,1}$) is represented by the (positively oriented, framed) $j$-th core sphere in the $i$-th connected summand with a dot (resp. without dots). By neck-cutting, one sees that the negatively oriented $j$-th core sphere in the $i$-th summand with a dot (resp. without dots) represents the element $A_{i,j,0}^{-1}$ (resp. $-A_{i,j,0}^{-2}A_{i,j,1}$) in $\mathcal S_0^2(D_{std})$.
\end{Ex}

\begin{Ex}\label{ex:S1S2I}
We claim that the inclusion map $i\colon I\times\partial D_{std}\hookrightarrow D_{std}$ as a collar neighborhood of the boundary induces an isomorphism $$\mathcal S_0^2(i)\colon\mathcal S_0^2(I\times\partial D_{std})\xrightarrow{\cong}\mathcal S_0^2(D_{std}).$$

Write $(\#^{m_i}(S^1\times S^2))^\circ=\#^{m_i}(S^1\times S^2)\backslash int(B^3)$. Then $I\times(\#^{m_i}(S^1\times S^2))^\circ\cong(\natural^{m_i}(S^1\times B^3))\natural(\natural^{m_i}(D^2\times S^2))$. Since $\mathcal S_0^2(S^1\times B^3)=\Q$ by general position and neck-cutting (see \cite[Theorem~1.5(a)]{manolescu2023skein}), we see by the boundary connected sum formula and the explicit description of generators of $\mathcal S_0^2(D_{std})$ in Example~\ref{ex:D2S2} that the composition of inclusions $I\times\sqcup_{i=1}^k(\#^{m_i}(S^1\times S^2))^\circ\hookrightarrow I\times\partial D_{std}\xhookrightarrow{i}D_{std}$ induces an isomorphism on $\mathcal S_0^2$. Since the first inclusion is given by attaching $3$-handles, its induced map on $\mathcal S_0^2$ is surjective by general position. It follows that $\mathcal S_0^2(i)$ is an isomorphism.
\end{Ex}

Thus, $\mathcal S_0^2(D_{std})\cong\mathcal S_0^2(I\times\partial D_{std})$ is an algebra under stacking along the $I$-direction.

\begin{Def}\label{def:shift_auto}
The \textit{shifting automorphism} by $\alpha\in H_2(D_{std})$ is the operator $$\mathrm{id}_\alpha\colon\mathcal S_0^2(D_{std})\to\mathcal S_0^2(D_{std})$$ defined by multiplication by $\prod_{i,j}A_{i,j,0}^{\alpha_{i,j}}$, where $\alpha=\sum_{i,j}\alpha_{i,j}e_{i,j}$.
\end{Def}

The element $\prod_{i,j}A_{i,j,0}^{\alpha_{i,j}}\in\mathcal S_0^2(D_{std})$ is characterized by being the unique nonzero element in trigrading $(0,0,\alpha)$ whose image under $\mathcal S_0^2(D_{std})\to\mathcal S_0^2(S^4)$ is the class represented by the empty skein, for any embedding $D_{std}\subset S^4$. Therefore, $\mathrm{id}_\alpha$ is independent of the parametrization of $D_{std}$. The shifting automorphisms assemble to a group action $H_2(D_{std})\to\mathrm{Aut}(\mathcal S_0^2(D_{std}))$.

\subsection{Rozansky-Willis homology for framed links}\label{sec:RW}
Fix an integer $m\ge0$. Rozansky \cite{rozansky2010categorification} and Willis \cite{willis2021khovanov} defined Khovanov homology for links in $\#^m(S^1\times S^2)$. For our purposes, we describe carefully the topological setup as follows.

Fix once and for all a surgery diagram for $\#^m(S^1\times S^2)$ as the $0$-surgery on an $m$-component unlink, drawn in a standard way on the plane. Explicitly, this means that we have the following data:
\begin{itemize}
\item a decomposition $\#^m(S^1\times S^2)=M\cup N$, where $N$ is a disjoint union of $m$ solid tori;
\item a diffeomorphism $M=S^3\backslash\nu(U_m)$ for an $m$-component unlink $U_m$;
\item a point $\infty\in int(M)\subset S^3$;
\item a diffeomorphism $S^3\backslash\{\infty\}=\R^3$ that sends $\nu(U_m)$ to the set of points within distance $0.01$ to $\sqcup_{i=1}^m\{(x,0,z)\colon(x-i)^2+z^2=0.01\}$, and sends $U_m$ to $\sqcup_{i=1}^m\{(x,-0.01z,z)\colon(x-i)^2+z^2=0.01\}$.
\end{itemize}

A framed oriented link $L\subset\#^m(S^1\times S^2)$ with $2$-divisible homology class is \textit{admissible} if $L\subset int(M)\backslash\{\infty\}$ and that the orthogonal projection of $L\cup U_m$ onto $\R^2\times\{0\}$ is generic, and standard near the (projection of the) surgery regions (say $\sqcup_{i=1}^m(i-0.2,i+0.2)\times(-0.02,0.02)\subset\R^2$) in the sense shown on the left of Figure~\ref{fig:admissibility}. The \textit{writhe} of an admissible link $L$, denoted $w(L)$, is the writhe of the framed oriented link $L$ regarded as a link in $S^3$ via $M\subset S^3$.

\begin{figure}
\centering
\includegraphics[width=0.9\linewidth]{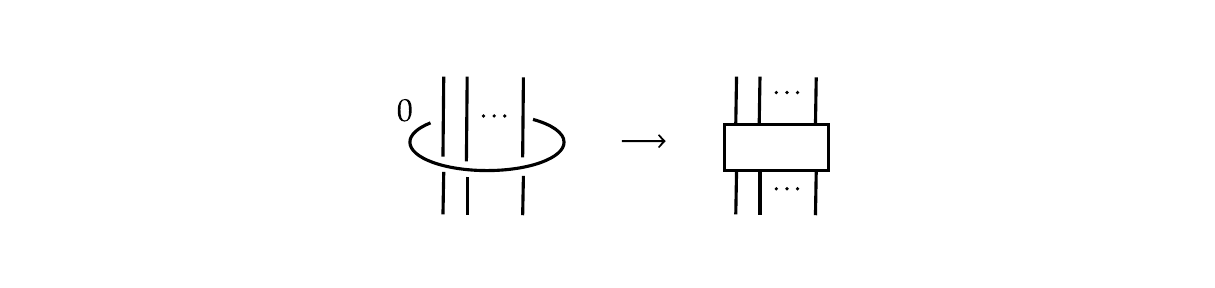}
\caption{Left: Diagram of an admissible link near a surgery region. Right: A Rozansky projector.}
\label{fig:admissibility}
\end{figure}

The \textit{unrenormalized Rozansky-Willis homology} associates to each admissible link $L\subset\#^m(S^1\times S^2)$ a bigraded vector space $KhR_2^+(L)$, which is defined by inserting Rozansky projectors in the link diagram of $L$ at each surgery region (shown pictorially in Figure~\ref{fig:admissibility}), and evaluating using the Khovanov-Rozansky $\gl_2$ homology functor. See Section~\ref{sec:projectors} for a discussion of Rozansky projectors. By our $\gl_2$ convention, $KhR_2^+(L)$ is related to $Kh(L)$ in \cite{willis2021khovanov} by
$$KhR_2^{+,h,q}(L)=(Kh^{-h,q+w(L)}(L))^*.$$
When $m=0$, this recovers the usual Khovanov-Rozansky $\gl_2$ homology $KhR_2$ for links in $S^3$.

In the rest of this paper, we will work with the \textit{renormalized Rozansky-Willis homology}, defined as $$\widetilde{KhR}_2^+(L)=(tq^{-1})^{w(L)/2}KhR_2^+(L),$$ where $t,q$ denote homological, quantum degree shifts, respectively. Note that its gradings take values in half-integers when $w(L)$ is odd.\footnote{Technically, an integral, or at least a mod $2$ homological grading is required for applying the Koszul sign convention in homological algebra. For this purpose, it is better to regard $\widetilde{KhR}_2^+(L)$ as $\tfrac12\Z_h\oplus\tfrac12\Z_q\oplus(\Z/2)_k$-graded, where $h,q$ are homological, quantum gradings taking values in half integers as defined above, and $k$ is the Koszul grading, defined as the mod $2$ homological grading of $KhR_2(L)$, which controls the Koszul sign convention.} In \cite{rozansky2010categorification,willis2021khovanov}, the isomorphism type of $\widetilde{KhR}_2^+(L)$ is shown to be an invariant of $L$ up to isotopy in $\#^m(S^1\times S^2)$. Since every framed oriented link with $2$-divisible homology class can be isotoped to be admissible, this gives an invariant up to isomorphism of isotopy classes of framed oriented links in $\#^m(S^1\times S^2)$ with $2$-divisible homology class.
For $m>0$, $\widetilde{KhR}_2^+$ is not known to be functorial with respect to link cobordisms in $I\times\#^m(S^1\times S^2)$. We will address its functoriality in Section~\ref{sec:concrete_functoriality}.

\begin{Rmk}\label{rmk:2_divisible}
\begin{enumerate}
\item The unrenormalized Rozansky-Willis homology $KhR_2^+(L)$ is only an invariant of $L$ up to overall grading shifts by multiples of $t^2q^{-2}$, unless $L$ is null-homologous.
\item As in \cite[Remark~5.5]{willis2021khovanov}, we could have removed the $2$-divisibility condition on links by declaring the Rozansky projector $P_{\ell,0}^\vee$ to be zero when $\ell$ is odd, so that $\widetilde{KhR}_2^+(L)=0$ for links that violate the $2$-divisibility condition. This is the correct definition in view of Theorem~\ref{thm:SZ}; see also Section~\ref{sec:gl2_projectors} and \cite[Section~5.2]{willis2021khovanov}. We impose the $2$-divisibility condition only for notational ease.
\end{enumerate}
\end{Rmk}

The homology group $\widetilde{KhR}_2^+(L)$ is usually infinite dimensional when $m>0$, but its homological degree is bounded from below. In fact, the space $\widetilde{KhR}_2^{+,\le h,*}(L)$ is finite-dimensional for each finite $h$. We formally define $$\widetilde{KhR}_2^-(L):=(\widetilde{KhR}_2^+(\bar L))^*,$$ where $\bar L$ denotes the mirror image of $L$, defined as the image of $L$ under an orientation-reversing involution $\iota$ on $\#^m(S^1\times S^2)$ that preserves the decomposition $\#^m(S^1\times S^2)=M\cup N$ and acts on $M$ by inverting the $z$-coordinate in $\R^3$ (note that after readjusting the position of $U_m$, $\iota$ respects admissibility of links). The star $*$ denotes the dual as graded vector spaces. Thus $\widetilde{KhR}_2^-(L)$ has homological degree bounded from above. When $m=0$, namely for admissible link in $S^3$, $\widetilde{KhR}_2^\pm$ agree.

Finally, a finite family of framed oriented links $L_i\subset\#^{m_i}(S^1\times S^2)$ with $2$-divisible homology classes ($i=1,\cdots,k$) determines a framed oriented link $L\subset\sqcup_{i=1}^k\#^{m_i}(S^1\times S^2)$ with $2$-divisible homology class. The link $L$ is \textit{admissible} if each $L_i$ is admissible, in which case we formally define $$\widetilde{KhR}_2^\pm(L):=\otimes_{i=1}^k\widetilde{KhR}_2^\pm(L_i).$$

\subsubsection{Properties of the Rozansky projector}\label{sec:projectors}
We collect some properties enjoyed by the Rozansky projectors that will be useful for us. This section is not required for reading Section~\ref{sec:outline}. We work with the Bar-Natan formalism, and this is the source of many sign ambiguities we will encounter later. We refer to Appendix~\ref{sec:gl2_projectors} for a construction of Rozansky projectors with $\gl_2$ webs and foams to remove the sign ambiguity.

Fix a nonnegative even integer $\ell$. The Rozansky projector on $\ell$ strands, denoted $P_{\ell,0}^\vee$, is a certain infinite complex in the Bar-Natan category of the disk with $2\ell$ endpoints, bounded from below. One way to construct $P_{\ell,0}^\vee$ is to take the family of maps $$\{1_\ell\to q^{\ell/2}\delta\otimes\delta^t\},$$ apply the cobar construction of Hogancamp \cite[Section~3]{hogancamp2020constructing}, and simplify by delooping the circle components. 
Here, $1_\ell$ is the identity tangle on $\ell$ strands, $\delta$ runs over all crossingless matchings of the $\ell$ points on the top, $\delta^t$ is the vertical reflection of $\delta$, $\otimes$ denotes vertical composition of tangles going from top to bottom, and each map $1_\ell\to q^{\ell/2}\delta\otimes\delta^t$ is given by the $\ell/2$ natural saddles. See Figure~\ref{fig:TL_family} for an illustration for $\ell=4$.

\begin{figure}
	\centering
    \includegraphics[width=0.7\linewidth]{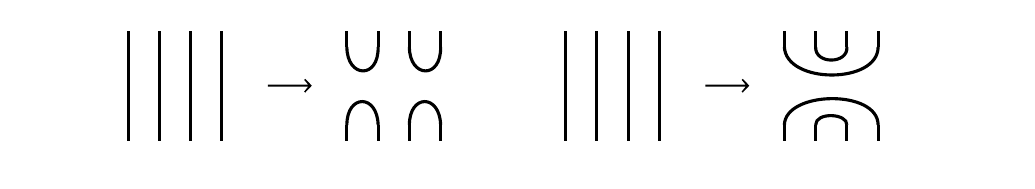}
	\caption{The family of maps that constructs $P_{4,0}^\vee$ via the cobar construction in \cite{hogancamp2020constructing} (grading shifts suppressed).}
	\label{fig:TL_family}
\end{figure}

The precise definition will not play a role in our paper. Instead, we make a list of properties enjoyed by the Rozansky projectors $P_{\ell,0}^\vee$. These will be justified in Section~\ref{sec:gl2_projectors} in the $\gl_2$ webs and foams setup.

\begin{Prop}\label{prop:projector_properties}
The Rozansky projectors $P_{\ell,0}^\vee$ satisfy the following properties.
\begin{enumerate}
\item\label{item:2.6.1} Each $P_{\ell,0}^\vee$ is a chain complex, where each term is a direct sum of $(\ell,\ell)$ Temperley-Lieb diagrams of through-degree zero, i.e. composites of the form , $\alpha\otimes\beta^t$, where $\alpha$ is an $(\ell,0)$ diagram and $\beta^t$ a $(0,\ell)$ diagram. Differentials between different terms are linear combinations of dotted cobordisms of through-degree zero, i.e. of the form $\sum(f\otimes g^t\colon\alpha\otimes\beta^t\to\alpha'\otimes\beta'^t)$ for some dotted cobordisms $f\colon\alpha\to\alpha',g\colon\beta\to\beta'$. Informally, this means that the projectors have an ``empty region in the middle.''
\item\label{item:2.6.2} $P_{\ell,0}^\vee$ is unital: it comes with a chain map $\iota_\ell\colon1_\ell\to P_{\ell,0}^\vee$, called the unit map.
\item\label{item:2.6.3} The Rozansky projector on $0$ strands is the empty diagram. The unit map $\iota_0\colon1_0\to P_{0,0}^\vee$ is the identity map.
\item\label{item:2.6.4} The maps $\iota_\ell\otimes\mathrm{id},\mathrm{id}\otimes\iota_\ell\colon P_{\ell,0}^\vee\to P_{\ell,0}^\vee\otimes P_{\ell,0}^\vee$ are chain homotopic.
\item\label{item:2.6.5} If $T$ is an $(\ell_0,\ell_1)$-tangle, then $P_{\ell_0,0}^\vee\otimes T\xrightarrow{\mathrm{id}\otimes\mathrm{id}\otimes\iota_{\ell_1}}P_{\ell_0,0}^\vee\otimes T\otimes P_{\ell_1,0}^\vee$ and $T\otimes P_{\ell_1,0}^\vee\xrightarrow{\iota_{\ell_0,0}\otimes\mathrm{id}\otimes\mathrm{id}}P_{\ell_0,0}^\vee\otimes T\otimes P_{\ell_1,0}^\vee$ are chain homotopy equivalences. In particular, $\iota_\ell\otimes\mathrm{id}$ and $\mathrm{id}\otimes\iota_\ell$ in (4) are chain homotopy equivalences.
\item\label{item:2.6.6} The unit map \raisebox{-25pt}{\includegraphics[width=0.23\linewidth]{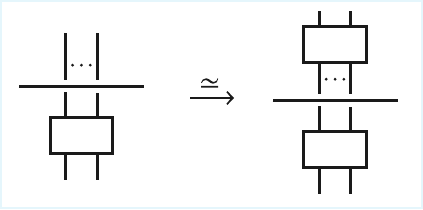}} (as well as its mirrored version) is a chain homotopy equivalence. 

\item\label{item:2.6.7} The unit maps in \raisebox{-15pt}{\includegraphics[width=0.45\linewidth]{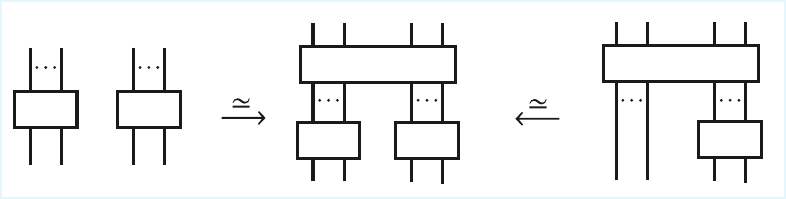}} are chain homotopy equivalences.
\item\label{item:2.6.8} $P_{\ell,0}^\vee$ is symmetric: there is a canonical chain homotopy equivalence between $P_{\ell,0}^\vee$ and the planar rotation of $P_{\ell,0}^\vee$ by $\pi$ that makes \raisebox{-30pt}{\includegraphics[width=0.45\linewidth]{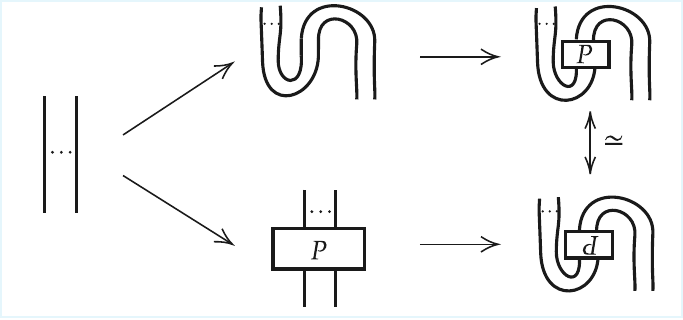}} commute up to homotopy.
\end{enumerate}
\end{Prop}

\subsection{A lasagna interpretation of Rozansky-Willis homology}\label{sec:SZ}
For our purposes, it will be most convenient to consider links in a disjoint union of connected sums of $S^1\times S^2$'s, although this might make the notation a bit awkward. We will start doing this in this section. The readers are welcome to take $k=1$ in the rest of the section.

For $k\ge0$, $m_1,\cdots,m_k\ge0$, the manifold $\sqcup_{i=1}^k\#^{m_i}(S^1\times S^2)$ is naturally the boundary of $D_{std}:=\#_{i=1}^k\natural^{m_i}(D^2\times S^2)$. Here, $D_{std}$ is thought of as equipped with a standard handle decomposition with one $0$-handle, $\sum_{i=1}^km_i$ $0$-framed $2$-handles, $k$ $3$-handles, and one $4$-handle, so that the $3$-handles are attached along $S^2(1)+(10i,0,0)\subset\R^3\subset S^3=\partial B^4$, $i=1,2,\cdots,k$, the $2$-handles in the $i$-th connected summand are attached inside $B^3(1/2)+(10i,0,0)$, along an unlink $U_{m_i}\subset S^3_i$ in standard position. Here, $S_i^3$ denotes the $S^3$ boundary component of $B^4\cup\text{($3$-handles)}$ containing $(10i,0,0)$, whose $\infty$ point is the inner belt point of the corresponding $3$-handle. In this description, $S^2(R)$ and $B^3(R)$ are the $2$-sphere and $3$-ball of radius $R$ centered at $0\in\R^3$, respectively. See Figure~\ref{fig:D_std}.

\begin{figure}
\centering
\includegraphics[width=0.8\linewidth]{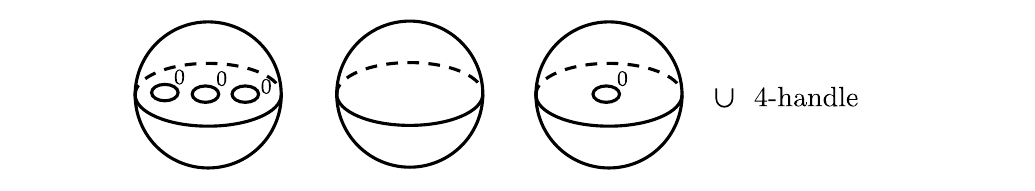}
\caption{The Kirby diagram for the concrete $\#_{i=1}^3\natural^{m_i}(D^2\times S^2)$ with $3$-handles drawn, $(m_1,m_2,m_3)=(3,0,1)$. The $4$-handle is the one touching the outer belt point of each $3$-handle.}
\label{fig:D_std}
\end{figure}

An admissible link $L\subset\partial D_{std}$ determines a canonical class $\alpha_L\in H_2^L(D_{std})$, characterized by having trivial algebraic intersections with the $2$-cocores of $D_{std}$. In \cite{sullivan2024kirby}, Sullivan--Zhang showed that one can recover Rozansky-Willis homology from Khovanov skein lasagna modules in the following sense.

\begin{Thm}[{\cite[Remark~1.6]{sullivan2024kirby}}]\label{thm:SZ}
For every admissible link $L\subset\partial D_{std}=\sqcup_{i=1}^k\#^{m_i}(S^1\times S^2)$, there is a canonical isomorphism of vector spaces
\begin{equation}\label{eq:SZ}
\mathcal S_0^2(D_{std};L)\cong\widetilde{KhR}_2^+(L)\otimes\mathcal S_0^2(D_{std})\tag{SZ}
\end{equation}
Moreover, in each skein grading $\alpha\in H_2^L(D_{std})$, the isomorphism in the forward direction is homogeneous with tridegree shift $(\alpha^2/2,-\alpha^2/2,-\alpha_L)$.

If $L\subset\partial D_{std}$ does not have $2$-divisible homology class, then $$\mathcal S_0^2(D_{std};L)=0.$$
\end{Thm}

The intersection pairing on $H_2^L(D_{std})$ appearing in the grading shifts above is defined by using the framing of $L$ as the boundary condition. The term $\widetilde{KhR}_2^+(L)$ in \eqref{eq:SZ} is set to concentrate in skein grading $0$.

As it will be of importance, in the rest of this section, we examine the isomorphism \eqref{eq:SZ} in more detail.

Write $m=\sum_{i=1}^km_i$. For $n\in\Z^m$, let $n_+,n_-,|n|\in\Z^m$ be defined by $(n_+)_j=\max(n_j,0)$, $(n_-)_j=\max(-n_j,0)$, $|n|_j=|n_j|$. Write $||n||=\sum_{j=1}^m|n|_j$.

Fix a homology class $\alpha\in H_2^L(D_{std})$. Write $\alpha=\alpha_L+n$ for some $n\in\Z^m\cong H_2(D_{std})$. At the skein class $\alpha$, the isomorphism \eqref{eq:SZ} is the composition of the following sequence of isomorphisms.

\begin{align}
&\mathcal S_0^2(D_{std};L;\alpha)\nonumber\\
\xleftarrow{\cong}&\,\mathrm{colim}_{r\to\infty}q^{-||n||-2||r||}KhR_2(L\cup(n_++r,n_-+r)\text{ belts})^{S_{|n|+2|r|}}\label{eq:SZ_MN}\tag{SZ1}\\
=&\,\mathrm{colim}_{r\to\infty}q^{-||n||-2||r||}(t^{-1}q)^{\alpha^2/2}\widetilde{KhR}_2^+(L\cup(n_++r,n_-+r)\text{ belts})^{S_{|n|+2|r|}}\label{eq:SZ_renom}\tag{SZ2}\\
\xrightarrow{\cong}&\,(t^{-1}q)^{\alpha^2/2}\mathrm{colim}_{r\to\infty}q^{-||n||-2||r||}\widetilde{KhR}_2^+(L^\circ\otimes(\otimes_{j=1}^mP_{\ell_j,0}^\vee)\cup(n_++r,n_-+r)\text{ belts})^{S_{|n|+2|r|}}\label{eq:SZ_res_1}\tag{SZ3}\\
\xrightarrow{\cong}&\,(t^{-1}q)^{\alpha^2/2}\mathrm{colim}_{r\to\infty}q^{-||n||-2||r||}\widetilde{KhR}_2^+(L^\circ\otimes(\otimes_{j=1}^mP_{\ell_j,0}^\vee)\sqcup U^{n_++r,n_-+r})^{S_{|n|+2|r|}}\label{eq:SZ_slide}\tag{SZ4}\\
\xrightarrow{\cong}&\,(t^{-1}q)^{\alpha^2/2}\widetilde{KhR}_2^+(L)\otimes\mathcal S_0^2(D_{std};\alpha-\alpha_L).\label{eq:SZ_simp}\tag{SZ5}
\end{align}

The individual steps are explained as follows.

\eqref{eq:SZ_MN}: This follows from the $2$-handlebody formula and the neck-cutting operation of Manolescu--Neithalath \cite[Proposition~3.8,Lemma~7.2]{manolescu2022skein}\footnote{The $Kh$ in Proposition~3.8 of \cite{manolescu2022skein} should be $KhR_2$ instead.}, formulated in terms of a filtered colimit as in \cite{hogancamp2025kirby}. In this formula, $L$ is regarded as a link in $\sqcup_{i=1}^kS^3$, and the belts consist of $(n_+)_j+r_j$ parallel copies of positively oriented $j$-th component of $\sqcup_{i=1}^kU_{m_i}$, and $(n_-)_j+r_j$ negatively oriented ones. The product symmetric group $S_{|n|+2|r|}=\prod_jS_{|n|_j+2|r|_j}$ acts on $KhR_2(L\cup(n_++r,n_-+r)\text{ belts})$ by permuting the belt circles (this is well-defined by Grigsby--Licata--Wehrli \cite{grigsby2018annular}). The colimit is taken over $r\in\Z_{\ge0}^m$ along $r\to r+e_j$, $j=1,\cdots,m$, where $e_j\in\Z^m$ is the $j$-th coordinate vector, and the corresponding maps on $KhR_2$ are given by symmetrized dotted annulus cobordism maps. Here, $KhR_2$ of a link in $\sqcup_{i=1}^kS^3$ is defined as the tensor product of $KhR_2$ of its components.

Explicitly, the isomorphism is given as follows. Let $v\in KhR_2(L\cup(n_++r,n_-+r)\text{ belts})^{S_{|n|+2|r|}}$. Then the class in the right hand side represented by $v$ is sent to the class in the left hand side represented by the lasagna filling $((I\times L) \cup(n_++r,n_-+r)\text{ cores},v)$, 
where the skein has $k$ input balls, the $i$-th of which is given by a slight shrunken collar neighborhood of $B^3(1)+(10i,0,0)\subset S^3=\partial B^4$ in the $0$-handle $B^4$, and the cores refer to parallel copies of cores of $2$-handles, slightly extended to the interior of $B^4$. The inverse of this isomorphism is given as follows. Any lasagna filling representing an element in the left hand side is equivalent to one of the form $((I\times L)\cup(n_++r,n_-+r)\text{ cores},v)$ by neck-cutting \cite[Lemma~7.2]{manolescu2022skein} along $3$-spheres given by cores of $3$-handles union boundary-parallel copies of $B^3(1)+(10i,0,0)\subset\partial B^4\subset B^4$, evaluating in the $B^4$ region containing the $4$-handle, and applying general position as in \cite{manolescu2022skein}. The class represented by such a lasagna filling is sent to the class on the right hand side represented by $\mathrm{Sym}(v)\in KhR_2(L\cup(n_++r,n_-+r)\text{ belts})^{S_{|n|+2|r|}}$, the symmetrization of $v$.

\eqref{eq:SZ_renom}: This is the renormalization by the writhe of $L\cup(n_++r,n_-+r)\text{ belts}$, which is equal to $\alpha^2$.

\eqref{eq:SZ_res_1}: Here, $L^\circ$ denotes $L$ regarded in $\sqcup_{i=1}^kS^3$, but with a neighborhood of the surgery regions removed, and $\ell_j$ denotes the geometric intersection number between $L$ and the disk in $\R^2\times\{0\}$ bounded by the $j$-th component of $\sqcup_{i=1}^kU_{m_i}$. The second $\otimes$ symbol indicates disjoint union, while the first represents insertion of each Rozansky projector $P_{\ell_j,0}^\vee$ at the deleted neighborhood of the $j$-th surgery region, placed immediately below the $j$-th collection of belts, as shown on the left of Figure~\ref{fig:SZ_slide}. The map is induced by the unit maps $1_{\ell_j}\to P_{\ell_j,0}^\vee$, and is shown to be an isomorphism by Sullivan--Zhang \cite{sullivan2024kirby}.

\begin{figure}
\centering
\includegraphics[width=0.7\linewidth]{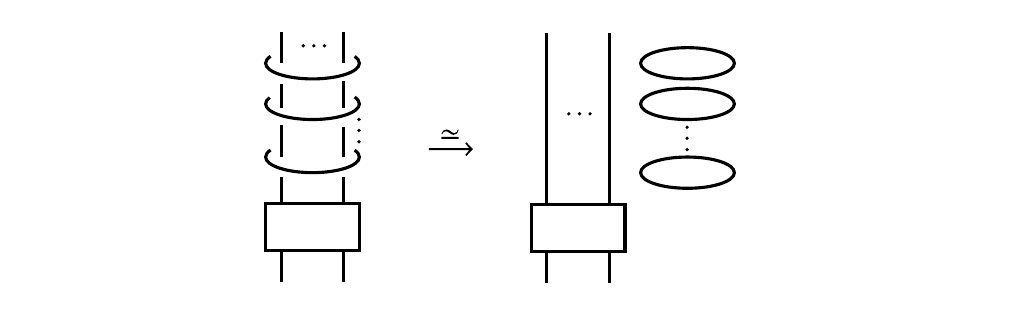}
\caption{The belt-slide isomorphism \eqref{eq:SZ_slide}, shown near one surgery region.}
\label{fig:SZ_slide}
\end{figure}

\eqref{eq:SZ_slide}: This follows from \cite[Lemma~3.12]{willis2021khovanov}. We sketch the proof. Each $P_{\ell_j,0}^\vee$ is a chain complex whose terms are direct sums of through-degree $0$ Temperley-Lieb diagrams with some quantum degree shifts. One can use simplifying Reidemeister II moves to slide the belts off each term in the complex through the ``empty region.'' These sliding-off isomorphisms (each well-defined up to sign) patch together to an isomorphism that slides the belts off the projectors, say to the right of the strands as shown in Figure~\ref{fig:SZ_slide}, by some compatible sign choices. Here $U^{a,b}$ denotes the $(a+b,0)$-cable of (a slightly shifted copy of) $U_m=\sqcup_{i=1}^kU_{m_i}$, with orientation on $b$ strands reversed, $a,b\in\Z_{\ge0}^m$. This isomorphism has an overall sign indeterminacy. To fix the sign, see Appendix~\ref{sec:sign_slide}.

\eqref{eq:SZ_simp}: By monoidality of the link homology, split disjoint unions give rise to tensor products. By definition, the first tensorial factor becomes $\widetilde{KhR}_2^+(L)$, the Rozansky-Willis homology of the admissible link $L\subset\partial D_{std}$. By the $2$-handlebody formula \cite[Proposition~3.8]{manolescu2022skein}, the second tensorial factor (under the colimit) becomes $\mathcal S_0^2(\natural^m(D^2\times S^2);\alpha-\alpha_L)$, which is isomorphic to $\mathcal S_0^2(D_{std};\alpha-\alpha_L)$ via the map induced by $3,4$-handle attachments.

\section{Outline}\label{sec:outline}
\subsection{The input homology}\label{sec:input_homology}
A \textit{$4$-dimensional $1$-handlebody} is an oriented $4$-manifold admitting a handle decomposition with only $0,1$-handles. A \textit{$4$-dimensional relative $1$-handlebody complement} $W$ is an oriented $4$-dimensional cobordism, i.e. an oriented $4$-manifold whose boundary comes with a partition $\partial W=(-\partial_-W)\sqcup\partial_+W$, such that $W\cong_\varphi X_1\backslash int(X_0)$ for some $4$-dimensional $1$-handlebodies $X_0,X_1$ with $X_0\subset int(X_1)$, so that the diffeomorphism $\varphi$ (usually dropped from the notation) restricts to $\partial_-W\cong\partial X_0$, $\partial_+W\cong\partial X_1$.

If $W$ is a $4$-dimensional relative $1$-handlebody complement, then $X_0,X_1$ are uniquely determined by $W$ up to diffeomorphisms. If $W_0,W_1$ are $4$-dimensional relative $1$-handlebody complements and $\phi\colon\partial_+W_0\to\partial_-W_1$ is an orientation-preserving diffeomorphism, then $W_0\cup_\phi W_1$ is also a $4$-dimensional relative $1$-handlebody complement; moreover, $W_i\cong X_{i+1}\backslash int(X_i)$, $i=0,1$, and $W_0\cup_\phi W_1\cong X_2\backslash int(X_0)$, for some $4$-dimensional $1$-handlebodies $X_0\subset X_1\subset X_2$.

If $M$ is an oriented manifold, an \textit{abstract $M$} is an oriented manifold $X$ that is diffeomorphic to $M$ with orientation. Thus, a $4$-dimensional $1$-handlebody is the same as an abstract $\sqcup_{i=1}^k\natural^{m_i}(S^1\times B^3)$ for some $k\ge0$, $m_1,\cdots,m_k\ge0$. The boundary of a $4$-dimensional $1$-handlebody, as well as the $\pm$-boundary of a $4$-dimensional relative $1$-handlebody complement, is some abstract $\sqcup_{i=1}^k\natural^{m_i}(S^1\times S^2)$.

\begin{Def}\label{def:Links_1}
Let $\mathbf{Links}_1$ denote the following category:
\begin{itemize}
\item Objects: $(S,L)$, where $S$ is an abstract $\sqcup_{i=1}^k\#^{m_i}(S^1\times S^2)$ for some $k\ge0$, $m_1,\cdots,m_k\ge0$, and $L\subset S$ is a framed oriented link with $2$-divisible homology class.
\item Morphisms\footnote{As usual, cobordisms are assumed to have standard collars near the boundaries, so that gluings are well-defined.}: $\mathrm{Hom}((S_0,L_0),(S_1,L_1))=\{(W,\Sigma)\}/\sim$, where $W$ is a $4$-dimensional relative $1$-handlebody complement with $\partial_-W=S_0$, $\partial_+W=S_1$, $\Sigma$ is a cobordism between $L_0,L_1$ in $W$, i.e. a properly embedded framed oriented surface with $(-\partial\Sigma)\cap S_0=L_0$, $\partial\Sigma\cap S_1=L_1$ (with framings and orientations). The equivalence relation $\sim$ is given by $(W_1,\Sigma_1)\sim(W_2,\Sigma_2)$ if there exists a diffeomorphism $W_1\cong W_2$ rel boundary that sends $\Sigma_1$ to $\Sigma_2$.
\item Identity morphism at $(S,L)$: The morphism represented by $(I\times S,I\times L)$, where $\partial_+(I\times S)=\{1\}\times S=S$ and $\partial_-(I\times S)=\{0\}\times S=S$ are the natural identifications.
\item Composition of morphisms: Taking union along the common boundary.
\end{itemize}
\end{Def}

The category $\mathbf{Links}_1$ is symmetric monoidal under disjoint union of ambient manifolds.

Let $f\mathbf{Vect}_\Q^{\Z\times\Z}$ be the category of bigraded $\Q$-vector spaces, finite-dimensional in every bidegree, and homogeneous maps between them. The input homology for our skein lasagna modules will be supplied by the next theorem, which is the main theorem of this paper.

\begin{Thm}\label{thm:KhR_2-}
There is a symmetric monoidal functor $\widetilde{KhR}_2^-\colon\mathbf{Links}_1\to f\mathbf{Vect}_\Q^{\Z\times\Z}$. If $L\subset\sqcup_{i=1}^k\#^{m_i}(S^1\times S^2)$ is an admissible link, then $\widetilde{KhR}_2^-(\sqcup_{i=1}^k\#^{m_i}(S^1\times S^2),L)$ is canonically isomorphic to $\widetilde{KhR}_2^-(L)$, the Rozansky-Willis homology of $L$. For a morphism $(W,\Sigma)$, the linear map $\widetilde{KhR}_2^-(W,\Sigma)$ is homogeneous of degree $(0,-\chi(\Sigma))$.
\end{Thm}
For $S$ an abstract $\sqcup_{i=1}^k\#^{m_i}(S^1\times S^2)$ and $L\subset S$ a framed oriented link with $2$-divisible homology class, $\widetilde{KhR}_2^-(S,L)$ will be called the \textit{Rozansky-Willis homology} of $(S,L)$, or more succinctly, of $L$.

\begin{Rmk}
Paralleling \cite[Section~5.1]{morrison2022invariants}, one could formulate Theorem~\ref{thm:KhR_2-} by saying that $\widetilde{KhR}_2^-$ defines an algebra over a certain $(4,1)$-lasagna colored operad.
\end{Rmk}

\subsection{Definition of Khovanov skein lasagna modules with \texorpdfstring{$1$}{1}-dimensional inputs}\label{sec:def_S02_bar}
In this section, assuming Theorem~\ref{thm:KhR_2-}, we construct the $1$-dimensional-input Khovanov skein lasagna modules and prove Theorem~\ref{thm:main}.

As in the setup of Section~\ref{sec:S02}, let $X$ be a compact oriented $4$-manifold and $L\subset\partial X$ be a framed oriented link.

In the context of the Khovanov skein lasagna module with $1$-dimensional inputs, a \textit{skein (with $1$-dimensional inputs)} in $X$ rel $L$ is a properly embedded framed oriented surface $\Sigma\subset X\backslash int(B)$ with $\partial\Sigma\cap\partial X=L$ (with framing and orientation), where $B\subset int(X)$ is a $4$-dimensional $1$-handlebody, called the \textit{input manifold} of $\Sigma$. The \textit{input link} of $\Sigma$ is the framed oriented link $(-\partial\Sigma)\cap\partial B$ in $\partial B$, which is required to have a $2$-divisible homology class. A \textit{lasagna filling (with $1$-dimensional inputs)} of $(X,L)$ is a pair $(\Sigma,v)$ where $\Sigma$ is a skein with some input manifold $B$, and $v\in\widetilde{KhR}_2^-(\partial B,(-\partial\Sigma)\cap\partial B)$ is an element in the Rozansky-Willis homology of the input link $(-\partial\Sigma)\cap\partial B$.

\begin{Def}\label{def:S02_bar}
The \textit{$1$-dimensional-input Khovanov skein lasagna module} of $(X,L)$ is the $\Q$-vector space (graded as indicated below)
$$\bar{\mathcal S}_0^2(X;L):=\Q\{\text{lasagna fillings (with $1$-dimensional inputs) of }(X,L)\}/\sim,$$ where $\sim$ is the equivalence relation generated by
\begin{itemize}
	\item Isotopy of the skein rel boundary;
	\item Linearity in the decoration: $(\Sigma,v)+\lambda(\Sigma,w)\sim(\Sigma,v+\lambda w)$, $\lambda\in\Q$;
	\item Enclosement relation: Let $(\Sigma,v)$ be a lasagna filling with input manifold $B$, and let $B'\subset int(X)$ be a $4$-dimensional $1$-handlebody that contains $B$ in its interior. Then $(\Sigma,v)\sim(\Sigma\backslash int(B'),\widetilde{KhR}_2^-(B'\backslash int(B),\Sigma\cap(B'\backslash int(B)))(v))$, where $\Sigma\backslash int(B')$ is regarded as a skein with input manifold $B'$.
\end{itemize}
The \textit{$1$-dimensional-input Khovanov skein lasagna module} of $X$ is $\bar{\mathcal S}_0^2(X):=\bar{\mathcal S}_0^2(X;\emptyset)$.

If $(\Sigma,v)$ is a lasagna filling of $(X,L)$ where $v$ is homogeneous of degree $(h,q)$, then the \textit{tridegree} of $(\Sigma,v)$ is defined to be $(h,q-\chi(\Sigma),[\Sigma])\in(\tfrac12\Z)^2\times H_2^L(X;\Z/2)$, where $H_2^L(X;\Z/2)$ is the preimage of $[L]$ under the connecting homomorphism $H_2(X,L;\Z/2)\to H_1(L;\Z/2)$, which is an $H_2(X;\Z/2)$-torsor. This descends to a trigrading on $\bar{\mathcal S}_0^2(X;L)$. The three gradings, usually denoted by $(h,q,\alpha)$, are called \textit{homological grading}, \textit{quantum grading}, \textit{skein grading}, respectively.
\end{Def}

In the case where $X$ is a $4$-dimensional $1$-handlebody, let $X_0$ be the complement of a collar neighborhood of $\partial X$. For every $L\subset\partial X$ with $2$-divisible homology class, we have an isomorphism of bigraded vector spaces $$\widetilde{KhR}_2^-(\partial X,L)\xrightarrow{\cong}\bar{\mathcal S}_0^2(X;L),\ v\mapsto[(I\times L,v)],$$ where $I\times L$ denotes a standard product skein with input manifold $X_0$. Note that in this case the skein grading on $\bar{\mathcal S}_0^2(X;L)$ is trivial.

\begin{Rmk}\label{rmk:bounding}
\begin{enumerate}
\item In the definition of the $1$-dimensional-input Khovanov skein lasagna modules, we could demand that all input links are null-homologous and take the input homology to be the restriction of $\widetilde{KhR}_2^-$ to the full subcategory of null-homologous links in $\mathbf{Links}_1$. We remark that it would be more natural to use the unrenormalized Rozansky-Willis homology $KhR_2^-$ as input in this case. One obtains a different theory of $1$-dimensional-input Khovanov skein lasagna modules, denoted $\bar{\mathcal S}_0^{2,O}$, for which the grading is by $\Z^2\times H_2^L(X)$. This admits a grading-preserving forgetful map $\mathcal S_0^2(X;L)\to\bar{\mathcal S}_0^{2,O}(X;L)$ for any pair $(X,L)$, which is surjective when $X$ is simply-connected. See also Remark~\ref{rmk:2_divisible}(1).
\item On the other end, we could have removed the $2$-divisibility constraint on input links by declaring $\widetilde{KhR}_2^-(S,L)=0$ for $L$ violating this condition. This will lead to exactly the same theory of $1$-dimensional-input skein lasagna modules. See also Remark~\ref{rmk:2_divisible}(2).
\end{enumerate}
\end{Rmk}

\subsection{Turning cobordisms inside out}\label{sec:upside_down}
The main idea of supplying the cobordism data for Theorem~\ref{thm:KhR_2-} is to turn the $1$-handlebodies inside out and examine Theorem~\ref{thm:SZ}. We will prove the following dual result, which is the dual main theorem of this paper.

\begin{Thm}\label{thm:KhR_2+}
There is a symmetric monoidal functor $\widetilde{KhR}_2^+\colon\mathbf{Links}_1^{op}\to f\mathbf{Vect}_\Q^{\Z\times\Z}$. For an admissible link $L\subset\#^m(S^1\times S^2)$ the vector space $\widetilde{KhR}_2^+(\#^m(S^1\times S^2),L)$ is canonically isomorphic to $\widetilde{KhR}_2^+(L)$. For a morphism $(W,\Sigma)$, the linear map $\widetilde{KhR}_2^+(W,\Sigma)$ is homogeneous of degree $(0,-\chi(\Sigma))$.
\end{Thm}
Theorem~\ref{thm:KhR_2-} is recovered from Theorem~\ref{thm:KhR_2+} by setting $\widetilde{KhR}_2^-:=(*)\circ\widetilde{KhR}_2^+\circ(-)$, where $(-)$ is the autoequivalence of $\mathbf{Links}_1$ defined by reversing the ambient orientation, and $(*)\colon (f\mathbf{Vect}_\Q^{\Z\times\Z})^{op}\to f\mathbf{Vect}_\Q^{\Z\times\Z}$ is the dualization functor. Explicitly, this means that $\widetilde{KhR}_2^-(S,L)=\widetilde{KhR}_2^+(-S,L)^*$ for objects, and $\widetilde{KhR}_2^-(W,\Sigma)=\widetilde{KhR}_2^+(-W,\Sigma)^*$ for morphisms.

The rest of this paper is devoted to proving the dual main Theorem~\ref{thm:KhR_2+}.\medskip

The classical Rozansky-Willis homology, introduced in Section~\ref{sec:RW}, is only defined for admissible links in concretely parametrized $\#^m(S^1\times S^2)$. Thus, the functor $\widetilde{KhR}_2^+$ described in Theorem~\ref{thm:KhR_2+} is not even defined on objects. Nevertheless, in the rest of the section, we ignore parametrization issues, and sketch how the cobordism maps for $\widetilde{KhR}_2^+$ are defined.

Suppose $(W,\Sigma)\colon(S_0,L_0)\to(S_1,L_1)$ is a morphism in $\mathbf{Links}_1$. Turning $(W,\Sigma)$ upside down and reversing the orientations, we obtain the transpose cobordism $(W^t,\Sigma^t)\colon(S_1,L_1)\to(S_0,L_0)$.

Write $W=X_1\backslash int(X_0)$ for $4$-dimensional $1$-handlebodies $X_0,X_1$. Choose an orientation-preserving embedding $X_1\hookrightarrow S^4$. Taking the complement and reversing the orientation, we get an embedding $-(S^4\backslash int(X_1))\hookrightarrow-(S^4\backslash int(X_0))$, where each $-(S^4\backslash int(X_j))$ is some abstract $\#_{i=1}^k\natural^{m_i}(D^2\times S^2)$.

By construction of the skein lasagna modules, we have a gluing map
\begin{equation}\label{eq:upside_down_las}
\mathcal S_0^2(W^t;\Sigma^t)\colon\mathcal S_0^2(-(S^4\backslash int(X_1));L_1)\to\mathcal S_0^2(-(S^4\backslash int(X_0));L_0).
\end{equation}
Upon fixing parametrizations $-(S^4\backslash int(X_j))\cong\#_{i=1}^{k^{(j)}}\natural^{m_i^{(j)}}(D^2\times S^2)$ making $L_j$ admissible on their boundaries, $j=0,1$, and using the isomorphism \eqref{eq:SZ}, \eqref{eq:upside_down_las} can be regarded as a map
\begin{equation}\label{eq:upside_down_RW}
\widetilde{KhR}_2^+(L_1)\otimes\mathcal S_0^2(\#_{i=1}^{k^{(1)}}\natural^{m_i^{(1)}}(D^2\times S^2))\to\widetilde{KhR}_2^+(L_0)\otimes\mathcal S_0^2(\#_{i=1}^{k^{(0)}}\natural^{m_i^{(0)}}(D^2\times S^2)).
\end{equation}
We will analyze the map \eqref{eq:upside_down_RW} and extract the desired morphism $\widetilde{KhR}_2^+(W,\Sigma)\colon\widetilde{KhR}_2^+(S_1,L_1)\to\widetilde{KhR}_2^+(S_0,L_0)$ as the ``first tensorial factor'' of \eqref{eq:upside_down_RW}, in an appropriate sense. As the gluing map \eqref{eq:upside_down_las} is manifestly functorial under composition, the functoriality of $\widetilde{KhR}_2^+(W,\Sigma)$ will be a formal consequence of our construction.

\subsection{Organization of the remaining sections}\label{sec:rest}
The sketch proof for Theorem~\ref{thm:KhR_2+} in Section~\ref{sec:upside_down} is rather imprecise, with two major omissions.

\begin{enumerate}[(A)]
\item We have not addressed the distinction between ``links in an abstract $\sqcup\#(S^1\times S^2)$'' and ``admissible links in the concrete $\sqcup\#(S^1\times S^2)$.''
\item The choice of the embedding $X_1\hookrightarrow S^4$ up to isotopy corresponds to a spin structure on $X_1$, which is in noncanonical one-to-one correspondence with $H^1(X_1;\Z/2)$. We have to remove this choice.
\end{enumerate}

In order to address item (A), we first need to prove the functoriality of Rozansky-Willis homology for cobordisms in concrete $\#(S^1\times S^2)$, namely Theorem~\ref{thm:concrete_functoriality_intro}. This will be carried out in Section~\ref{sec:concrete_functoriality}. For a given cobordism between admissible links, we examine the map \eqref{eq:upside_down_RW}, whose definition entails no choice in this setup. It will be sufficient to examine \eqref{eq:upside_down_RW} for a set of elementary cobordisms that generate all cobordisms.

In Section~\ref{sec:abstract_spin_homology}, we address (A), with the extra assumption that our abstract $\sqcup\#(S^1\times S^2)$ is equipped with a spin structure. This entails a study of the diffeomorphism group of $\#\natural(D^2\times S^2)$ rel boundary. After proving Theorem~\ref{thm:diff_D2S2_rel_boundary_intro} using Gabai's $4$-dimensional lightbulb theorem \cite{gabai20204}, we reduce the problem mainly to understanding the action on $\mathcal S_0^2(\#\natural(D^2\times S^2);L)$ by barbell diffeomorphisms defined by Budney--Gabai \cite{budney2019knotted}, for admissible links $L\subset\sqcup\#(S^1\times S^2)$.

In Section~\ref{sec:abstract_spin_functoriality}, we carry out the heuristics in Section~\ref{sec:upside_down} carefully for morphisms that come with ambient spin structures, by decomposing an arbitrary such morphism into elementary ones, and examining \eqref{eq:upside_down_RW} for each elementary morphism.

In Section~\ref{sec:nonspin}, we remove the spin assumption and address (B). This requires a construction of induced maps on $\mathcal S_0^2$ by Gluck twists on framed embedded $2$-spheres in $4$-manifolds. Such induced maps will be isomorphisms that transform the grading in a specified way, proving Theorem~\ref{thm:gluck_intro}.

Sign ambiguities will come up at various stages of our proofs. Appendix~\ref{sec:sign} employs the $\gl_2$ webs and foams formalism to fix the signs.

\section{Functoriality in concrete \texorpdfstring{$\#(S^1\times S^2)$}{\#S1*S2}}\label{sec:concrete_functoriality}
Fix $k\ge0$, $m_1,\cdots,m_k\ge0$. Write for short $S_{std}:=\sqcup_{i=1}^k\#^{m_i}(S^1\times S^2)$.
\subsection{The statements}
Let $\mathbf{Links}_{S_{std}}$ denote the category whose objects are admissible links in $S_{std}$, and whose morphisms are link cobordisms in $I\times S_{std}$ up to isotopy rel boundary. The identity morphisms and compositions are defined in the natural way. The goal of this section is to extend Rozansky-Willis homology to a functor on $\mathbf{Links}_{S_{std}}$.

\begin{Thm}\label{thm:concrete_functoriality}
There is a functor $$\widetilde{KhR}_2^+\colon\mathbf{Links}_{S_{std}}^{op}\to f\mathbf{Vect}_\Q^{\Z\times\Z}$$ that extends the definition of $\widetilde{KhR}_2^+$ on objects as in Section~\ref{sec:RW}. For a morphism $\Sigma$, $\widetilde{KhR}_2^+(\Sigma)$ is homogeneous of degree $(0,-\chi(\Sigma))$.
\end{Thm}

For a link cobordism $\Sigma\colon L_0\to L_1$ between admissible links in $S_{std}$, let $\Sigma^t\colon L_1\to L_0$ denote its transpose, defined by turning $\Sigma$ upside down, reversing its orientation and the ambient orientation. Of course, one can obtain a covariant functor extending $\widetilde{KhR}_2^-$ relating to our contravariant functor via the transpose functor $(\cdot)^t\colon\mathbf{Links}_{S_{std}}^{op}\to\mathbf{Links}_{S_{std}}$ (which is the identity map on objects), recovering Theorem~\ref{thm:concrete_functoriality_intro}. We state everything in the contravariant way, to match Theorem~\ref{thm:KhR_2+}. This difference will only be essential from Section~\ref{sec:abstract_spin_functoriality} onward.

Write $D_{std}:=\#_{i=1}^k\natural^{m_i}(D^2\times S^2)$. Thus $\partial D_{std}=S_{std}$. The following concept plays an important role in our paper.
\begin{Def}
For an admissible link $L\subset S_{std}$ and any $\alpha\in H_2^L(D_{std})$, the \textit{lasagna quantum grading} on $\mathcal S_0^2(D_{std};L;\alpha)$ is the contribution to the quantum grading coming from the second tensorial factor in the right hand side of the bigrading-preserving isomorphism $$\mathcal S_0^2(D_{std};L;\alpha)\cong((t^{-1}q)^{\alpha^2/2}\widetilde{KhR}_2^+(L))\otimes\mathcal S_0^2(D_{std};\alpha-\alpha_L)$$ from Theorem~\ref{thm:SZ}.
\end{Def}

\begin{Thm}\label{thm:concrete_functoriality_las}
For a link cobordism $\Sigma\colon L_0\to L_1$ between admissible links in $S_{std}$, the gluing map $$\mathcal S_0^2(I\times S_{std};\Sigma^t)\colon\mathcal S_0^2(D_{std};L_1)\to\mathcal S_0^2(D_{std};L_0)$$ is nonincreasing in the lasagna quantum grading. The corresponding associated graded map $gr\mathcal S_0^2(I\times S_{std};\Sigma^t)$, under the isomorphism \eqref{eq:SZ}, is of the form 
\begin{equation}\label{eq:concrete_functoriality_las}
\widetilde{KhR}_2^+(\Sigma)\otimes gr(\mathrm{id}_{\alpha_{L_1}*[\Sigma^t]-\alpha_{L_0}})\colon\widetilde{KhR}_2^+(L_1)\otimes gr\mathcal S_0^2(D_{std})\to\widetilde{KhR}_2^+(L_0)\otimes gr\mathcal S_0^2(D_{std})
\end{equation}
for some map $\widetilde{KhR}_2^+(\Sigma)\colon\widetilde{KhR}_2^+(L_1)\to\widetilde{KhR}_2^+(L_0)$.
\end{Thm}

We recall that $\mathrm{id}_\alpha$ in \eqref{eq:concrete_functoriality_las} denotes the shifting automorphism on $\mathcal S_0^2(D_{std})$ by $\alpha$, defined in Definition~\ref{def:shift_auto}. The star $*$ denotes the natural concatenation map on homology.

The proof of Theorem~\ref{thm:concrete_functoriality_las} takes up the bulk of Section~\ref{sec:concrete_functoriality}. In Section~\ref{sec:concrete_handleslide} we decompose it into various cases, which are treated individually in Sections~\ref{sec:concrete_far}--\ref{sec:concrete_handleslide}. Before going there, we deduce Theorem~\ref{thm:concrete_functoriality} as a consequence of Theorem~\ref{thm:concrete_functoriality_las}.

\begin{proof}[Proof of Theorem~\ref{thm:concrete_functoriality} assuming Theorem~\ref{thm:concrete_functoriality_las}]
The functor $\widetilde{KhR}_2^+$ is defined on objects in Section~\ref{sec:RW}. Let $\Sigma\colon L_0\to L_1$ be a morphism. If the map \eqref{eq:concrete_functoriality_las} is nonzero, $\widetilde{KhR}_2^+(\Sigma)$ is determined uniquely; if it is zero, define $\widetilde{KhR}_2^+(\Sigma)$ to be zero. Since $\mathcal S_0^2(I\times S_{std};\Sigma^t)$ is homogeneous with bidegree shift $(0,-\chi(\Sigma))$, the same is true for $\widetilde{KhR}_2^+(\Sigma)$.

It remains to check functoriality. If $\Sigma\colon L\to L$ is the identity morphism, then $\mathcal S_0^2(I\times S_{std};\Sigma^t)$, hence \eqref{eq:concrete_functoriality_las}, is the identity map. It follows that $\widetilde{KhR}_2^+(\Sigma)$ is the identity map. If $\Sigma_i\colon L_i\to L_{i+1}$, $i=0,1$, are two morphisms in $\mathbf{Links}_{S_{std}}$, then $\mathcal S_0^2(I\times S_{std};(\Sigma_1\circ\Sigma_0)^t)=\mathcal S_0^2(I\times S_{std};\Sigma_0^t)\circ\mathcal S_0^2(I\times S_{std};\Sigma_1^t)$. After taking the associated graded maps, we see again from \eqref{eq:concrete_functoriality_las} that $\widetilde{KhR}_2^+(\Sigma_1\circ\Sigma_0)=\widetilde{KhR}_2^+(\Sigma_0)\circ\widetilde{KhR}_2^+(\Sigma_1)$.
\end{proof}

We remark that taking associated graded maps is necessary for Theorem~\ref{thm:concrete_functoriality_las}. If $\Sigma\colon\emptyset\to\emptyset$ is the cobordism between two empty links given by the $j$-th core $2$-sphere in the $i$-th disjoint union summand in $\{1/2\}\times S$, the induced map $\mathcal S_0^2(D_{std})\to\mathcal S_0^2(D_{std})$ is nonzero (it is the multiplication map by $A_{i,j,1}$ in the notation of Example~\ref{ex:D2S2}). However, since $\widetilde{KhR}_2^+(\emptyset)=\Q$, any map of degree $(0,-\chi(\Sigma))=(0,-2)$ is necessarily zero. We will see that taking associated graded maps, as well as introducing nontrivial shifting automorphisms $\mathrm{id}_\alpha$, is necessary only for one elementary move (handleslide) described in Proposition~\ref{prop:concrete_decomposition}.

\subsection{Decomposition into elementary cobordisms}
If Theorem~\ref{thm:concrete_functoriality_las} holds for two composable morphisms, then it holds for their composition as well. Therefore, it suffices to decompose any cobordism into a composition of some elementary cobordisms, and check Theorem~\ref{thm:concrete_functoriality_las} for these elementary ones.

\begin{Prop}\label{prop:concrete_decomposition}
Every morphism in $\mathbf{Links}_{S_{std}}$ is a composition of some elementary morphisms of the following forms. See Figure~\ref{fig:concrete_decomposition}.
\begin{enumerate}[(i)]
\item Isotopies via admissible links;\label{item:concrete_adm_isotopy}
\item Reidemeister moves or a Morse move (birth, death, saddle) away from the surgery regions;\label{item:concrete_R_M}
\item Finger moves, also in reverse;\label{item:concrete_finger}
\item Crossing moves, also in reverse;\label{item:concrete_push}
\item Overpass/underpass moves;\label{item:concrete_pass}
\item Handleslides.\label{item:concrete_slide}
\end{enumerate}
\end{Prop}
\begin{figure}
\centering
\includegraphics[width=0.6\linewidth]{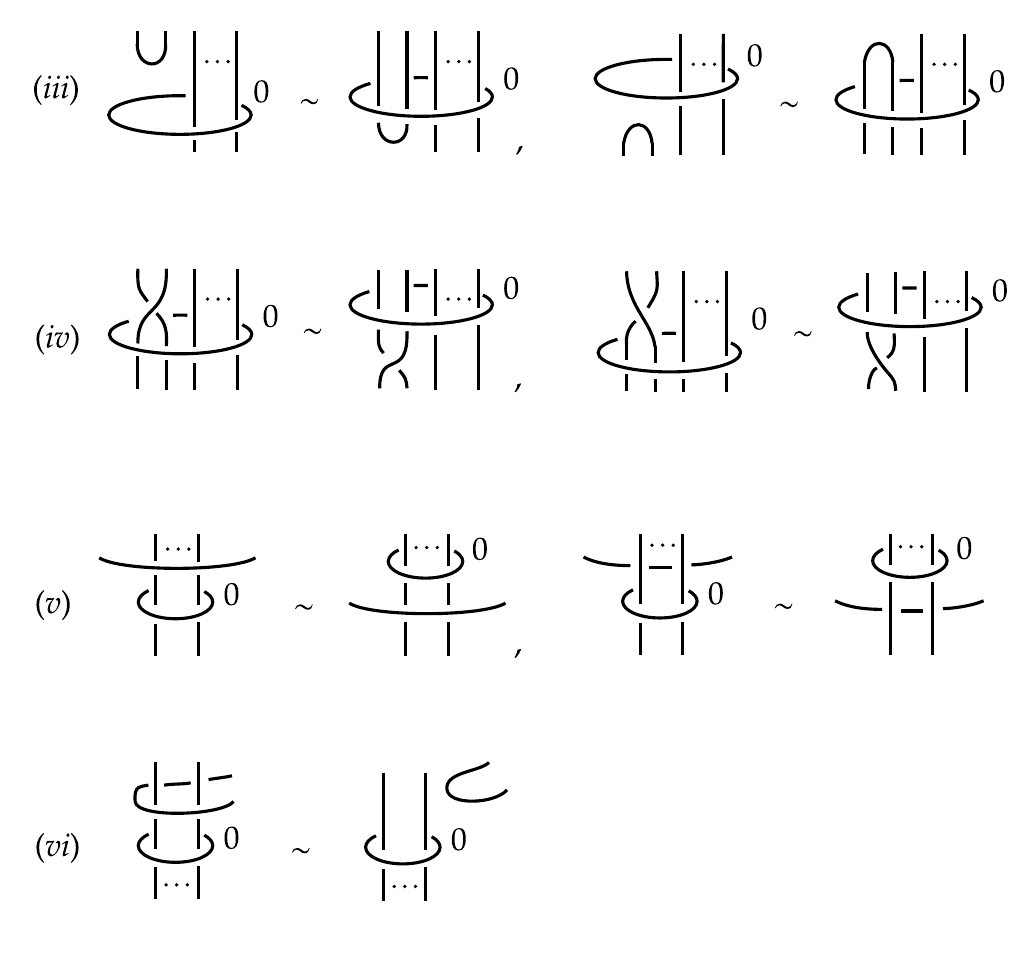}
\caption{Type \eqref{item:concrete_finger}\listsymbol\eqref{item:concrete_slide} moves in Proposition~\ref{prop:concrete_decomposition}. Here, for (iv), we have only drawn, for simplicity, the case of pushing a crossing between the first two strands, but it is understood that the similar pushing is allowed for crossings between any two adjacent strands.}
\label{fig:concrete_decomposition}
\end{figure}
\begin{proof}
This follows from a general position argument, see e.g. \cite[Proposition~3.2]{willis2021khovanov}.
\end{proof}

\subsection{Moves away from surgery regions}\label{sec:concrete_far}
We prove Theorem~\ref{thm:concrete_functoriality_las} for moves \eqref{item:concrete_adm_isotopy} and \eqref{item:concrete_R_M} in Proposition~\ref{prop:concrete_decomposition}.

By the description of the isomorphism \eqref{eq:SZ_MN} in Section~\ref{sec:SZ}, at each $\alpha\in H_2^L(D_{std})$, the map $\mathcal S_0^2(I\times S_{std};\Sigma^t)$ in terms of row \eqref{eq:SZ_MN} is induced termwise in the colimit by the map induced by $\Sigma^t$, namely
\begin{align}\label{eq:colim_model_maps_local}
KhR_2(\Sigma^t\cup(I\times(n_++r,n_-+r)\text{ belts}))&\colon\nonumber\\ KhR_2(L_1\cup(n_++r,n_-+r)\text{ belts})^{S_{|n|+2|r|}}&\to KhR_2(L_0\cup(n_++r,n_-+r)\text{ belts})^{S_{|n|+2|r|}}
\end{align}
on each colimit summand. Here $n=\alpha-\alpha_{L_1}\in H_2(D_{std})=\Z^{\sum_{i=1}^km_i}$. The same description applies for the map in terms of row \eqref{eq:SZ_renom}.

The isomorphisms \eqref{eq:SZ_res_1} and \eqref{eq:SZ_slide} are local near the surgery regions, and thus intertwine with the maps on $\widetilde{KhR}_2^+$ induced by $\Sigma^t$ on rows \eqref{eq:SZ_renom}\listsymbol\eqref{eq:SZ_slide} summand-wise. It follows that $\mathcal S_0^2(I\times S_{std};\Sigma^t)=\widetilde{KhR}_2^+(\Sigma)\otimes\mathrm{id}$ in terms of row \eqref{eq:SZ_simp}, where $\widetilde{KhR}_2^+(\Sigma)\colon\widetilde{KhR}_2^+(L_1)\to\widetilde{KhR}_2^+(L_0)$ is induced by an isotopy of the link diagram (for move \eqref{item:concrete_adm_isotopy}) or a Reidemeister/Morse-induced map on the link diagram (for move \eqref{item:concrete_R_M}). In particular, it preserves the lasagna quantum grading, and the associated map is of the form \eqref{eq:concrete_functoriality_las}, as desired (note $\alpha_{L_1}*[\Sigma^t]-\alpha_{L_0}=0$ since $\Sigma^t$ is disjoint from the $2$-cocores of $D_{std}$).

\subsection{Finger moves}\label{sec:concrete_finger}
We prove Theorem~\ref{thm:concrete_functoriality_las} for move \eqref{item:concrete_finger} in Proposition~\ref{prop:concrete_decomposition}. We only consider the move pushing a downward-pointing finger up, as the reverse is similar. Thus, the reverse $\Sigma^t$ is the move pushing a downward-pointing finger down, shown as the left to right direction in Figure~\ref{fig:concrete_decomposition}(iii).

\begin{figure}
\centering
\includegraphics[width=0.6\linewidth]{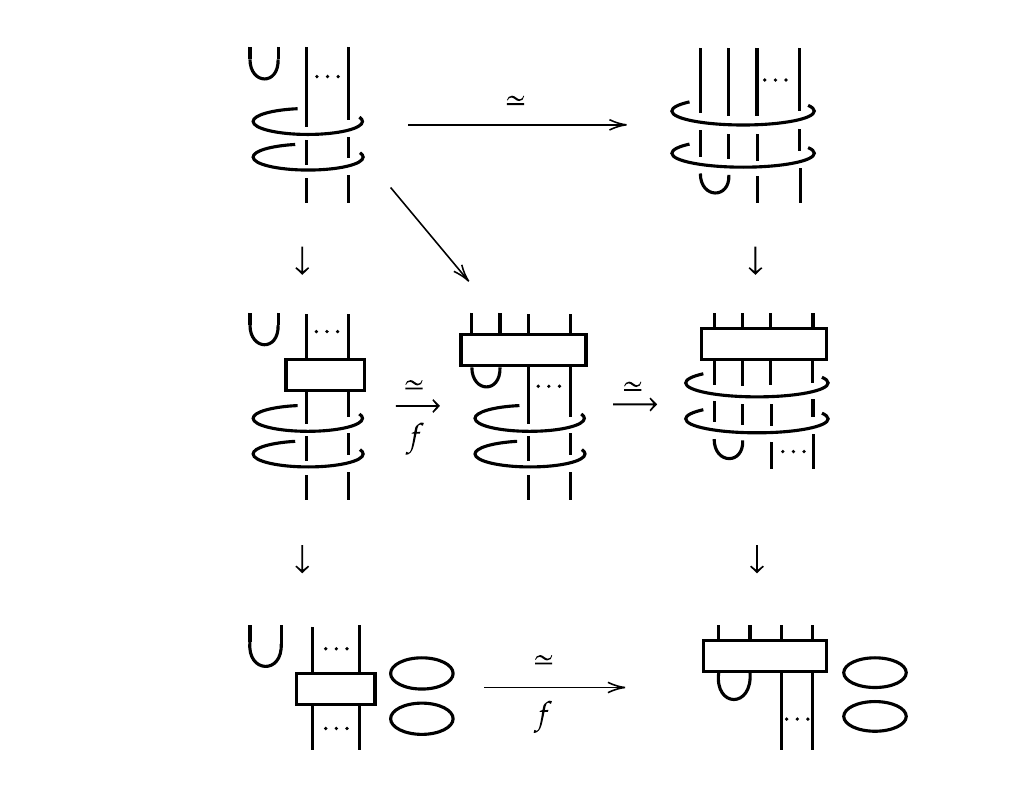}
\caption{Rows \eqref{eq:SZ_renom}\listsymbol\eqref{eq:SZ_slide} of a finger move.}
\label{fig:finger_diagram}
\end{figure}

\begin{figure}
\centering
\includegraphics[width=0.65\linewidth]{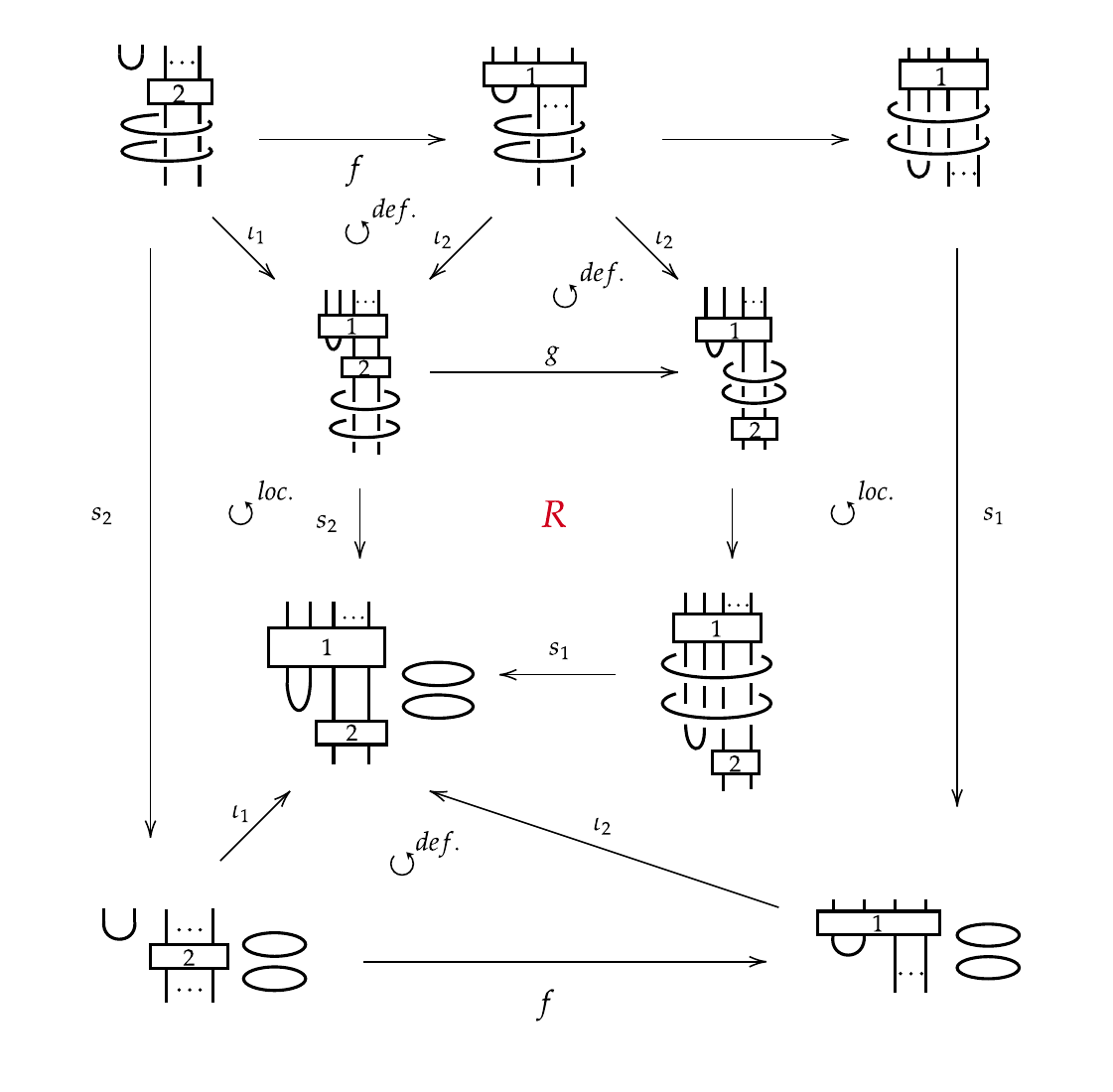}
\caption{The lower rectangle in Figure~\ref{fig:finger_diagram}.}
\label{fig:finger_diagram_low}
\end{figure}

By the description of the isomorphism \eqref{eq:SZ_MN} in Section~\ref{sec:SZ}, the map $\mathcal S_0^2(I\times S_{std};\Sigma^t)$ in terms of rows \eqref{eq:SZ_MN} and \eqref{eq:SZ_renom} is each induced by summand-wise cobordism maps pushing the finger down across the belts, as shown in the first row of Figure~\ref{fig:finger_diagram}. For simplicity, we only depict the case of $2$ belts at the surgery region with some orientations that are omitted from the diagram. Similar simplifications in figures will not be further remarked upon.

We claim that in terms of row \eqref{eq:SZ_res_1}, $\mathcal S_0^2(I\times S_{std};\Sigma^t)$ is equal to a map that is summand-wise given by pushing the finger first down the relevant projector $P_{\ell,0}^\vee$, then down across the belts, as shown in the second row of Figure~\ref{fig:finger_diagram}. Here, the push-down map across the projector, denoted $f$, is the ``sandwich'' composition
\begin{figure}[H]
\centering
\includegraphics[width=0.65\linewidth]{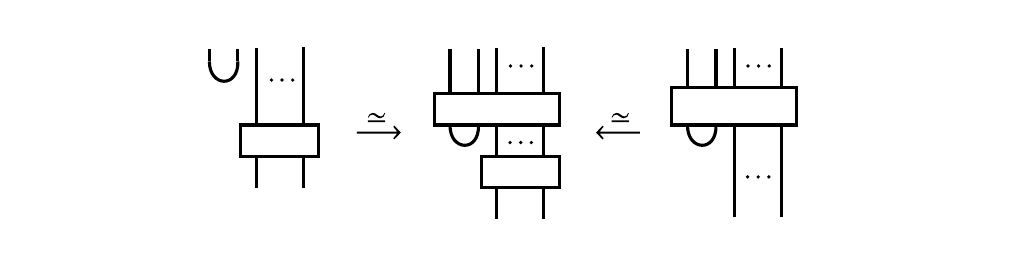}
\end{figure}
(the two maps are isomorphisms on homology by Proposition~\ref{prop:projector_properties}(5)). We need to prove the upper half of Figure~\ref{fig:finger_diagram} commutes. This is straightforward: the quadrilateral commutes by locality, and the triangle commutes by definition of $f$ and locality.

We claim that in terms of row \eqref{eq:SZ_slide}, $\mathcal S_0^2(I\times S_{std};\Sigma^t)$ is equal to a map that is summand-wise pushing the cup down across the belts, as shown in the third row of Figure~\ref{fig:finger_diagram}. Namely, we need to prove the lower half of Figure~\ref{fig:finger_diagram} commutes. We fill this lower rectangle into Figure~\ref{fig:finger_diagram_low}, where the middle map $g$ is the ``sandwich'' map defined similarly to $f$, namely as the composition
\begin{figure}[H]
\centering
\includegraphics[width=0.7\linewidth]{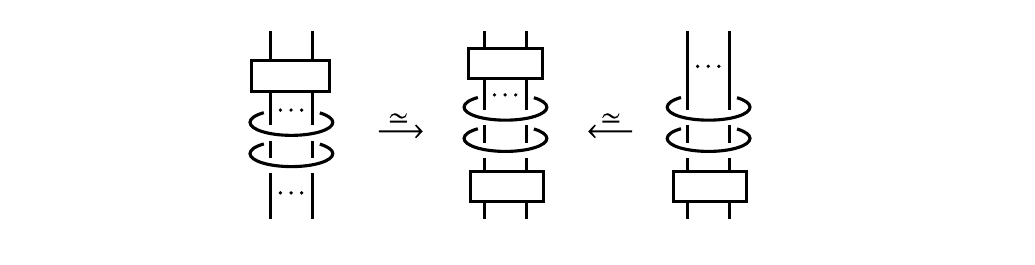}.
\end{figure}
For convenience, projectors in Figure~\ref{fig:finger_diagram_low} are labeled by $1$ or $2$, and the subscripts in the unit maps $\iota_\bullet$ and sliding maps $s_\bullet$ refer to the corresponding projectors. All regions in Figure~\ref{fig:finger_diagram_low} commute by definition or locality, except the middle rectangle labeled $R$. We fill in the rectangle $R$ into the diagram of isomorphisms
\begin{figure}[H]
\centering
\includegraphics[width=0.7\linewidth]{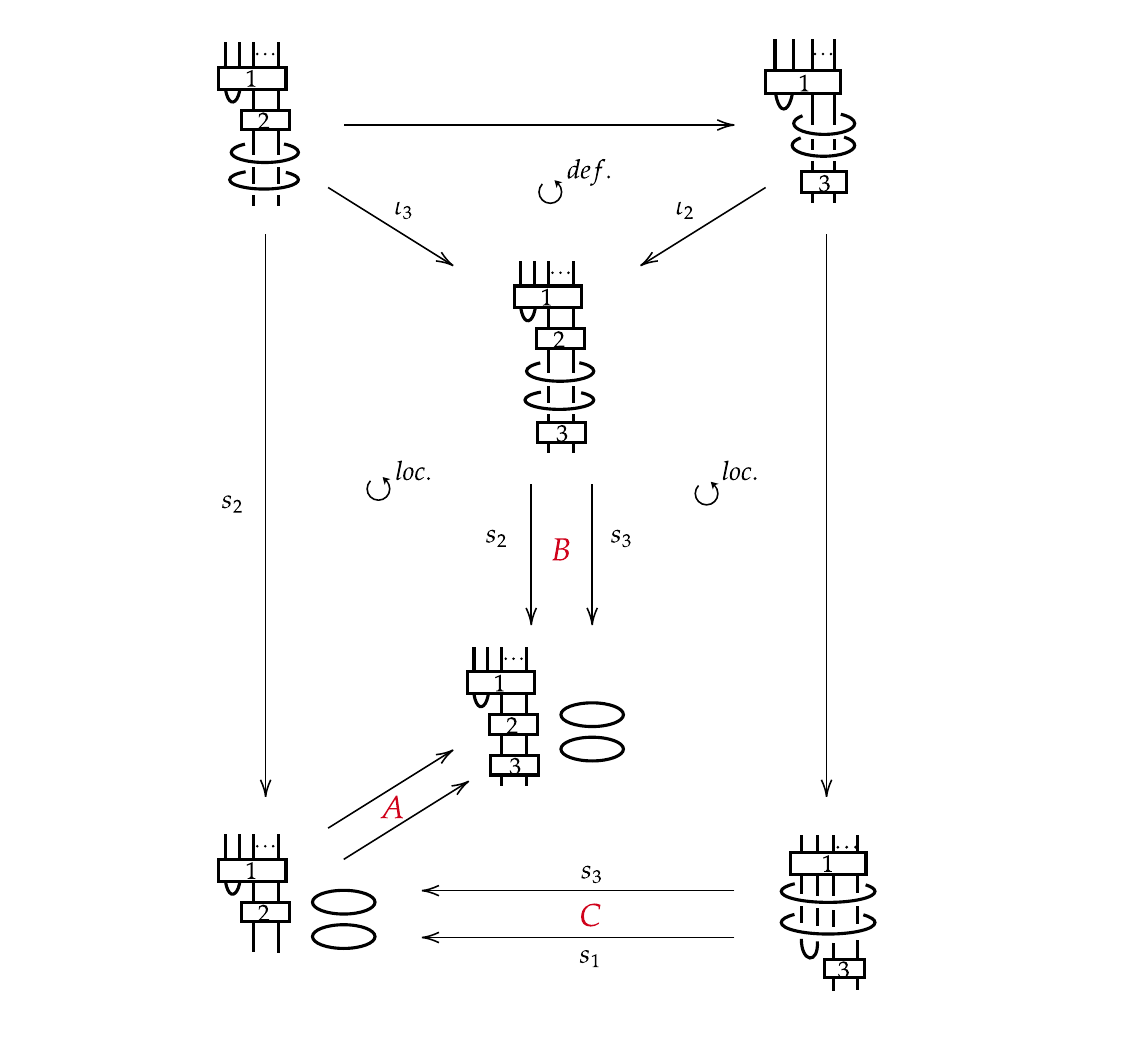},
\end{figure}
where all regions except $A,B,C$ commute by definition or locality. We redraw regions $A,B,C$:
\begin{figure}[H]
\centering
\includegraphics[width=0.7\linewidth]{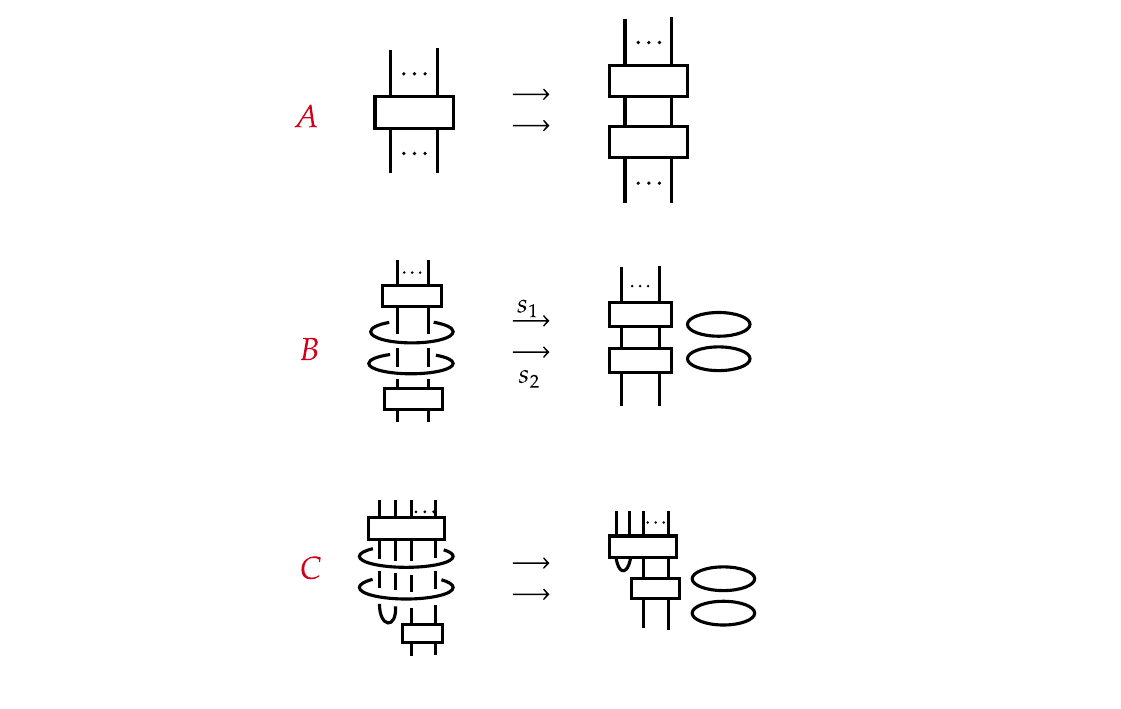}.
\end{figure}

The commutativity of region $A$ is exactly Proposition~\ref{prop:projector_properties}(4).

We consider region $B$. On the chain level, each projector is a chain complex of through-degree $0$ Temperley-Lieb diagrams. Each of the two chain maps $s_1,s_2$ is given termwise by sliding the belts down the Temperley-Lieb diagram from above/below using simplifying Reidemeister II moves, 
for which no cross terms exist.
Explicit chain homotopy inverses $s_1^{-1},s_2^{-1}$ can also be written down using \cite[Section~3]{gugenheim1972chain}; see also \cite[Proof of Proposition~2.10]{willis2021khovanov} for an account in our situation, the relevant property used here being that simplifying Reidemeister II moves are \textit{very strong deformation retracts} \cite[Definition~2.9]{willis2021khovanov}. As all terms in the complex of the chain model of the right hand side of region $B$ have internal homological degree $0$, the formulas in \cite{gugenheim1972chain} show that the composition $s_2\circ s_1^{-1}$ is given termwise by compositions of Reidemeister-induced maps, for which no cross terms exist.
Each such termwise composition takes the form
\begin{figure}[H]
\centering
\includegraphics[width=0.6\linewidth]{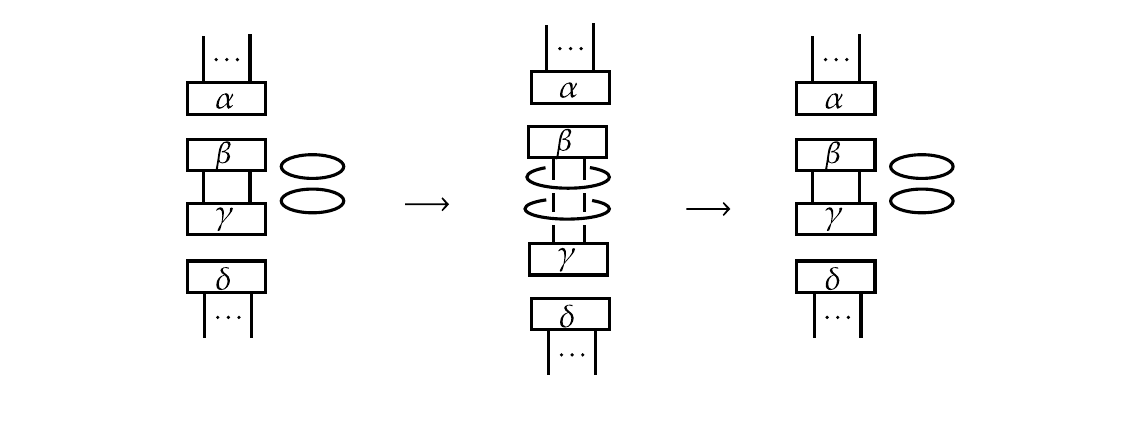},
\end{figure}
for some crossingless tangles $\alpha,\beta,\gamma,\delta$, where the first map slides the belts into the strands from the top, and the second map slides them out from the bottom. We claim this composition is the identity chain map up to sign. By locality, we may ignore the $\alpha,\delta$ boxes. The composition is identity on homology up to sign because the composite cobordism is isotopic to the identity cobordism in $I\times S^3$, and Khovanov homology is functorial in $S^3$ \cite[Theorem~1.1]{morrison2022invariants}. Since the diagram is crossingless, this implies that the composition is identity on the chain level up to sign. In Appendix~\ref{sec:sign_region_B_R1}, we show that under the appropriate setup using webs and foams, the corresponding composition map is equal to the identity chain map, rather than its negation, proving the commutativity of region $B$.

The commutativity of region $C$ is proved in exactly the same way.

Now, it follows that the map $\mathcal S_0^2(I\times S_{std};\Sigma^t)=\widetilde{KhR}_2^+(\Sigma)\otimes\mathrm{id}$ in terms of row \eqref{eq:SZ_simp}, where $\widetilde{KhR}_2^+(\Sigma)$ is the map on homology induced by pushing the finger down the projector. The statement follows.

\subsection{Crossing moves}
We prove Theorem~\ref{thm:concrete_functoriality_las} for move \eqref{item:concrete_push} in Proposition~\ref{prop:concrete_decomposition}. We only consider the case of pushing a positive crossing up the belts. The cases of pushing down or of a negative crossing are similar.

The proof is analogous to the case of finger moves. The only difference is that we need to prove the commutativity of the following analog of region $C$ in Section~\ref{sec:concrete_finger}
\begin{figure}[H]
\centering
\includegraphics[width=0.7\linewidth]{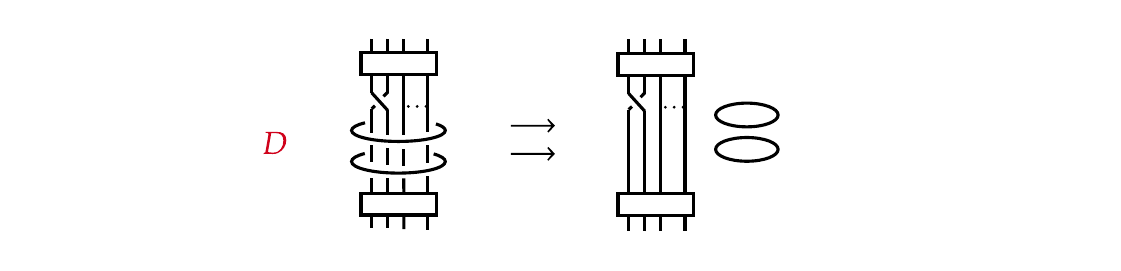}.
\end{figure}
Assuming the commutativity of region $D$, we can conclude that $\mathcal S_0^2(I\times S_{std};\Sigma^t)=\widetilde{KhR}_2^+(\Sigma)\otimes\mathrm{id}$ where $\widetilde{KhR}_2^+(\Sigma)$ is the map on homology induced by pushing the crossing down the projector, defined as the ``sandwich'' composition
\begin{figure}[H]
\centering
\includegraphics[width=0.75\linewidth]{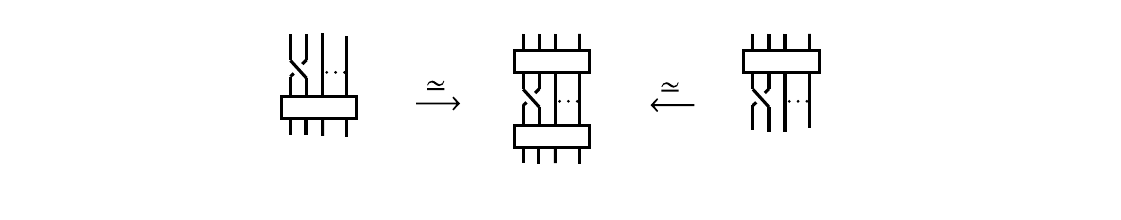}.
\end{figure}

We prove the commutativity of region $D$. The same argument for region $B$ in Section~\ref{sec:concrete_finger} no longer applies, because the terms in the twisted complex of the chain model of the right hand side of region $D$ have nontrivial internal homological degrees, hence cross terms may exist when applying \cite{gugenheim1972chain}. Nevertheless, \cite{gugenheim1972chain} implies that if $\Phi$ denotes the map going from the right hand side to itself by sliding the belts into the strands from the top and sliding them off from the bottom, then $\Phi$ is nondecreasing in the external homological degree. By judiciously choosing Reidemeister III induced chain maps when passing the belts across the middle crossing (see \cite[Section~3.5]{morrison2022invariants}, in the $\gl_2$ webs and foams context), we may assume that the induced map $\Phi$ is nondecreasing\footnote{Note that \cite{morrison2022invariants} uses the homological convention instead of the cohomological convention that we adopt here. Therefore, the words ``increasing'' and ``decreasing'' in their Lemma~3.12 have the opposite meaning to ours.} in the internal homological degree, and its internal-degree-preserving part on the two resolutions of the middle crossing, within each term of the twisted complex, is each given by the Reidemeister-induced maps. Since $\Phi$ preserves the total homological degree, it must therefore preserve both the internal and external homological degrees. Consequently, by the same argument as before, $\Phi$ is given by a termwise chain map that is identity on each term of the twisted complex up to termwise signs. In Appendix~\ref{sec:sign_region_B_R1}, we remove the sign ambiguity and show that $\Phi$ is termwise the identity map, hence the identity map overall, proving the commutativity of region $D$.

\subsection{Overpass/underpass moves}\label{sec:concrete_pass}
We prove Theorem~\ref{thm:concrete_functoriality_las} for move \eqref{item:concrete_pass} in Proposition~\ref{prop:concrete_decomposition}. We only consider the case of an overpass move, as the underpass case is similar.

The proof is again analogous to the case of finger moves. The only difference is that we need to prove the commutativity of the following analog of region $C$ in Section~\ref{sec:concrete_finger}
\begin{figure}[H]
\centering
\includegraphics[width=0.7\linewidth]{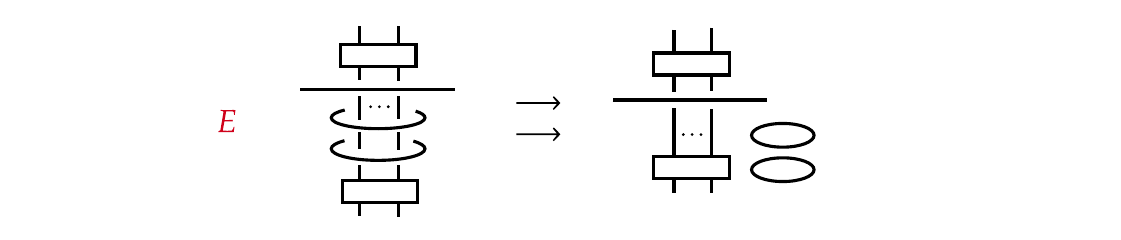}.
\end{figure}

Assuming the commutativity of region $E$, we can conclude that $\mathcal S_0^2(I\times S_{std};\Sigma^t)=\widetilde{KhR}_2^+(\Sigma)\otimes\mathrm{id}$ where $\widetilde{KhR}_2^+(\Sigma)$ is the map on homology induced by pushing the overpassing strand down the projector, defined as the ``sandwich'' composition (both maps are isomorphisms on homology by Proposition~\ref{prop:projector_properties}(6))
\begin{figure}[H]
\centering
\includegraphics[width=0.6\linewidth]{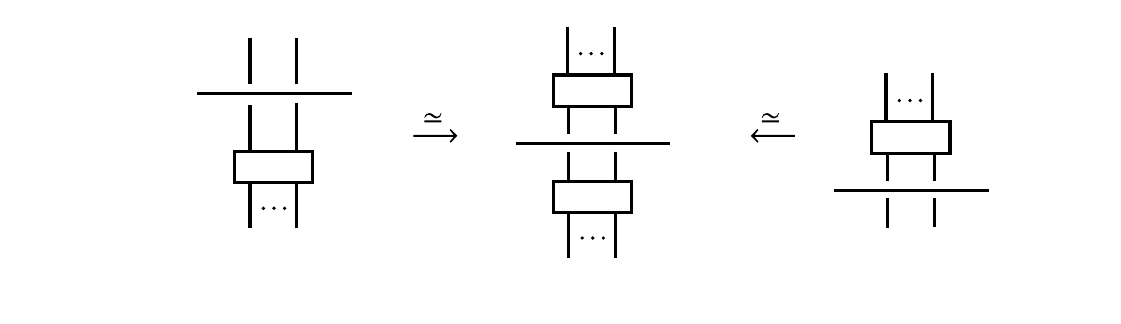}.
\end{figure}

We prove the commutativity of region $E$. Consider the diagram
\begin{figure}[H]
\centering
\includegraphics[width=0.6\linewidth]{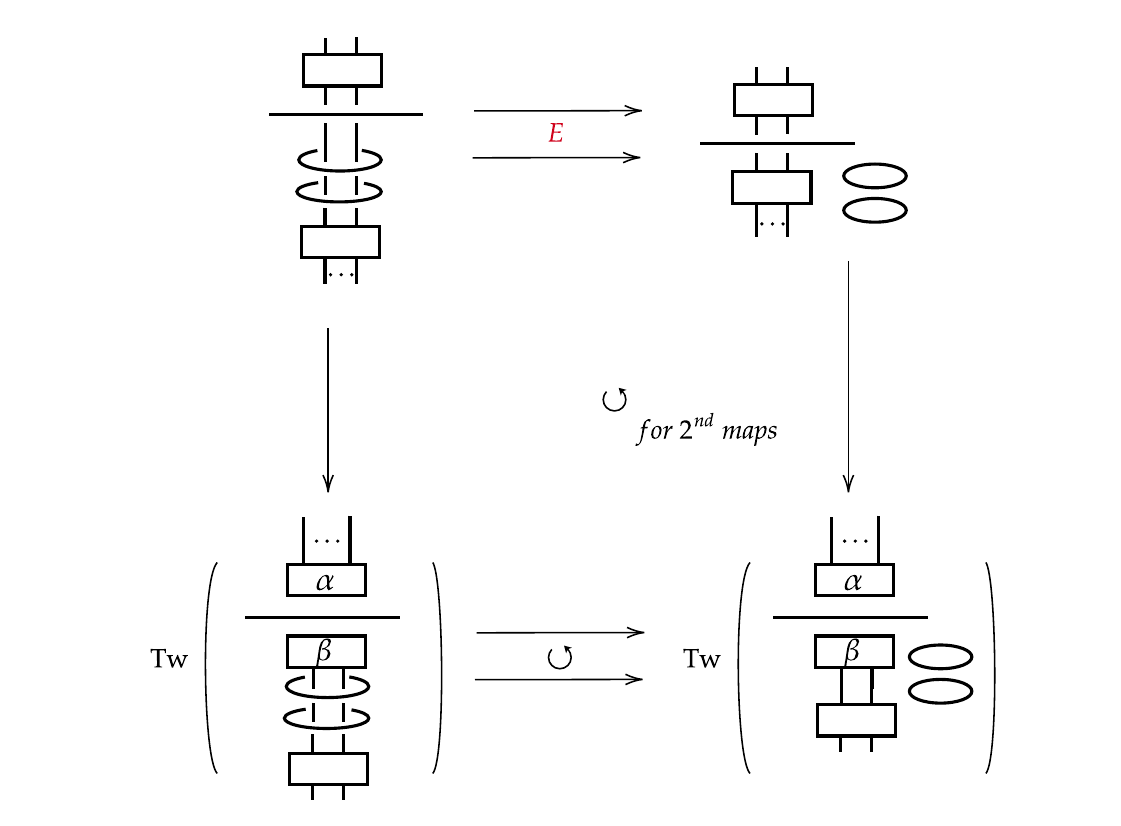}
\end{figure}
where we expand the first projector to rewrite the diagrams as twisted complexes, and the vertical maps slide the overstrand into the empty region by termwise Reidemeister II simplifying maps. The two horizontal maps in the second row are equal by the same argument as for region $B$ in Section~\ref{sec:concrete_finger}. If we take the pair of horizontal maps that slide the belts off the second projector, one for each row in the diagram, then the square commutes by locality. Thus, to prove the commutativity of region $E$, it suffices to show the square also commutes for the other pair of horizontal maps.

The proof is similar to the one for region $B$ in Section~\ref{sec:concrete_finger}. We repeat the argument again. Redraw the relevant parts of this square termwise:
\begin{figure}[H]
\centering
\includegraphics[width=0.6\linewidth]{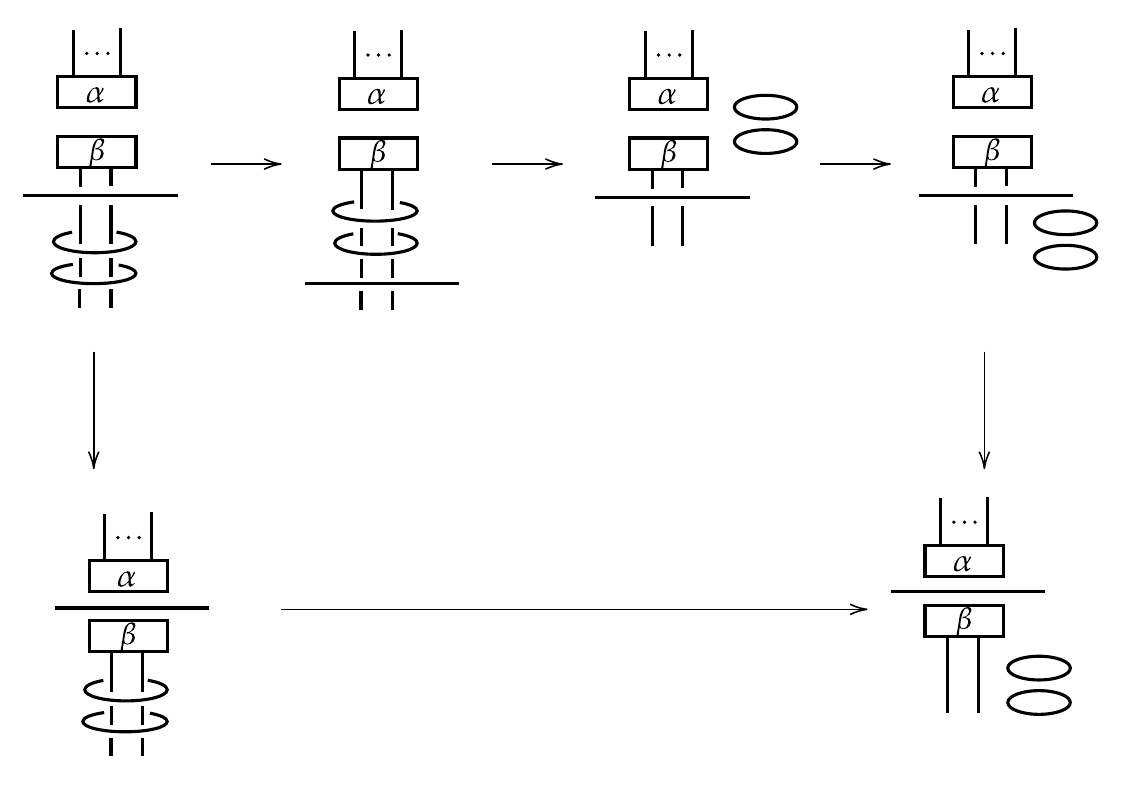}.
\end{figure}
The two total chain maps are assembled from these termwise maps, for which no cross terms exist. If we start from the lower-right corner and compose the chain maps or their explicit homotopy inverses (which are termwise Reidemeister-induced maps, since simplifying Reidemeister II give very strong deformation retracts) in a loop, we obtain the termwise identity chain map up to individual signs, because the termwise composite cobordisms are isotopic to identity and the source and target diagrams are crossingless. The sign ambiguity is removed in Appendix~\ref{sec:sign_region_B_R1}. This proves the commutativity of region $E$.

\subsection{Handleslides}\label{sec:concrete_handleslide}
We prove Theorem~\ref{thm:concrete_functoriality_las} for move \eqref{item:concrete_slide} in Proposition~\ref{prop:concrete_decomposition}. We only consider the case where the sign of the handleslide is negative, namely where $\Sigma^t\subset I\times S_{std}\subset D_{std}$ intersects the $2$-cocores positively at one point, see Figure~\ref{fig:concrete_positive_handleslide}. The positive case is similar. Suppose the strand is sliding over the $j$-th $2$-handle of $D_{std}$.

\begin{figure}
\centering
\includegraphics[width=0.75\linewidth]{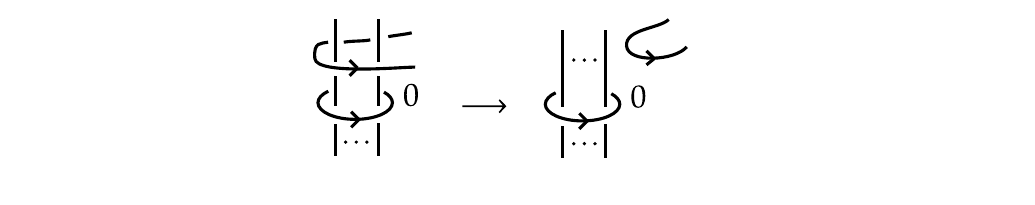}
\caption{The reverse of a negative handleslide, with orientations on the $2$-handle attaching curve and the sliding strand indicated.}
\label{fig:concrete_positive_handleslide}
\end{figure}

We claim that the map $\mathcal S_0^2(I\times S_{std};\Sigma^t)$ in terms of row \eqref{eq:SZ_MN} is induced by the summand-wise symmetrized saddle map 
\begin{align*}
KhR_2(L_1\cup(n_++r,n_-+r)\text{ belts})^{S_{|n|+2|r|}}&\xrightarrow{\text{saddle}}KhR_2(L_0\cup(n_++e_j+r,n_-+r)\text{ belts})\\
&\xrightarrow{\mathrm{Sym}}KhR_2(L_0\cup(n_++e_j+r,n_-+r)\text{ belts})^{S_{|n|+e_j+2|r|}},
\end{align*}
or in local diagram, similar to the annular model situation in \cite[Section~4]{hogancamp2025kirby}, as 
\begin{figure}[H]
\centering
\includegraphics[width=0.7\linewidth]{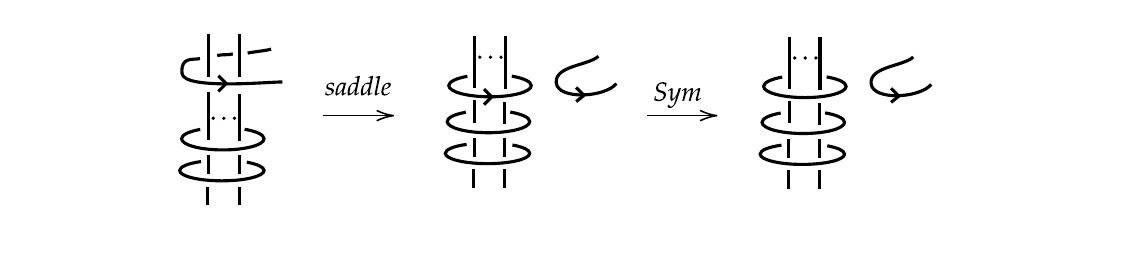}.
\end{figure}

On the lasagna filling level, the map $\mathcal S_0^2(I\times S_{std};\Sigma^t)$ attaches to the standard skein $(I\times L_1)\cup((n_++r,n_-+r)\text{ cores})$ the cobordism $\Sigma^t$. If one enlarges the input ball a bit to make the new skein $\Sigma^t\cup(I\times L_1)\cup((n_++r,n_-+r)\text{ cores})$ standard of the form $(I\times L_0)\cup((n_++e_j+r,n_-+r)\text{ cores})$, the Khovanov decoration changes according to the cobordism between the two balls, which is exactly given by the claimed saddle. This proves the claim.

\begin{figure}
\centering
\includegraphics[width=0.7\linewidth]{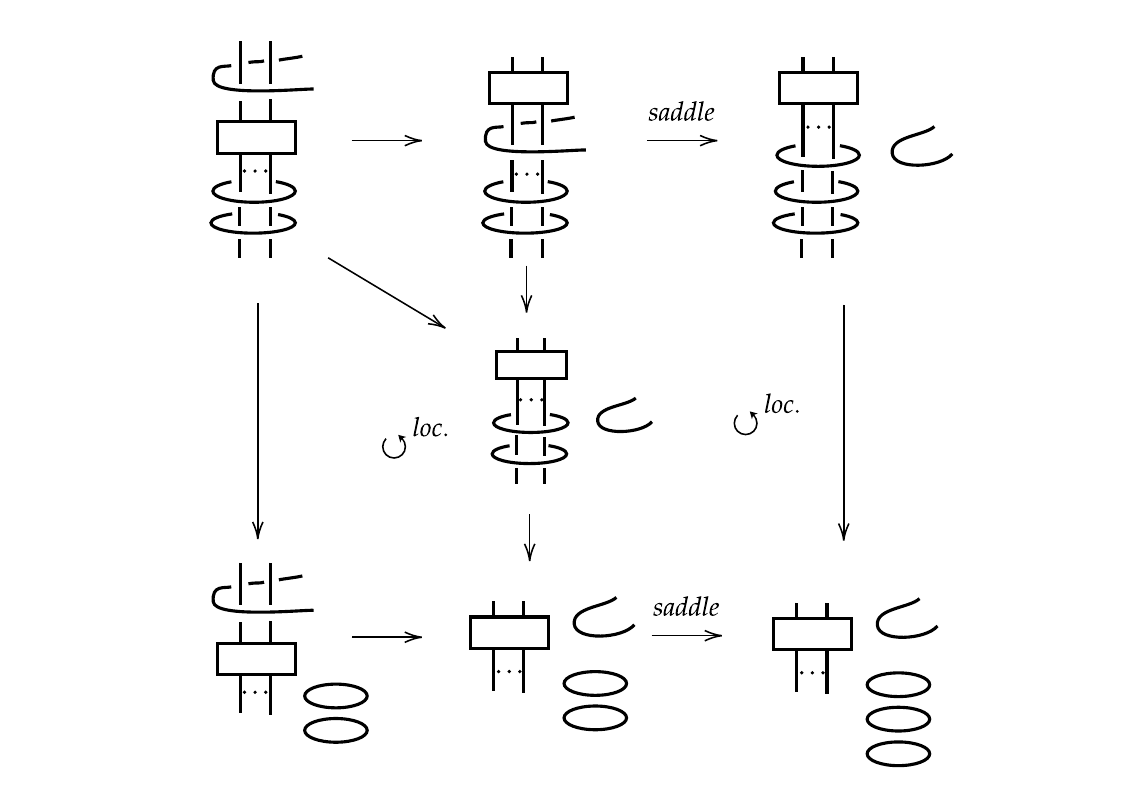}
\caption{Rows \eqref{eq:SZ_res_1} and \eqref{eq:SZ_slide} of a reverse handleslide move (before symmetrization).}
\label{fig:concrete_handleslide_diagram}
\end{figure}

The same description of $\mathcal S_0^2(I\times S_{std};\Sigma^t)$ applies to row \eqref{eq:SZ_renom}. By the same argument as in Section~\ref{sec:concrete_finger} using an analog of the upper half of Figure~\ref{fig:finger_diagram}, the map $\mathcal S_0^2(I\times S_{std};\Sigma^t)$ in terms of row \eqref{eq:SZ_res_1} is given by the first row of Figure~\ref{fig:concrete_handleslide_diagram}, postcomposed with symmetrization. Here the first map is the ``sandwich'' map as before.

We claim that Figure~\ref{fig:concrete_handleslide_diagram} commutes. As sliding off the belts commutes with symmetrization, this will imply $\mathcal S_0^2(I\times S_{std};\Sigma^t)$ in terms of row \eqref{eq:SZ_slide} is given by the second row of Figure~\ref{fig:concrete_handleslide_diagram}, postcomposed with symmetrization. The two rectangles commute by locality, while the commutativity of the triangle follows from the commutativity of region $A$ in Section~\ref{sec:concrete_finger} and region $F$:
\begin{figure}[H]
\centering
\includegraphics[width=0.75\linewidth]{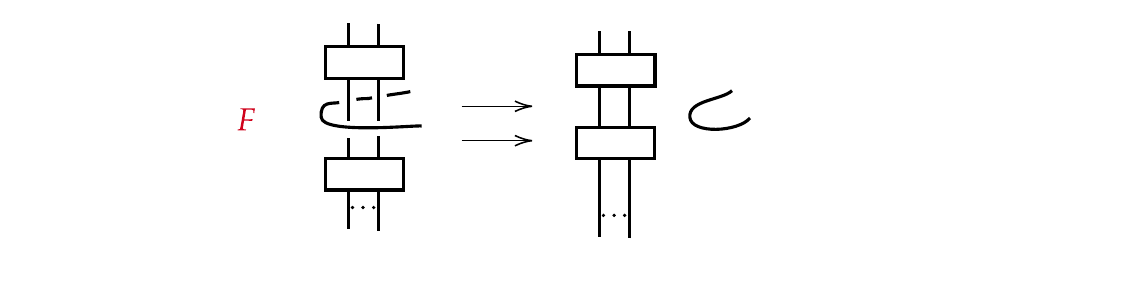}.
\end{figure}
The commutativity of region $F$ follows from the same argument for that of region $B$ in Section~\ref{sec:concrete_finger}. This proves the claim.

Now, the saddle map in the second row of Figure~\ref{fig:concrete_handleslide_diagram} is equal to $(birth)\otimes(dot)+(dotted\ birth)\otimes1$. In terms of row \eqref{eq:SZ_simp}, after symmetrization, the first term decreases the lasagna quantum grading while the second term preserves it. This proves that $\mathcal S_0^2(I\times S_{std};\Sigma^t)$ is nonincreasing in the lasagna quantum grading. The associated graded map is given by $\widetilde{KhR}_2^+(\Sigma)\otimes gr(u)$, where $\widetilde{KhR}_2^+(\Sigma)\colon\widetilde{KhR}_2^+(L_1)\to\widetilde{KhR}_2^+(L_0)$ is the slide-off map
\begin{figure}[H]
\centering
\includegraphics[width=0.75\linewidth]{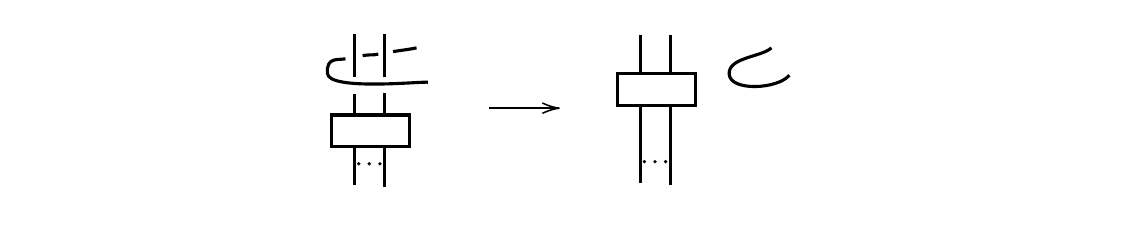},
\end{figure}
and $u\colon\mathcal S_0^2(D_{std})\to\mathcal S_0^2(D_{std})$ is the gluing map that attaches a collar of the boundary containing a copy of the dotted $j$-th core $S^2$ as skein. By the description of $\mathcal S_0^2(D_{std})$ in Section~\ref{sec:D2S2}, $u=\mathrm{id}_{e_j}$. Since $e_j=\alpha_{L_1}*[\Sigma^t]-\alpha_{L_0}$, this finishes the proof of Theorem~\ref{thm:concrete_functoriality_las}.

\section{Homology in abstract spin \texorpdfstring{$\#(S^1\times S^2)$}{\#S1*S2}}\label{sec:abstract_spin_homology}
As in the previous section, fix $k\ge0$, $m_1,\cdots,m_k\ge0$, and write for short $S_{std}:=\sqcup_{i=1}^k\#^{m_i}(S^1\times S^2)$, $D_{std}:=\#_{i=1}^k\natural^{m_i}(D^2\times S^2)$. Equip $D_{std}$ with its unique spin structure, and $S_{std}=\partial D_{std}$ with the induced boundary spin structure. When $k=0$, $S_{std}=\emptyset$, and statements in Section~\ref{sec:abstract_spin_homology_statements} are trivial. Assume from now on $k\ge1$.

\subsection{The statements}\label{sec:abstract_spin_homology_statements}
If $M$ is a spin manifold, an \textit{abstract spin} $M$ is a spin manifold that is spin diffeomorphic to $M$. The goal of this section is to construct the Rozansky-Willis homology for links in an abstract spin $S_{std}$, as stated in the following theorem.

\begin{Thm}\label{thm:abstract_spin_homology}
Let $S$ be an abstract spin $S_{std}$ and $L\subset S$ be a framed oriented link with $2$-divisible homology class. There is a bigraded vector space $\widetilde{KhR}_2^+(S,L)$, called the \textit{Rozansky-Willis homology} of $(S,L)$, which is an invariant of the pair $(S,L)$ up to spin diffeomorphisms. When $S=S_{std}$ and $L$ is admissible, $\widetilde{KhR}_2^+(S,L)$ is canonically isomorphic to $\widetilde{KhR}_2^+(L)$.
\end{Thm}

Let $(S,L)$ be as in Theorem~\ref{thm:abstract_spin_homology}. A spin parametrization $\phi\colon S\xrightarrow{\cong}S_{std}$ is \textit{$L$-admissible} if $\phi(L)$ is an admissible link in $S_{std}$. If $\rho$ is such a parametrization, then $\widetilde{KhR}_2^+(\rho(L))$ is a natural candidate for $\widetilde{KhR}_2^+(S,L)$. For this idea to work, if $\rho'$ is another $L$-admissible parametrization of $S$, we need to identify $\widetilde{KhR}_2^+(\rho(L))$ with $\widetilde{KhR}_2^+(\rho'(L))$ in a canonical and functorial way. This is the content of the following proposition.

\begin{Prop}\label{prop:spin_repara}
Let $L\subset S_{std}$ be an admissible link and $\phi\in\Diff^{spin}(S_{std})$ be an $L$-admissible spin diffeomorphism of $S_{std}$. There is a canonical isomorphism
\begin{equation}\label{eq:spin_repara}
\widetilde{KhR}_2^+(\phi)\colon\widetilde{KhR}_2^+(\phi(L))\xrightarrow{\cong}\widetilde{KhR}_2^+(L)
\end{equation}
of graded vector spaces, which is the identity map when $\phi=\mathrm{id}$. If $\phi'\in\Diff^{spin}(S_{std})$ is $\phi(L)$-admissible, then $$\widetilde{KhR}_2^+(\phi)\circ\widetilde{KhR}_2^+(\phi')=\widetilde{KhR}_2^+(\phi'\circ\phi)\colon\widetilde{KhR}_2^+(\phi'(\phi(L)))\xrightarrow{\cong}\widetilde{KhR}_2^+(L).$$
\end{Prop}

Again, one could state the theorem in the covariant way by inverting the diffeomorphisms.

\begin{proof}[Proof of Theorem~\ref{thm:abstract_spin_homology} assuming Proposition~\ref{prop:spin_repara}]
Let $P^{spin}(S,L)$ denote the set of $L$-admissible spin parametrizations of $S$. Then, $\widetilde{KhR}_2^+(S,L)$ can be defined as the ``cross-section'' subspace of $\prod_{\rho\in P^{spin}(S,L)}\widetilde{KhR}_2^+(\rho(L))$ consisting of elements $(v_\rho)_\rho$ with $v_{\rho'}=\widetilde{KhR}_2^+(\rho\circ(\rho')^{-1})(v_\rho)$ for all $\rho,\rho'\in P^{spin}(S,L)$.

For any $\rho\in P^{spin}(S,L)$, the projection map $\widetilde{KhR}_2^+(S,L)\xrightarrow{\cong}\widetilde{KhR}_2^+(\rho(L))$ is an isomorphism of graded vector spaces. The case $S=S_{std}$ and $\rho=\mathrm{id}$ gives the last statement of the theorem.
\end{proof}

Every spin diffeomorphism $\phi\in\Diff^{spin}(S_{std})$ admits a lift $\tilde\phi\in\Diff^+(D_{std})$ in the orientation-preserving diffeomorphism group of $D_{std}$, as can be seen in Section~\ref{sec:abstract_spin_homology_decomposition}. Proposition~\ref{prop:spin_repara} is therefore a formal consequence of the following theorem.

\begin{Thm}\label{thm:spin_repara_las}
Let $L,\phi$ be as in Proposition~\ref{prop:spin_repara}, and let $\tilde\phi\in\Diff^+(D_{std})$ be a lift of $\phi$. The pushforward map $$\mathcal S_0^2(\tilde\phi^{-1})\colon\mathcal S_0^2(D_{std};\phi(L))\to\mathcal S_0^2(D_{std};L)$$ induces a map $gr_0\mathcal S_0^2(D_{std};\phi(L))\to gr_0\mathcal S_0^2(D_{std};L)$ on the $0$-th associated graded spaces with respect to the lasagna quantum grading. Under the isomorphism \eqref{eq:SZ}, this induced map is uniquely of the form
\begin{equation}\label{eq:spin_repara_las}
\widetilde{KhR}_2^+(\phi)\otimes gr_0(\mathrm{id}_\alpha\circ\phi_*^{-1})\colon\widetilde{KhR}_2^+(\phi(L))\otimes gr_0\mathcal S_0^2(D_{std})\to\widetilde{KhR}_2^+(L)\otimes gr_0\mathcal S_0^2(D_{std})
\end{equation}
for some isomorphism $\widetilde{KhR}_2^+(\phi)$ independent of $\tilde\phi$, and some $\alpha\in H_2(D_{std})$ depending on $\tilde\phi$ and $[L]\in H_1(S_{std})$. Here, $\phi_*$ in \eqref{eq:spin_repara_las} is defined as the pushforward isomorphism $\mathcal S_0^2(\tilde\phi)\colon\mathcal S_0^2(D_{std})\xrightarrow{\cong}\mathcal S_0^2(D_{std})$.
\end{Thm}
Note that by the description of $\mathcal S_0^2(D_{std})$ in Section~\ref{sec:D2S2}, the map $\phi_*=\mathcal S_0^2(\tilde\phi)\colon\mathcal S_0^2(D_{std})\to\mathcal S_0^2(D_{std})$ only depends on the action of $\tilde\phi$ on $H_2(D_{std})$ (hence, in particular, is independent of the choice of $\tilde\phi$). Note also that the uniqueness statement in Theorem~\ref{thm:spin_repara_las} is trivial.

The proof of Theorem~\ref{thm:spin_repara_las} takes up the bulk of Section~\ref{sec:abstract_spin_homology}. In Section~\ref{sec:abstract_spin_homology_decomposition} we decompose it into various cases, which are treated individually in Sections~\ref{sec:diff_rel_bdy}--\ref{sec:handleslides}. Proposition~\ref{prop:spin_repara} is then a consequence.

\begin{proof}[Proof of Proposition~\ref{prop:spin_repara} assuming Theorem~\ref{thm:spin_repara_las}]
We check that the assignment $\phi\mapsto\widetilde{KhR}_2^+(\phi)$ given by Theorem~\ref{thm:spin_repara_las} is functorial. If $\phi=\mathrm{id}$, taking $\tilde\phi=\mathrm{id}$ in Theorem~\ref{thm:spin_repara_las} yields $\widetilde{KhR}_2^+(\phi)=\mathrm{id}$. If $\phi_j\in\Diff^{spin}(S_{std})$ with lifts $\tilde\phi_j\in\Diff^+(D_{std})$, $j=0,1$, where $\phi_0$ is $L$-admissible and $\phi_1$ is $\phi_0(L)$-admissible, then $\tilde\phi_1\circ\tilde\phi_0$ is a lift of $\phi_1\circ\phi_0$. Since $\mathcal S_0^2((\tilde\phi_1\circ\tilde\phi_0)^{-1})=\mathcal S_0^2(\tilde\phi_0^{-1})\circ\mathcal S_0^2(\tilde\phi_1^{-1})$, we know from Theorem~\ref{thm:spin_repara_las} that $$\widetilde{KhR}_2^+(\phi_1\circ\phi_0)\otimes gr_0(\mathrm{id}_\alpha\circ\mathcal S_0^2((\tilde\phi_1\circ\tilde\phi_0)^{-1}))=(\widetilde{KhR}_2^+(\phi_0)\circ\widetilde{KhR}_2^+(\phi_1))\otimes gr_0(\mathrm{id_{\alpha_0}\circ\mathcal S_0^2(\tilde\phi_0^{-1})}\circ\mathrm{id_{\alpha_1}\circ\mathcal S_0^2(\tilde\phi_1^{-1})})$$ for some $\alpha,\alpha_0,\alpha_1\in H_2(D_{std})$. It follows that $\alpha=\alpha_0+(\tilde\phi_0)_*^{-1}(\alpha_1)$ and $\widetilde{KhR}_2^+(\phi_1\circ\phi_0)=\widetilde{KhR}_2^+(\phi_0)\circ\widetilde{KhR}_2^+(\phi_1)$.
\end{proof}

An important special case of Theorem~\ref{thm:spin_repara_las} is when $\phi=\mathrm{id}$. We state this as a proposition by itself, which will in turn be an ingredient of the proof of Theorem~\ref{thm:spin_repara_las}.

\begin{Prop}\label{prop:D2S2_repara}
Let $L\subset S_{std}$ be an admissible link and $\psi\in\Diff_\partial(D_{std})$ be a diffeomorphism of $D_{std}$ rel boundary. The pushforward map $$\mathcal S_0^2(\psi)\colon\mathcal S_0^2(D_{std};L)\to\mathcal S_0^2(D_{std};L)$$ induces a map on $gr_0$, which under the isomorphism \eqref{eq:SZ} takes the form $$\mathrm{id}\otimes gr_0(\mathrm{id}_{\alpha})\colon\widetilde{KhR}_2^+(L)\otimes gr_0\mathcal S_0^2(D_{std})\to\widetilde{KhR}_2^+(L)\otimes gr_0\mathcal S_0^2(D_{std})$$ for some $\alpha\in H_2(D_{std})$ depending on $\psi$ and $[L]\in H_1(S_{std})$.
\end{Prop}

\subsection{Decomposition into elementary diffeomorphisms}\label{sec:abstract_spin_homology_decomposition}
The independence of the map $\widetilde{KhR}_2^+(\phi)$ in Theorem~\ref{thm:spin_repara_las} on the choice of $\tilde\phi$ follows from Proposition~\ref{prop:D2S2_repara}. Thus, assuming Proposition~\ref{prop:D2S2_repara}, Theorem~\ref{thm:spin_repara_las} can be regarded as a statement for the pair $(L,\tilde\phi)$.

If Theorem~\ref{thm:spin_repara_las} holds for $(L,\tilde\phi)$ and $(\phi(L),\tilde\phi')$, then it holds for $(L,\tilde\phi'\circ\tilde\phi)$ as well. Therefore, it suffices to decompose any $\tilde\phi$ with $\phi$ being $L$-admissible into a composition of elementary diffeomorphisms $\tilde\phi_m\circ\cdots\circ\tilde\phi_1$ with each $\phi_i$ being $\phi_{i-1}\circ\cdots\circ\phi_1(L)$-admissible, and prove Theorem~\ref{thm:spin_repara_las} for such elementary diffeomorphisms.

\begin{Prop}\label{prop:diff_spin_S_std}
The spin mapping class group of $S_{std}$, $\pi_0(\Diff^{spin}(S_{std}))$, is generated by mapping classes represented by the following spin diffeomorphisms:
\begin{enumerate}[(i)]
\item Switching the $i$-th and $i'$-th connected component when $m_i=m_{i'}$, $1\le i<i'\le k$;
\item Switching the $j$-th and $(j+1)$-th connected summands in the $i$-th connected component, when $1\le j<j+1\le m_i$;
\item Inverting the first connected summand in the $i$-th connected component by reflecting both the $S^1$ factor and the $S^2$ factor, when $m_i>0$;
\item Sliding the first $0$-framed surgery circle negatively over the second $0$-framed surgery circle in the $i$-th connected component, when $m_i>1$.
\end{enumerate}
\end{Prop}
\begin{proof}
This is standard. See e.g. \cite[Theorem~III.4.3]{laudenbach1974topologie}.
\end{proof}

A \textit{strongly admissible link} in $S_{std}$ is an admissible link so that on each connected component of $S_{std}$ diffeomorphic to some $\#^m(S^1\times S^2)$ with $m>0$, the diagram of the link is standard near the (projection of the) collective surgery regions (say $(0.8,m+0.2)\times(-0.02,0.02)\subset\R^2$) in that no strand passes through the regions between adjacent surgery regions; see Figure~\ref{fig:strong_admissibility}. A diffeomorphism $\phi$ of $S_{std}$ is \textit{$L$-strong-admissible} if $\phi(L)$ is strongly admissible.

\begin{figure}
\centering
\includegraphics[width=0.6\linewidth]{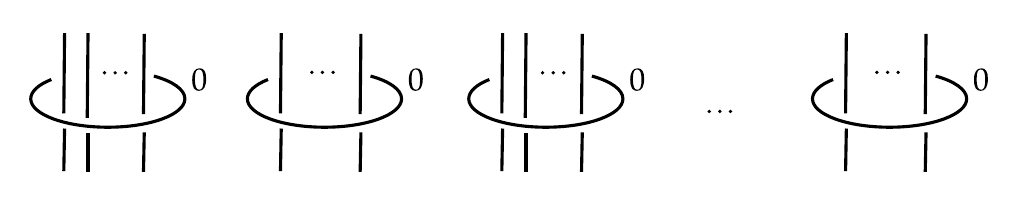}
\caption{Diagram of a strongly admissible link near the surgery regions.}
\label{fig:strong_admissibility}
\end{figure}

\begin{Prop}\label{prop:diffeomorphism_decomposition}
Every diffeomorphism $\tilde\phi\in\Diff^+(D_{std})$ can be decomposed into a composition of some elementary diffeomorphisms of the following forms.
\begin{enumerate}[(i)]
\item Diffeomorphisms rel boundary;\label{item:rel_boundary}
\item Isotopy insertions: a diffeomorphism that is trivial outside a collar neighborhood $[-1,0]\times S_{std}$ of the boundary, and is an isotopy starting from $\mathrm{id}_{S_{std}}$ in this parametrized collar neighborhood;\label{item:isotopy_insertion}
\item Connected summand rearrangements: exchanging the $i$-th and $i'$-th connected summands, when $m_i=m_{i'}$, $1\le i<i'\le k$;\label{item:exchange_1}
\item Boundary summand rearrangements: exchanging the $j$-th and $(j+1)$-th boundary summands in the $i$-th connected summand, when $1\le j<j+1\le m_i$;\label{item:exchange_2}
\item Inversions: inverting the first boundary summand in the $i$-th connected summand by reflecting both the $D^2$ factor and the $S^2$ factor, when $m_i>0$;\label{item:inversion}
\item Negative handleslides: sliding the first $0$-framed $2$-handle negatively over the second $0$-framed $2$-handle in the $i$-th connected summand, when $m_i>1$.\label{item:handleslide}
\end{enumerate}
Moreover, we can arrange the following:
\begin{itemize}
\item Type \eqref{item:exchange_1}\listsymbol\eqref{item:handleslide} diffeomorphisms can be taken to be of some standard forms: type \eqref{item:exchange_1} ones exchanges the $i$-th and $i'$-th $3$-handles together with the two corresponding collections of $2$-handles while preserving all other regions (cf. Figure~\ref{fig:D_std}); type \eqref{item:exchange_2}\listsymbol\eqref{item:handleslide} ones can be locally visualized on the boundary in the presence of a strongly admissible link as in Figure~\ref{fig:diffeo_moves}.
\item If $\phi=\tilde\phi|_{S_{std}}$ is $L$-admissible for an admissible link $L\subset S_{std}$, then the decomposition $\tilde\phi=\tilde\phi_m\circ\cdots\circ\tilde\phi_1$ can be chosen so that each $\phi_i$ is $\phi_{i-1}\circ\cdots\circ\phi_1(L)$-admissible, and $\phi_{i-1}\circ\cdots\circ\phi_1(L)$-strong-admissible if $\phi_{i+1}$ is of type \eqref{item:exchange_1}\listsymbol\eqref{item:handleslide}.
\end{itemize}
\end{Prop}
\begin{figure}
\centering
\includegraphics[width=0.6\linewidth]{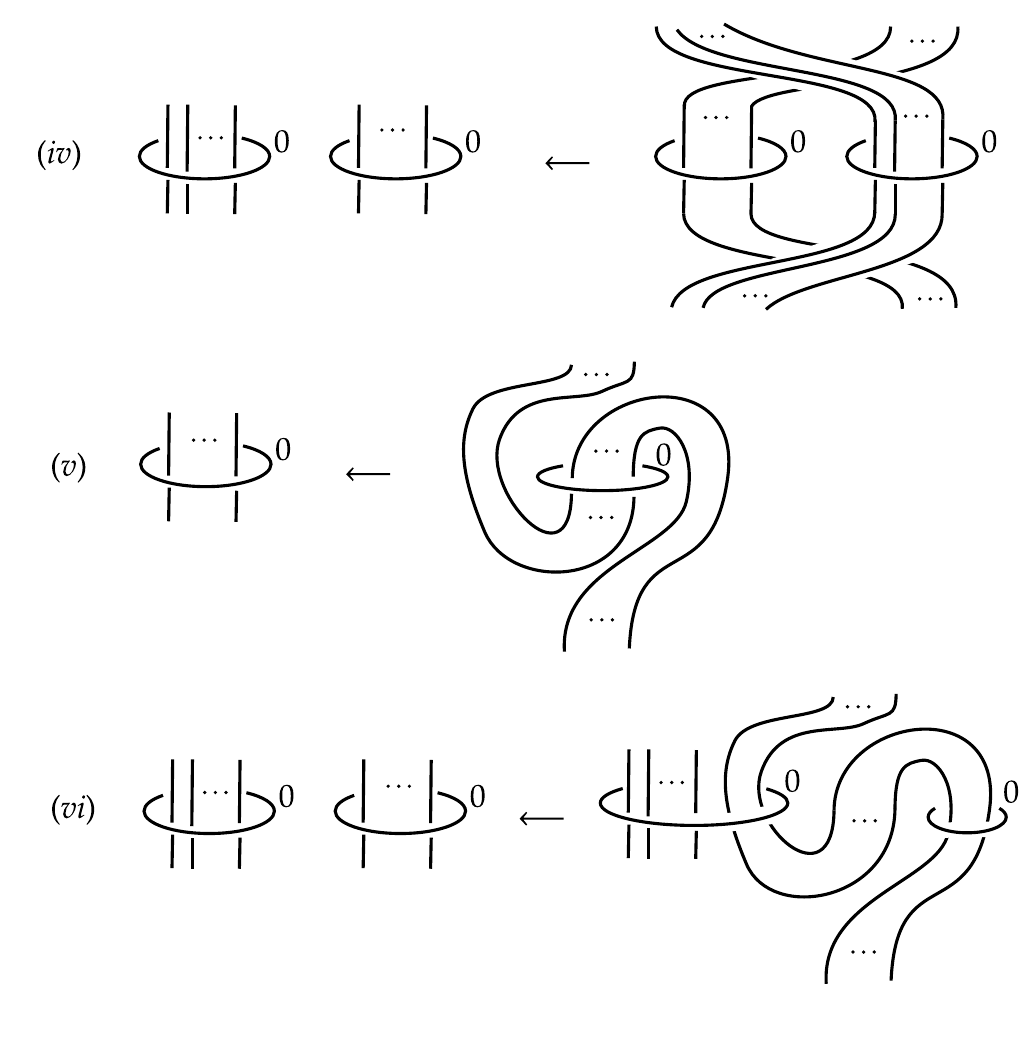}
\caption{Elementary diffeomorphisms of types \eqref{item:exchange_2}\listsymbol\eqref{item:handleslide} in Proposition~\ref{prop:diffeomorphism_decomposition} carrying the strongly admissible links on right hand sides to the admissible links on the left hand sides.}
\label{fig:diffeo_moves}
\end{figure}
\begin{proof}
Pick a sequence $\phi_1,\cdots,\phi_r\in\Diff^{spin}(S_{std})$, each of some standard form described in Proposition~\ref{prop:diff_spin_S_std}, taken to be compatible with Figure~\ref{fig:diffeo_moves}, such that $\phi_r\circ\cdots\circ\phi_1$ and $\phi$ represent the same mapping class. Pick standard lifts $\tilde\phi_1,\cdots,\tilde\phi_r\in\Diff^+(D_{std})$. Then $\tilde\phi=\tilde\phi_{r+2}\circ\cdots\circ\tilde\phi_1$ for some type \eqref{item:isotopy_insertion} diffeomorphism $\tilde\phi_{r+1}$ and some type \eqref{item:rel_boundary} diffeomorphism $\tilde\phi_{r+2}$. The (strong) admissibility conditions on the decomposition in the presence of an admissible link can be achieved by inserting extra type \eqref{item:isotopy_insertion} diffeomorphisms between each pair of adjacent $\tilde\phi_i$'s as well as before $\tilde\phi_1$.
\end{proof}

\subsection{The diffeomorphism group of \texorpdfstring{$\#\natural(D^2\times S^2)$}{D2*S2} rel boundary}
In order to prove Proposition~\ref{prop:D2S2_repara}, we need to understand $\Diff_\partial(D_{std})$, the diffeomorphism group of $D_{std}$ rel boundary.

Write $m=\sum_{i=1}^km_i$, the second Betti number of $D_{std}$. Let $S_j$ denote the sphere obtained from the $j$-th $2$-core capped off by a standard disk in $B^4$, and $C_j$ denote the $j$-th $2$-cocore of $D_{std}$, for $j=1,\cdots,m$. The boundary belt circles of the cocores, denoted $U_1,\cdots,U_m$, are equipped with the framings coming from the cocores. The intersection number can be defined between any two homology classes in the set $H_2(D_{std})\cup(\cup_{j=1}^mH_2^{U_j}(D_{std}))$, where $H_2^{U_j}(D_{std})\subset H_2(D_{std},U_j)$ denotes the preimage of $[U_j]\in H_1(U_j)$ under the connecting homomorphism $H_2(D_{std},U_j)\to H_1(U_j)$. When evaluating on the sequence $([S_1],\cdots,[S_m],[C_1],\cdots,[C_m])$, the intersection pairing takes the matrix form $\left(\begin{smallmatrix}0&I_m\\I_m&0\end{smallmatrix}\right)$, where $I_m$ is the identity $m\times m$ matrix.

If $\psi\in\Diff_\partial(D_{std})$, then $\psi_*=\mathrm{id}$ on $H_2(D_{std})$. However, this need not be the case on $H_2^{U_j}(D_{std})$. With respect to the sequence $[S_j],[C_j]$, $\psi_*$ takes a matrix form $\left(\begin{smallmatrix}I_m&X\\0&I_m\end{smallmatrix}\right)$ for some $m\times m$ integral matrix $X$. Since $\psi_*$ respects the intersection pairing, $X\in\mathfrak o(m;\Z)$ is an $m\times m$ integral skew-symmetric matrix. We have constructed a group homomorphism $$h_1\colon\Diff_\partial(D_{std})\to\mathfrak o(m;\Z).$$
On the other hand, the set of spin structures on $D_{std}$ rel the standard spin structure on $S_{std}$, denoted $Spin_\partial(D_{std})$, is affine over $H^1(D_{std},S_{std};\Z/2)\cong H_3(D_{std};\Z/2)\cong(\Z/2)^{k-1}$. Every element in $\Diff_\partial(D_{std})$ acts on $Spin_\partial(D_{std})$ as a translation by some class in $H_3(D_{std};\Z/2)$, giving rise to a map $$h_2\colon\Diff_\partial(D_{std})\to H_3(D_{std};\Z/2),$$ which is a group homomorphism since $\Diff_\partial(D_{std})$ acts trivially on $H_3(D_{std};\Z/2)$, noting that $H_3(D_{std};\Z/2)$ is generated from the boundary.

Recall from the introduction that $\Diff_{\partial,loc}(D_{std})$ denotes the subgroup of $\Diff_{\partial}(D_{std})$ consisting of diffeomorphisms that can be isotoped rel boundary to be supported in a local $4$-ball.

\begin{Thm}\label{thm:diff_D2S2_rel_boundary}
The group homomorphism $h=h_1\times h_2$ fits into the short exact sequence $$1\to\Diff_{\partial,loc}(D_{std})\to\Diff_\partial(D_{std})\xrightarrow{h}\mathfrak o(m;\Z)\times H_3(D_{std};\Z/2)\to1.$$
\end{Thm}

\begin{proof}
By choosing a local $4$-ball disjoint from all $S_j,C_j$, we see that $\Diff_{\partial,loc}(D_{std})\subset\ker(h)$.

We show the exactness at the last term. The map $h_2$ is surjective because we can insert Dehn twists along the $3$-core spheres of $D_{std}$. To prove the surjectivity of $h_1$, it suffices to exhibit for each $1\le i<j\le m$ a diffeomorphism $\beta_{i,j}\in\Diff_\partial(D_{std})$ that sends $[C_i]$ to $[C_i]+[S_j]$, $[C_j]$ to $[C_j]-[S_i]$, and each other $[C_k]$ to itself.

The union of the $0$-handle, the $i$-th $2$-handle and the $j$-th $2$-handle in $D_{std}$ form a standard $\natural^2(D^2\times S^2)$. Thus it suffices to exhibit a diffeomorphism $\beta=\beta_{1,2}\in\Diff_\partial(\natural^2(D^2\times S^2))$ with $\beta_*[C_1]=[C_1]+[S_2]$, $\beta_*[C_2]=[C_2]-[S_1]$. This is provided by the barbell diffeomorphism defined by Budney--Gabai \cite[Section~5]{budney2019knotted}. For our purposes, we give an alternative description of this diffeomorphism below.

See Figure~\ref{fig:barbell}. Consider the positive
Hopf link $U_1\cup U_2\subset B^3$. Let $D_i$ be a spanning disk for $U_i$ that intersects $U_{i+1}$ transversely at one point, $i=1,2$, index modulo $2$. The manifold $\natural^2(D^2\times S^2)$ is obtained from attaching $0$-framed $2$-handles along $\{i\}\times U_i$, $i=1,2$, in $\partial B^4$, where $B^4=[1,2]\times B^3$. Let $K_1,K_2$ denote the cores of the $2$-handles, with the usual orientation convention that $\partial K_i=-\{i\}\times U_i$, $i=1,2$. The cocores of $\natural^2(D^2\times S^2)$ are given by $C_1:=-\{1\}\times D_2$, $C_2:=\{2\}\times D_1$ (interior slightly pushed into the interior of $B^4$). Define disks $C_1':=(I\times U_2)\cup K_2$, $C_2':=(I\times U_1)\cup(-K_1)$. Then $[C_1']=[C_1]+[S_2]$, $[C_2']=[C_2]-[S_1]$. One can check, for example by cutting along $\partial[1,2]\times B^3$ and regluing, that the complement of $C_1'\cup C_2'$ in $\natural^2(D^2\times S^2)$ is diffeomorphic to $B^4$ via some diffeomorphism canonical up to isotopy rel boundary. Thus there is a diffeomorphism $\beta\in\Diff_\partial(\natural^2(D^2\times S^2))$ that carries $C_i$ to $C_i'$, as desired. The mapping class of $\beta$ is well-defined.

\begin{figure}
\centering
\includegraphics[width=0.7\linewidth]{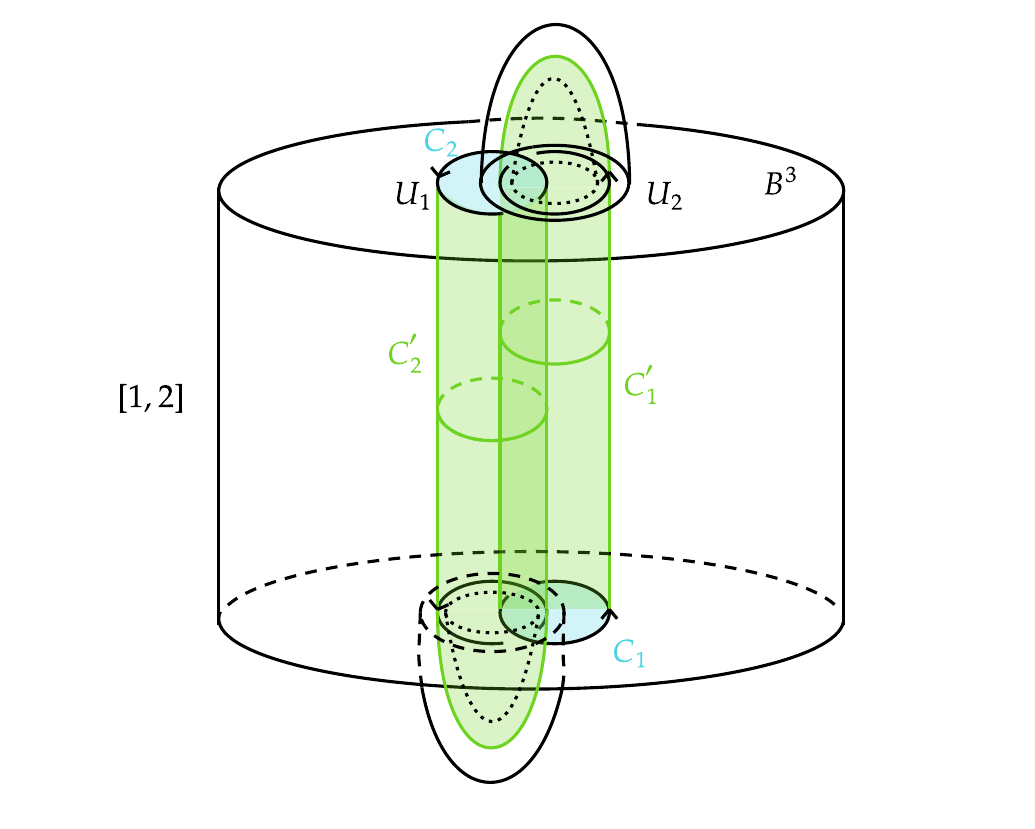}
\caption{The barbell diffeomorphism of $\natural^2(D^2\times S^2)$, determined by sending $C_j$ to $C_j'$, $j=1,2$.}
\label{fig:barbell}
\end{figure}

We show the exactness at the middle term. Supposing $\psi\in\ker(h)$, we show that $\psi$ can be isotoped to be supported in a local $4$-ball. By the relative Hurewicz theorem, each $\psi(C_j)$ is homotopic to $C_j$. Since the $C_j$'s admit disjoint dual $2$-spheres, and have trivial normal bundles rel boundary, Gabai's $4$-dimensional lightbulb theorem \cite[Theorem~10.1]{gabai20204} (see also \cite[Theorem~0.6]{gabai2021self}) implies that $\psi$ is isotopic rel boundary to some $\psi'$ that fixes each neighborhood of $C_j$ pointwisely. The complement of a neighborhood of $\partial D_{std}\cup(\cup_jC_j)$ is a $4$-sphere with $k$ open $4$-balls removed. Choose $k-1$ disjoint embedded arcs in this holed $S^4$ connecting the holes. Isotope $\psi'$ to some $\psi''$ rel the exterior of this holed $S^4$ that is identity on the $k-1$ arcs. Since $\psi$ acts trivially on $Spin_\partial(D_{std})$, we may further assume that $\psi''$ preserves the framings of the arcs, hence fixes a neighborhood of the arcs. It follows that $\psi''$ is supported in a local $4$-ball, as desired.
\end{proof}

\subsection{Diffeomorphisms rel boundary}\label{sec:diff_rel_bdy}
In this section we prove Proposition~\ref{prop:D2S2_repara}, and consequently Theorem~\ref{thm:spin_repara_las} for type \eqref{item:rel_boundary} diffeomorphisms in Proposition~\ref{prop:diffeomorphism_decomposition}. Since the statement respects compositions in $\psi$, it suffices to check on a set of generators of $\Diff_\partial(D_{std})$, which by the proof of Theorem~\ref{thm:diff_D2S2_rel_boundary} can be taken to be consisting of elements of $\Diff_{\partial,loc}(D_{std})$, Dehn twists along $3$-spheres, and barbell diffeomorphisms implanted standardly in $D_{std}$.

If $\psi\in\Diff_{\partial,loc}(D_{std})$, then $\mathcal S_0^2(\psi)=\mathrm{id}$ because we can localize $\psi$ to avoid any given skein. If $\psi$ is a $3$-sphere twist, by putting lasagna fillings in a general position with respect to the twisting sphere, neck-cutting, twisting and regluing, we know that $\mathcal S_0^2(\psi)=\mathrm{id}$ (here, we used the fact that $\pi_1(\Diff(S^3))=\Z/2$ acts trivially on Khovanov-Rozansky $\gl_2$ homology).

Let $\beta_{i,j}\in\Diff_\partial(D_{std})$ denote the barbell diffeomorphism implanted to the union of the $0$-handle and the $i,j$-th $2$-handles of $D_{std}$, $i<j$. It remains to check Proposition~\ref{prop:D2S2_repara} for $\beta_{i,j}$.

We formulate the following special case of Proposition~\ref{prop:D2S2_repara}. A \textit{standard belt link} is an admissible link in $S_{std}$ which is some parallel cable of the union of the belt circles of the $2$-cocores in $D_{std}$, with various orientations, that takes a standard form as indicated locally on the left of Figure~\ref{fig:belt_link}. A \textit{belt link} is a framed oriented link in $S_{std}$ isotopic to a standard belt link. For a standard belt link $U$, define $1\in\widetilde{KhR}_2^+(U)$ to be the image of the class $1$ ($=1\otimes\cdots\otimes1$) in $KhR_2$ of $U$ viewed as an unlink in $\sqcup_{i=1}^kS^3$, under the unit maps that create Rozansky projectors at the surgery regions (see the right of Figure~\ref{fig:belt_link}).

\begin{figure}
\centering
\includegraphics[width=0.7\linewidth]{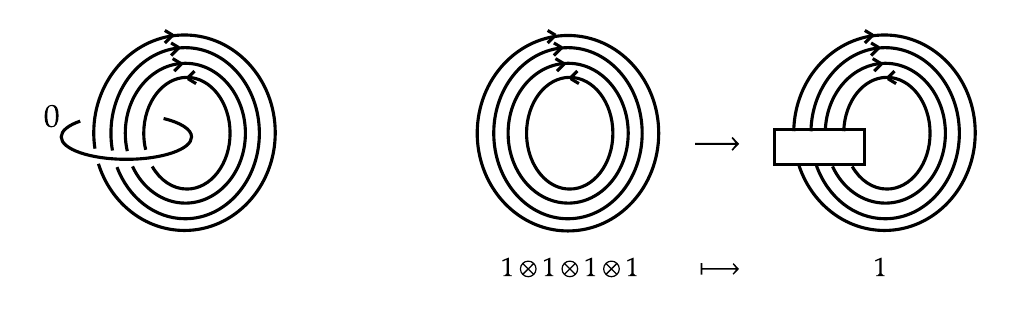}
\caption{Left: Diagram of a standard belt link near a surgery region. Right: The element $1$ in the Rozansky-Willis homology of a standard belt link.}
\label{fig:belt_link}
\end{figure}

\begin{Lem}\label{lem:barbell_las}
Let $U\subset S_{std}$ be a standard belt link. The pushforward map $\mathcal S_0^2(\beta_{i,j})\colon\mathcal S_0^2(D_{std};U)\to\mathcal S_0^2(D_{std};U)$, under the isomorphism \eqref{eq:SZ}, sends $1\otimes1$ to $1\otimes\mathrm{id}_\alpha(1)$ plus terms with lower lasagna quantum gradings, for some $\alpha$ depending on $i,j$, and $[U]\in H_1(S_{std})$.
\end{Lem}

Note that by the description of \eqref{eq:SZ} in Section~\ref{sec:SZ}, the element $1\otimes1\in\mathcal S_0^2(D_{std};U)$ is represented by the skein given by the collection of cocores that caps $U$ off.

\begin{proof}
Let $C$ denote the disjoint union of $2$-cocores in $D_{std}$ that cap $U$ off. As observed, $1\otimes1$ is represented by the skein $C$ with empty decoration. The diffeomorphism $\beta_{i,j}$ preserves each cocore except those parallel to the $i,j$-th $2$-cocores $C_i,C_j$ of $D_{std}$, which are sent to disks parallel to $C_i',C_j'$ shown as $C_1',C_2'$ in Figure~\ref{fig:barbell}.

The element $\mathcal S_0^2(\beta_{i,j})(1\otimes1)$ is represented by the skein $\beta_{i,j}(C)$ with empty decoration. We track this element in the sequence of isomorphisms relating $\mathcal S_0^2(D_{std};U)$ to $\widetilde{KhR}_2^+(U)\otimes\mathcal S_0^2(D_{std})$ in Section~\ref{sec:SZ}.

Let $n_+,n_-$ denote the number of positively, negatively oriented $i$-th belt circles in $U$ and $m_+,m_-$ denote similarly those numbers of $j$-th belt circles. Thus, the homology class of $\beta_{i,j}(C)$ is $\alpha_U+(n_+-n_-)[S_j]-(m_+-m_-)[S_i]$.

Delete a slightly shrunken $0$-handle from $D_{std}$ and evaluate the skein $\beta_{i,j}(C)$ inside this ball. We see that in terms of row \eqref{eq:SZ_MN}\footnote{Technically, in the description of the isomorphism \eqref{eq:SZ_MN}, the standard skein comes with $k$ input balls instead of a single input ball. This difference is insignificant.}, $\mathcal S_0^2(\beta_{i,j})(1\otimes1)\in S_0^2(D_{std};U;\alpha_U+(m_--m_+)[S_i]+(n_+-n_-)[S_j])$ is represented by symmetrized $1\otimes\mathrm{coev}(1)\in KhR_2(U')\otimes KhR_2(H_+^{(m,n)}\sqcup\overline{H_+^{(m,n)}})\cong KhR_2(U\cup(m_-e_i+n_+e_j,m_+e_i+n_-e_j)\text{ belts})$ in the colimit summand $r=\min(m_+,m_-)e_i+\min(n_+,n_-)e_j$. Here $U'\subset S^3$ are components of $U$ that are not parallel to the $i,j$-th belt circles, $H_+^{(m,n)}$ is the cable of the positive Hopf link defined as indicated in Figure~\ref{fig:hopf_cable}, and $\mathrm{coev}\colon\Q\to KhR_2(H_+^{(m,n)}\sqcup\overline{H_+^{(m,n)}})$ is the coevaluation map. The copy $H_+^{(m,n)}$ is thought of as the components of $U$ that are parallel to the $j$-th belt circle together with $(n_+,n_-)$ belts coming from components of $\beta_{i,j}(C)$ parallel to $C_i'$, and a similar description applies to the mirror copy $\overline{H_+^{(m,n)}}$. See Figure~\ref{fig:skein_beta_C} for a sketch of $\beta_{i,j}(C)$, especially the induced orientation on the $H_+^{(m,n)}\sqcup\overline{H_+^{(m,n)}}$ part of the input link.

\begin{figure}
\centering
\includegraphics[width=0.75\linewidth]{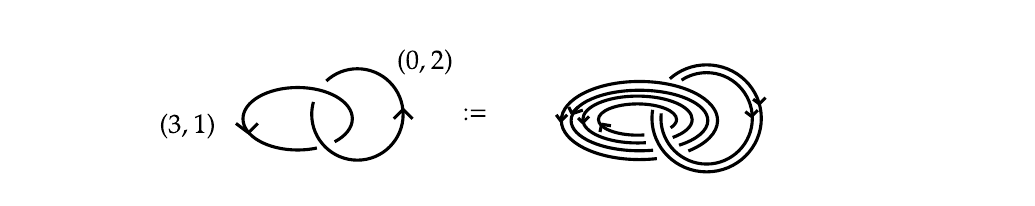}
\caption{The link $H_+^{(m,n)}\subset S^3$ for $(m_+,m_-)=(3,1)$, $(n_+,n_-)=(0,2)$.}
\label{fig:hopf_cable}
\end{figure}

\begin{figure}
\centering
\includegraphics[width=0.65\linewidth]{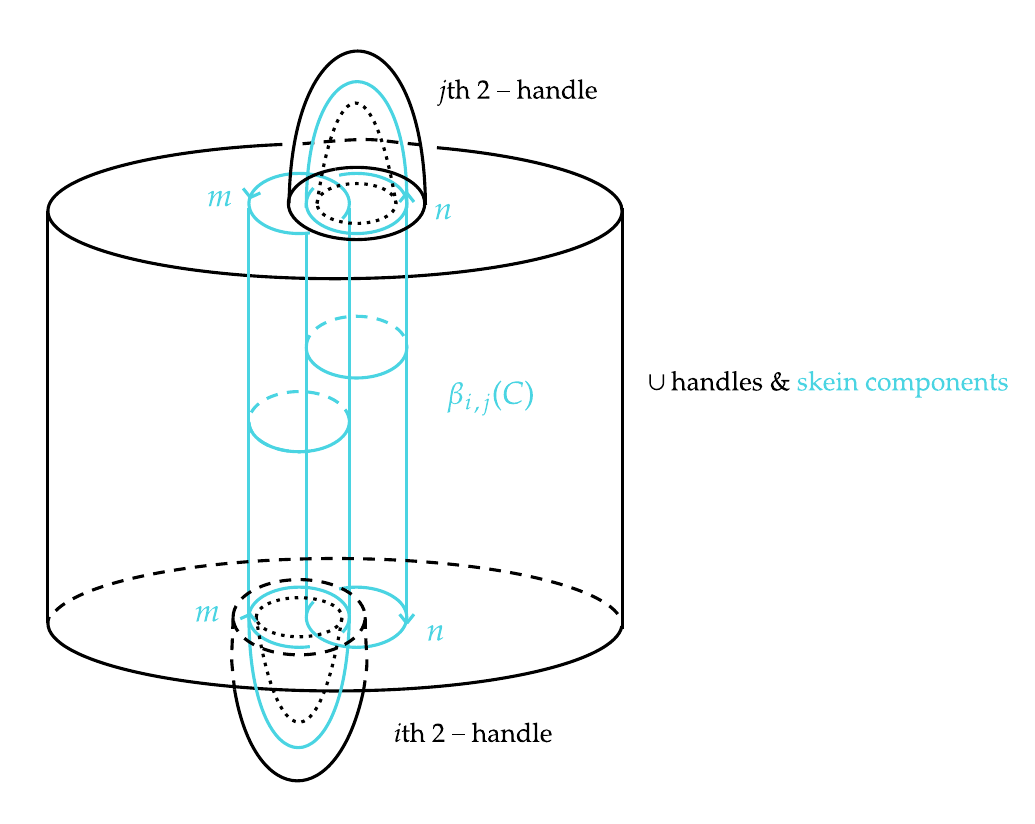}
\caption{The skein $\beta_{i,j}(C)$ is an $n$-cable on $C_i'$ union an $m$-cable on $C_j'$, union other cocore components.}
\label{fig:skein_beta_C}
\end{figure}

Tracing $\mathcal S_0^2(\beta_{i,j})(1\otimes1)$ further down \eqref{eq:SZ_renom}\listsymbol\eqref{eq:SZ_slide}, we see that the corresponding elements in these three rows (before symmetrization) are given by $1\in\widetilde{KhR}_2^+(U')$ tensor the image of $1\in\widetilde{KhR}_2^+(\emptyset)$ under the composition of maps shown in Figure~\ref{fig:barbell_check}.
\begin{figure}
\centering
\includegraphics[width=0.75\linewidth]{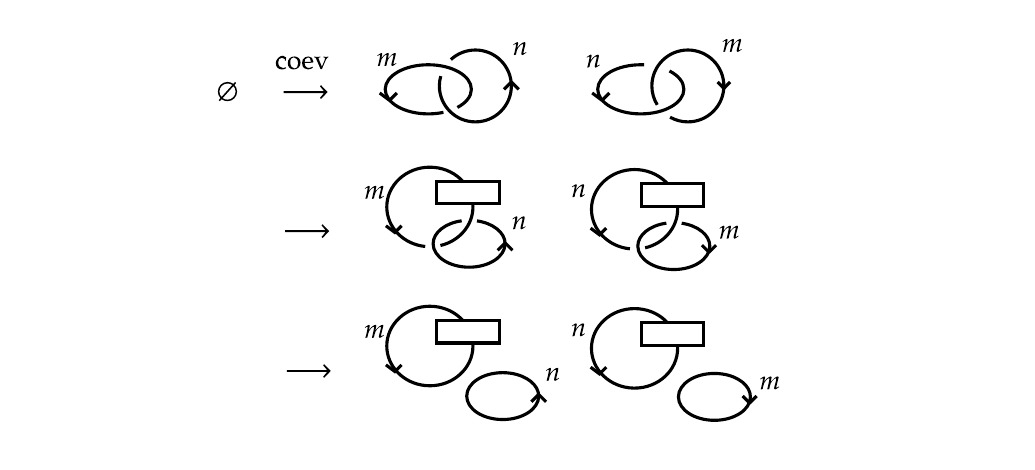}
\caption{The barbell move check.}
\label{fig:barbell_check}
\end{figure}

In order to prove the lemma, with $\alpha=(m_--m_+)e_i+(n_+-n_-)e_j$, in view of Example~\ref{ex:D2S2}, Definition~\ref{def:shift_auto}, and the proof of \eqref{eq:SZ}, it suffices to show that the image of $1$ in the last row of Figure~\ref{fig:barbell_check} is equal to $1\otimes1\otimes \Khdot\otimes \Khdot\in\widetilde{KhR}_2^+(1_m)\otimes\widetilde{KhR}_2^+(1_n)\otimes\widetilde{KhR}_2^+(U^{n_+,n_-})\otimes\widetilde{KhR}_2^+(U^{m_-,m_+})$ plus terms with lower quantum gradings in the $\widetilde{KhR}_2^+(U^{n_+,n_-})\otimes\widetilde{KhR}_2^+(U^{m_-,m_+})$ factor. Here $1_\ell$ denotes a standard belt link in $S^1\times S^2$ with $\ell_+$ strands positively oriented and $\ell_-$ ones negatively oriented, $U^{a,b}$ has the same meaning as in \eqref{eq:SZ_slide}, and $\Khdot\in\widetilde{KhR}_2^+(U^{a,b})$ denotes the element represented by $\Khdot\otimes\cdots\otimes \Khdot$.

Following \cite{rozansky2010categorification,willis2021khovanov}, Manolescu--Marengon--Sarkar--Willis \cite[Corollary~2.2]{manolescu2023generalization} showed that the Rozansky projector can be approximated by full twists in a quantitative sense. For us, this implies that $$\widetilde{KhR}_2^{+,0,-|\ell|}(1_\ell)\cong\widetilde{KhR}_2^{+,0,-|\ell|}(T(|\ell|,|\ell|)_{\ell_+,\ell_-})\cong(Kh^{|\ell|^2/2,3|\ell|^2/2-|\ell|}(T(|\ell|,|\ell|)))^*\cong\Q,$$ where $T(|\ell|,|\ell|)_{\ell_+,\ell_-}$ denotes the positive torus link $T(|\ell|,|\ell|)$ equipped with an orientation where $\ell_-$ of the strands have the orientation reversed. Similarly, $$\widetilde{KhR}_2^{+,0,<-|\ell|}(1_\ell)\cong(Kh^{|\ell|^2/2,<3|\ell|^2/2-|\ell|}(T(|\ell|,|\ell|)))^*=0,$$ $$\widetilde{KhR}_2^{+,<0,*}(1_\ell)\cong(Kh^{>|\ell|^2/2,*}(T(|\ell|,|\ell|)))^*=0.$$ Here, the equalities about torus links $T(|\ell|,|\ell|)$ follow from the calculation of St\v osi\'c \cite[Theorem~1, Theorem~3]{stovsic2009khovanov}, suitably renormalized. Consequently, by degree reason, the image of $1$ under the maps in Figure~\ref{fig:barbell_check} is a scalar $c$ times $1\otimes1\otimes \Khdot\otimes \Khdot$ plus lower terms. Note that we used the fact that $1\in\widetilde{KhR}_2^{+,0,-|\ell|}(1_\ell)$ is nonzero, which can be verified for example by tracking the element $1\otimes\cdots\otimes1$ in the $KhR_2$ of the $|\ell|$-component unlink in the diagram in Figure~\ref{fig:element_1_belt_link}, where the top right map is an equivalence by Proposition~\ref{prop:projector_properties}(3)(5).
\begin{figure}
\centering
\includegraphics[width=0.6\linewidth]{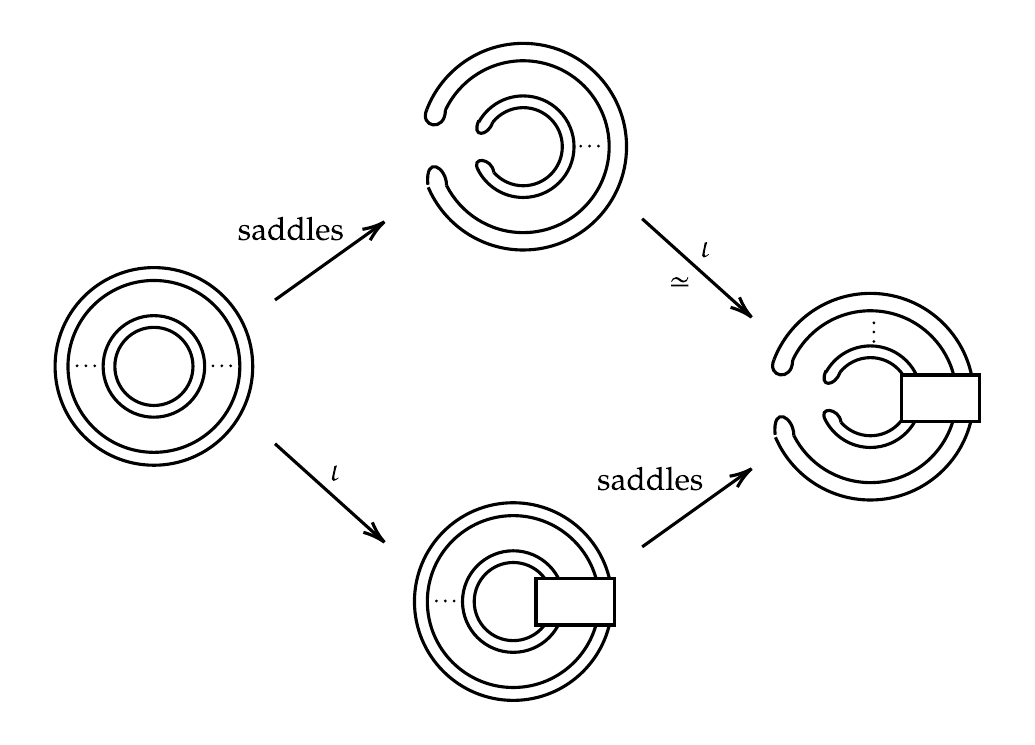}
\caption{Tracking the element $1\otimes\cdots\otimes1$.}
\label{fig:element_1_belt_link}
\end{figure}

Now we prove $c=1$. In each row of Figure~\ref{fig:barbell_check}, one may cap off the $m$ strands in the left half picture by $|m|/2$ dotted annuli, and the $n$ strands in the right half picture by $|n|/2$ dotted annuli. For the strands with Rozansky projectors, capping off means following the right half of Figure~\ref{fig:element_1_belt_link} to go from the bottom term to the top term, then capping circles off by dotted caps. This capping off procedure commutes (up to sign) with the maps of Figure~\ref{fig:barbell_check} (passing from the middle row to the last row of Figure~\ref{fig:barbell_check} requires a termwise check, and technically we have termwise sign ambiguities; this will be fixed in Appendix~\ref{sec:sign_barbell}), so that we can consider the image of $1\in\widetilde{KhR}_2^+(\emptyset)$ in two ways. If we perform all of the maps in Figure~\ref{fig:barbell_check} before capping off, we see  the element $c(1\otimes1\otimes \Khdot\otimes \Khdot)$ mapping to $\pm c(\Khdot\otimes \Khdot)\in\widetilde{KhR}_2^+(U^{n_+,n_-}\sqcup U^{m_-,m_+})$ after capping off. If instead we cap off immediately in the first row, the unit maps and belt pull-offs are identity maps and we see a cobordism $\emptyset\to H_+^{(m,n)}\sqcup\overline{H_+^{(m,n)}}\to U^{n_+,n_-}\sqcup U^{m_-,m_+}$ isotopic to a disjoint union of dotted annulus creations, so that $1\in\widetilde{KhR}_2^+(\emptyset)$ maps to $\pm \Khdot\otimes \Khdot$, proving that $c=\pm1$. We will fix the sign $c=1$ in Appendix~\ref{sec:sign_barbell}.
\end{proof}

We are now ready to deduce Proposition~\ref{prop:D2S2_repara} from Lemma~\ref{lem:barbell_las}.

\begin{proof}[Proof of Proposition~\ref{prop:D2S2_repara}]
As already explained, it suffices to prove the statement for $\psi=\beta_{i,j}$.

Let $x\in\mathcal S_0^2(D_{std};L)$ be represented by some lasagna filling $(\Sigma,v)$. By an isotopy rel boundary, we may assume that $\beta_{i,j}$ is supported by $\beta_{i,j}\in\Diff_\partial(D_{std}')$ where $D_{std}'$ is a shrunk copy of $D_{std}$. By an isotopy, we may assume that the input balls of $\Sigma$ are disjoint from $D_{std}'$. By neck-cutting \cite[Lemma~7.2]{manolescu2022skein}, we may assume that $\Sigma$ is disjoint from the $3$-handles of $D_{std}$. By general position and isotopy, we may assume that $\Sigma$ intersects $D_{std}'$ in a disjoint union of $2$-cocores with standard framings and various orientations, and that $U:=\Sigma\cap\partial D_{std}'$ is a standard belt link. By tubing and isotopy, we may assume that $\Sigma$ has a single input ball, which is contained in a collar neighborhood $[0,1]\times S_{std}$ of $\partial D_{std}'$ outside $D_{std}'$ in which $\Sigma$ takes the form $([0,1]\times U)\cup([1/2,1]\times L_0)$ where $L_0$ is the input link of $\Sigma$, identified with a local link in $S_{std}$ distant from $U$ and the surgery regions. Let $D_{std}''=D_{std}'\cup([0,1]\times S_{std})$ and let $\Sigma^\circ$ denote the part of $\Sigma$ in $D_{std}\backslash int(D_{std}'')=[1,2]\times S_{std}$. We have the commutative diagram
\begin{equation}\label{eq:localize_barbell_las}
\begin{tikzcd}
\mathcal S_0^2(D_{std}';U)\ar[rr,"\mathcal S_0^2(\beta_{i,j})"]\arrow[d,"v\otimes\bullet"']&&\mathcal S_0^2(D_{std}';U)\arrow[d,"v\otimes\bullet"]\\
\mathcal S_0^2(D_{std}'';U\cup L_0)\arrow[d,"\mathcal S_0^2({[1,2]}\times S_{std};\Sigma^\circ)"']&&\mathcal S_0^2(D_{std}'';U\cup L_0)\arrow[d,"\mathcal S_0^2({[1,2]}\times S_{std};\Sigma^\circ)"]\\
\mathcal S_0^2(D_{std};L)\ar[rr,"\mathcal S_0^2(\beta_{i,j})"]&&\mathcal S_0^2(D_{std};L).
\end{tikzcd}
\end{equation}

Tracing the element $1\otimes1$ in the top left corner of \eqref{eq:localize_barbell_las}, by Theorem~\ref{thm:concrete_functoriality_las} and Lemma~\ref{lem:barbell_las}, we get
$$\begin{tikzcd}
1\otimes1\ar[r,mapsto]\ar[d,mapsto]&1\otimes\mathrm{id}_\alpha(1)+\cdots\ar[d,mapsto]\\
(v\otimes1)\otimes1\ar[d,mapsto]&(v\otimes1)\otimes\mathrm{id}_\alpha(1)+\cdots\ar[d,mapsto]\\
x=\widetilde{KhR}_2^+(\Sigma^\circ)(v\otimes1)\otimes\mathrm{id}_{\alpha'}(1)+\cdots\ar[r,mapsto]&\widetilde{KhR}_2^+(\Sigma^\circ)(v\otimes1)\otimes\mathrm{id}_{\alpha+\alpha'}(1)+\cdots,
\end{tikzcd}$$
where $\cdots$ are terms with negative lasagna quantum gradings. The statement follows.
\end{proof}

\subsection{Isotopy insertions}
We prove Theorem~\ref{thm:spin_repara_las} for type \eqref{item:isotopy_insertion} diffeomorphisms in Proposition~\ref{prop:diffeomorphism_decomposition}.

By assumption, $\tilde\phi$ is supported in a collar neighborhood of the boundary, on which it takes the form $$\Phi\colon[-1,0]\times S_{std}\to[-1,0]\times S_{std},\ (t,x)\mapsto(t,\phi_t(x))$$ for some isotopy $\phi_t$ between $\mathrm{id}$ and $\phi$. By considering the action on lasagna fillings, we see that $\mathcal S_0^2(\tilde\phi^{-1})$ is equal to the gluing map $$\mathcal S_0^2(I\times S_{std};\Phi^{-1}([-1,0]\times\phi(L)))\colon\mathcal S_0^2(D_{std};\phi(L))\to\mathcal S_0^2(D_{std};L).$$ Now the statement follows from Theorem~\ref{thm:concrete_functoriality_las}, with $\widetilde{KhR}_2^+(\phi)=\widetilde{KhR}_2^+((\Phi^{-1}([-1,0]\times\phi(L)))^t)$.

\subsection{Connected summand exchanges}\label{sec:connect_summand_exchange}
We prove Theorem~\ref{thm:spin_repara_las} for type \eqref{item:exchange_1} diffeomorphisms in Proposition~\ref{prop:diffeomorphism_decomposition}.

A standard such diffeomorphism acts on a standard lasagna filling of $L$ in the sense of the explanation of \eqref{eq:SZ_MN} in Section~\ref{sec:SZ} by exchanging the $i$-th and $i'$-th input balls, carrying together the $i$-th and $i'$-th collections of $2,3$-handles and skeins within. 

Alternatively, in terms of row \eqref{eq:SZ_MN}, $\mathcal S_0^2(\tilde\phi^{-1})$ acts by exchanging the $i$-th and $i'$-th tensorial factors if one decomposes the $KhR_2$ of $L\cup\text{belts}\subset\sqcup_{i=1}^kS^3$ into a tensor product. This description propagates along the sequence of isomorphisms \eqref{eq:SZ_renom}\listsymbol\eqref{eq:SZ_slide}. We conclude that the statement holds (even without taking $0$-th associated graded maps) with $\widetilde{KhR}_2^+(\phi)\colon\widetilde{KhR}_2^+(\phi(L))\to\widetilde{KhR}_2^+(L)$ being the map that exchanges the $i$-th and $i'$-th tensorial factors and $\alpha=0$.

\subsection{Boundary summand exchanges}
We prove Theorem~\ref{thm:spin_repara_las} for type \eqref{item:exchange_2} diffeomorphisms in Proposition~\ref{prop:diffeomorphism_decomposition}.

We follow the strategy in Section~\ref{sec:concrete_finger}. We claim that the map $\mathcal S_0^2(\tilde\phi^{-1})$ in terms of rows \eqref{eq:SZ_renom}\listsymbol\eqref{eq:SZ_slide} is given by the rows in Figure~\ref{fig:exchange_diagram}, where the second maps on the second and third rows are the ``sandwich'' maps defined as before.

\begin{figure}
\centering
\includegraphics[width=0.6\linewidth]{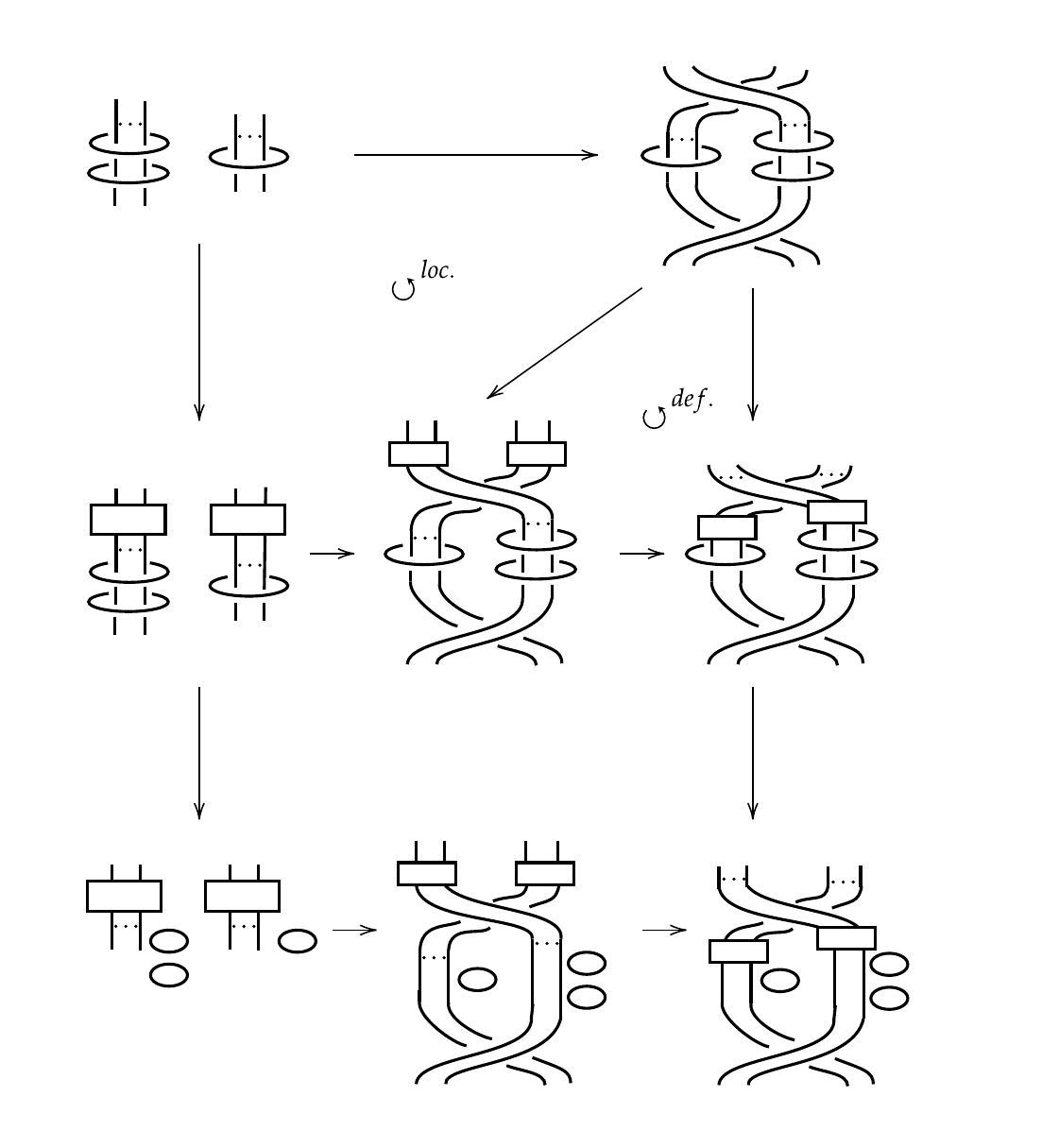}
\caption{Rows \eqref{eq:SZ_renom}\listsymbol\eqref{eq:SZ_slide} of a reverse standard boundary-summand-exchange diffeomorphism.}
\label{fig:exchange_diagram}
\end{figure}

The claim for row \eqref{eq:SZ_renom} comes from an easy examination of the isomorphism \eqref{eq:SZ_MN} as before, and the claim for row \eqref{eq:SZ_res_1} follows because the upper rectangle of Figure~\ref{fig:exchange_diagram} commutes. To show the claim for row \eqref{eq:SZ_slide}, it suffices to prove the lower rectangle of Figure~\ref{fig:exchange_diagram} commutes. By conjugating and sliding one belt-projector combination at a time, it suffices to show the commutativity of the boundary of the following diagram of isomorphisms:
\begin{figure}[H]
\centering
\includegraphics[width=0.55\linewidth]{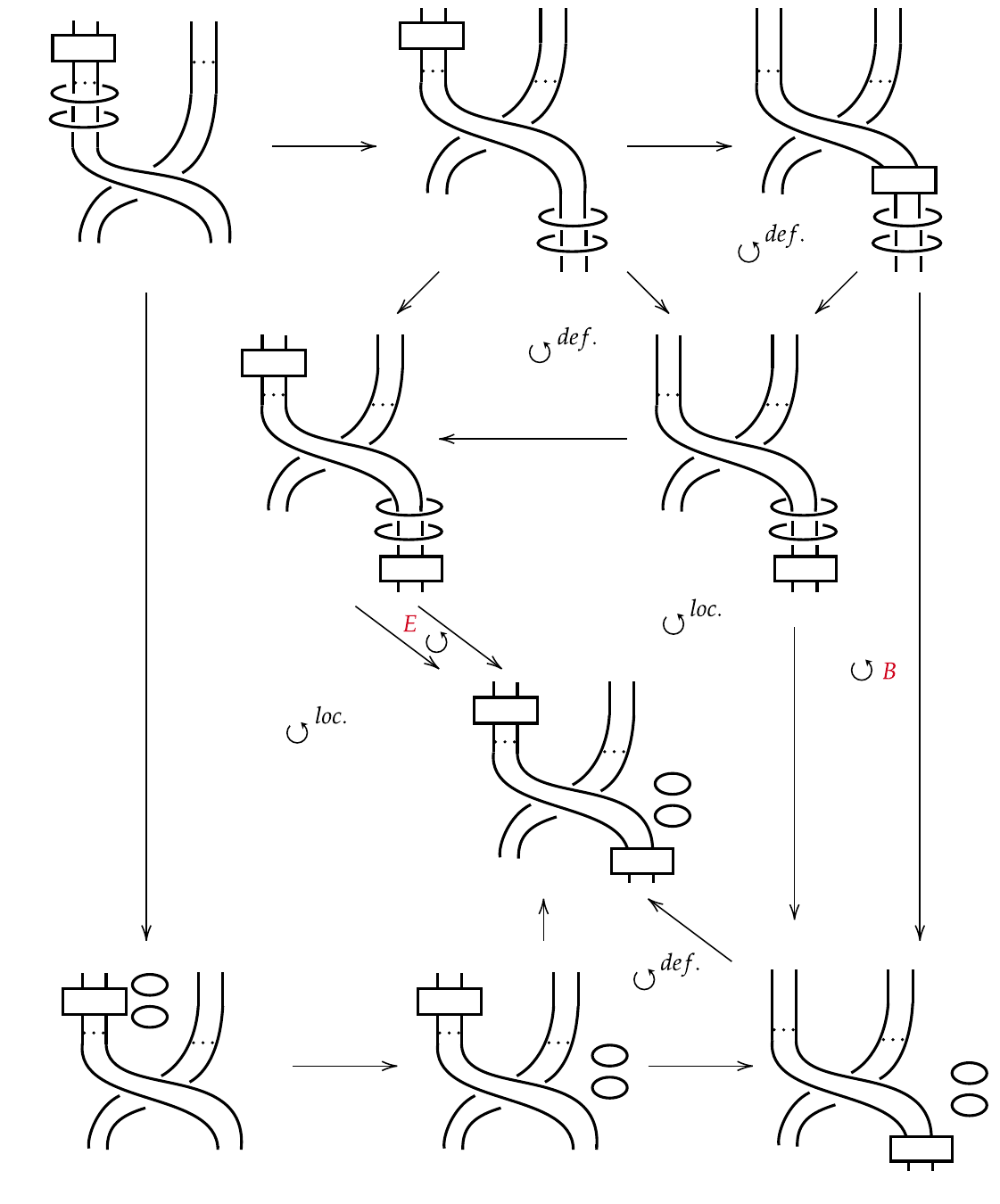}.
\end{figure}
Here, all but two regions commute by definition and locality, and the remaining two of them commute by the commutativity of region $B$ in Section~\ref{sec:concrete_finger} and a variant of region $E$ in Section~\ref{sec:concrete_pass}. The claim follows.

In particular, tracing through the isomorphism \eqref{eq:SZ_simp}, we see that Theorem~\ref{thm:spin_repara_las} holds with $\widetilde{KhR}_2^+(\phi)$ given by the composition
\begin{figure}[H]
\centering
\includegraphics[width=0.6\linewidth]{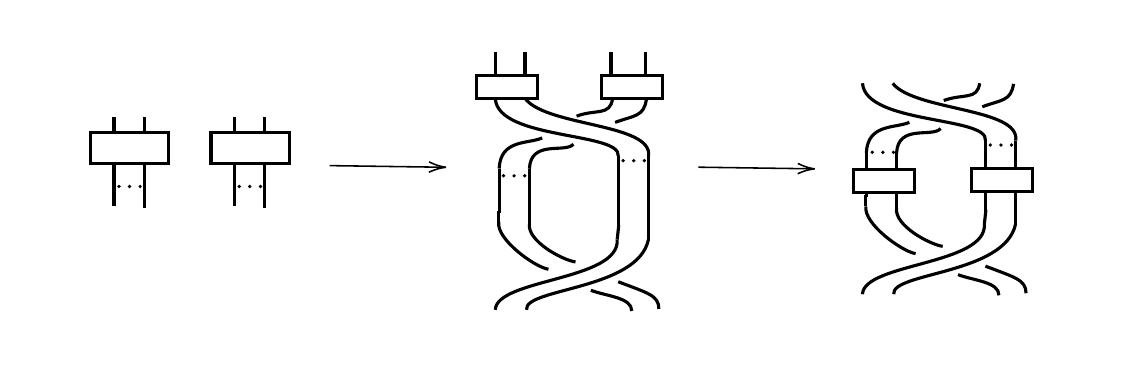}
\end{figure}
and $\alpha=0$.

\subsection{Inversions}\label{sec:inversion}
We prove Theorem~\ref{thm:spin_repara_las} for type \eqref{item:inversion} diffeomorphisms in Proposition~\ref{prop:diffeomorphism_decomposition}.

We claim that the map $\mathcal S_0^2(\tilde\phi^{-1})$ in terms of rows \eqref{eq:SZ_renom}\listsymbol\eqref{eq:SZ_slide} is given by the rows of Figure~\ref{fig:inversion_diagram}. Here the second map on the second row is given by the usual ``sandwich'' map, and the second map on the third row is induced by the chain map that moves, termwise, the unknotted circles through the empty region of the projector (see Proposition~\ref{prop:projector_properties}(1)).

\begin{figure}
\centering
\includegraphics[width=0.6\linewidth]{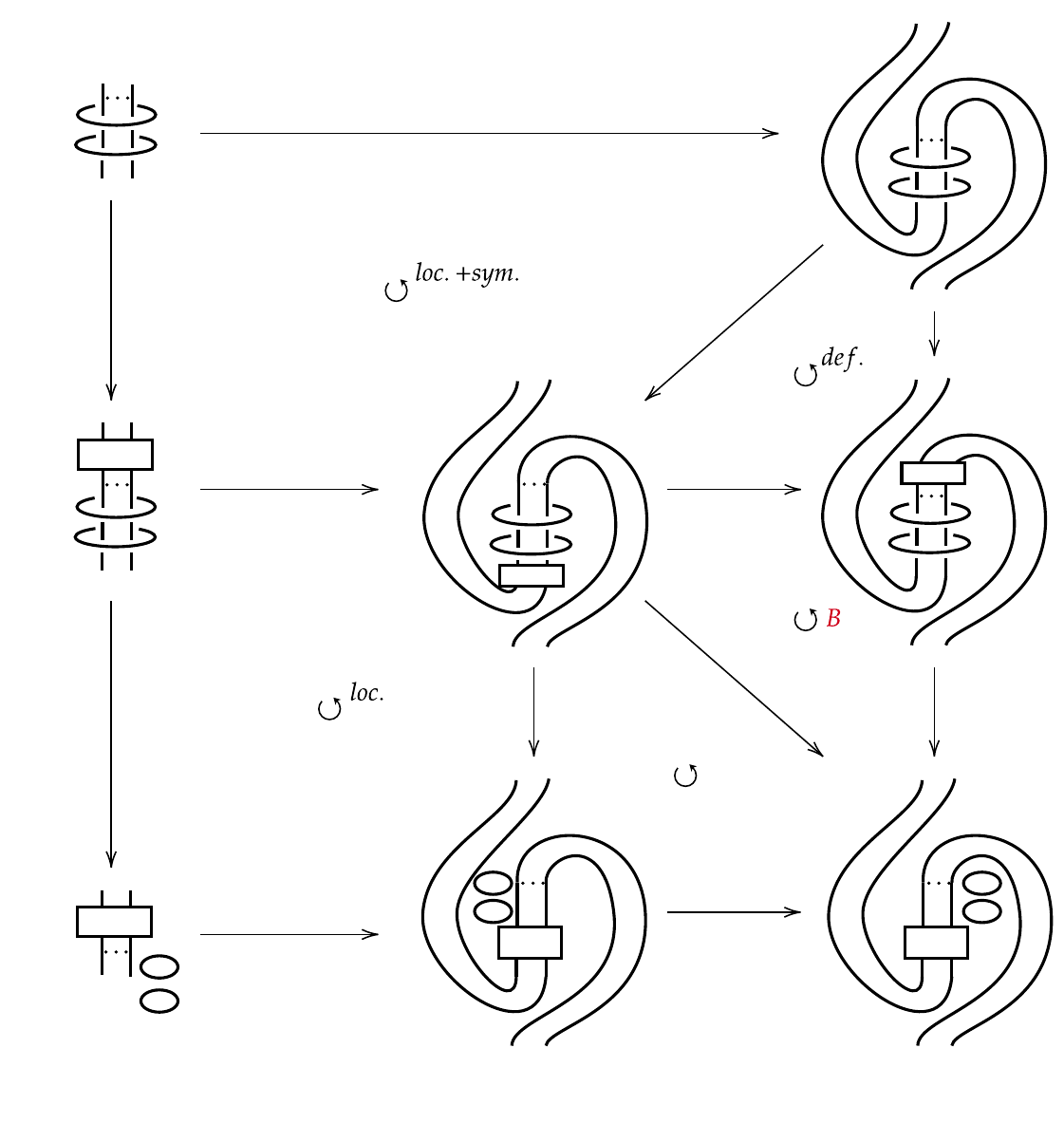}
\caption{Rows \eqref{eq:SZ_renom}\listsymbol\eqref{eq:SZ_slide} of a reverse standard inversion diffeomorphism.}
\label{fig:inversion_diagram}
\end{figure}

It suffices to justify the commutativity of the diagram. We comment on the regions whose commutativity are not immediate. The rectangle region in the upper half commutes because the Rozansky projector is symmetric in the sense of Proposition~\ref{prop:projector_properties}(8). The lower right triangle region commutes by the commutativity of region $B$ in Section~\ref{sec:concrete_finger}, while the lower-most triangle region commutes termwise on the chain level by definition, hence also on homology.

We conclude that Theorem~\ref{thm:spin_repara_las} holds with $\widetilde{KhR}_2^+(\phi)$ given by the rotation map
\begin{figure}[H]
\centering
\includegraphics[width=0.6\linewidth]{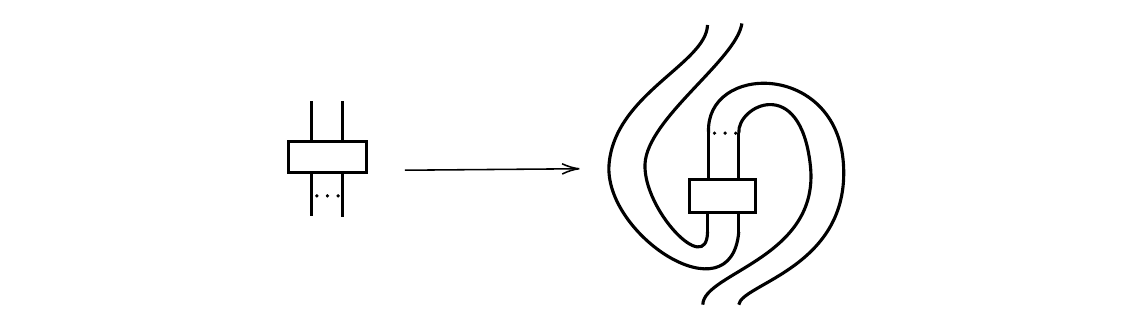}
\end{figure}
and $\alpha=0$ (note that the belt orientations get flipped after the first map in the second row of Figure~\ref{fig:inversion_diagram}, which reflects the action of $\tilde\phi^{-1}$ on $H_2(D_{std})$ appearing in the second factor of \eqref{eq:spin_repara_las}.)

\subsection{Handleslides}\label{sec:handleslides}
We prove Theorem~\ref{thm:spin_repara_las} for type \eqref{item:handleslide} diffeomorphisms in Proposition~\ref{prop:diffeomorphism_decomposition}.

For ease of drawing diagrams, we draw $\tilde\phi^{-1}$ near the handleslide region as
\begin{figure}[H]
\centering
\includegraphics[width=0.8\linewidth]{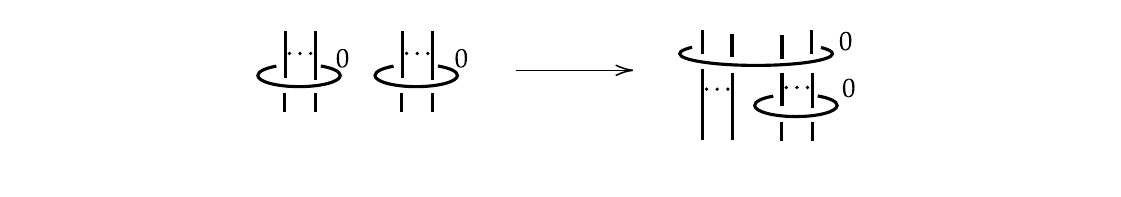}
\end{figure}
with the isotopy putting the first $2$-handle back to its standard position omitted.

We claim that $\mathcal S_0^2(\tilde\phi^{-1})$ in terms of rows \eqref{eq:SZ_renom}\listsymbol\eqref{eq:SZ_slide} is equal to rows $1,2,4$ of Figure~\ref{fig:slide_handle_diagram} postcomposed with symmetrizations. Here, the second map in the second row (and other appearances of this local picture) is given by the composition of isomorphisms (see Proposition~\ref{prop:projector_properties}(7))
\begin{equation}\label{eq:2_to_4_projector}
\vcenter{\hbox{\includegraphics[width=0.6\linewidth]{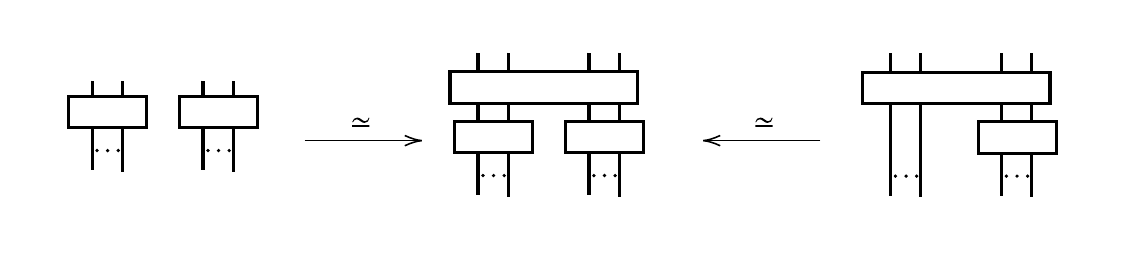}}},
\end{equation}
and the third map in the second row is the usual ``sandwich'' map.

\begin{figure}
\centering
\includegraphics[width=0.65\linewidth]{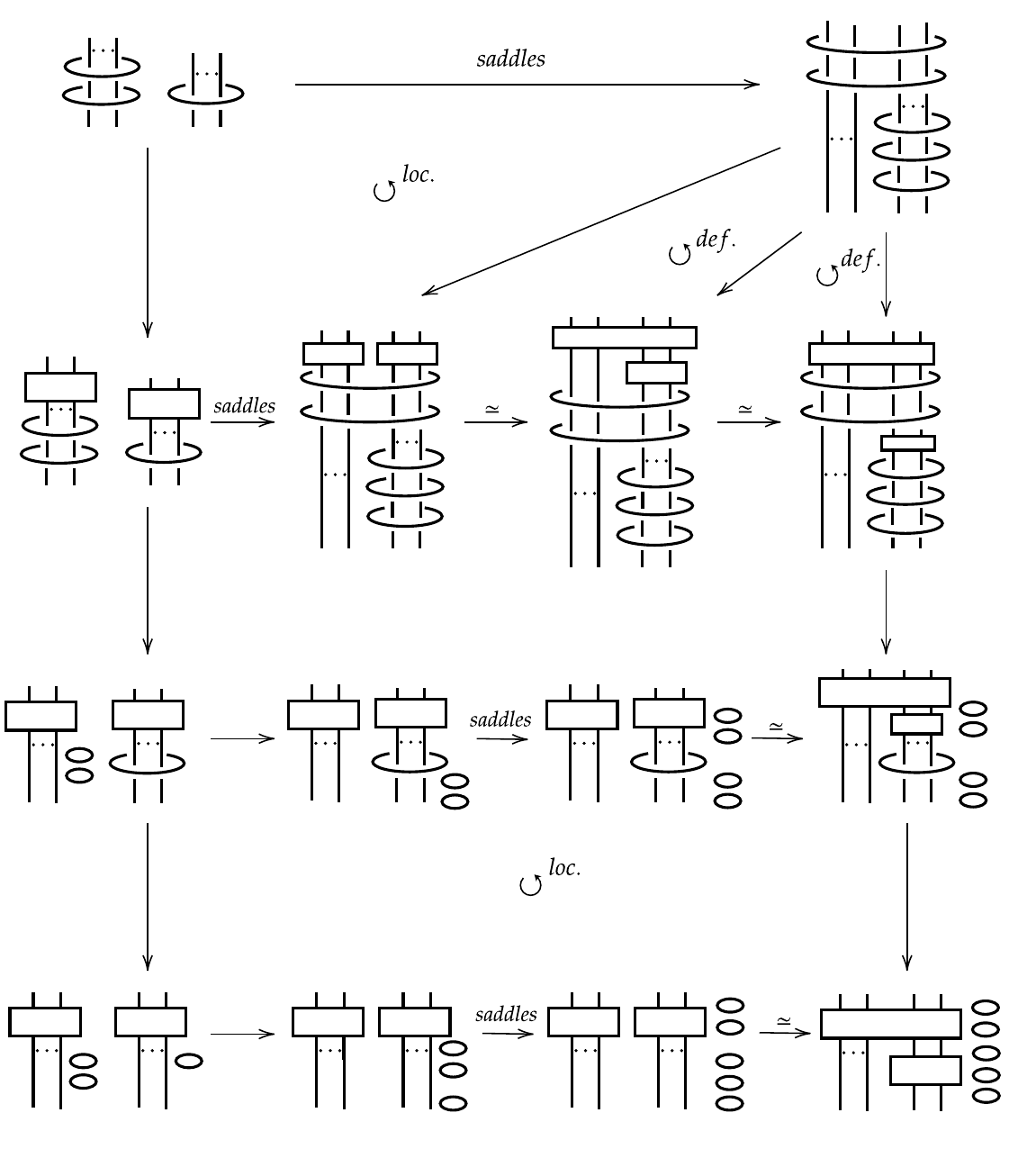}
\caption{Rows \eqref{eq:SZ_renom}\listsymbol\eqref{eq:SZ_slide} of a reverse standard negative handleslide diffeomorphism (before symmetrization), shown as rows $1,2,4$ in the figure.}
\label{fig:slide_handle_diagram}
\end{figure}

To see the claim for row \eqref{eq:SZ_renom}, start with a lasagna filling $(I\times L\cup(n_++r,n_-+r)\text{ cores},v)$, standard in the sense of the explanation of \eqref{eq:SZ_MN} in Section~\ref{sec:SZ}. When we slide the $j$-th $2$-handle over the $(j+1)$-st using $\tilde\phi^{-1}$, keep the input balls invariant and wrap the $j$-th collection of cores in the skein over the $(j+1)$-st $2$-core. After enlarging input balls of the skein, the new skein now has input link $L\cup(n_++r+(n_-)_je_{j+1},n_-+r+(n_+)_je_{j+1})\text{ belts}$, and the evaluation map performed is given by the claimed saddles.

The claim for row \eqref{eq:SZ_res_1} follows by the commutativity of the upper rectangle of Figure~\ref{fig:slide_handle_diagram}.

To show the claim for row \eqref{eq:SZ_slide}, note that the lower rectangle of Figure~\ref{fig:slide_handle_diagram} commutes, hence it remains to show the commutativity of the middle rectangle. We keep the relevant parts of this rectangle and focus on one belt at a time; thus, it suffices to show the commutativity of the boundary of the following diagram:
\begin{figure}[H]
\centering
\includegraphics[width=0.65\linewidth]{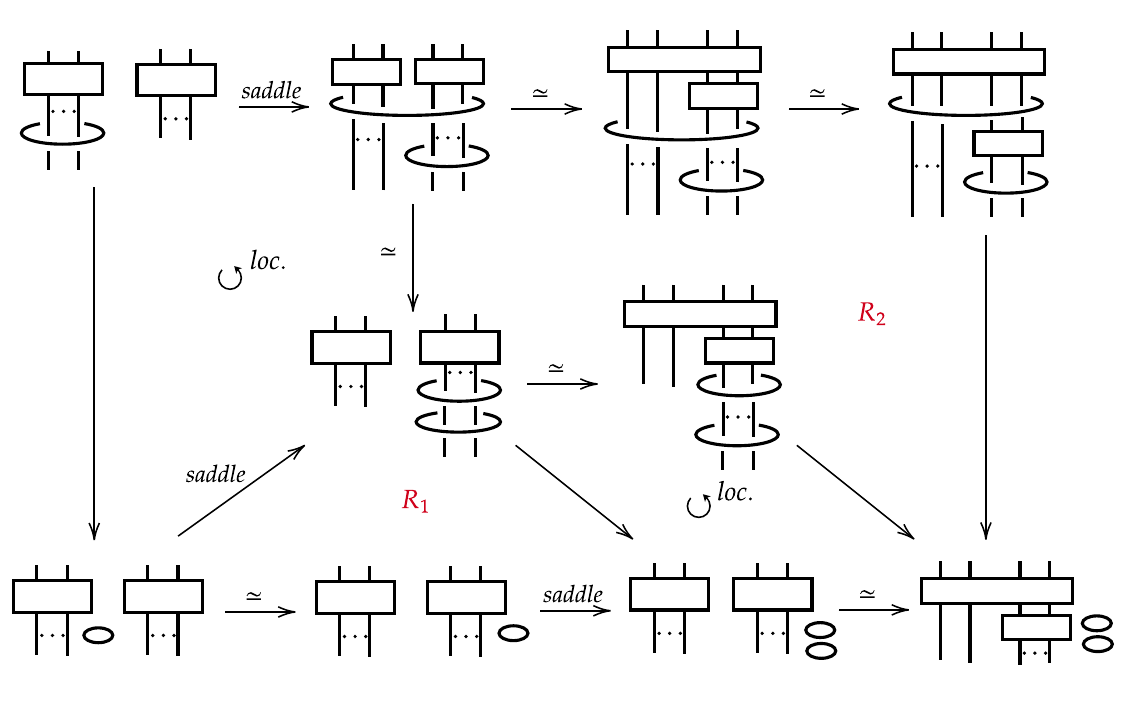}.
\end{figure}
Here, two of the rectangles commute by locality. We justify the commutativity of regions $R_1$ and $R_2$ below.

The relevant parts of region $R_1$ are redrawn as
\begin{figure}[H]
\centering
\includegraphics[width=0.7\linewidth]{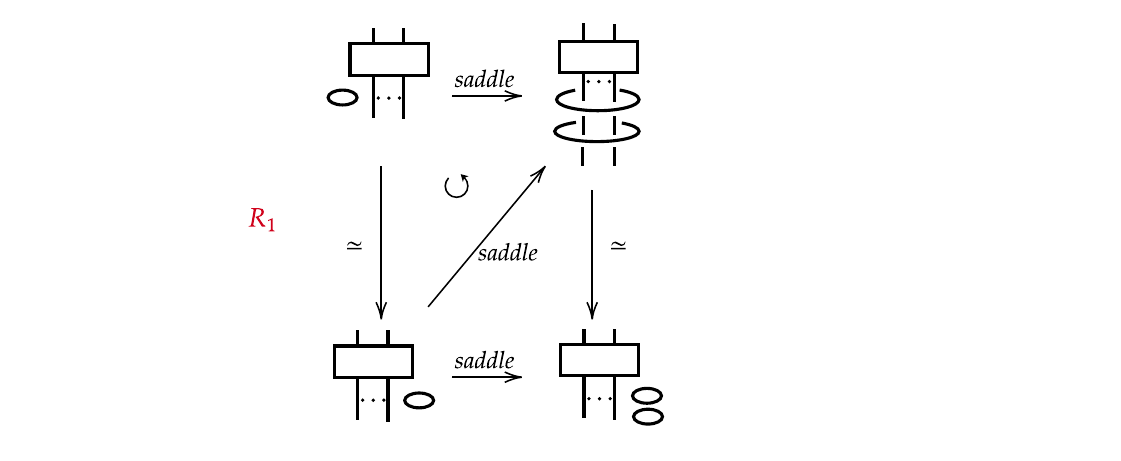},
\end{figure}
where the upper triangle commutes. The commutativity of the lower triangle up to sign follows from the fact that on the chain level, the termwise diagrams
\begin{figure}[H]
\centering
\includegraphics[width=0.65\linewidth]{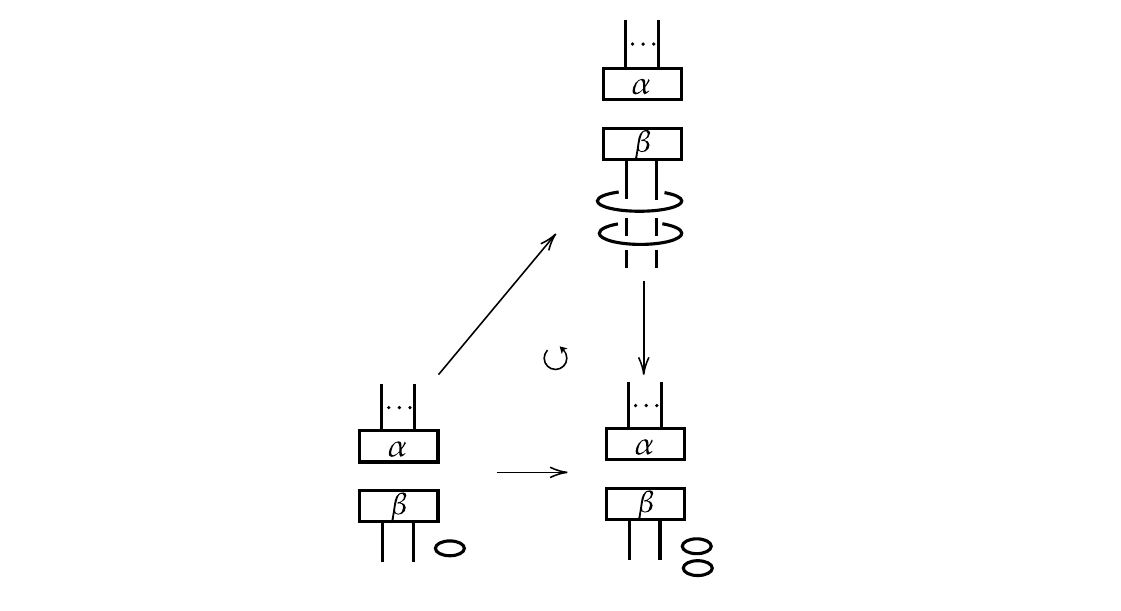}
\end{figure}
commute, and that the lower-right corner of $R_1$ is crossingless, hence no cross term exists. The sign is fixed in Appendix~\ref{sec:sign_region_B_R1}.

The relevant parts of region $R_2$ are redrawn as
\begin{figure}[H]
\centering
\includegraphics[width=0.6\linewidth]{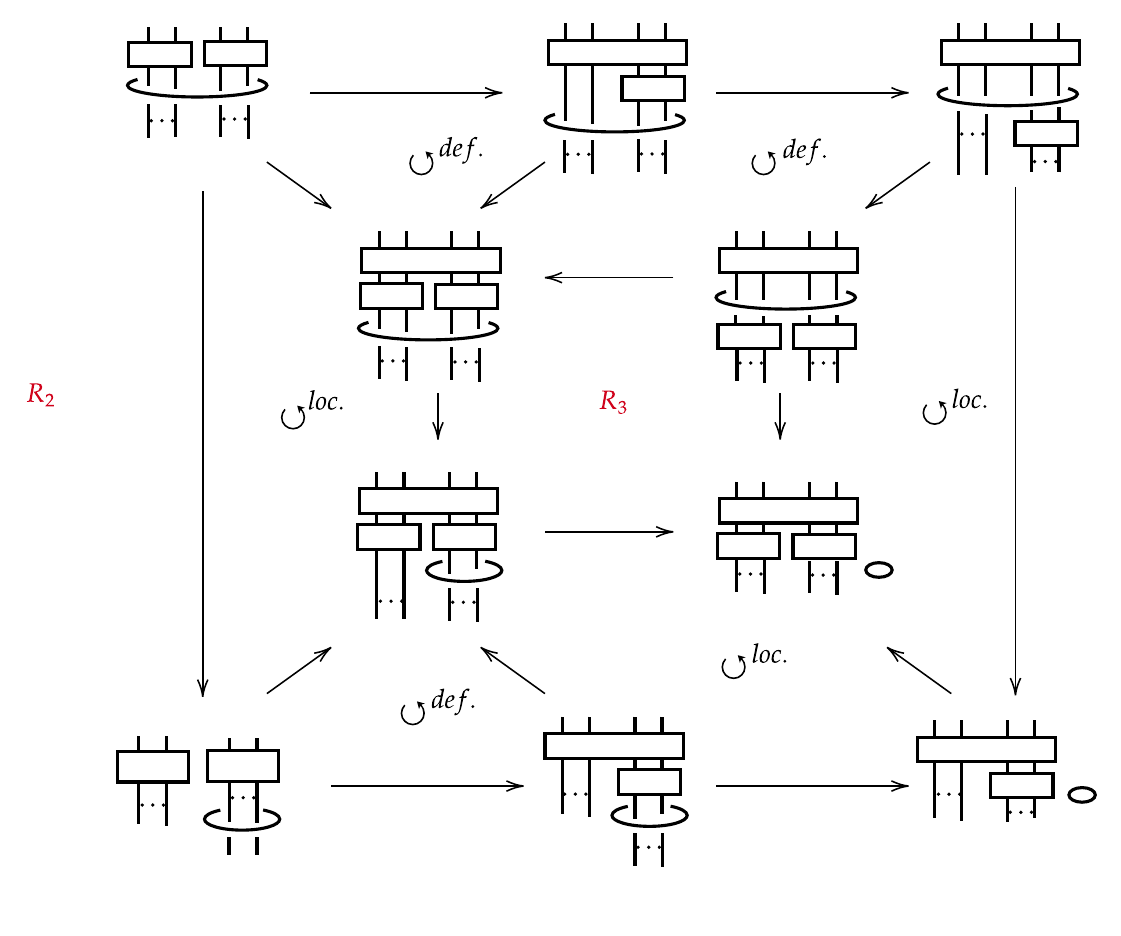},
\end{figure}
where all regions except $R_3$ commutes by definition or locality. We fill in region $R_3$ (and omit one vertex) as
\begin{figure}[H]
\centering
\includegraphics[width=0.6\linewidth]{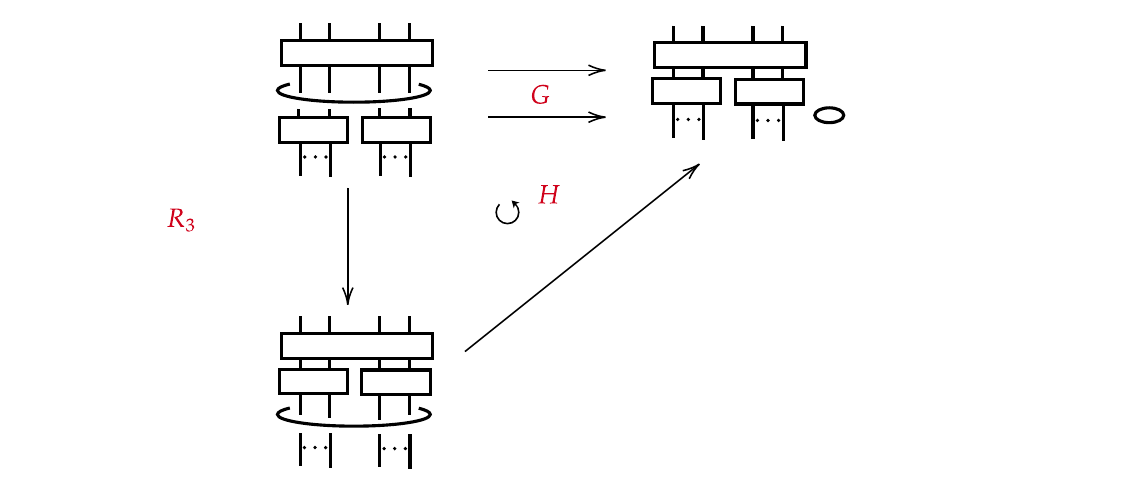}.
\end{figure}
Here, region $G$ commutes by the same argument as the commutativity of region $B$ in Section~\ref{sec:concrete_finger}, and the lower triangle commutes because one can prove the commutativity of the following region $H$ using again the same proof as region $B$:
\begin{figure}[H]
\centering
\includegraphics[width=0.65\linewidth]{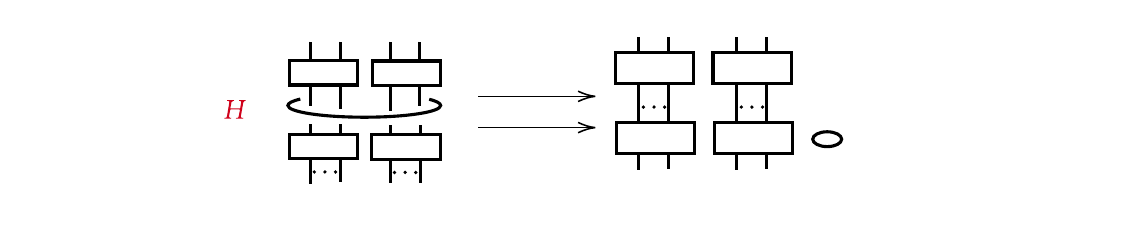}.
\end{figure}
This proves the claim for row \eqref{eq:SZ_slide}.

At lasagna quantum degree $0$, each unknotted circle appearing in row $4$ of Figure~\ref{fig:slide_handle_diagram} carries an $\Khdot$ label, which is sent to $\Khdot\otimes \Khdot$ under a saddle map. We conclude that Theorem~\ref{thm:spin_repara_las} holds for $\widetilde{KhR}_2^+(\phi)$ given by the composition map \eqref{eq:2_to_4_projector} post-composed with an isotopy that puts the first projector in standard position, and $\alpha=0$. This finishes the proof of Theorem~\ref{thm:spin_repara_las}.

\section{Functoriality for spin morphisms}\label{sec:abstract_spin_functoriality}
\subsection{The statements}
In this section, we prove a special case of Theorem~\ref{thm:KhR_2+}, assuming all objects and morphisms are spin.

More precisely, let $\mathbf{Links}_1^{spin}$ denote the category defined similarly as $\mathbf{Links}_1$ in Definition~\ref{def:Links_1}, except that each $S$ in an object $(S,L)$ is assumed to be an abstract spin $\sqcup_{i=1}^k\#^{m_i}(S^1\times S^2)$, and each $W$ in a morphism $(W,\Sigma)\colon(S_0,L_0)\to(S_1,L_1)$ comes with a spin structure making it a spin cobordism from $S_0$ to $S_1$. The goal of this section is to prove the following theorem.

\begin{Thm}\label{thm:abstract_spin_functoriality}
There is a symmetric monoidal functor $\widetilde{KhR}_2^+\colon(\mathbf{Links}_1^{spin})^{op}\to f\mathbf{Vect}_\Q^{\Z\times\Z}$ that extends the definition of $\widetilde{KhR}_2^+$ on objects as in Section~\ref{sec:abstract_spin_homology}. For a morphism $(W,\Sigma)$, $\widetilde{KhR}_2^+(W,\Sigma)$ is homogeneous of degree $(0,-\chi(\Sigma))$.
\end{Thm}

Theorem~\ref{thm:abstract_spin_functoriality} is a formal consequence of the ``turning cobordisms inside out'' trick described in Section~\ref{sec:upside_down}, together with the following theorem.

\begin{Thm}\label{thm:spin_functoriality_las}
Let $D_{std,j}:=\#_{i=1}^{k^{(j)}}\natural^{m_i^{(j)}}(D^2\times S^2)$, $S_{std,j}:=\partial D_{std,j}$, and $L_j\subset S_{std,j}$ be an admissible link, $j=0,1$. If $i\colon D_{std,1}\hookrightarrow int(D_{std,0})$ is a smooth embedding such that $W^t:=D_{std,0}\backslash int(D_{std,1})$ is a $4$-dimensional relative $1$-handlebody complement with $\partial_-W^t=-S_{std,0}$, $\partial_+W^t=-S_{std,1}$, and $\Sigma^t\subset W^t$ is a cobordism between $L_1$ and $L_0$, then the map $$\mathcal S_0^2(W^t;\Sigma^t)\colon\mathcal S_0^2(D_{std,1};L_1)\to\mathcal S_0^2(D_{std,0};L_0)$$ induces a map on the $0$-th associated graded spaces with respect to the lasagna quantum grading. Under the isomorphism \eqref{eq:SZ}, this induced map is of the form
$$\widetilde{KhR}_2^+(\Sigma)\otimes gr_0(\mathrm{id}_{\alpha}\circ i_*)\colon\widetilde{KhR}_2^+(L_1)\otimes gr_0\mathcal S_0^2(D_{std,1})\to\widetilde{KhR}_2^+(L_0)\otimes gr_0\mathcal S_0^2(D_{std,0})$$
for some $\widetilde{KhR}_2^+(\Sigma)=\widetilde{KhR}_2^+(W,\Sigma)\colon\widetilde{KhR}_2^+(L_1)\to\widetilde{KhR}_2^+(L_0)$ and some $\alpha\in H_2(D_{std,0})$. Here, $i_*$ denotes the pushforward map $\mathcal S_0^2(i)\colon\mathcal S_0^2(D_{std,1})\to\mathcal S_0^2(D_{std,0})$.
\end{Thm}

The proof of Theorem~\ref{thm:spin_functoriality_las} takes up the bulk of Section~\ref{sec:abstract_spin_functoriality}. In Section~\ref{sec:abstract_spin_functoriality_decomposition} we decompose it into various cases, which are treated individually in Sections~\ref{sec:sixpartone}--\ref{sec:sixpartlast}. Before going there, we deduce Theorem~\ref{thm:abstract_spin_functoriality} as a consequence of Theorem~\ref{thm:spin_functoriality_las}.

\begin{proof}[Proof of Theorem~\ref{thm:abstract_spin_functoriality} assuming Theorem~\ref{thm:spin_functoriality_las}]
Let $(W,\Sigma)\colon(S_0,L_0)\to(S_1,L_1)$ be a morphism in $\mathbf{Links}_1^{spin}$. Write $W=X_1\backslash int(X_0)$ for spin $4$-dimensional $1$-handlebodies $X_0,X_1$. Choose a spin embedding $X_1\hookrightarrow S^4$, we get an embedding $-(S^4\backslash int(X_1))\hookrightarrow-(S^4\backslash int(X_0))$. Choose parametrizations $\tilde\phi_j\colon-(S^4\backslash int(X_j))\xrightarrow{\cong} D_{std,j}=\#_{i=1}^{k^{(j)}}\natural^{m_i^{(j)}}(D^2\times S^2)$ so that $\phi_j=\tilde\phi_j|_{S_j}$ is $L_j$-admissible, $j=0,1$. Then $(W,\Sigma)$ gives rise to a map 
\begin{equation}\label{eq:las_glue_phi_0_W}
\mathcal S_0^2(\tilde\phi_0(W^t);\tilde\phi_0(\Sigma^t))\colon\mathcal S_0^2(D_{std,1};\phi_1(L_1))\to\mathcal S_0^2(D_{std,0};\phi_0(L_0)),
\end{equation}
which determines a map $\widetilde{KhR}_2^+(\tilde\phi_0(W),\tilde\phi_0(\Sigma))\colon\widetilde{KhR}_2^+(\phi_1(L_1))\to\widetilde{KhR}_2^+(\phi_0(L_0))$ by Theorem~\ref{thm:spin_functoriality_las}, uniquely so if we insist it to be zero whenever \eqref{eq:las_glue_phi_0_W} is zero. This in turn uniquely determines a map $\widetilde{KhR}_2^+(W,\Sigma)$ making $$\begin{tikzcd}
\widetilde{KhR}_2^+(\phi_1(L_1))\ar[rrr,"{\widetilde{KhR}_2^+(\tilde\phi_0(W),\tilde\phi_0(\Sigma))}"]\ar[d,"\cong"]&&&\widetilde{KhR}_2^+(\phi_0(L_0))\ar[d,"\cong"])\\
\widetilde{KhR}_2^+(S_1,L_1)\ar[rrr,"{\widetilde{KhR}_2^+(W,\Sigma)}"]&&&\widetilde{KhR}_2^+(S_0,L_0)
\end{tikzcd}$$ commute.

We check that $\widetilde{KhR}_2^+(W,\Sigma)$ is independent of $\tilde\phi_0,\tilde\phi_1$. For another choice $\tilde\phi_0',\tilde\phi_1'$, we have the commutative diagram
$$\begin{tikzcd}
\mathcal S_0^2(D_{std,1};\phi_1'(L_1))\ar[rrr,"{\mathcal S_0^2(\tilde\phi_0'(W^t);\tilde\phi_0'(\Sigma^t))}"]\ar[d,"\cong","\mathcal S_0^2(\tilde\phi_1\circ\tilde\phi_1'^{-1})"']&&&\mathcal S_0^2(D_{std,0};\phi_0'(L_0))\ar[d,"\cong","\mathcal S_0^2(\tilde\phi_0\circ\tilde\phi_0'^{-1})"']\\
\mathcal S_0^2(D_{std,1};\phi_1(L_1))\ar[rrr,"{\mathcal S_0^2(\tilde\phi_0(W^t);\tilde\phi_0(\Sigma^t))}"]&&&\mathcal S_0^2(D_{std,0};\phi_0(L_0))
\end{tikzcd}$$
By Theorem~\ref{thm:spin_repara_las} and Theorem~\ref{thm:spin_functoriality_las}, this diagram descends to a commutative diagram on $gr_0$, which under isomorphism \eqref{eq:SZ} becomes a commutative diagram
$$\begin{tikzcd}
\widetilde{KhR}_2^+(\phi_1'(L_1))\otimes gr_0\mathcal S_0^2(D_{std,1})\ar[rrrrr,"{\widetilde{KhR}_2^+(\tilde\phi_0'(W),\tilde\phi_0'(\Sigma))\otimes gr_0(\mathrm{id}_{\alpha_0}\circ i'_*)}"]\ar[dd,"\cong","\begin{array}{c}\widetilde{KhR}_2^+(\phi_1'\circ\phi_1^{-1})\otimes\\gr_0(\mathrm{id}_{\alpha_2}\circ(\phi_1\circ\phi_1'^{-1})_*)\end{array}"']&&&&&\widetilde{KhR}_2^+(\phi_0'(L_0))\otimes gr\mathcal S_0^2(D_{std,0})\ar[dd,"\cong","\begin{array}{c}\widetilde{KhR}_2^+(\phi_0'\circ\phi_0^{-1})\otimes\\gr_0(\mathrm{id}_{\alpha_3}\circ(\phi_0\circ\phi_0'^{-1})_*)\end{array}"']\\\\
\widetilde{KhR}_2^+(\phi_1(L_1))\otimes gr_0\mathcal S_0^2(D_{std,1})\ar[rrrrr,"{\widetilde{KhR}_2^+(\tilde\phi_0(W),\tilde\phi_0(\Sigma))\otimes gr_0(\mathrm{id}_{\alpha_1}\circ i_*)}"]&&&&&\widetilde{KhR}_2^+(\phi_0(L_0))\otimes gr\mathcal S_0^2(D_{std,0})
\end{tikzcd}$$
for some $\alpha_0,\alpha_1,\alpha_2,\alpha_3$, where $i,i'\colon D_{std,1}\hookrightarrow D_{std,0}$ are naturally determined by the parametrizations. The two composition maps on the second tensorial factors are both nonzero, as they each send the element $gr_0(1)$ ($1$ is the element represented by the empty skein) to a nonzero element. Therefore, the two composition maps $\widetilde{KhR}_2^+(\phi_0'\circ\phi_0^{-1})\circ\widetilde{KhR}_2^+(\tilde\phi_0'(W),\tilde\phi_0'(\Sigma))$ and $\widetilde{KhR}_2^+(\tilde\phi_0(W),\tilde\phi_0(\Sigma))\circ\widetilde{KhR}_2^+(\phi_1'\circ\phi_1^{-1})$ on the first tensorial factors are equal up to some scalar $\lambda\in\Q$. If they are nonzero, then the two compositions on the second tensorial factors are equal up to $\lambda^{-1}$. Since either composition on the second tensorial factors, after postcomposing with the map on $gr_0\mathcal S_0^2$ induced by an embedding $D_{std,0}\subset S^4$, sends $gr_0(1)\in gr_0\mathcal S_0^2(D_{std,1})$ to $gr_0(1)\in gr_0\mathcal S_0^2(S^4)\cong\Q$, we have $\lambda=1$. This shows that $\widetilde{KhR}_2^+(W,\Sigma)$ is independent of the choices of $\tilde\phi_j$, $j=0,1$. It is also independent of the choice of the spin embedding $X_1\hookrightarrow S^4$, as any two such embeddings are isotopic. Hence, $\widetilde{KhR}_2^+(W,\Sigma)$ is well-defined.

The functoriality of $\widetilde{KhR}_2^+(W,\Sigma)$ is proved similarly. The identity morphism induces the identity map by Theorem~\ref{thm:spin_functoriality_las}. If $(W_j,\Sigma_j)\colon(S_j,L_j)\to(S_{j+1},L_{j+1})$, $j=0,1$, are two composable morphisms in $\mathbf{Links}_1^{spin}$, then we can write $W_j=X_{j+1}\backslash int(X_j)$, $j=0,1$, for some spin $4$-dimensional $1$-handlebodies $X_0,X_1,X_2$. Choose a spin embedding $X_2\hookrightarrow S^4$ and parametrizations $\tilde\phi_j\colon-(S^4\backslash int(X_j))\xrightarrow{\cong}D_{std,j}$, so that $\phi_j\colon S_j\xrightarrow{\cong} S_{std,j}$ is $L_j$-admissible, $j=0,1,2$. We obtain concrete models for the abstract induced maps $\widetilde{KhR}_2^+(W_j,\Sigma_j)$, $j=0,1$, as well as $\widetilde{KhR}_2^+(W_1\circ W_0,\Sigma_1\circ\Sigma_0)$. A similar argument as before using Theorem~\ref{thm:spin_functoriality_las} shows that compositions are functorial on these concrete models, hence on the abstract induced maps themselves.

The symmetric monoidality of $\widetilde{KhR}_2^+$ is clear from the construction.
\end{proof}

\subsection{Decomposition into elementary morphisms}\label{sec:abstract_spin_functoriality_decomposition}
If Theorem~\ref{thm:spin_functoriality_las} holds for embeddings $D_{std,j+1}\hookrightarrow int(D_{std,j})$ and composable link cobordisms in $D_{std,j}\backslash int(D_{std,j+1})$, $j=0,1$, then it holds for the composition as well.

Therefore, it suffices to decompose any morphism into a composition of some elementary morphisms, and check Theorem~\ref{thm:spin_functoriality_las} for these elementary ones. We first state such a decomposition result for abstract morphisms, namely morphisms in $\mathbf{Links}_1$.

\begin{Prop}\label{prop:abstract_decomposition}
Every morphism in $\mathbf{Links}_1$ is a composition of some elementary morphisms of the following forms. See Figure~\ref{fig:12_handles}.
\begin{enumerate}[(i)]
\item Product morphisms: abstract $(I\times S_{std},\Sigma)\colon(S_{std},L_0)\to(S_{std},L_1)$ for some $S_{std}=\sqcup_{i=1}^k\#^{m_i}(S^1\times S^2)$.\label{item:product}
\item Ball creations: abstract $((I\times S_{std})\sqcup B^4,I\times L)\colon(S_{std},L)\to(S_{std}\sqcup S^3,L)$ for some $S_{std}=\sqcup_{i=1}^k\#^{m_i}(S^1\times S^2)$.\label{item:ball_creation}
\item Connected sums: abstract $(W_\#,I\times L)\colon(S_{std},L)\to(S_{std,\#},L)$ for some $S_{std}=\sqcup_{i=1}^k\#^{m_i}(S^1\times S^2)$, $k\ge2$, where $S_{std,\#}=(\sqcup_{i=1}^{k-2}\#^{m_i}(S^1\times S^2))\sqcup(\#^{m_{k-1}+m_k}(S^1\times S^2))$, $W_\#$ is a (standard) $1$-handle attachment between the $(k-1)$-th and $k$-th component of $S_{std}$ attached near $\infty\in\#^{m_i}(S^1\times S^2)$, $i=k-1,k$, and $I\times L$ is the trace of $L$ in $W_\#$.\label{item:connected_sum}
\item $1$-handle attachments: abstract $(W_+,I\times L)\colon(S_{std},L)\to(S_{std,+},L)$ for some $S_{std}=\sqcup_{i=1}^k\#^{m_i}(S^1\times S^2)$, $k\ge1$, where $S_{std,+}=(\sqcup_{i=1}^{k-1}\#^{m_i}(S^1\times S^2))\sqcup(\#^{m_k+1}(S^1\times S^2))$, $W_+$ is a (standard) $1$-handle attachment on the $k$-th component of $S_{std}$ that misses $L$, and $I\times L$ is the trace of $L$ in $W_+$.\label{item:1_handle}
\item Canceling $2$-handle attachments: abstract $(W_-,I\times L)\colon(S_{std},L)\to(S_{std,-},L)$ for some $S_{std}=\sqcup_{i=1}^k\#^{m_i}(S^1\times S^2)$, $k\ge1$, $m_k\ge1$, where $S_{std,-}=(\sqcup_{i=1}^{k-1}\#^{m_i}(S^1\times S^2))\sqcup(\#^{m_k-1}(S^1\times S^2))$, $W_-$ is a (standard) $2$-handle attachment on the $k$-th component of $S_{std}$ that misses $L$ and cancels the last $S^1\times S^2$ connected summand, and $I\times L$ is the trace of $L$ in $W_-$.\label{item:2_handle_cancel}
\end{enumerate}
Moreover, in each concrete model $(W,\Sigma)\colon(S_{std,0},L_0)\to(S_{std,1},L_1)$ described above, $L_0,L_1$ can be assumed to be admissible.
\end{Prop}
The usage of ``standard'' in Proposition~\ref{prop:abstract_decomposition} is in a similar sense to that in Proposition~\ref{prop:diffeomorphism_decomposition}. We will not be pedantic about this distinction below.

\begin{figure}
\centering
\includegraphics[width=0.7\linewidth]{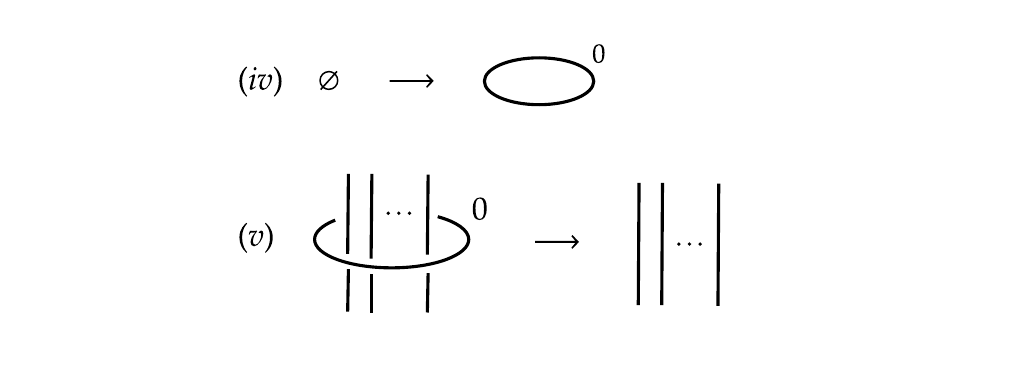}
\caption{Forward: Type \eqref{item:1_handle} and \eqref{item:2_handle_cancel} morphisms in Proposition~\ref{prop:abstract_decomposition}. Backward: Type \eqref{item:3_handle_cancel} and \eqref{item:2_handle} morphisms in Corollary~\ref{cor:abstract_decomposition_reverse}.}
\label{fig:12_handles}
\end{figure}

\begin{proof}
We claim that it suffices to perform the decomposition of a morphism $(W,\Sigma)\colon(S_0,L_0)\to(S_1,L_1)$ on the $4$-manifold level. Indeed, if such a $4$-manifold-level decomposition is given, by general position, further decomposing and making the handles thin in the cocore direction, we may assume that
\begin{itemize}
\item The boundary of each layer intersects $\Sigma$ transversely.
\item Each $1$-handle is disjoint from $\Sigma$.
\item Each $2$-handle intersects $\Sigma$ in a disjoint union of interior cocores. In particular, the attaching region of each $2$-handle is disjoint from $\Sigma$.
\end{itemize}
By further splitting off product pieces, we can assume that $2$-handles have short cores, thus by general position,
\begin{itemize}
\item Each $2$-handle is disjoint from $\Sigma$.
\end{itemize}
Admissibility of boundary links can be achieved by choosing nice parametrizations. The claim follows.

Now we forget about $L_0,L_1,\Sigma$ and exhibit a $4$-manifold-level decomposition for $W$. Write $W=X_1\backslash int(X_0)$ for $4$-dimensional $1$-handlebodies $X_0,X_1$. Without loss of generality, say $W$, hence $X_1$, is nonempty and connected. By the ball creation or the connected sum operation, we may assume $X_0$ to be nonempty and connected as well.

Fix a decomposition of $X_0,X_1$ into handlebodies each with a single $0$-handle. Assume the $0$-handle of $X_0$ is contained in that of $X_1$. Now $\pi_1(X_1)$ is a free group with generators given by the $1$-handles of $X_1$. For each $1$-handle of $X_1$, attach a $1$-handle to $X_0$ inside $int(X_1)$ that represents the corresponding generator in $X_1$ (up to isotopy, there is no choice for this attachment). Next, each original $1$-handle in $X_0$ can slide over these new $1$-handles ambiently in $X_1$ so that it represents the trivial element in $\pi_1(X_1)$. Again, there is only one possible configuration, hence we see after sliding, each original $1$-handle can be canceled by an ambient canceling $2$-handle, making the rest of the $X_0$ complement a product.
\end{proof}

\begin{Cor}\label{cor:abstract_decomposition_reverse}
Any $(i\colon D_{std,1}\hookrightarrow D_{std,0},\Sigma^t\subset W^t)$ in the statement of Theorem~\ref{thm:spin_functoriality_las} can be decomposed into a composition of elementary ones of the following forms.
\begin{enumerate}[(i)]
\item Cobordisms in twisted products: $i\colon D_{std,1}\xrightarrow[\cong]{\tilde\phi}D_{std,0}\hookrightarrow D_{std,0}$ where $\tilde\phi$ is an orientation-preserving diffeomorphism and the second map is a collar-thickening; $\Sigma^t$ is any cobordism between admissible links.\label{item:twisted_product}
\item Ball annihilations: $D_{std,1}=D_{std,0}\#B^4$, $i$ is the $4$-handle attachment that caps off the last $S^3$ boundary component; $\Sigma^t$ is the trace in $W^t$ of an admissible link $L\subset S_{std,1}$ missing the last component.\label{item:ball_annihilation}
\item Separating $3$-handle attachments: $D_{std,1}=\#_{i=1}^{k-2}\natural^{m_i}(D^2\times S^2)\#\natural^{m_{k-1}+m_k}(D^2\times S^2)$, $D_{std,0}=\#_{i=1}^{k}\natural^{m_i}(D^2\times S^2)$, $k\ge2$, $i$ is a standard $3$-handle attachment onto the last summand of $S_{std,1}$ that separates the first $m_k$ and the last $m_{k+1}$ $2$-handles; $\Sigma^t$ is the trace in $W^t$ of an admissible link $L\subset S_{std,1}$ that misses the $3$-handle attaching region.\label{item:3_handle_sep}
\item Canceling $3$-handle attachments: $D_{std,1}=\#_{i=1}^{k-1}\natural^{m_i}(D^2\times S^2)\#\natural^{m_k+1}(D^2\times S^2)$, $D_{std,0}=\#_{i=1}^k\natural^{m_i}(D^2\times S^2)$, $k\ge1$, $i$ is a standard $3$-handle attachment that cancels the last $2$-handle; $\Sigma^t$ is the trace in $W^t$ of an admissible link $L\subset S_{std,1}$ that misses the $3$-handle attaching region.\label{item:3_handle_cancel}
\item Local $0$-framed $2$-handle attachments: $D_{std,1}=\#_{i=1}^{k-1}\natural^{m_i}(D^2\times S^2)\#\natural^{m_k-1}(D^2\times S^2)$, $D_{std,0}=\#_{i=1}^k\natural^{m_i}(D^2\times S^2)$, $k\ge1$, $m_k\ge1$, $i$ is a standard local $0$-framed $2$-handle attachment in the last boundary component; $\Sigma^t$ is the trace in $W^t$ of an admissible link $L\subset S_{std,1}$ that misses the $2$-handle attaching region.\label{item:2_handle}
\end{enumerate}
\end{Cor}
\begin{proof}
Turn the moves in Proposition~\ref{prop:abstract_decomposition} upside down and use type \eqref{item:twisted_product} morphisms to absorb parametrization changes. 
\end{proof}

\subsection{Cobordisms in twisted products}
\label{sec:sixpartone}
We prove Theorem~\ref{thm:spin_functoriality_las} for type \eqref{item:twisted_product} morphisms in Corollary~\ref{cor:abstract_decomposition_reverse}.

We first decompose such a morphism into a composition of the following more elementary ones.
\begin{enumerate}[(a)]
\item A collar-thickening: $i\colon D_{std}\hookrightarrow D_{std}$ is a collar-thickening; $\Sigma^t$ is any cobordism between admissible links.
\item A reparametrization: $i\colon D_{std}\xrightarrow[\cong]{\tilde\phi}D_{std}\hookrightarrow D_{std}$ where $\tilde\phi\in\Diff^+(D_{std})$ with $\phi=\tilde\phi|_{S_{std}}$ $L$-admissible for an admissible link $L\subset S_{std}$, and the second map is a collar-thickening; $\Sigma^t$ is the product cobordism from $L$ to $\phi(L)$;
\item A name change: $i\colon D_{std,1}\xrightarrow[\cong]{\tilde\phi}D_{std,0}\hookrightarrow D_{std,0}$ where $\tilde\phi$ is a standard orientation-preserving diffeomorphism that exchanges some $i$-th and $(i+1)$-th connected summands of $D_{std,1}$, and the second map is a collar thickening; $\Sigma^t$ is the product cobordism from some admissible link $L$ to $\phi(L)$.
\end{enumerate}

Theorem~\ref{thm:spin_functoriality_las} for each type of morphisms above is now a consequence of our previous work.

(a): The statement follows from Theorem~\ref{thm:concrete_functoriality_las}.

(b): The statement follows from Theorem~\ref{thm:spin_repara_las}.

(c): The statement follows from an argument similar to that of Section~\ref{sec:connect_summand_exchange}.

\subsection{Ball annihilations}\label{sec:ball_annilhilation}
We prove Theorem~\ref{thm:spin_functoriality_las} for type \eqref{item:ball_annihilation} morphisms in Corollary~\ref{cor:abstract_decomposition_reverse}.

In the sequence of isomorphisms \eqref{eq:SZ_MN}\listsymbol\eqref{eq:SZ_simp} leading to \eqref{eq:SZ}, the effect of the extra $4$-handle $i\colon D_{std,1}\hookrightarrow D_{std,0}$ comes in nowhere. Therefore, terms in each row \eqref{eq:SZ_MN}\listsymbol\eqref{eq:SZ_simp} for $(D_{std,1},L)$ and $(D_{std,0},L)$ are isomorphic via the obvious isomorphisms. We conclude that $\mathcal S_0^2(W^t;\Sigma^t)$ is equal to $\widetilde{KhR}_2^+(\Sigma)\otimes i_*$ in terms of row \eqref{eq:SZ_simp}, where $\widetilde{KhR}_2^+(\Sigma)$ is induced by the natural isomorphism $\widetilde{KhR}_2^+(\emptyset)\cong\Q$, $\emptyset$ being the empty link in the last $S^3$ factor.

\subsection{Separating \texorpdfstring{$3$}{3}-handles}
We prove Theorem~\ref{thm:spin_functoriality_las} for type morphisms \eqref{item:3_handle_sep} in Corollary~\ref{cor:abstract_decomposition_reverse}.

The argument is as in Section~\ref{sec:ball_annilhilation}, as the separating $3$-handle missing $L$ intertwines with the isomorphisms \eqref{eq:SZ_MN}\listsymbol\eqref{eq:SZ_simp} in the evident way. We conclude that $\mathcal S_0^2(W^t;\Sigma^t)$ is equal to $\widetilde{KhR}_2^+(\Sigma)\otimes i_*$ in terms of row \eqref{eq:SZ_simp}, where $\widetilde{KhR}_2^+(\Sigma)$ is induced by the canonical isomorphism $\widetilde{KhR}_2^+(L_{k-1}\sqcup L_k)\xrightarrow{\cong}\widetilde{KhR}_2^+(L_{k-1})\otimes\widetilde{KhR}_2^+(L_k)$. Here $L_i$ denote the part of $L\subset S_{std,0}$ in the $i$-th boundary component.

\subsection{Canceling \texorpdfstring{$3$}{3}-handles}
We prove Theorem~\ref{thm:spin_functoriality_las} for type morphisms \eqref{item:3_handle_cancel} in Corollary~\ref{cor:abstract_decomposition_reverse}. 

This morphism can be visualized locally as the reverse of (iv) in Figure~\ref{fig:12_handles}. In each colimit summand in the term corresponding to $\mathcal S_0^2(D_{std,1};L)$ in each of the rows \eqref{eq:SZ_MN}\listsymbol\eqref{eq:SZ_slide}, a tensorial factor corresponding to belts coming from the last $2$-handle can be split off. The canceling $3$-handle evaluates the terms in these tensorial factors to scalars by sending $\Khdot$ to $1$ and $1$ to $0$. We conclude that $\mathcal S_0^2(W^t;\Sigma^t)$ is equal to $\widetilde{KhR}_2^+(\Sigma)\otimes i_*$ in terms of row \eqref{eq:SZ_simp}, where $\widetilde{KhR}_2^+(\Sigma)$ is the map that forgets the empty projector $P_{0,0}^\vee$ coming from the last surgery region in $S_{std,1}$.

\subsection{Local \texorpdfstring{$2$}{2}-handles}
\label{sec:sixpartlast}
We prove Theorem~\ref{thm:spin_functoriality_las} for type morphisms \eqref{item:2_handle} in Corollary~\ref{cor:abstract_decomposition_reverse}.

This morphism can be visualized locally as the reverse of (v) in Figure~\ref{fig:12_handles}. In terms of rows \eqref{eq:SZ_MN} or \eqref{eq:SZ_renom}, $\mathcal S_0^2(W^t;\Sigma^t)$ is equal to the inclusion as the term with $(n_+)_m=(n_-)_m=r_m=0$, where $m=\sum_{i=1}^km_i$. In terms of rows \eqref{eq:SZ_res_1} or \eqref{eq:SZ_slide}, $\mathcal S_0^2(W^t;\Sigma^t)$ is equal to the map to this $(n_+)_m=(n_-)_m=r_m=0$ term induced by the unit map $1_{\ell_m}\to P_{\ell_m,0}^\vee$ at the last surgery region in $S_{std,0}$. We conclude that $\mathcal S_0^2(W^t;\Sigma^t)$ is equal to $\widetilde{KhR}_2^+(\Sigma)\otimes i_*$ in terms of row \eqref{eq:SZ_simp}, where $\widetilde{KhR}_2^+(\Sigma)$ is the map induced by the unit map $1_{\ell_m}\to P_{\ell_m,0}^\vee$. This finishes the proof of Theorem~\ref{thm:spin_functoriality_las}.

\section{Remove the spin assumption}\label{sec:nonspin}
In this section, we remove the spin assumption in Theorem~\ref{thm:abstract_spin_functoriality} and promote it to Theorem~\ref{thm:KhR_2+}. To this end, we review the Gluck twist operation in Section~\ref{sec:gluck} and define induced maps on Khovanov skein lasagna modules by Gluck twists in Section~\ref{sec:las_gluck}. This allows us define to $\widetilde{KhR}_2^+$ on objects of $\mathbf{Links}_1$ in Section~\ref{sec:nonspin_obj} and on morphisms of $\mathbf{Links}_1$ in Section~\ref{sec:nonspin_mor}, strengthening results in Section~\ref{sec:abstract_spin_homology} and Section~\ref{sec:abstract_spin_functoriality}, respectively.

\subsection{The Gluck twist operation}\label{sec:gluck}
Let $X$ be a compact oriented $4$-manifold and $S\subset int(X)$ be an embedded unoriented $2$-sphere with trivial normal bundle. The closed tubular neighborhood $\overline{\nu(S)}$ of $S$ is diffeomorphic to $D^2\times S^2$. The boundary $S^1\times S^2$ admits a nonspin diffeomorphism $\tau$ given by a Dehn twist along the $S^2$-factor, or more explicitly $(\theta,x)\mapsto(\theta,\mathrm{rot}_\theta(x))$ where $\mathrm{rot}_\theta\colon S^2\to S^2$ is the rotation-by-$\theta$ map along some fixed axis. The \textit{Gluck twist} of $X$ along $S$, denoted $X_S$, is the $4$-manifold obtained by cutting out $\nu(S)$ and regluing it back by a $\tau$-twist. We make some elementary observations:

\begin{enumerate}
\item Since the natural inclusion $O(2)\times O(3)\subset\Diff(S^1\times S^2)$ is a homotopy equivalence \cite{hatcher1981diffeomorphism}, $\pi_1(\Diff^+(D^2\times S^2))\xrightarrow{\partial}\pi_1(\Diff^+(S^1\times S^2))$ is surjective, so the manifold $X_S$ is well-defined up to a canonical diffeomorphism (up to isotopy, omitted below).
\item If $S$ is isotopic to $S'$, then $X_S$ is diffeomorphic to $X_{S'}$ via a diffeomorphism determined by an isotopy from $S$ to $S'$.
\item If $S\subset X$ is unknotted, then $X_S$ is diffeomorphic to $X$ via some diffeomorphism determined by a bounding $3$-ball $B$.
\item The manifold $X_S$ contains another copy of $S$, and the iterated Gluck twist $(X_S)_S$ is canonically diffeomorphic to $X$.
\item If $S=S_0\cup\cdots\cup S_k$ is disjoint union of unoriented $2$-spheres in $int(X)$ with trivial normal bundles, then the iterated Gluck twist $(\cdots(X_{S_0})_{S_1}\cdots)_{S_k}$ is independent of the ordering of $S_0,\cdots,S_k$ up to canonical diffeomorphisms. Write $X_S$ or $X_{S_0,\cdots,S_k}$ for this iterated Gluck twist.
\item Let $S_0,S_1\subset int(X)$ be disjoint embedded unoriented $2$-spheres with trivial normal bundles, and $\gamma$ be a path in $int(X)$ equipped with a nonvanishing normal vector field, connecting $S_0$ and $S_1$ with interior disjoint from $S_0\cup S_1$, whose tangent and normal vectors at the endpoints are transverse to $S_0,S_1$. Let $S_0\#S_1=S_0\#_\gamma S_1$ be the connected sum of $S_0$ and $S_1$ along $\gamma$. Then $X_{S_0\#S_1}$ is canonically diffeomorphic to $X_{S_0,S_1}$ via some diffeomorphism determined by $S_0,S_1,\gamma$, as can be seen from relative Kirby diagrams of the twisted $\nu(S_0\cup\gamma\cup S_1)$'s rel the common boundary, as shown in Figure~\ref{fig:gluck_conn_sum_kirby}. If one switches the roles of $S_0,S_1$, then the diffeomorphism changes by a barbell diffeomorphism implanted from $\nu(S_0\cup\gamma\cup S_1)$.
\begin{figure}
\centering
\includegraphics[width=0.75\linewidth]{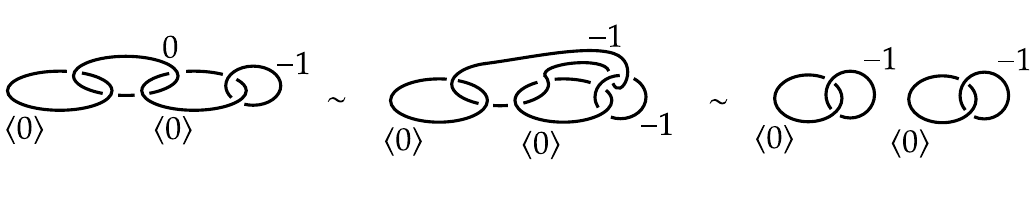}
\caption{Kirby moves exhibiting a diffeomorphism from $(\nu(S_0\cup\gamma\cup S_1))_{S_0\#S_1}$ to $(\nu(S_0\cup\gamma\cup S_1))_{S_0,S_1}$ rel boundary. Each move slides the $2$-handle on the left across one other handle.}
\label{fig:gluck_conn_sum_kirby}
\end{figure}
\item Let $X,S$ be as above and $L\subset\partial X$ be an oriented link. We assign two distinguished isomorphisms
\begin{equation}\label{eq:gluck_H2}
H_2(\tau_S)_\pm\colon H_2^L(X)\xrightarrow{\cong}H_2^L(X_S)
\end{equation}
as follows. Let $\Sigma\subset X$ be an oriented surface bounding $L$ that intersects $S$ transversely. Take $\nu(S)$ small so that $\Sigma\cap\nu(S)$ is a disjoint union of cocore disks. Normally frame $\Sigma$ near $\Sigma\cap\nu(S)$ and give $\Sigma\cap\partial\nu(S)$ the induced framing. Take $\Sigma'\looparrowright X_S$ to be the immersed surface given by $\Sigma$ outside $\nu(S)$, and $\#U$ embedded disks capping $\tau(U)$ off inside the twisted $\nu(S)$, each having self-intersection number $\pm1$. We demand $H_2(\tau_S)_\pm([\Sigma])=[\Sigma']$. One can check that this is well-defined. If $S$ is unknotted, then $H_2(\tau_S)_\pm=\mathrm{id}$, where the codomain and the domain are identified via the canonical diffeomorphism $X_S\cong X$ determined by a given bounding $3$-ball of $S$. Similarly, under the canonical isomorphism $(X_S)_S\cong X$, we have $H_2(\tau_S)_\mp\circ H_2(\tau_S)_\pm=\mathrm{id}_{H_2^L(X)}$ and $H_2(\tau_S)_\pm^2\colon H_2^L(X)\xrightarrow{\cong}H_2^L(X)$ is given by $\alpha\mapsto\alpha\pm(\alpha\cdot[S])[S]$, where we give $S$ an arbitrary orientation to regard $[S]$ as a class in $H_2(X)$.
\end{enumerate}

\subsection{Lasagna induced maps by Gluck twists}\label{sec:las_gluck}
\begin{Thm}\label{thm:gluck}
Let $X$ be a compact oriented $4$-manifold, $L\subset\partial X$ be a framed oriented link, and $S\subset int(X)$ be an embedded unoriented $2$-sphere with trivial normal bundle. There are two natural maps
\begin{equation}\label{eq:gluck}
\tau_{X,L,S,\pm}\colon\mathcal S_0^2(X;L)\to\mathcal S_0^2(X_S;L)
\end{equation}
of $\Q$-vector spaces. Moreover,
\begin{enumerate}[(1)]
\item For any $\alpha\in H_2^L(X)$, $\tau_{X,L,S,\pm}$ restricts to a map $\mathcal S_0^2(X;L;\alpha)\to\mathcal S_0^2(X_S;L;H_2(\tau_S)_\pm(\alpha))$ homogeneous with bidegree shift $(\mp(\alpha\cdot[S])^2/2,\pm(\alpha\cdot[S])^2/2)$. Here $H_2(\tau_S)_\pm$ is the isomorphism \eqref{eq:gluck_H2}.
\item If $S$ is unknotted, then $\tau_{X,L,S,\pm}=\mathrm{id}$, where the codomain and the domain of $\tau_{X,L,S,\pm}$ are identified via the canonical diffeomorphism $X_S\cong X$ determined by a given bounding $3$-ball of $S$.
\item Under the canonical diffeomorphism $(X_S)_S\cong X$, $\tau_{X_S,L,S,\mp}\circ\tau_{X,L,S,\pm}=\mathrm{id}$. In particular, $\tau_{X,L,S,\pm}$ is an isomorphism.
\item If $S=S_0\cup\cdots\cup S_k\subset int(X)$ is a disjoint union unoriented $2$-spheres, then the induced map $\tau_{X_{S_0,\cdots,S_{k-1}},L,S_k,\pm}\circ\cdots\circ\tau_{X,L,S_0,\pm}\colon\mathcal S_0^2(X;L)\to\mathcal S_0^2(X_{S_0,\cdots,S_k};L)$ is independent of the ordering of $S_0,\cdots,S_k$ up to canonical diffeomorphisms. Write for short $\tau_{X,L,S,\pm}$ for this composition.
\end{enumerate}
\end{Thm}
As we will only be using $\tau_{X,L,S,+}$ in the sequel, write for short $\tau_{X,L,S}=\tau_{X,L,S,+}$.

The proof of Theorem~\ref{thm:gluck} takes up most of Section~\ref{sec:las_gluck}.

A precursor of the fact that the Khovanov skein lasagna module (over $\Q$) is invariant under Gluck twists was obtained in \cite[Section~6.10]{ren2024khovanov}, where it was shown that the Khovanov skein lasagna module does not detect potential exotic $4$-spheres obtained from Gluck twists.

\begin{Cor}
Let $X,L,S$ be as in Theorem~7.1. If $S$ is null-homologous, then $\mathcal S_0^2(X;L)$ and $\mathcal S_0^2(X_S;L)$ are isomorphic as graded vector spaces.\qed
\end{Cor}

\begin{Cor}
Let $X,L,S$ be as in Theorem~7.1. For any $\alpha\in H_2^L(X)$, we have an isomorphism $$\mathcal S_0^2(X;L;\alpha+(\alpha\cdot[S])[S])\cong\mathcal S_0^2(X;L;\alpha)$$ with bidegree shift $((\alpha\cdot[S])^2,-(\alpha\cdot[S])^2)$.
\end{Cor}
\begin{proof}
$(\tau_{X_S,L,S}\circ\tau_{X,L,S})^{-1}$ gives such an isomorphism.
\end{proof}

We take a slight detour before giving the construction of \eqref{eq:gluck}.

A belt link in $S^1\times S^2$, in the sense of Section~\ref{sec:diff_rel_bdy}, is a framed oriented link that is isotopic to a union of some even number of $S^1$ fibers with standard framing and various orientations. We defined, for a standard belt link $U$ as shown on the left of Figure~\ref{fig:belt_link}, a distinguished class $1\in\mathcal S_0^2(D^2\times S^2;U)$ as the class $1\otimes1\in\widetilde{KhR}_2^+(U)\otimes\mathcal S_0^2(D^2\times S^2)$ under the isomorphism \eqref{eq:SZ}. It has homological degree $0$, quantum degree $-\#U$, and skein degree $\alpha_U$. Alternatively, it is the class represented by the standard union of cocores that cap off $U$ as a lasagna filling without input balls. We claim that in fact every collection of disks capping off $U$ with framing and orientation represents the class $1$, and consequently, as every belt link is isotopic to a standard one, there is a well-defined element $1\in\mathcal S_0^2(D^2\times S^2;U)$ for every belt link $U$ which is independent of the parametrization of the pair $(D^2\times S^2,U)$. When $U=\emptyset$ there is nothing to show. When $U\ne\emptyset$, to see the claim, use Gabai's $4$-dimensional lightbulb theorem \cite[Theorem~10.1]{gabai20204} to isotope one component $C_1$ of the collection of disks to standard position. The complement of $\nu(C_1)$ is diffeomorphic to a $4$-ball, in which the other $\#U-1$ components of $U$ form an unlink on the boundary, capped off by other disks in the collection. These $\#U-1$ disks evaluate to the standard element $1\otimes\cdots\otimes1$ on the boundary (this can be seen by passing to Lee homology). Hence, one can replace these $\#U-1$ disk components by standard cocores without changing the evaluation, and the claim follows.

A \textit{twisted belt link} is a framed oriented link in $S^1\times S^2$ that is $\tau(U)$ for some belt link $U\subset S^1\times S^2$. A \textit{standard positive/negative twisted belt link} is a twisted belt link that takes a standard form as shown in Figure~\ref{fig:twisted_belt_link}, which in particular is admissible. For a standard positive/negative twisted belt link $T$, define a distinguished class $1_\pm\in\mathcal S_0^2(D^2\times S^2;T)$ as the class $1_\pm\otimes1\in\widetilde{KhR}_2^+(T)\otimes\mathcal S_0^2(D^2\times S^2)$ under the isomorphism \eqref{eq:SZ}. Here, if $U$ denotes the standard belt link with strands having orientations matching those of $T$, then we have an isomorphism of bigraded vector spaces $\widetilde{KhR}_2^+(T;\Z)\cong\widetilde{KhR}_2^+(U;\Z)$ via termwise simplifying Reidemeister I, II maps (see \cite[Lemma~2.26]{willis2021khovanov}), well-defined up to sign, and $1_\pm\in\widetilde{KhR}_2^{+,0,-\#T}(T)$ is defined as the element corresponding to $1\in\widetilde{KhR}_2^{+,0,-\#U}(U)$ under this isomorphism rationalized. To fix the sign, see Appendix~\ref{sec:sign_gluck}. The class $1_\pm\in\mathcal S_0^2(D^2\times S^2;T)$ has homological degree $-\alpha_T^2/2$, quantum degree $\alpha_T^2/2-\#T$, and skein degree $\alpha_T$.

Alternatively, if $T$ is a standard positive belt link, write $n=\#T$, and say $T$ has $n_+$ (resp. $n_-$) strands oriented upward (resp. downward). Then up to sign, $1_+\in\widetilde{KhR}_2^+(T)$ is the image of $1$ under $\Z\cong\widetilde{KhR}_2^{+,0,-n}(T(n,n)_{n_+,n_-};\Z)\xrightarrow[\cong]{\iota}\widetilde{KhR}_2^{+,0,-n}(T;\Z)\to\widetilde{KhR}_2^{+,0,-n}(T)$ where $T(n,n)_{n_+,n_-}$ is the torus link $T(n,n)$ in $S^3$ with orientation on $n_-$ of the strands reversed, $\iota$ is the unit map creating the Rozansky projector which is an isomorphism by the proofs in \cite{manolescu2023generalization}, and the first isomorphism is due to \cite[Theorem~3]{stovsic2009khovanov}, suitably renormalized. To see this alternative description, it suffices to show that $1\in\widetilde{KhR}_2^+(U;\Z)\cong\widetilde{KhR}_2^+(T;\Z)$ is primitive, which follows from the fact that the Rozansky projector is idempotent up to homotopy. By an abuse of notation, below we also write frequently $1\in\widetilde{KhR}_2^+(T(n,n)_{n_+,n_-})$ for $1\in\Z\cong\widetilde{KhR}_2^{+,0,-n}(T(n,n)_{n_+,n_-};\Z)$ rationalized, and $1\in KhR_2^+(T(n,n)_{n_+,n_-})$ for its image under the renormalization. The sign of $1\in\widetilde{KhR}_2^+(T(n,n)_{n_+,n_-})$ is fixed by demanding it to map to $1_+\in\widetilde{KhR}_2^+(T)$ under the unit map. By this alternative description, the element $1_+\in\mathcal S_0^2(D^2\times S^2;T)$ is represented by the standard lasagna filling $(I\times T(n,n)_{n_+,n_-},1)$ of $(D^2\times S^2,T)$ with one input ball being a shrunk $0$-handle, input link $T(n,n)_{n_+,n_-}$ with label $1\in KhR_2^+(T(n,n)_{n_+,n_-})\cong(t^{-1}q)^{(n_+-n_-)^2/2}\widetilde{KhR}_2^+(T(n,n)_{n_+,n_-})$, and a product skein contained in a collar neighborhood of the boundary of the $0$-handle.

\begin{figure}
\centering
\includegraphics[width=0.65\linewidth]{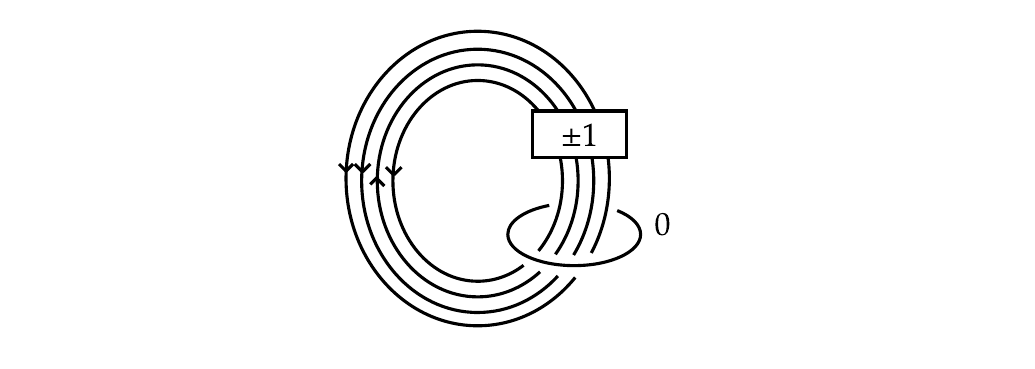}
\caption{Diagram of a standard positive/negative twisted belt link in $S^1\times S^2$.}
\label{fig:twisted_belt_link}
\end{figure}

\begin{Lem}\label{lem:twist_belt_link}
\begin{enumerate}[(1)]
\item For a standard positive/negative twisted belt link $T\subset S^1\times S^2$, the class $1_\pm\in\mathcal S_0^2(D^2\times S^2;T)$ is independent of the parametrization of the pair $(D^2\times S^2,T)$. In particular, for any twisted belt link $T$, there are two distinguished classes $1_\pm\in\mathcal S_0^2(D^2\times S^2;T)$.
\item If $\Sigma\colon T\to T'$ is an annular cobordism annihilating two components of a twisted belt link $T\subset S^1\times S^2$, where the annihilating annulus component is $\partial$-parallel, then $\mathcal S_0^2(I\times S^1\times S^2;\Sigma)\colon\mathcal S_0^2(D^2\times S^2;T)\to\mathcal S_0^2(D^2\times S^2;T')$ maps $1_\pm$ to $0$.
\item Let $\Sigma\colon T\to T'$ be the cobordism in (2) with an extra dot on the annihilating annulus component, then $\mathcal S_0^2(I\times S^1\times S^2;\Sigma)\colon\mathcal S_0^2(D^2\times S^2;T)\to\mathcal S_0^2(D^2\times S^2;T')$ maps $1_\pm$ to $1_\pm$.
\item Let $T_0\subset S^1\times S^2$ be a standard negative twisted belt link. Let $T_1\subset S^1\times S^2$ (resp. $U\subset S^1\times S^2$) be the standard positive twisted belt link (resp. standard belt link) with the same number of components (with orientations) as $T_0$. The gluing map $\mathcal S_0^2(D^2\times S^2;T_0)\otimes\mathcal S_0^2(D^2\times S^2;T_1)\to\mathcal S_0^2(D^2\times S^2;U)$ as shown in Figure~\ref{fig:1_+_1_-_to_1} maps $1_-\otimes1_+$ to $1$.
\begin{figure}
\centering
\includegraphics[width=0.6\linewidth]{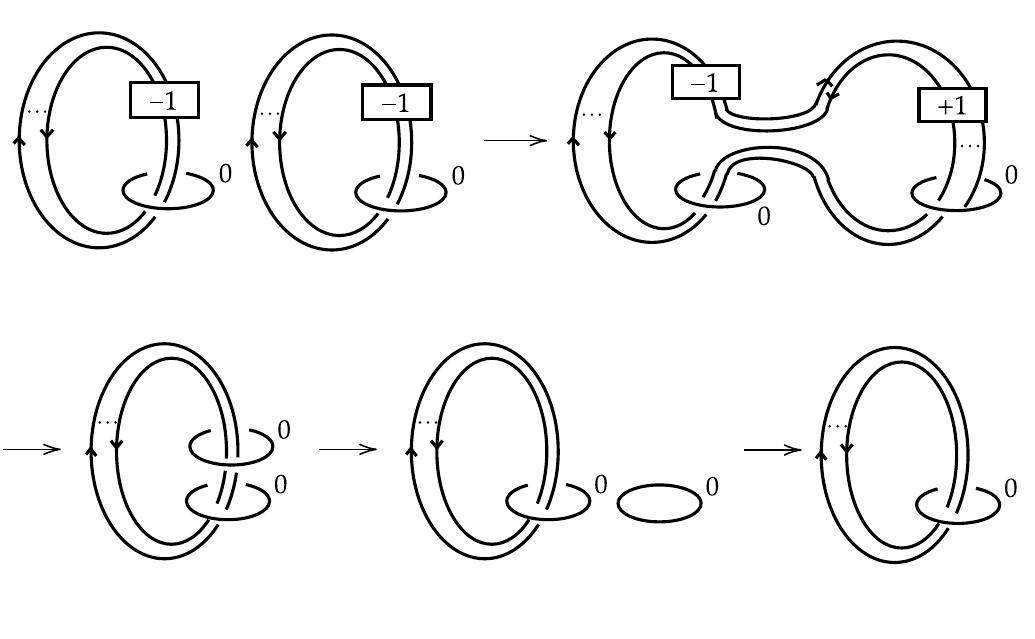}
\caption{A gluing map $(D^2\times S^2\sqcup D^2\times S^2;T_0\sqcup T_1)\to(D^2\times S^2;U)$.}
\label{fig:1_+_1_-_to_1}
\end{figure}
\item Let $U_0,U_1\subset S^1\times S^2$ be belt links and $T_0,T_1\subset S^1\times S^2$ be twisted belt links, so that $T_i$ has the same number of components (with orientations) as $U_i$, $i=0,1$, as shown in Figure~\ref{fig:1_1_1_to_1_+_1_+}. Suppose $U_0$ and $U_1$ together have $n$ components, of which $n_+$ are oriented upward, and $n_-$ downward. The image of $1\otimes1\otimes1\in\mathcal S_0^2(D^2\times S^2;U_0)\otimes\mathcal S_0^2(D^2\times S^2;U_1)\otimes\widetilde{KhR}_2^+(T(n,n)_{n_+,n_-})$ in $\mathcal S_0^2(D^2\times S^2;T_0)\otimes\mathcal S_0^2(D^2\times S^2;T_1)$ under the gluing map shown in Figure~\ref{fig:1_1_1_to_1_+_1_+} is equal to $1_+\otimes\mathrm{id}_\alpha(1_+)$ plus terms with lower lasagna quantum gradings, for some $\alpha\in H_2(D^2\times S^2)$.
\begin{figure}
\centering
\includegraphics[width=0.7\linewidth]{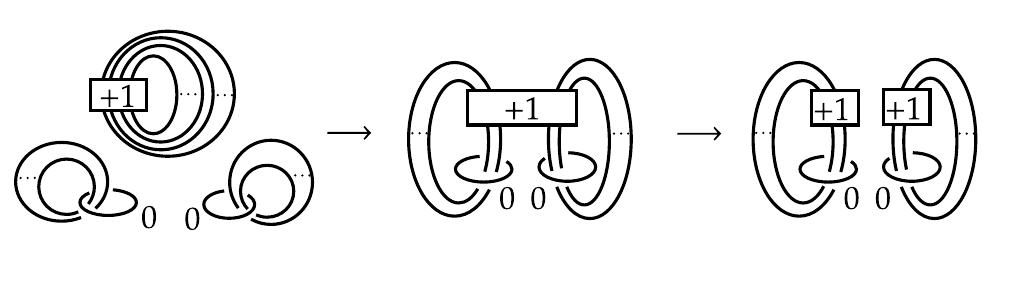}
\caption{A gluing map $(D^2\times S^2\sqcup D^2\times S^2\sqcup B^4;U_0\sqcup U_1\sqcup T(n,n)_{n_+,n_-})\to(D^2\times S^2\natural D^2\times S^2;T_0\cup T_1)$. The second map slides the strands on the left over the $2$-handle on the right.}
\label{fig:1_1_1_to_1_+_1_+}
\end{figure}
\end{enumerate}
\end{Lem}

\begin{proof}
(1) Let $T,T'\subset S^1\times S^2$ be standard twisted belt links, both of which are positive/negative, and $\tilde\phi\colon D^2\times S^2\to D^2\times S^2$ be an orientation-preserving diffeomorphism mapping $T$ to $T'$ with framing and orientation. We need to show that $\mathcal S_0^2(\widetilde{\phi})\colon\mathcal S_0^2(D^2\times S^2;T)\to\mathcal S_0^2(D^2\times S^2;T')$ sends $1_\pm$ to $1_\pm$.

We decompose $\tilde\phi\colon D^2\times S^2\to D^2\times S^2$ into the composition of the following four diffeomorphisms:

\begin{enumerate}[(i)]
\item A diffeomorphism $\tilde\phi_1$ of $D^2\times S^2$ with $\phi_{1,*}=\phi_*$ on $H_*(S^1\times S^2)$ (here and below, dropping the tilde indicates restricting to the boundary) that sends $T$ to a standard positive/negative twisted belt link $T_1$, which is either
\begin{enumerate}[(a)]
\item the identity diffeomorphism; or
\item a diffeomorphism that sends the $0$-handle (resp. $2$-handle) of $D^2\times S^2$ to itself, preserving the core and the cocore of the $2$-handle but reversing each of their orientations.
\end{enumerate}
\item A diffeomorphism $\tilde\phi_2$ that is the identity on the $2$-handle of $D^2\times S^2$ as well as on a shrunk copy of the $0$-handle, and the trace of a braiding away from the $2$-handle attaching region that permutes strands of $T_1$ in a collar neighborhood of the boundary of the $0$-handle, that sends $T_1$ to $T'$.
\item A diffeomorphism $\tilde\phi_3$ that is the identity on a shrunk copy of $D^2\times S^2$, and the trace of an isotopy of $S^1\times S^2$ in a collar neighborhood of $\partial(D^2\times S^2)$ which preserves $T'$ as a set, so that $\phi_3\circ\phi_2\circ\phi_1=\phi$.
\item A diffeomorphism $\tilde\phi_4$ that is rel boundary.
\end{enumerate}

Here, if $T\ne\emptyset$, changing the braiding in (ii) by a pure braid if necessary, one can always find $\tilde\phi_3$ in (iii) as claimed, thanks to Waldhausen's classical work \cite{waldhausen1968irreducible} applied to the Haken manifold $S^1\times S^2\backslash T'$. If $T=\emptyset$, the existence of $\tilde\phi_3$ is trivial.

We show that each $\tilde\phi_i$ sends $1_\pm$ to $1_\pm$.

If $\tilde\phi_1\ne\mathrm{id}$, then we may choose it so that its effect on $T$ is the composition of the inverse of a standard inversion as in Proposition~\ref{prop:diffeomorphism_decomposition}(v), and the trace of some isotopy in $S_{std}$ that consists of only overpass/underpass moves and isotopies via admissible links, in the sense of Proposition~\ref{prop:concrete_decomposition}(i)(v). We note that in the proofs of Theorem~\ref{thm:concrete_functoriality_las} and Theorem~\ref{thm:spin_repara_las}, taking associated graded spaces and introducing shifting isomorphism were only necessary for the handleslide move (Proposition~\ref{prop:concrete_decomposition}(vi)), the barbell move and the trace of an isotopy involving handleslides (special cases of Proposition~\ref{prop:diffeomorphism_decomposition}(i)(ii)). Since $\tilde\phi_1^{-1}$ admits a decomposition in which none of these special cases arise, the proofs of these theorems show that $\mathcal S_0^2(\tilde\phi_1)$ is exactly equal to $\widetilde{KhR}_2^+(\phi_1^{-1})\otimes\mathrm{id}\colon\widetilde{KhR}_2^+(T)\otimes\mathcal S_0^2(D^2\times S^2)\to\widetilde{KhR}_2^+(T_1)\otimes\mathcal S_0^2(D^2\times S^2)$ under the isomorphism \eqref{eq:SZ}, where $\widetilde{KhR}_2^+(\phi_1^{-1})$ is the composite map
\begin{figure}[H]
\centering
\includegraphics[width=0.8\linewidth]{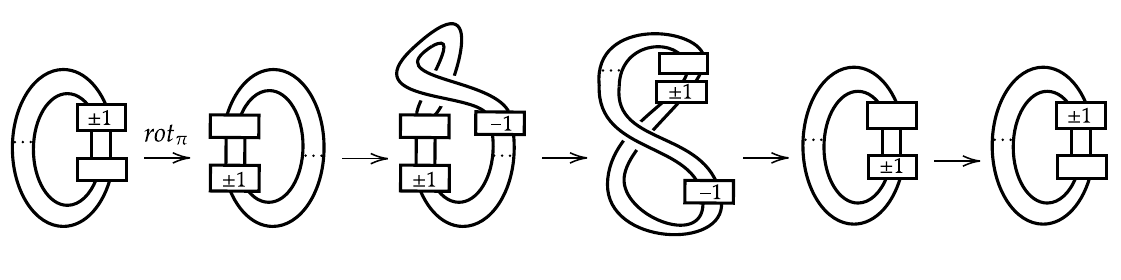}.
\end{figure}

To check that $\widetilde{KhR}_2^+(\phi_1^{-1})$ sends $1_\pm$ to $1_\pm$, by exploiting strategies similar to those in Section~\ref{sec:concrete_functoriality}, one reduces to show the commutativity of regions $TB,TE$ shown as follows:
\begin{figure}[H]
\centering
\includegraphics[width=0.7\linewidth]{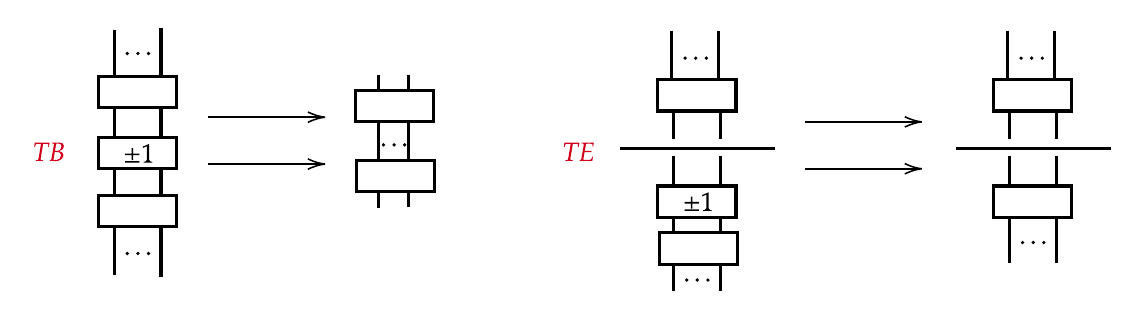}.
\end{figure}
Here, in each region, the second map is given termwise by absorbing the $\pm1$ twist into the second projector, and the first map is given termwise by pushing the $\pm1$ twist up and absorbing it into the first projector.

The commutativity of these regions is proved similarly as regions $B,E$ in Section~\ref{sec:concrete_functoriality}. For region $TB$, the composition of the inverse of the first map with the second map is given termwise by rotating the middle crossingless unlink by $2\pi$, hence is the identity chain map up to sign. To fix the sign, see Appendix~\ref{sec:sign_gluck}. This proves the commutativity of region $TB$. The proof of region $TE$ exactly follows that of region $E$. This proves that $\mathcal S_0^2(\tilde\phi_1)$ sends $1_\pm$ to $1_\pm$.

Analogously, $\mathcal S_0^2(\tilde\phi_2)$ is exactly equal to $\widetilde{KhR}_2^+(\phi_2^{-1})\otimes\mathrm{id}$, where $\widetilde{KhR}_2^+(\phi_2^{-1})\colon\widetilde{KhR}_2^+(T_1)\to\widetilde{KhR}_2^+(T')$ is the map that braids the strands according to $\tilde\phi_2$, shown as a composition of maps of the form (the braiding can happen between any two adjacent strands, either positively or negatively; only a special case is depicted)
\begin{figure}[H]
\centering
\includegraphics[width=0.7\linewidth]{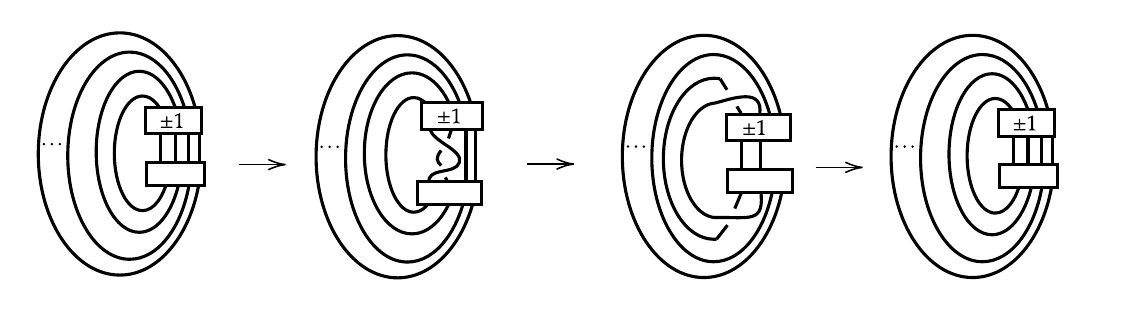}.
\end{figure}
To check that $\widetilde{KhR}_2^+(\phi_2^{-1})$ maps $1_\pm$ to $1_\pm$, in addition to the commutativity of region $TB$ above, one also need the commutativity of region $TD$ (and its mirrored version):
\begin{figure}[H]
\centering
\includegraphics[width=0.7\linewidth]{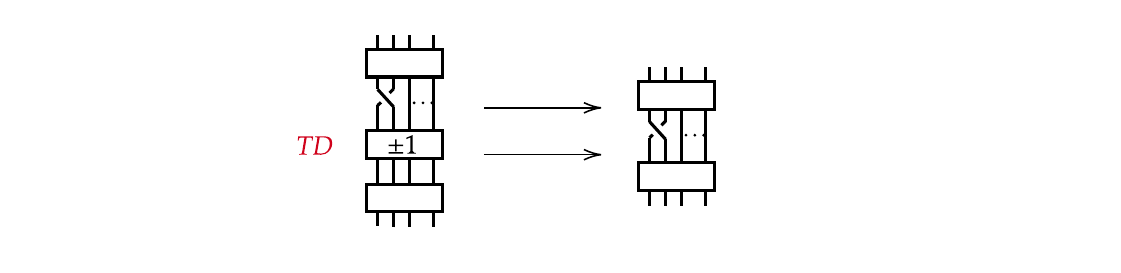}.
\end{figure}
In order to apply the strategy for region $D$, we need to show that the chain map that pushes the $\pm1$ twist up past the crossing (called $X$) is nondecreasing in the homological degree contributed by the distinguished crossing $X$. This can be realized by choosing the twist-pushing map to be the composition of the creation of a canceling pair of twists above $X$ (a $\pm1$ twist above a $\mp1$ twist) and a rotation map that cancels the $\mp1$ twist above and the $\pm1$ twist below $X$, carefully chosen as the composition of two ``$\varphi$'' maps in \cite[Lemma~4.5]{chen2025flip}. This proves that $\mathcal S_0^2(\tilde\phi_2)$ sends $1_\pm$ to $1_\pm$.

The diffeomorphism $\tilde\phi_3$ induces the identity map, since any lasagna filling of $(D^2\times S^2,T')$ can be isotoped to be $I\times T'$ in a collar neighborhood of the boundary.

Finally, by Gabai's $4$-dimensional lightbulb theorem (or Theorem~\ref{thm:diff_D2S2_rel_boundary}), $\tilde\phi_4$ is isotopic rel boundary to a diffeomorphism supported on a local $4$-ball, hence also induces the identity map.\smallskip

(2) Pick $T,T'$ to be standard positive or negative twisted belt links and $\Sigma$ to be standard. Then, by the same argument used for $\tilde\phi_1,\tilde\phi_2$ in (1) above, the induced map takes the form $\widetilde{KhR}_2^+(\Sigma)\otimes\mathrm{id}$ under the isomorphism \eqref{eq:SZ}. This claim now follows from the fact that $\widetilde{KhR}_2^{+,0,-\#T'-2}(T')=0$ by \cite{manolescu2023generalization,stovsic2009khovanov}.\smallskip

(3) As in (2), pick $T,T',\Sigma$ to be standard. One shows that $\widetilde{KhR}_2^+(\Sigma)$ maps $1_\pm\in\widetilde{KhR}_2^+(T)$ to $1_\pm\in\widetilde{KhR}_2^+(T')$ by using the commutativity of region $TB$ above, as well as the commutativity of region $TC$:
\begin{figure}[H]
\centering
\includegraphics[width=0.7\linewidth]{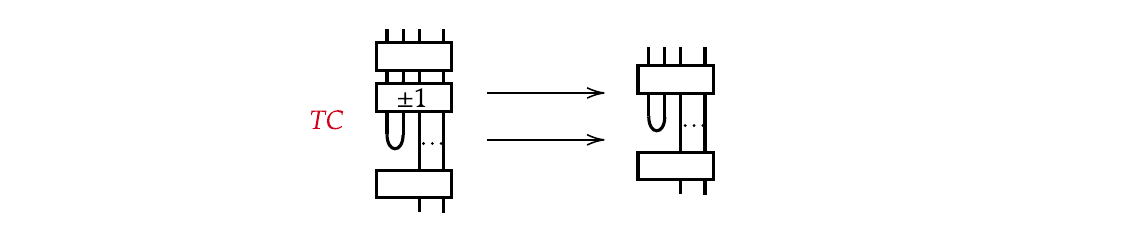},
\end{figure}
which is proved in the same way as for region $C$.\smallskip

(4) Let $n=\#T_i$ with $n_+$ (resp. $n_-$) strands oriented upward (resp. downward). By the description preceding Lemma~\ref{lem:twist_belt_link}, $1_+\in\mathcal S_0^2(D^2\times S^2;T_0)$ is the image of $1\in KhR_2^+(T(n,n)_{n_+,n_-})\cong\mathcal S_0^2(B^4;T(n,n)_{n_+,n_-})$ under the natural $2$-handle attachment map. Thus, the image of $1_-\otimes1_+$ under the stated gluing map is also the image of $1_-\otimes1\in\mathcal S_0^2(D^2\times S^2;T_0)\otimes\mathcal S_0^2(B^4;T(n,n)_{n_+,n_-})$ under the gluing map given by $n$ saddles followed by an isotopy. The claim that the image equals $1$ follows from the commutativity of region $TX$:
\begin{figure}[H]
\centering
\includegraphics[width=0.7\linewidth]{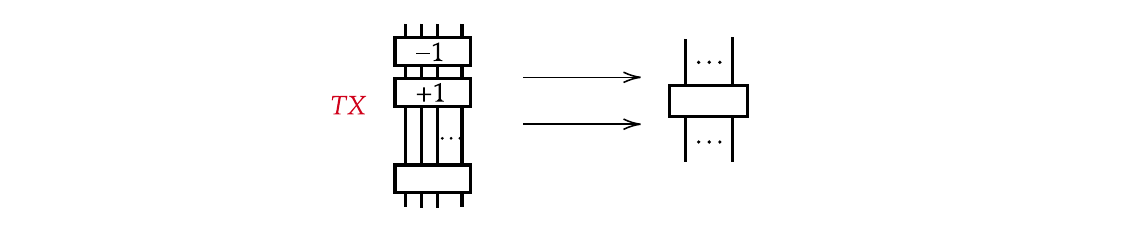},
\end{figure}
where the first map is the composition of two absorptions of twists into the projector, and the second map is the composition of Reidemeister-induced maps that cancel the two twists. As in the proof of other similar regions, the composition of the inverse of the first map with the second map is termwise the identity chain map up to sign as the composition of the cobordisms is isotopic to identity. To fix the sign, see Appendix~\ref{sec:sign_gluck}.

(5) Consider the diagram of isomorphisms on $\widetilde{KhR}_2^+$ in Figure~\ref{fig:gluck_sum_check}, where the first map on the second row is given by termwise maps that slide the strands on the left across the middle opening region of the second projector. We first reduce to checking the commutativity of Figure~\ref{fig:gluck_sum_check}.

\begin{figure}
\centering
\includegraphics[width=0.7\linewidth]{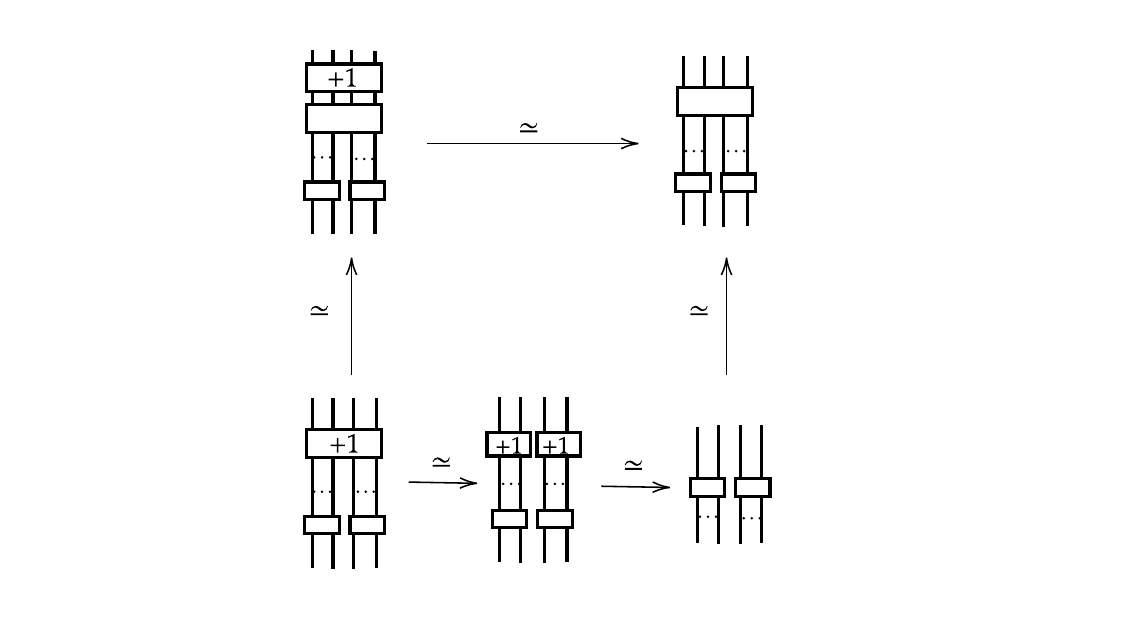}
\caption{Check compatibility of various element $1$'s under connected sums.}
\label{fig:gluck_sum_check}
\end{figure}

The image of $1\otimes1\otimes1\in\mathcal S_0^2(D^2\times S^2;U_0)\otimes\mathcal S_0^2(D^2\times S^2;U_1)\otimes\widetilde{KhR}_2^+(T(n,n)_{n_+,n_-})$ in the middle term in Figure~\ref{fig:1_1_1_to_1_+_1_+} is $v\otimes1\in\widetilde{KhR}_2^+(L)\otimes\mathcal S_0^2(D^2\times S^2\natural D^2\times S^2)\cong\mathcal S_0^2(D^2\times S^2\natural D^2\times S^2;L)$ where $L$ is the admissible link shown on the boundary, and $v$ is the element in the (closure of the) bottom left term in Figure~\ref{fig:gluck_sum_check} whose image in the (closure of the) bottom right term under the composite isomorphism that goes through the first row is equal to $1\otimes1$. The image of $v\otimes1$ in the last term in Figure~\ref{fig:1_1_1_to_1_+_1_+}, by the description of handleslide maps in Section~\ref{sec:concrete_handleslide} and the definition of the element $1$'s, is equal to $f(v)\otimes(1\otimes\mathrm{id}_\alpha(1))\in\widetilde{KhR}_2^+(T_0\cup T_1)\otimes\mathcal S_0^2(D^2\times S^2\natural D^2\times S^2)$ plus terms with lower lasagna quantum gradings, where $f$ is the first map on the second row of Figure~\ref{fig:gluck_sum_check}. To check $f(v)=1_+\otimes1_+$, it is thus sufficient to check the commutativity of Figure~\ref{fig:gluck_sum_check}.

In turn, the commutativity of Figure~\ref{fig:gluck_sum_check} is reduced to the commutativity of region $TG$:
\begin{figure}[H]
\centering
\includegraphics[width=0.7\linewidth]{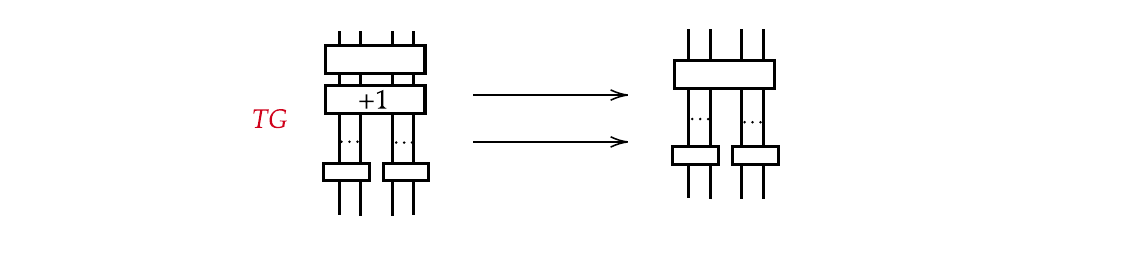},
\end{figure}
where the first map is the twist absorption into the projector on the top, and the second map is obtained by sliding the strands connected to the bottom-left projector across the bottom-right, thereby undoing the $+1$ linking between the two groups of strands, and then performing two twist absorptions into the bottom projectors. The commutativity of region $TG$ is checked termwise as before.
\end{proof}

We also prove the following technical lemma as a consequence of Lemma~\ref{lem:twist_belt_link}(5). This is not needed for the proof of Theorem~\ref{thm:gluck}.

\begin{Lem}\label{lem:gluck_sum}
Write $D_{std}:=\#_{i=1}^k\natural^{m_i}(D^2\times S^2)$ and $S_{std}=\partial D_{std}$. Let $X$ be an abstract $D_{std}$ together with an orientation-preserving identification $\partial X=S_{std}$. Let $S_0,S_1\subset int(X)$ be unoriented $2$-spheres, $S_0\#S_1$ be the connected sum of them along a given path $\gamma$, and $L\subset S_{std}$ be an admissible link. Suppose that
\begin{itemize}
\item $\psi\colon X_{S_0,S_1}\xrightarrow{\cong}D_{std}$ and $\psi'\colon X_{S_0\#S_1}\xrightarrow{\cong}D_{std}$ are orientation-preserving diffeomorphisms rel boundary;
\item The complement of $\nu(S_0\cup\gamma\cup S_1)$ in $X$ is a $4$-dimensional relative $1$-handlebody complement $W^t$ with $\partial_-W^t=-S_{std}$ and $\partial_+W^t=-\partial\nu(S_0\cup\gamma\cup S_1)$.
\end{itemize}
Then, the composition $$\mathcal S_0^2(D_{std};L)\xrightarrow[\cong]{\mathcal S_0^2(\psi^{-1})}\mathcal S_0^2(X_{S_0,S_1};L)\xrightarrow[\cong]{\tau_{X,L,S_0\cup S_1}^{-1}}\mathcal S_0^2(X;L)\xrightarrow[\cong]{\tau_{X,L,S_0\#S_1}}\mathcal S_0^2(X_{S_0\#S_1};L)\xrightarrow[\cong]{\mathcal S_0^2(\psi')}\mathcal S_0^2(D_{std};L)$$ induces a map on the $0$-th associated graded spaces with respect to the lasagna quantum grading. Under the isomorphism \eqref{eq:SZ}, this induced map is equal to $\mathrm{id}\otimes gr_0(\mathrm{id}_\alpha)$ for some $\alpha\in H_2(D_{std})$.
\end{Lem}
\begin{proof}
Orient $S_0,S_1$ so that the connected sum operation respects the orientations. Let $X_1=\nu(S_0\cup\gamma\cup S_1)$, which is naturally identified with $D_{std,1}:=\natural^2(D^2\times S^2)$ using the orientations of $S_0,S_1$.

A generic lasagna filling $(\Sigma_+^t,v)$ representing some element $x\in\mathcal S_0^2(X;L)$ can be assumed to have input balls away from $X_1$ and skein $\Sigma_+^t$ intersecting $X_1$ in a disjoint union of cocores transverse to $S_0\cup S_1$. The boundary of these cocores is a belt link $U_0\cup U_1$ in $\partial X_1\cong S_{std}$, as shown in the middle of Figure~\ref{fig:S_0_S_1}. Let $(\Sigma^t,v)$ denote the part of $(\Sigma_+^t,v)$ outside $X_1$, which can be thought of as a lasagna filling of $(W^t,-(U_0\cup U_1)\sqcup L)$. The spheres $S_0,S_1,S_0\#S_1$ are also shown in the middle of Figure~\ref{fig:S_0_S_1}. The Gluck twist $(X_1)_{S_0,S_1}$ can be naturally reidentified with $D_{std,1}$ via a diffeomorphism sending $U_0\cup U_1$ to the twisted belt link $T_0\cup T_1$ shown on the left of Figure~\ref{fig:S_0_S_1}. More explicitly, the Gluck twist on $S_j$ is performed along the copy of $S_j$ on $\partial\nu(S_j)$ closest to $\partial X_1$ by a counterclockwise full rotation around the center of the part of $S_j$ shown in Figure~\ref{fig:S_0_S_1} (the north pole) and the center of the $2$-core part of $S_j$ (the south pole), $j=0,1$, while the identification $(X_1)_{S_0,S_1}\cong D_{std,1}$ pushes the effect of these Gluck twists towards the boundary. By construction, $\tau_{X,L,S_0\cup S_1}(x)$ is the image of $1_+\otimes 1_+\in\mathcal S_0^2(D_{std,1};T_0\cup T_1)$ under gluing the lasagna filling $(\Sigma^t,v)$.

\begin{figure}
\centering
\includegraphics[width=0.75\linewidth]{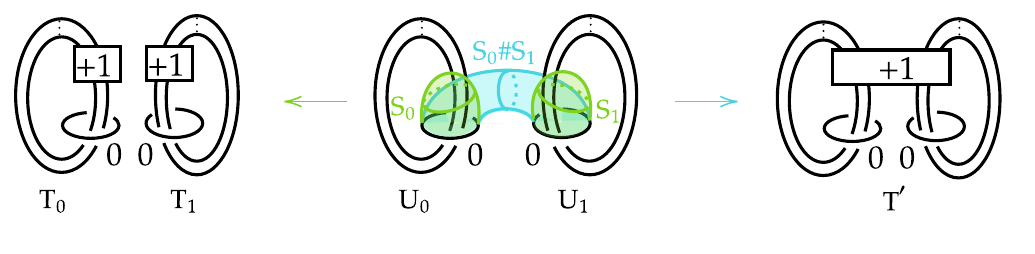}
\caption{Middle: The neighborhood $X_1=\nu(S_0\cup\gamma\cup S_1)\cong D_{std,1}$, with the belt link $U_0\cup U_1$ shown on the boundary, the spheres $S_0,S_1$ shown in green, and the sphere $S_0\#S_1$ shown in blue. The parts of spheres shown are slightly pushed into the interior of the $0$-handle. Left: $(X_1)_{S_0,S_1}$ reparametrized as $D_{std,1}$, with $U_0\cup U_1$ shown as the twisted belt link $T_0\cup T_1$ on the boundary. Right: $(X_1')_{S_0\#S_1}$ reparametrized as $D_{std,1}$, with $U_0\cup U_1$ shown as the admissible link $T'$ on the boundary.}
\label{fig:S_0_S_1}
\end{figure}

On the other hand, let $X_1'$ denote a slight shrunk copy of $X_1$. The Gluck twist $(X_1')_{S_0\#S_1}$ can be naturally reidentified with $D_{std,1}$ via a diffeomorphism sending $U_0\cup U_1$ to the admissible link $T'$ shown on the right of Figure~\ref{fig:S_0_S_1}, by a description analogous to the previous case. The boundary parametrizations of $(X_1)_{S_0,S_1}$ and $(X_1')_{S_0\#S_1}$ differ by $\phi=\tau_{S_0\#S_1}\circ(\tau_{S_0}\circ\tau_{S_1})^{-1}\in\Diff^{spin}(\partial D_{std,1})$, where here we abuse the notation and use $S_0,S_1$ to denote the two core spheres of $\partial D_{std,1}$, and $\tau_S$ to denote the Dehn twist along the sphere $S$. The mapping class of $\phi$ is trivial, hence the parametrization $(X_1')_{S_0\#S_1}\cong D_{std,1}$ extends by a levelwise diffeomorphism to a parametrization $(X_1)_{S_0\#S_1}\cong D_{std,1}$ whose boundary parametrization agrees with that of $(X_1)_{S_0,S_1}\cong D_{std,1}$. This levelwise diffeomorphism can be chosen to isotope the link $T'$ to $T_0\cup T_1$ by sliding the strands on the left over the second $2$-handle. Let $X_1''$ be a slight shrunk copy of $X_1'$, and push $S_0\#S_1$ slightly outside $X_1''$, within $X_1'$. To calculate $\tau_{X,L,S_0\#S_1}(x)$, we start with $1\in\mathcal S_0^2(X_1'';U_0\cup U_1)$. Add in an input ball in $int(X_1')\backslash X_1''$ with an input link $T(n,n)_{n_+,n_-}$ carrying the label $1\in KhR_2(T(n,n)_{n_+,n_-})$, as shown on the left of Figure~\ref{fig:1_1_1_to_1_+_1_+}, where $n_+$ and $n_-$ denote the numbers of positively and negatively oriented strands of $U_0\cup U_1$, respectively, and $n=n_++n_-$. Then evaluate to $((X_1')_{S_0\#S_1},T')$, implementing the effect of the lasagna Gluck twist along $S_0\#S_1$ within $X_1'$. Next, evaluate to $((X_1)_{S_0\#S_1},T_0\cup T_1)$ by gluing in the trace of the isotopy $T'\sim T_0\cup T_1$. Finally, glue in $(\Sigma^t,v)$ to obtain $\tau_{X,L,S_0\#S_1}(x)\in\mathcal S_0^2(X_{S_0\#S_1};L)$. By Lemma~\ref{lem:twist_belt_link}(5), $\tau_{X,L,S_0\#S_1}(x)$ is thus equal to the image of $1_+\otimes\mathrm{id}_\alpha(1_+)+\cdots\in\mathcal S_0^2(D_{std,1};T_0\cup T_1)$ under gluing in the lasagna filling $(\Sigma^t,v)$, for some $\alpha\in H_2(D^2\times S^2)$, where $\cdots$ are terms with negative lasagna quantum degrees.

Now, change $\psi'$ by a diffeomorphism rel boundary if necessary (which does not affect the statement thanks to Proposition~\ref{prop:D2S2_repara}), we may assume $\psi$ and $\psi'$ agree outside $X_1$, and differ by the composite diffeomorphism $(X_1)_{S_0,S_1}\cong D_{std,1}\cong(X_1)_{S_0\#S_1}$ inside $X_1$. As such, the element $\mathcal S_0^2(\psi)(\tau_{X,L,S_0\cup S_1}(x))$ (resp. $\mathcal S_0^2(\psi')(\tau_{X,L,S_0\#S_1}(x))$) is equal to the image of $z_1:=1_+\otimes1_+\in\mathcal S_0^2(D_{std,1};T_0\cup T_1)$ (resp. $z_2:=1_+\otimes\mathrm{id}_\alpha(1)+\cdots\in\mathcal S_0^2(D_{std,1};T_0\cup T_1)$) under gluing $\psi_*(\Sigma^t,v)$, the pushforward of $(\Sigma^t,v)$ under $\psi$. We drop $\psi$ from the notation near $D_{std,1}$ in what follows.

By changing $\psi_*(\Sigma^t,v)$ in its equivalence class, we may assume that it contains a single input ball, which is near $\partial D_{std,1}$. Further decompose $\psi_*(\Sigma^t,v)$ into
\begin{enumerate}[(i)]
\item A lasagna filling in a collar neighborhood $I\times\partial D_{std,1}$ of $\partial D_{std,1}$ rel $-(T_0\cup T_1)\sqcup(T_0\cup T_1\sqcup L_0)$ with an input link $L_0$ and skein $I\times(T_0\cup T_1)\cup[1/2,1]\times L_0$. Here, when regarded as a link in $\partial D_{std,1}$, $L_0$ is admissible and distant from $T_0\cup T_1$.
\item A link cobordism in (a shrunk copy of) $\psi(W^t)$ from $T_0\cup T_1\sqcup L_0$ to $L$.
\end{enumerate}
Since gluing (i) amounts to putting some label $v$ in the new tensorial factor $\widetilde{KhR}_2^+(L_0)$ of $\widetilde{KhR}_2^+(T_0\cup T_1\sqcup L_0)$ (which do not contribute to lasagna quantum degrees), the images of $z_1,z_2$ under (i) also differ by a shifting isomorphism plus terms with negative lasagna quantum degrees. Finally, in view of Theorem~\ref{thm:spin_functoriality_las}, their images under the whole gluing of $(\Sigma^t,v)$ differ by a shifting isomorphism plus terms with negative lasagna quantum degrees, as desired.
\end{proof}

We now give the proof of Theorem~\ref{thm:gluck}.

\begin{proof}[Proof of Theorem~\ref{thm:gluck}]
The map $\tau_{X,L,S,\pm}$ in \eqref{eq:gluck} is constructed as follows.

Let $(\Sigma,v)$ be a lasagna filling of $(X,L)$ representing a given element $x\in\mathcal S_0^2(X;L)$. By general position, we may assume that the input balls of $\Sigma$ are disjoint from $S$, and that $\Sigma$ intersects $S$ transversely at some finitely many points. Let $\overline{\nu(S)}\cong D^2\times S^2$ be a closed tubular neighborhood of $S$, so that $\Sigma$ intersects $D^2\times S^2$ in some finite number of cocore disks, with some boundary $U\subset S^1\times S^2$.

If $U$ has an odd number of components, then Theorem~\ref{thm:SZ} implies that $(\Sigma,v)$ defines the zero class when restricted to a lasagna filling of $(D^2\times S^2,U)$, hence it also defines the zero class in $\mathcal S_0^2(X;L)$.

Suppose now $U$ has an even number of components. Then $U$ is a belt link. Write $(X,\Sigma)=(X\backslash\nu(S),\Sigma\backslash\nu(S))\cup(D^2\times S^2,\Sigma\cap D^2\times S^2)$, we see that $x=[(\Sigma,v)]$ is the image of $1$ under the map $$\mathcal S_0^2(D^2\times S^2;U)\to\mathcal S_0^2(X;L)$$ that glues in the lasagna filling $(\Sigma\backslash\nu(S),v)$ of $(X\backslash\nu(S),L\cup(-U))\cong(X_S\backslash\nu(S),L\cup(-\tau(U)))$.

The Gluck twist $X_S$ is obtained from $X$ by cutting out $D^2\times S^2$ and regluing it back via the Dehn twist $\tau$. The belt link $U$, on the boundary of this glued-in $D^2\times S^2$, is thus the twisted belt link $\tau(U)$. We define $\tau_{X,L,S,\pm}(x)$ to be the image of $1_\pm$ under the gluing map $$\mathcal S_0^2(D^2\times S^2;\tau(U))\to\mathcal S_0^2(X_S;L)$$ that glues in the lasagna filling $(\Sigma\backslash\nu(S),v)$ of $(X\backslash\nu(S),L\cup(-U))$.\medskip

We check that $\tau_{X,L,S,\pm}$ is well-defined.

First, by Lemma~\ref{lem:twist_belt_link}(1), $\tau_{X,L,S,\pm}([(\Sigma,v)])$ is independent of the parametrization $\overline{\nu(S)}\cong D^2\times S^2$.

Next, we check that $\tau_{X,L,S,\pm}([(\Sigma,v)])$ is independent of $(\Sigma,v)$. It is clearly linear in the label $v$. If $(\Sigma',v')$ is another lasagna filling whose input balls contains those of $\Sigma$ and are disjoint from $S$, the two gluing maps leading to $\tau_{X,L,S,\pm}(x)$ are equal. Hence, it remains to check that $\tau_{X,L,S,\pm}([(\Sigma,v)])$ is invariant under isotoping the skein $\Sigma$ rel boundary.

By general position, every isotopy of the skein $\Sigma$ rel boundary is decomposed into some finger/Whitney moves that create/annihilate transverse intersections $\Sigma\cap S$ in pairs, and some isotopies that do not create/annihilate intersections with $S$. We need to show that $\tau_{X,L,S,\pm}([(\Sigma,v)])$ is invariant under a finger move from $\Sigma$ to some $\Sigma'$ across $S$.

By a neck-cutting, we may assume that the component of $\Sigma$ that undergoes the finger move is a local $2$-sphere $S_0$, which can be either undotted or dotted. When $S_0$ is undotted, $\tau_{X,L,S,\pm}([(\Sigma,v)])$ is zero since $S_0$ evaluates to zero. For $(\Sigma',v)$, since a collar neighborhood of the boundary of $\nu(S)$ outside $\nu(S)$ contains $\Sigma'$ as an undotted annulus annihilation map, we conclude from Lemma~\ref{lem:twist_belt_link}(2) that $\tau_{X,L,S,\pm}([(\Sigma',v)])$ is also zero. Similarly, when $S_0$ is dotted, we find using Lemma~\ref{lem:twist_belt_link}(3) that $\tau_{X,L,S,\pm}([(\Sigma,v)])$ and $\tau_{X,L,S,\pm}([(\Sigma',v)])$ are equal. This shows that $\tau_{X,L,S,\pm}$ is well-defined.

It is clear from the definition that $\tau_{X,L,S,\pm}$ preserves the homological and the quantum degrees, and covers $H_2(\tau_S)_\pm$ in the skein degree.\smallskip

We prove the addenda (1) to (4). Item (4) is trivial. Item (1) follows by comparing the degree of $1\in\mathcal S_0^2(D^2\times S^2;U)$ with those of $1_\pm\in\mathcal S_0^2(D^2\times S^2;\tau(U))$.

(2) By general position, we may isotope any skein in $X$ rel $L$ to be disjoint from a given $3$-ball bounding $S$. The statement follows.

(3) Push $S\subset X_S$ to be disjoint from $\nu(S)\subset X_S$, and let $S'$ denote this pushoff copy. Find a copy of $D^2\times S^2$ in $int(X)$ containing $\nu(S)\cup\nu(S')$ as the tubular neighborhood of two core spheres. Any given lasagna filling $(\Sigma,v)$ of $(X,L)$ can be assumed to have input balls disjoint from $D^2\times S^2$ and $\Sigma$ intersecting $D^2\times S^2$ in some number of cocores. The statement follows from Lemma~\ref{lem:twist_belt_link}(4) applied to the gluing of $((D^2\times S^2)\backslash(\nu(S)\cup\nu(S')),\Sigma\backslash(\nu(S)\cup\nu(S')))$ onto $(\nu(S)\sqcup\nu(S'),\tau(\partial\nu(S)\cap\Sigma)\sqcup\tau(\partial\nu(S')\cap\Sigma))$.
\end{proof}

\subsection{\texorpdfstring{$\widetilde{KhR}_2^+$}{KhR2+} on objects}\label{sec:nonspin_obj}
As before, let $D_{std}=\#_{i=1}^k\natural^{m_i}(D^2\times S^2)$ and $S_{std}=\partial D_{std}=\sqcup_{i=1}^k\#^{m_i}(S^1\times S^2)$. We prove the following theorem.

\begin{Thm}\label{thm:abstract_homology}
Theorem~\ref{thm:abstract_spin_homology} is still true when $S$ is only assumed to be an abstract $S_{std}$.\footnote{We warn the readers that the notation $S$ is used both for a $3$-manifold that is an abstract $S_{std}$, and a union of disjoint embedded $2$-spheres in a $4$-manifold. In the rest of this section, $S$ will only appear for the second purpose.}
\end{Thm}

In other words, if $L\subset S_{std}$ is admissible and $\phi\in\Diff^+(S_{std})$ is any $L$-admissible orientation-preserving diffeomorphism, we wish to assign an isomorphism $\widetilde{KhR}_2^+(\phi)\colon\widetilde{KhR}_2^+(\phi(L))\xrightarrow{\cong}\widetilde{KhR}_2^+(L)$ functorially in $\phi$. This is provided by Theorem~\ref{thm:nonspin_repara_las}, which will feature in the later proof of Theorem~\ref{thm:abstract_homology}.

\begin{Thm}\label{thm:nonspin_repara_las}
Let $L\subset S_{std}$ be an admissible link and $\phi\in\Diff^+(S_{std})$ be an orientation-preserving $L$-admissible diffeomorphism. Let $S\subset int(D_{std})$ be a $\partial$-parallel finite union of disjoint embedded unoriented $2$-spheres, and $\tilde\phi\colon D_{std}\xrightarrow{\cong}(D_{std})_S$ be a diffeomorphism that extends $\phi$. Then the composition $$\mathcal S_0^2(D_{std};\phi(L))\xrightarrow{\tau_{D_{std},\phi(L),S}}\mathcal S_0^2((D_{std})_S;\phi(L))\xrightarrow{\mathcal S_0^2(\tilde\phi^{-1})}\mathcal S_0^2(D_{std};L)$$ induces a map on the $0$-th associated graded spaces with respect to the lasagna quantum grading. Under the isomorphism \eqref{eq:SZ}, this induced map is uniquely of the form 
\begin{equation}\label{eq:nonspin_repara_las}
\widetilde{KhR}_2^+(\phi)\otimes gr_0(\mathrm{id}_\alpha\circ\phi_*^{-1})\colon\widetilde{KhR}_2^+(\phi(L))\otimes gr_0\mathcal S_0^2(D_{std})\to\widetilde{KhR}_2^+(L)\otimes gr_0\mathcal S_0^2(D_{std})
\end{equation}
for some isomorphism $\widetilde{KhR}_2^+(\phi)$ independent of $S$ and $\tilde\phi$, and some $\alpha\in H_2(D_{std})$. Here, $\phi_*$ in \eqref{eq:nonspin_repara_las} is defined as the composition $\mathcal S_0^2(D_{std})\cong\mathcal S_0^2(I\times S_{std})\xrightarrow[\cong]{\mathcal S_0^2(I\times\phi)}\mathcal S_0^2(I\times S_{std})\cong\mathcal S_0^2(D_{std})$.
\end{Thm}

The existence of an extension $\tilde\phi$ as in Theorem~\ref{thm:nonspin_repara_las} can be seen as follows. The set of spin structures on $S_{std}$ is affine over $H^1(S_{std};\Z/2)$. There is a distinguished spin structure, denoted $\mathfrak s_0$, which is the restriction of the unique spin structure on $D_{std}$. Write $\phi_*(\mathfrak s_0)=\mathfrak s_0+\beta$ for some $\beta\in H^1(S_{std};\Z/2)\cong H_2(S_{std};\Z/2)\cong H_2(D_{std};\Z/2)$, represented by some $S=\sqcup_{i=1}^mS_i\subset int(D_{std})$, a $\partial$-parallel union of disjoint embedded unoriented $2$-spheres. The restriction of the spin structure on $(D_{std})_S$ on $S_{std}$ is $\mathfrak s_0+[S]=\phi_*(\mathfrak s_0)$, and the existence of $\tilde\phi$ follows from the surjectivity of $\Diff^+(D_{std})\to\Diff^{spin}(S_{std})$.

We will prove Theorem~\ref{thm:nonspin_repara_las} after some topological observations.

\begin{Lem}\label{lem:disjoint_2_spheres}
If $S_0,S_1\subset S_{std}$ are two collections of disjoint embedded unoriented $2$-spheres, then $S_0$ is related to some other collection $S_0'$ that is disjoint from $S_1$ via a sequence of (isotopies and) inverses of the connected sum operation.
\end{Lem}

\begin{proof}
Put $S_0$ and $S_1$ into transverse position. Surger $S_0$ along innermost disks on $S_1$ bounded by $S_0\cap S_1$.
\end{proof}

\begin{Lem}\label{lem:relate_2_spheres}
If $S_0,S_1\subset S_{std}$ are two collections of disjoint embedded unoriented $2$-spheres with $[S_0]=[S_1]\in H_2(S_{std};\Z/2)$, then $S_0$ is related to $S_1$ via a sequence of (isotopies, and) connected sums, separating $2$-sphere creations, and their inverses.
\end{Lem}

\begin{proof}
Without loss of generality, say $S_{std}$ is connected. Take connected sums to make both $S_0$ and $S_1$ connected. If $[S_0]=[S_1]=0$, then $S_0,S_1$ are both separating, hence they are related by the moves. If $[S_0]=[S_1]\ne0$, apply Lemma~\ref{lem:disjoint_2_spheres} to make $S_0$ disjoint from $S_1$. Since $S_1$ is nonseparating, we may take connected sums outside $S_1$ to make $S_0$ connected again. Now form a connected sum $S:=S_0\#S_1$, which may be assumed to be disjoint from $S_0\cup S_1$. Since $[S]=[S_0]+[S_1]=0$, $S$ is separating. Therefore, $S_0$ is related to $S_1$ via one separating sphere creation and one connected sum: $S_0\sim S_0\cup S\sim S_1$.
\end{proof}

\begin{Lem}\label{lem:partial_parallel_complement}
Let $S_0,S_1\subset S_{std}=\partial D_{std}$ be two disjoint $2$-spheres, $\gamma\subset S_{std}$ be an arc connecting them, and $D_{std}\subset S^4$ be the standard embedding with complement a $4$-dimensional $1$-handlebody. Then the complement of a neighborhood of $S_0\cup\gamma\cup S_1$ in $S^4$ is a $4$-dimensional $1$-handlebody.
\end{Lem}
\begin{proof}
Since $S_0,S_1$ lie on the same connected component of $S_{std}$, by capping off extra components if necessary, we may assume $D_{std}=\natural^m(D^2\times S^2)$. It suffices to show that $S_0,S_1$ bound disjoint $3$-balls $B_0,B_1$ in $S^4$ whose interiors are disjoint from $\gamma$. We divide into four cases.

\textbf{Case 1}: $S_0,S_1$ are both separating.

Write $D_{std}=(D^2\times S^2)\natural\cdots\natural(D^2\times S^2)$ and $S_{std}=(S^1\times S^2)\#\cdots\#(S^1\times S^2)$. By a spin parametrization change of $S_{std}$ (which extends to $D_{std}$), we may assume that $S_j$ is a copy of the $a_j$-th connected sum $2$-sphere in $S_{std}$, $j=0,1$, $0\le a_0\le a_1\le m-1$ (when $a_j=0$, we interpret this as saying that $S_j$ bounds a ball in $S_{std}$). We may choose $B_j$ to be a copy of the $a_j$-th boundary connected sum $3$-ball in $D_{std}$, $j=0,1$ (when $a_j=0$, we interpret as saying that $B_j$ is a $\partial$-parallel $3$-ball).

\textbf{Case 2}: $S_0$ is separating, $S_1$ is nonseparating.

As in Case 1, $S_0$ bounds a $3$-ball $B_0$ in $D_{std}$. By applying a spin parametrization change of $S_{std}$ to make $S_1$ standard, we see that $S_1$ is the belt sphere of the cocore of a $1$-handle in the $4$-dimensional $1$-handlebody $S^4\backslash D_{std}$. Hence we may choose $B_1$ to be this cocore $3$-ball.

\textbf{Case 3}: $S_0\cup S_1$ is nonseparating.

By a spin parametrization change of $S_{std}$, we see that $S_0,S_1$ are belt spheres of the cocores of two different $1$-handles in $S^4\backslash D_{std}$. Hence we may choose $B_0,B_1$ to be these cocore $3$-balls.

\textbf{Case 4}: $S_0,S_1$ are nonseparating, but $S_0\cup S_1$ is separating.

If $S_0$ and $S_1$ are isotopic, they bound two parallel cocores of a $1$-handle in $S^4\backslash D_{std}$. If they are nonisotopic, by a spin parametrization change of $S_{std}=(S^1\times S^2)\#\cdots\#(S^1\times S^2)$, we may assume that they are embedded in the $i$-th connected summand in $S_{std}$ as two nonisotopic $S^2$ factors, for some $1<i<n$. It is clear that they bound disjoint $3$-balls $B_0,B_1$ in the $i$-th boundary connected summand of $S^4\backslash D_{std}=(S^1\times B^3)\natural\cdots\natural(S^1\times B^3)$.
\end{proof}

\begin{proof}[Proof of Theorem~\ref{thm:nonspin_repara_las}]
The independence of the statement on the choice of $\tilde\phi$ follows from Proposition~\ref{prop:D2S2_repara}.

We prove its independence on $S$, a $\partial$-parallel finite union of disjoint unoriented $2$-spheres. Suppose $S'$ is another such union of $2$-spheres, then $[S]=[S']=\phi_*(\mathfrak s_0)-\mathfrak s_0\in H_2(D_{std};\Z/2)\cong H_2(S_{std};\Z/2)$. Note that every separating $2$-sphere in $S_{std}$ is unknotted in $D_{std}$, thus by Lemma~\ref{lem:relate_2_spheres}, $S$ and $S'$ are related by a sequence of
\begin{enumerate}[(a)]
\item unknotted $2$-sphere creations,
\item $\partial$-parallel connected sums,
\end{enumerate}
or their inverses, through $\partial$-parallel union of $2$-spheres. This sequence of moves determine a canonical identification $(D_{std})_S=(D_{std})_{S'}$ rel boundary. Theorem~\ref{thm:gluck}(2) implies the independence of $S$ under type (a) moves, while Lemma~\ref{lem:partial_parallel_complement} and Lemma~\ref{lem:gluck_sum} imply the independence of $S$ under type (b) moves.

In view of the commutative diagram \eqref{eq:gluck_las_functorial}, it suffices to decompose $\phi$ into elementary diffeomorphisms and prove the theorem for each elementary one. The mapping class group of $S_{std}$ is generated by the spin mapping class group $\pi_0(\Diff^{spin}(S_{std}))$ together with Dehn twists along each of the $\sum_{i=1}^km_i$ core $S^2$ factors, each of which can be taken to have a standard form, visualized as in Figure~\ref{fig:Dehn_twist} in the presence of an admissible link.

\begin{figure}
\centering
\includegraphics[width=0.7\linewidth]{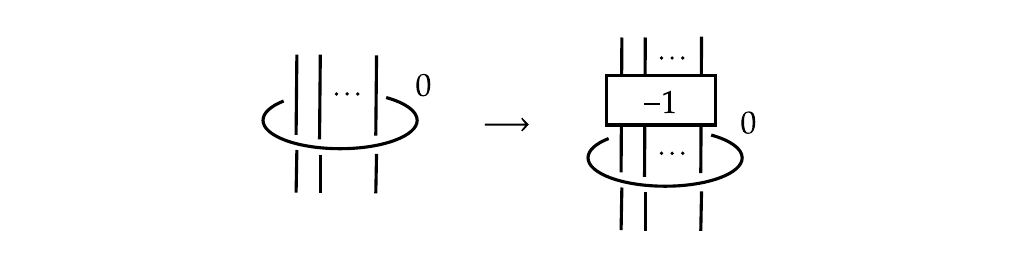}
\caption{The Dehn twist on $S_{std}$ along a core $2$-sphere.}
\label{fig:Dehn_twist}
\end{figure}

When $\phi\in\Diff^{spin}(S_{std})$, choose $S=\emptyset$, the statement follows from the spin case, Theorem~\ref{thm:spin_repara_las}.

When $\phi$ is the negative Dehn twist along the $j$-th $S^2$ factor in $S_{std}$, let $S\subset int(D_{std})$ be a slight pushin of the Dehn twist $2$-sphere in $S_{std}$ and $\tilde\phi\colon D_{std}\to(D_{std})_S$ be the natural diffeomorphism extending $\phi$, given by pushing the effect of the twist towards the boundary. Take a standard lasagna filling $(I\times\phi(L)\cup(n_+,n_-)\text{ cores},v)$, $v\in\widetilde{KhR}_2^+(\phi(L)\cup(n_+,n_-)\text{ belts})$ representing some element $x\in\mathcal S_0^2(D_{std};\phi(L))$, $n_\pm\in\Z^{\sum_{i=1}^km_i}$. By construction of the Gluck twist map \eqref{eq:gluck} in the proof of Theorem~\ref{thm:gluck} and the alternative description of the element $1_+$ preceding Lemma~\ref{lem:twist_belt_link}, the image of $x$ under the composition map $\mathcal S_0^2(\tilde\phi^{-1})\circ\tau_{D_{std},\phi(L),S}$ is represented by the standard lasagna filling $(I\times L\cup(n_+,n_-)\text{ cores},w)$, where $w\in\widetilde{KhR}_2^+(L\cup(n_+,n_-)\text{ belts})$ is the image of $v\otimes1\in\widetilde{KhR}_2^+(\phi(L)\cup(n_+,n_-)\text{ belts})\otimes\widetilde{KhR}_2^+(T(\ell,\ell)_{\ell_+,\ell_-})$ under the composition (shown near the $j$-th surgery region)
\begin{figure}[H]
\centering
\includegraphics[width=0.7\linewidth]{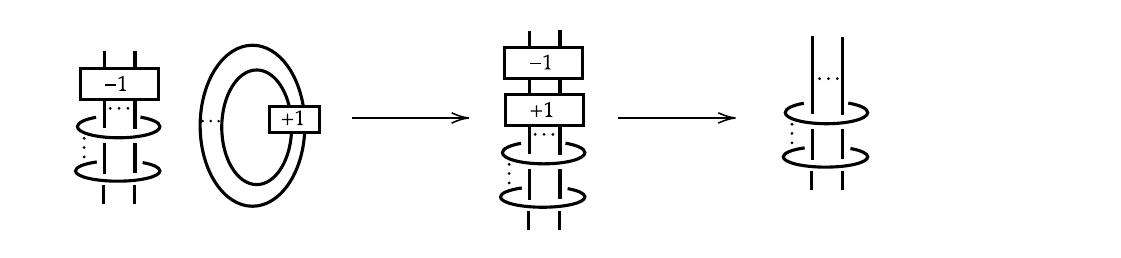},
\end{figure}
where $\ell_+,\ell_-$ denote the number of times $L$ intersects the $j$-th surgery region positively, negatively, respectively, and $\ell=\ell_++\ell_-$. Since this map is distant from the surgery region, locality implies that the statement holds for $S,\tilde\phi$ with $\alpha=0$ and $\widetilde{KhR}_2^+(\phi)$ given by saddling in $1\in\widetilde{KhR}_2^+(T(\ell,\ell)_{\ell_+,\ell_-})$ right above the $j$-th projector, composed with Reidemeister-induced maps that cancels the $\pm1$ twists.
\end{proof}

\begin{proof}[Proof of Theorem~\ref{thm:abstract_homology}] 
This is similar to deducing Theorem~\ref{thm:abstract_spin_homology} from Theorem~\ref{thm:spin_repara_las}. The only difference is that the lasagna level functoriality itself is more involved, which we now address.

Suppose $\phi_j\in\Diff^+(S_{std})$ with lifts $\tilde\phi_j\colon D_{std}\xrightarrow{\cong}(D_{std})_{S_j}$, $j=0,1$, so that $\phi_0$ is $L$-admissible and $\phi_1$ is $\phi_0(L)$-admissible. Let $N:=(-1,0]\times S_{std}$ be a collar neighborhood of the boundary of $D_{std}$. After some isotopies, we may assume that
\begin{itemize}
\item $S_j\subset\{-(j+1)/3\}\times S_{std}\subset N$, $j=0,1$;
\item $\tilde\phi_j=\mathrm{id}\times\phi_j$ on $N$, $j=0,1$.
\end{itemize}
The composition map $\phi:=\phi_1\circ\phi_0$ extends to a map $\tilde\phi\colon D_{std}\to(D_{std})_S$, where $S=\tilde\phi_1(S_0)\cup S_1$. More precisely, we take $\tilde\phi$ as the composition $$D_{std}\xrightarrow{\tilde\phi_0}(D_{std})_{S_0}\xrightarrow{(\tilde\phi_1)_{S_0}}((D_{std})_{S_1})_{\tilde\phi_1(S_0)}\cong(D_{std})_S,$$ where the middle map is the natural map induced by $\tilde\phi_1$.

Consider the following diagram of isomorphisms:
\begin{equation}\label{eq:gluck_las_functorial}
\begin{tikzcd}
\mathcal S_0^2(D_{std};\phi(L))\ar[rr,"{\tau_{D_{std},\phi(L),S}}"]\ar[dd,"{\tau_{D_{std},\phi(L),S_1}}"']&&\mathcal S_0^2((D_{std})_S;\phi(L))\ar[rr,"{\mathcal S_0^2(\tilde\phi^{-1})}"]\ar[rrdd,"{\mathcal S_0^2((\tilde\phi_1)_{S_0}^{-1})}"]&&\mathcal S_0^2(D_{std};L)\\\\
\mathcal S_0^2((D_{std})_{S_1};\phi(L))\ar[rr,"{\mathcal S_0^2(\tilde\phi_1^{-1})}"']\ar[uurr,"{\tau_{(D_{std})_{S_1},\phi(L),\tilde\phi_1(S_0)}}"]&&\mathcal S_0^2(D_{std};\phi_0(L))\ar[rr,"{\tau_{D_{std},\phi_0(L),S_0}}"']&&\mathcal S_0^2((D_{std})_{S_0};\phi_0(L))\ar[uu,"{\mathcal S_0^2(\tilde\phi_0^{-1})}"'].
\end{tikzcd}
\end{equation}

The upper left triangle is commutative by Theorem~\ref{thm:gluck}(4). The lower quadrilateral is commutative by naturality of the lasagna Gluck twist construction. The upper right triangle is trivially commutative. Note that the first row of \eqref{eq:gluck_las_functorial} is a lasagna defining map for $\widetilde{KhR}_2^+(\phi)\colon\widetilde{KhR}_2^+(\phi(L))\to\widetilde{KhR}_2^+(L)$ in the sense of Theorem~\ref{thm:nonspin_repara_las}, even though $S$ is not necessarily $\partial$-parallel. This is because one may apply Lemma~\ref{lem:disjoint_2_spheres} to change $\tilde\phi_1(S_0)$ by some inverse connected sum operations within $\{-1/3\}\times S_{std}$ to make $S$ $\partial$-parallel, while Lemma~\ref{lem:partial_parallel_complement} and Lemma~\ref{lem:gluck_sum} guarantee that the induced map of the first row of \eqref{eq:gluck_las_functorial}, in terms of \eqref{eq:nonspin_repara_las}, is unchanged except possibly by altering the shift $\alpha\in H_2(D_{std})$. Theorem~\ref{thm:abstract_homology} now follows from Theorem~\ref{thm:nonspin_repara_las} by an argument analogous to that in the proof of Theorem~\ref{thm:abstract_spin_homology}. In particular, $\widetilde{KhR}_2^+$ is defined on objects of $\mathbf{Links}_1$ as a ``cross-section'' of $\prod_{\rho\in P(S,L)}\widetilde{KhR}_2^+(\rho(L))$ as in the proof of Theorem~\ref{thm:abstract_spin_homology}, where $P(S,L)$ is the set of $L$-admissible parametrizations of $S$.
\end{proof}

\subsection{\texorpdfstring{$\widetilde{KhR}_2^+$}{KhR2+} on morphisms}\label{sec:nonspin_mor}
We are finally able to prove Theorem~\ref{thm:KhR_2+}. 

\begin{proof}[Proof of Theorem~\ref{thm:KhR_2+}]
Let $(W,\Sigma)\colon(S_0,L_0)\to(S_1,L_1)$ be a morphism in $\mathbf{Links}_1$, where $W=X_1\backslash int(X_0)$ for $4$-dimensional $1$-handlebodies $X_0,X_1$. As in the proof of Theorem~\ref{thm:abstract_spin_functoriality}, choose an orientation-preserving embedding $i\colon X_1\hookrightarrow S^4$ and parametrizations $\tilde\phi_j\colon-(S^4\backslash int(X_j))\xrightarrow{\cong}D_{std,j}=\#_{i=1}^{k^{(j)}}\natural^{m_i^{(j)}}(D^2\times S^2)$ making $\phi_j:=\tilde\phi_j|_{S_j}$ $L_j$-admissible, $j=0,1$. Then the map $\widetilde{KhR}_2^+(\tilde\phi_0(W),$ $\tilde\phi_0(\Sigma))\colon\widetilde{KhR}_2^+(\phi_1(L_1))\to\widetilde{KhR}_2^+(\phi_0(L_0))$ determines a map $\widetilde{KhR}_2^+(W,\Sigma)\colon\widetilde{KhR}_2^+(S_1,L_1)\to\widetilde{KhR}_2^+(S_0,L_0)$. We have to check that $\widetilde{KhR}_2^+(W,\Sigma)$ is independent of the embedding $i\colon X_1\hookrightarrow S^4$. The functoriality will be automatic, as we may pick any spin structure on the relevant $4$-dimensional $1$-handlebodies and repeat the arguments in the proof of Theorem~\ref{thm:abstract_spin_functoriality}. Furthermore, it suffices to check the independence on $i$ when $(W,\Sigma)$ is an elementary morphism as described in Proposition~\ref{prop:abstract_decomposition}.

Suppose $i'\colon X_1\hookrightarrow S^4$ is another choice of embedding, and $\tilde\phi_j'\colon-(S^4\backslash int(X_j))\xrightarrow{\cong}D_{std,j}$ are parametrizations, $j=0,1$. Choose a diffeomorphism 
$$\psi_1\colon-(S^4\backslash int(i'(X_1)))\xrightarrow{\cong}(-(S^4\backslash int(i(X_1))))_S$$
rel boundary for some $\partial$-parallel union of $2$-spheres $S$, and let
$$\psi_0\colon-(S^4\backslash int(i'(X_0)))\xrightarrow{\cong}(-(S^4\backslash int(i(X_0))))_S$$
be the extension of $\psi_1$ rel $W$.

We have the following commutative diagram
$$\begin{tikzcd}
\mathcal S_0^2(D_{std,1};\phi_1(L_1))\ar[rrr,"{\mathcal S_0^2(\tilde\phi_0(i(W^t));\tilde\phi_0(i(\Sigma^t)))}"]&&&\mathcal S_0^2(D_{std,0};\phi_0(L_0))\\
\mathcal S_0^2(-(S^4\backslash int(i(X_1)));L_1)\ar[rrr,"{\mathcal S_0^2(i(W^t);i(\Sigma^t))}"]\ar[u,"{\mathcal S_0^2(\tilde\phi_1)}","\cong"']\ar[d,"{\tau_{-(S^4\backslash int(i(X_1))),L_1,S}}"',"\cong"]&&&\mathcal S_0^2(-(S^4\backslash int(i(X_0)));L_0)\ar[u,"{\mathcal S_0^2(\tilde\phi_0)}"',"\cong"]\ar[d,"{\tau_{-(S^4\backslash int(i(X_0))),L_0,S}}","\cong"']\\
\mathcal S_0^2((-(S^4\backslash int(i(X_1))))_S;L_1)\ar[rrr,"{\mathcal S_0^2(i(W^t);i(\Sigma^t))}"]&&&\mathcal S_0^2((-(S^4\backslash int(i(X_0))))_S;L_0)\\
\mathcal S_0^2(-(S^4\backslash int(i'(X_1)));L_1)\ar[rrr,"{\mathcal S_0^2(i'(W^t);i'(\Sigma^t))}"]\ar[u,"{\mathcal S_0^2(\psi_1)}","\cong"']\ar[d,"{\mathcal S_0^2(\tilde\phi_1')}"',"\cong"]&&&\mathcal S_0^2(-(S^4\backslash int(i'(X_0)));L_0)\ar[u,"{\mathcal S_0^2(\psi_0)}"',"\cong"]\ar[d,"{\mathcal S_0^2(\tilde\phi_0')}","\cong"']\\
\mathcal S_0^2(D_{std,1};\phi_1'(L_1))\ar[rrr,"{\mathcal S_0^2(\tilde\phi_0'(i'(W^t));\tilde\phi_0'(i'(\Sigma^t)))}"]&&&\mathcal S_0^2(D_{std,0};\phi_0'(L_0)).
\end{tikzcd}$$
Here, the composition down the first column gives the lasagna defining map for $\widetilde{KhR}_2^+(\phi_1\circ\phi_1'^{-1})\colon\widetilde{KhR}_2^+(\phi_1(L_1))\to\widetilde{KhR}_2^+(\phi_1'(L_1))$ in the sense of Theorem~\ref{thm:nonspin_repara_las}, because one can commute the first two isomorphisms.

Similarly, we claim that the composition down the second column serves as lasagna defining map for $\widetilde{KhR}_2^+(\phi_0\circ\phi_0'^{-1})$, which will finish the proof that $\widetilde{KhR}_2^+(W,\Sigma)$ is well-defined. Since we have assumed $(W,\Sigma)$ to be an elementary morphism as described in Proposition~\ref{prop:abstract_decomposition}, we check for each case. 

\eqref{item:product}\eqref{item:2_handle_cancel}: $S\subset\partial_+W$ is parallel to $\partial_-W$.

\eqref{item:ball_creation}: Up to capping off some unknots in $W$ (which does not affect the composite map by Theorem~\ref{thm:gluck}(2)), $S$ is parallel to $\partial_-W$.

\eqref{item:connected_sum}\eqref{item:1_handle}: One can surger as in Lemma~\ref{lem:disjoint_2_spheres} to make $S\subset\partial_+W$ disjoint from the belt sphere of the $1$-handle (which does not affect the induced map on $\widetilde{KhR}_2^+$ by Lemma~\ref{lem:partial_parallel_complement} and Lemma~\ref{lem:gluck_sum}), so that it is parallel to $\partial_-W$.

The proof is complete.
\end{proof}

\appendix
\section{Sign fixes}\label{sec:sign}
In this appendix, we prove Theorem~\ref{thm:gl2_webs_functorial}, the precise version of Theorem~\ref{thm:gl2_webs_functorial_intro}, and explain how to use the $\gl_2$ webs and foams formalism to fix various sign ambiguities appearing in the paper.

In Appendix~\ref{sec:gl2_top}, we recall the topological setup of $\gl_2$ webs in $S^3$ and $\gl_2$ foams between them, and introduce singular $\gl_2$ foams that are of interest to us. In Appendix~\ref{sec:gl2_functoriality}, we show that the closed Lee foam evaluation in $\R^4$ agrees with the abstract Lee foam evaluation, assign maps on $\gl_2$ homology induced by singular foams, and deduce Theorem~\ref{thm:gl2_webs_functorial}. In Appendix~\ref{sec:gl2_projectors}, we sketch the definition of $\gl_2$ Rozansky projectors. Throughout, we work over $\Z$, except finally in Section~\ref{sec:gl2_sign} where we work over $\Q$ and address all sign issues in the main text of this paper.

\subsection{\texorpdfstring{$\gl_2$}{gl2} webs and singular \texorpdfstring{$\gl_2$}{gl2} foams}\label{sec:gl2_top}
We set up the notion of (embedded) $\gl_2$ webs and foams that is relevant for us. See \cite{queffelec2024movie} for a more general topological setup.

For us, a \textit{$\gl_2$ web} in $\R^3$ (resp. $S^3$) is an embedded trivalent graph $W\subset\R^3$ (resp. $S^3$) together with the following data:
\begin{enumerate}
\item A label $1$ or $2$ on each edge;
\item An orientation on each edge;
\item An oriented ribbon $R$ of $W$, i.e. a smoothly embedded oriented surface $R\subset\R^3$ (resp. $S^3$) that has $W$ as its core.
\end{enumerate}
The data are subject to the following constraints\footnote{\cite{queffelec2024movie} further requires that the tangent vectors of all three edges at a trivalent vertex to point in the same direction, as this would cut down the number of generic Reidemeister-type moves and movie moves (a similar condition is posed on $\gl_2$ foams in $I\times\R^3$ or $I\times S^3$). We ignore this difference.}:
\begin{enumerate}
\item At each vertex, two of the adjacent edges are labeled $1$ and one is labeled $2$;
\item At each vertex, the two $1$-labeled edges induce the same orientation on the vertex, which is opposite to the orientation induced by the $2$-labeled edge.
\end{enumerate}
Each vertex of a $\gl_2$ web $W$ is assigned the orientation induced from the $2$-labeled edge adjacent to it. Moreover, the orientation of the ribbon $R$ at a vertex coupled with the vertex orientation induces a cyclic ordering of the three adjacent edges: for a positive vertex, we take the cyclic ordering determined by the orientation of the ribbon, and for a negative one the opposite.\smallskip

A \textit{(regular) $\gl_2$ foam} between $\gl_2$ webs $W_0,W_1$ in $\R^3$ (resp. $S^3$) is an embedded singular surface $F\subset I\times\R^3$ (resp. $I\times S^3$) cobounding $\{0\}\times W_0$ and $\{1\}\times W_1$, whose interior points have local neighborhood diffeomorphic to either $\R^2\subset\R^4$ or $Y\times\R\subset\R^4$, where $Y\subset\R^2\subset\R^3$ is the neighborhood of a trivalent vertex of an embedded graph in $\R^2$ (a ``Y shape''), together with the following additional data:
\begin{enumerate}
\item A label $1$ or $2$ on each face;
\item An orientation on each face;
\item An oriented ribbon $R$ of $F$, i.e. a smoothly properly embedded oriented $3$-manifold $R\subset I\times\R^3$ (resp. $I\times S^3$) that has $F$ as its core.
\end{enumerate}
Points on $F$ whose neighborhoods are of the form $Y\times\R$ form a $1$-manifold in $I\times\R^3$ (resp. $I\times S^3$) cobounding vertices of $W_0$ and vertices of $W_1$, each component of which is called a \textit{seam} of $F$. Each component of the exterior of seams in $F$ is a \textit{face} of $F$. The additional data of $F$ are subject to the following constraints:
\begin{enumerate}
\item Around each seam, two of the adjacent faces are labeled $1$ and one is labeled $2$;
\item Around each seam, the two $1$-labeled faces induces the same orientation on the seam, which is opposite to the orientation induced by the $2$-labeled face;
\item The labels and orientations on faces are compatible with the labels and orientations on edges when restricted to the boundary (here, as usual, for $W_0$ this means its edge orientations are given by $-\partial F$);
\item The ribbon restricts to the ribbons of the boundary webs, with compatible orientations.
\end{enumerate}
Each seam of a $\gl_2$ foam $F$ is assigned the orientation induced from the $1$-labeled faces adjacent to it. Moreover, the orientation of the ribbon $R$ at a seam coupled with the seam orientation induces a cyclic ordering of the three adjacent faces by the right-hand rule. As a consequence, both the orientations on seams and the cyclic orientations around them are compatible with those for vertices of the boundary webs.

\begin{figure}
\centering
\includegraphics[width=0.75\linewidth]{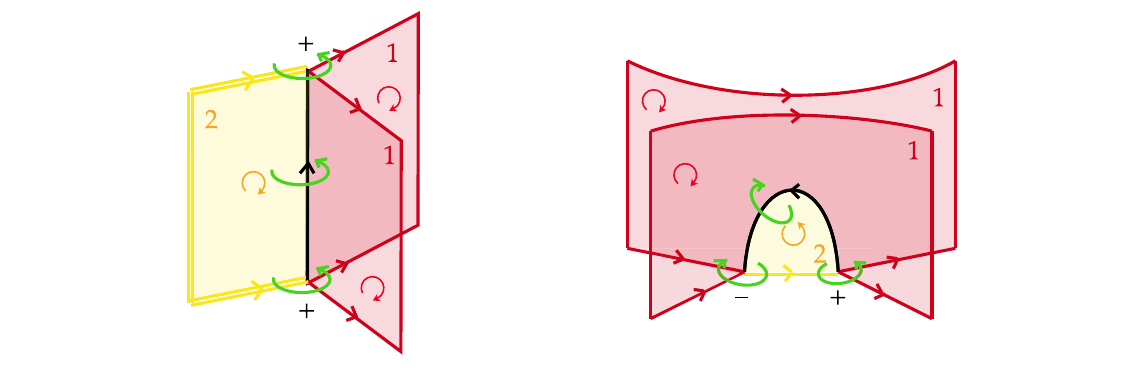}
\caption{Examples of $\gl_2$ foams between $\gl_2$ webs that are contained in $I\times\R^2\times\{0\}\subset I\times\R^3$, shown locally, with ribbons given by thickenings of the pictures in $I\times\R^2\times\{0\}$ with the induced orientations from $I\times\R^2\times\{0\}$. The orientations of the vertices, edges, seams, and faces, as well as cyclic orientations around vertices and seams are indicated.}
\label{fig:foam_examples}
\end{figure}
See Figure~\ref{fig:foam_examples} for two examples of $\gl_2$ foams contained in $I\times\R^2\times\{0\}\subset I\times\R^3$ shown locally, with orientation data indicated in the picture.\medskip

Below, when we speak about $\gl_2$ webs and foams without specifications, we always mean $\gl_2$ webs in $\R^3$ or $S^3$ and $\gl_2$ foams in $I\times\R^3$ or $I\times S^3$.

A \textit{dotted} $\gl_2$ foam is a $\gl_2$ foam with finitely many distinguished points (called \textit{dots}) on the interior of its $1$-labeled faces.

For our purposes, we also wish to allow more general $\gl_2$ foams.
\begin{Def}
A \textit{singular $\gl_2$ foam} is a dotted $\gl_2$ foam, but with finitely many singularities away from the seams and the dots of the following types allowed:
\begin{enumerate}
\item Transverse double points between $1$- and $2$-labeled faces or between $2$-labeled faces. The ribbon is also immersed near the transverse double points, but is embedded when restricted to a neighborhood in each sheet.
\item Framing points, which are $\Z$-labeled points on the interior of faces of the foam away from double points. In a closed neighborhood of such a $k$-labeled point, the foam is regular of the form $D^2=D^2\times\{0\}\subset D^2\times D^2$, but the ribbon is only immersed of the form $D^2\times(-\epsilon,\epsilon)\to D^2\times D^2$, $(z,t)\mapsto(z,tz^k)$, where we denote points in $D^2\subset\C$ using complex numbers.
\end{enumerate}
\end{Def}
\begin{Rmk}
\begin{enumerate}
\item After taking the Khovanov-Rozansky $\gl_2$ functor (see Section~\ref{sec:gl2_functoriality}), the framing points in a singular $\gl_2$ foam acts the same role as framing-changing input balls in \cite[Definition~2.5]{morrison2024invariants} when $N=2$ (note that they use a different renormalization convention).
\item We could have also allowed transverse double points between $1$-labeled edges, and insist that they act as immersion point input balls (as considered in the Lee case in \cite[Example~3.7]{morrison2024invariants}) on the $\gl_2$ homology. Theorem~\ref{thm:gl2_webs_functorial_intro} will still be valid once we fix cocycles in the Khovanov-Rozansky $\gl_2$ chain complexes of the positive/negative Hopf links (see the construction in Section~\ref{sec:gl2_functoriality}). However, the resulting cobordism maps will not be invariant under finger/Whitney moves between $1$-labeled faces.
\end{enumerate}
\end{Rmk}

We think of $1,2$-labeled edges and faces as having thickness $1,2$, respectively. The \textit{writhe} of a $\gl_2$ web $W$, denoted $w(W)$, is the linking number between $W$ and a push-off of itself in the normal direction of the ribbon. Analogously, by interpreting framing points as introducing local twistings, one could also define the \textit{self-intersection number} of a singular $\gl_2$ foam $F\colon W_0\to W_1$; more explicitly, it is $[F]\cdot[F]:=2i(F)+i_1(F)+4i_2(F)$, where $i(F)$ is twice the sum of the signed intersection number between $1$- and $2$-labeled faces plus four times the sum of the signed intersection number between $2$-labeled faces, and $i_k(F)$ is the sum of labels on the framing points on $k$-labeled faces, $k=1,2$. Thus, $[F]\cdot[F]=w(W_1)-w(W_0)$. In particular, the writhe of a $\gl_2$ web in $S^3$ is the self-intersection number of any singular $\gl_2$ foam in $B^4$ that bounds it (such a singular foam always exists). A regular $\gl_2$ foam has self-intersection number $0$, and the writhe of a $\gl_2$ web in $S^3$ is a complete obstruction to having a regular $\gl_2$ foam in $B^4$ bounding it.

\subsection{Functoriality of singular \texorpdfstring{$\gl_2$}{gl2} foams}\label{sec:gl2_functoriality}
For $M=\R^3,S^3$, let $\mathbf{Links}_{M}$ denote the category of admissible framed oriented links in $M$ and framed oriented link cobordisms between them up to isotopy rel boundary, where admissibility of the link is in the same sense as in Section~\ref{sec:RW}, namely that the link is contained in $\R^3$ and the projection onto $\R^2\times\{0\}$ is generic. The link diagram of an admissible framed link comes with $\Z$-labeled framing points away from crossings, which may move freely along the link components, combine or split in a weight-preserving way, and $0$-labeled framing points may be created or annihilated at will\footnote{Strictly speaking, a generic projection of the ribbon would only equip the diagram with $\pm$-half-framing points, which create or cancel in pairs only when we isotope the ribbon. Since framing points only affect the homology by global bigrading shifts, we ignore this technical difference. Once we restrict to integral framing points (by imposing this as a part of the admissibility condition), homological shifts are always even.}. Reidemeister I moves come at a cost of $\pm1$-labeled framing points.

Similarly, let $\TanWeb{M}$ (resp. $\TanWebsing{M}$) denote the category of admissible $\gl_2$ webs in $M$ and dotted (resp. singular) $\gl_2$ foams between them up to isotopy rel boundary (resp. isotopy rel boundary, weight-preserving collision/separation of labeled framing points on the same face, and creation/annihilation of $0$-labeled framing points). Here, in addition to the requirements of generality as in the case of links, the admissibility of $\gl_2$ webs further requires that the projection to $\R^2\times\{0\}$ is orientation-preserving on the ribbon at each trivalent vertex. Framing points in a web diagram are not allowed to move across trivalent vertices. We have the following diagram of functors
\begin{equation}\label{eq:links_webs_sing}
\begin{tikzcd}
\mathbf{Links}_{\R^3}\ar[r]\ar[d,hook]&\mathbf{Links}_{S^3}\ar[d,hook]\\
\TanWeb{\R^3}\ar[r]\ar[d,hook]&\TanWeb{S^3}\ar[d,hook]\\
\TanWebsing{\R^3}\ar[r]&\TanWebsing{S^3}.
\end{tikzcd}
\end{equation}
Here, all vertical arrows are inclusions of categories. All horizontal arrows are full, but not faithful. All functors except the vertical ones from the first row to the second are bijective on objects.

The Khovanov-Rozansky $\gl_2$ homology was first defined for links in $\R^3$ and $S^3$ and link cobordisms by Khovanov \cite{khovanov2000categorification} (referred to as Khovanov homology, where the renormalization convention is different than ours), although its functoriality turns out to be more difficult. The functoriality for links in $\R^3$, up to sign, was proved by Jacobsson \cite{jacobsson2004invariant}. Many sign fixes appeared in the literature, but the one that first introduced webs and foams in the language we will be using is due to Blanchet \cite{blanchet2010oriented}; thus, we obtain a functor 
\begin{equation}\label{eq:CKhR_2_links_R3}
CKhR_2\colon\mathbf{Links}_{\R^3}\to\Kb(\Z)^\Z,
\end{equation}
where the target is the bounded homotopy category of cochain complexes of quantum $\Z$-graded abelian groups, with quantum grading shifts allowed for morphisms. The functoriality for links in $S^3$ requires an additional check for a global elementary movie move called the sweep-around move, and was obtained only recently by Morrison--Walker--Wedrich \cite{morrison2022invariants}. This means \eqref{eq:CKhR_2_links_R3} descends to a functor $$CKhR_2\colon\mathbf{Links}_{S^3}\to\Kb(\Z)^\Z.$$
The Khovanov-Rozansky $\gl_2$ homology was also extended to $\gl_2$ webs in $\R^3$ or $S^3$ as the $N=2$ special case of Wu \cite{wu2014colored} and studied by many other authors. For us, singularities in the diagram of a $\gl_2$ web, in addition to trivalent vertices, are crossings with various kinds and $\Z$-labeled framing points where the ribbon twists along the strands by multiples of full turns rel the blackboard framing. One resolves the singularities and builds a cube of resolutions using the rules in Figure~\ref{fig:gl2_convention}.
\begin{figure}
\centering
\includegraphics[width=0.8\linewidth]{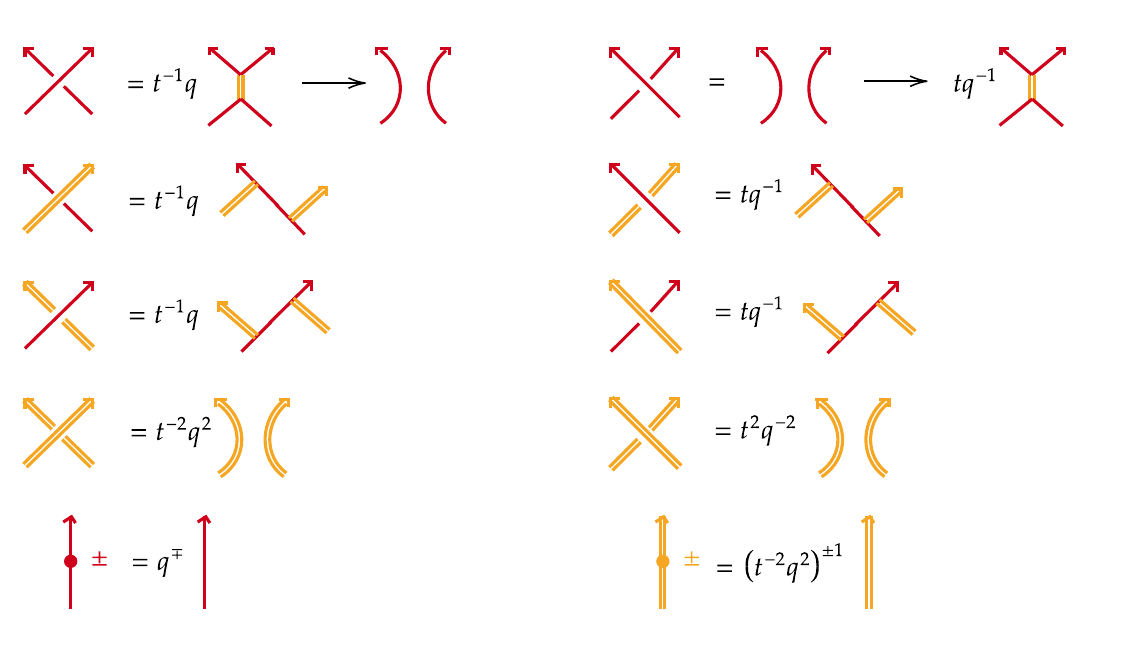}
\caption{Resolution of crossings or framing points in a web diagram into complexes of planar webs. Here $t$ and $q$ denote the homological and quantum degree shifts, respectively.}
\label{fig:gl2_convention}
\end{figure}
The functoriality of $CKhR_2$ for webs in $\R^3$ was proved recently by Queffelec \cite{queffelec2022gl2}. This means \eqref{eq:CKhR_2_links_R3} extends to a functor $$CKhR_2\colon\TanWeb{\R^3}\to\Kb(\Z)^\Z.$$ We warn the readers that our convention differs from that of Queffelec by mirroring the webs and foams (or on the level of diagrams, changing the signs of all crossings and framing points).

One could also define the universal version of Khovanov-Rozansky $\gl_2$ homology in the various settings above, by replacing the underlying Frobenius algebra $\Z[\Khdot]/(\Khdot^2)$ by its universal deformation $\Z[E_1,E_2,\Khdot]/(\Khdot^2-E_1\Khdot+E_2)$ over $\Z[E_1,E_2]$. The proof of functoriality is similar, and one obtains a functor $CKhR_2^{univ}$ to $\Kb(\Z[E_1,E_2])^\Z$, the bounded homotopy category of cochain complexes of quantum $\Z$-graded $\Z[E_1,E_2]$-modules, from $\mathbf{Links}_{\R^3},\mathbf{Links}_{S^3}$, or $\TanWeb{\R^3}$ (see \cite[Theorem~2.2]{morrison2024invariants} for the functoriality for links in $S^3$). Here, the free variables $E_1,E_2$ have $q$-degrees $2,4$, respectively. By setting $E_1=E_2=0$, one recovers the functoriality results for the undeformed theories.

\begin{Thm}\label{thm:gl2_webs_functorial}
The universal Khovanov-Rozansky $\gl_2$ homology on objects extends to a functor 
\begin{equation}\label{eq:CKhR_2_webs_S3_sing}
CKhR_2^{univ}\colon \TanWebsing{S^3}\to \HChb(\Z[E_1,E_2])^\Z.
\end{equation}
Here, $\HChb$ denotes the cohomology category of the dg category of bounded chain complexes, i.e. the extension of $\Kb$ where homological degree shifts are allowed for morphisms\footnote{In fact, morphisms in the image of $CKhR_2^{univ}$ are always of even homological degree.}. Moreover,
\begin{enumerate}
\item The bigrading shifts of $CKhR_2^{univ}(F)\colon CKhR_2^{univ}(W_0)\to CKhR_2^{univ}(W_1)$ for a singular foam $F\colon W_0\to W_1$ is $(t^{-1}q)^{i(F)+2i_2(F)}q^{-\chi(u(F))-i_1(F)+2\#(dots)}$, where $u(F)$ denotes the closure of the union of $1$-labeled faces in $F$.
\item $CKhR_2^{univ}(F)$ (up to homotopy) is independent of the embedding of the interior of the $2$-labeled faces of $F$. In fact, up to sign, it is determined by the abstract dotted surface $u(F)$. The sign is further determined by the germ of $u(F)$ in $F$ with all data (orientations, ribbon, dots, and framing points) and the mod $4$ total Euler characteristic of the $2$-labeled faces.
\end{enumerate}
\end{Thm}
\begin{Rmk}
The functor $CKhR_2^{univ}$ in Theorem~\ref{thm:gl2_webs_functorial} gives rise to functors from every term in the diagram \eqref{eq:links_webs_sing}. When restricted to the first two rows of \eqref{eq:links_webs_sing}, the image is contained in the subcategory $\Kb(\Z[E_1,E_2])^\Z$ where morphisms preserve the homological degree. Thus, Theorem~\ref{thm:gl2_webs_functorial} is a simultaneous generalization of all previous functoriality results in the context of $\gl_2$ homology.
\end{Rmk}

The proof of Theorem~\ref{thm:gl2_webs_functorial} is divided into two steps. In Section~\ref{sec:gl2_webs_functorial_S3}, by proving a $4$-dimensional Lee foam evaluation formula, we upgrade Queffelec's functoriality to webs in $S^3$ and nonsingular foams between them. In Section~\ref{sec:gl2_webs_functorial_singular}, we further extend the functoriality to singular foams. This allows the construction of (braided) monoidal $2$-categories in Section~\ref{ssec:braidedcat}. In Section~\ref{sec:sylleptic} we interpret the moves involving singular foams in terms of sylleptic centers.

\subsubsection{Functoriality of regular \texorpdfstring{$\gl_2$}{gl2} foams}\label{sec:gl2_webs_functorial_S3}
In this section, we show that Queffelec's functor 
\begin{equation}\label{eq:CKhR_2_webs_R3}
CKhR_2^{univ}\colon\TanWeb{\R^3}\to\Kb(\Z[E_1,E_2])^\Z
\end{equation}
descends to a functor 
\begin{equation}\label{eq:CKhR_2_webs_S3}
CKhR_2^{univ}\colon\TanWeb{S^3}\to\Kb(\Z[E_1,E_2])^\Z.
\end{equation}
This amounts to showing that the sweep-around movie move \cite[(1-1)]{morrison2022invariants} induces the identity chain map up to chain homotopy (for us, the strand that sweeps around can be either $1$-labeled or $2$-labeled).

For a $\gl_2$ web $W$ (resp. $\gl_2$ foam $F$), let $u(W)$ (resp. $u(F)$) denote the closure of the union of $1$-labeled edges (resp. faces), thought of as an unoriented link (resp. link cobordism). If $F\colon W_0\to W_1$ is a $\gl_2$ foam between admissible $\gl_2$ webs, then $u(F)\colon u(W_0)\to u(W_1)$ is an orientable unoriented link cobordism between admissible unoriented links, and the Bar-Natan formalism of Khovanov homology gives a chain map $CKhR_2^{univ}(u(F))\colon CKhR_2^{univ}(u(W_0))\to CKhR_2^{univ}(u(W_1))$, well-defined up to sign and chain homotopy, between the Khovanov chain complexes of links $u(W_0),u(W_1)$, where the bigrading is only well-defined up to an overall even shift. The main result of Beliakova--Hogancamp--Putyra--Wehrli \cite{beliakova2023functoriality} implies that there are isomorphisms $\iota_\bullet$ canonical up to signs making the following diagram commute up to sign and chain homotopy (as homologically $\Z/2$-graded, $\Z$-relatively graded chain complexes).
\begin{equation}\label{eq:forget_thick}
\begin{tikzcd}
CKhR_2^{univ}(W_0)\ar[rr,"CKhR_2^{univ}(F)"]\ar[d,"\iota_{W_0}","\cong"']&&CKhR_2^{univ}(W_1)\ar[d,"\iota_{W_1}","\cong"']\\
CKhR_2^{univ}(u(W_0))\ar[rr,"CKhR_2^{univ}(u(F))"]&&CKhR_2^{univ}(u(W_1)).
\end{tikzcd}
\end{equation}

Now, let $F\colon W\to W$ be the movie of a sweep-around move. Then $u(F)\colon u(W)\to u(W)$ is either the movie of a sweep-around move (if a $1$-labeled edge sweeps around) or the identity movie (if a $2$-labeled edge sweeps around) for links and link cobordisms, whose induced chain map $CKhR_2^{univ}(u(F))$ is the identity chain map up to sign and chain homotopy. Indeed, any version of Khovanov homology with analogs of \eqref{eq:forget_thick} commuting up to global sign and homotopy inherit the triviality of the sweep-around move up to sign from $CKhR_2^{univ}$, for which it is proven in \cite[Theorem~2.2]{morrison2024invariants} following \cite{morrison2022invariants}. By the commutativity of \eqref{eq:forget_thick}, we deduce that $CKhR^{univ}_2(F)$ is chain homotopic to the identity map up to sign.

To fix the sign, as usual, it suffices to do so on the level of an appropriately defined Lee homology. To this end, we denote by $$CKhR_{Lee}\colon\TanWeb{\R^3}\to\Kb(\Q)^\Z$$
the result of base-changing $CKhR_2^{univ}$ to $\Q$ by tensoring all complexes with the $\Z[E_1,E_2]$-module $\Q$, on which $E_1$ acts by $0$ and $E_2$ by $-1$. Further, we write $KhR_{Lee}$ for the homology of the complexes computed by $CKhR_{Lee}$ and refer to this as \emph{Lee homology}.

Returning to the movie of a sweep-around move $F\colon W\to W$, we now check the induced map $KhR_{Lee}(F)\colon KhR_{Lee}(W)\to KhR_{Lee}(W)$ is the identity map. Since the underlying abstract $\gl_2$ foam (by which we mean the underlying singular surface together with labels, orientations, and cyclic orderings around seams) of $F$ is the same as that of the identity foam cobordism $W\to W$, this follows from the next theorem using a tracing argument.

\begin{Thm}[$4$-dimensional Lee foam evaluation]\label{thm:Lee_foam_evaluation}
Let $F\colon\emptyset\to\emptyset$ be a closed dotted $\gl_2$ foam in $\R^4$. Then the following two rational numbers are equal:
\begin{enumerate}[(a)]
\item The $4$-dimensional Lee evaluation $\langle F\rangle_{Lee}$ of $F$, i.e. the eigenvalue of the endomorphism $KhR_{Lee}(F)$ of $KhR_{Lee}(\emptyset)=\Q$ provided by the functoriality of Lee homology.
\item The Lee evaluation of $F$ as an abstract foam, defined as in Blanchet \cite[Section~1.5 and 4]{blanchet2010oriented}.
\end{enumerate}
\end{Thm}

\begin{Cor}
Let $F\colon\emptyset\to\emptyset$ be a closed dotted $\gl_2$ foam in $\R^4$. The $4$-dimensional Khovanov-Rozansky $\gl_2$ evaluation of $F$ agrees with the $\gl_2$ evaluation of $F$ as an abstract foam, defined as in \cite{blanchet2010oriented}.
\end{Cor}
Thus, if $u(F)$ contains a component that is not a one-dotted sphere or an undotted torus, then the $4$-dimensional evaluation is $\langle F\rangle_{\gl_2}=0$. Otherwise, $\langle F\rangle_{\gl_2}=\pm2^{\#\{\text{torus components in $u(F)$}\}}$, with the sign determined by the remaining data of $F$ viewed as an abstract foam. In particular, the sign is positive if $F$ has no $2$-labeled faces.
\begin{proof}
If the $q$-degree shift of $KhR_2(F)\colon KhR_2(\emptyset)\to KhR_2(\emptyset)$ is nonzero, then both the concrete and the abstract $\gl_2$ evaluations of $F$ are zero. If the $q$-degree shift of $KhR_2(F)$ is zero, then $\langle F\rangle_{\gl_2}=\langle F\rangle_{Lee}$ is equal to the abstract Lee/$\gl_2$ evaluation.
\end{proof}

For our purpose, it is convenient to allow \textit{Lee idempotent-colored foams}, namely (undotted) $\gl_2$ foams whose $1$-labeled faces are decorated with colors $+=(1+x)/2$ or $-=(1-x)/2$ (which are idempotents in the Lee Frobenius algebra $\Q[x]/(x^2-1)$). Every $\gl_2$ foam can be rewritten as a formal linear combination of idempotent-colored ones, and both evaluations in Theorem~\ref{thm:Lee_foam_evaluation} are extended to closed idempotent-colored foams by linearity.

We refer the readers to \cite[(3.1)-(3.19)]{queffelec2022gl2} for a collection of local skein relations satisfied by the Lee evaluation of abstract $\gl_2$ foams, some of which will be useful to us. Relations (3.1) and (3.3)-(3.12) therein can be converted into idempotent-colored skein relations in the natural way. We also note one more skein relation that if the two $1$-labeled faces around a seam are colored by the same idempotent, then the abstract Lee evaluation is zero.

The local pictures in these skein relations can also be interpreted as foams embedded in $B^3$, with ribbon given by a thickening of the core surface, oriented as a submanifold of $B^3$. To evaluate in Lee homology an idempotent-colored foam $F$ in $\R^4$, one can apply these local relations to simplify $F$. More explicitly, this means that if in some local $B^3$, $F$ with its ribbon is given by one side of a skein relation, then one can replace the local picture of $F$ by the other side of the the skein relation (this is justified by Queffelec's functoriality, since one can rotate the local $B^3$ to sit in the first three coordinates of $I\times\R^3$ (with orientation) and evaluate).

\begin{proof}[Proof of Theorem~\ref{thm:Lee_foam_evaluation}]
It suffices to prove the case when $F$ is idempotent-colored by $\pm$. If two of the $1$-labeled faces around a seam in $F$ are colored by the same idempotent, then both evaluations (a) and (b) are zero. Thus, it remains to prove Theorem~\ref{thm:Lee_foam_evaluation} for \textit{compatibly} idempotent-colored foams, namely the ones where two $1$-labeled faces around each seam have opposite colors $\pm$.

Let $F_1,F_2$ denote the closures of the union of $1,2$-labeled faces in $F$, respectively. As we observed above, one may simplify $F$ by the skein relations \cite[(3.1)-(3.19)]{queffelec2022gl2} without affecting the truth of the statement.

\textbf{Step 1}: We may assume $F_1$ to be connected.

This is because we can find a collection of paths connecting different components of $F_1$ in the complement of $F$, and perform the inverses of \cite[(3.15) or (3.18)]{queffelec2022gl2}.

\textbf{Step 2}: We may assume $F_1$ and $F_2$ to be disjoint smoothly embedded closed oriented surfaces, or equivalently, $F$ has empty seam.

The collection of seams is an embedded multicurve $\gamma$ on the closed surface $F_1$. The inverses of \cite[(3.4)]{queffelec2022gl2} allow us to change $\gamma$ by oriented band surgeries on $F_1$. If $\gamma\ne\emptyset$, since $F_1$ is compatibly colored, we may apply band surgeries to assume $\gamma$ has a single component, which is necessarily a separating curve on $F_1$. By further band surgeries, we may assume $\gamma$ is a contractible curve on $F_1$. Since $\gamma$ bounds the surface $F_2$ in the complement of $F_1$, the twisting of the germ of $F_2$ along $\gamma$ around $F_1$ is zero, hence we may apply the neck-cutting relation \cite[(3.2)]{queffelec2022gl2} and then \cite[(3.6)]{queffelec2022gl2} to detach $F_2$ from $F_1$.

\textbf{Step 3}: The ribbon of $F$ is the same as framings on the embedded surfaces $F_1,F_2\subset\R^4$. Changing these framings does not affect the Lee foam evaluation.

This is because $F$ with two different framings differ in a generic movie presentation levelwise by some framing points, and the assignment of induced maps is insensitive to framing points.

\textbf{Step 4}: We may assume $F=\tau(S^1\times H)$, where $H\subset B^3$ is the positive Hopf link with one $1$-labeled and one $2$-labeled component, both $0$-framed, and $\tau\colon S^1\times B^3\hookrightarrow\R^4$ is the twisted embedding induced by an oddly framed circle in $\R^4$.

Suppose $W_1\cup W_2\subset I\times\R^4$ is a paired oriented cobordism between $F_1\cup F_2\subset\R^4$ and some $F_1'\cup F_2'\subset\R^4$. Upon changing the framings of $F_1,F_2$, we may assume $W_1\cup W_2$ to be a framed cobordism (note that closed oriented $3$-manifolds embedded in $\R^5$ have trivial normal bundles). By making the projection $W_1\cup W_2\subset I\times\R^4\to I$ Morse, we see that $(F_1,F_2)$ and $(F_1',F_2')$ (with all components of $F_1'$ colored by the color of $F_1$) are related by a sequence of skein relations \cite[(3.2)(3.15)(3.17)]{queffelec2022gl2} and their inverses. The claim now follows from the result by Sanderson \cite[Example~1.3]{sanderson1987bordism} that the oriented unframed cobordism group of the pair $(F_1,F_2)$ in $\R^4$ is isomorphic to $\Z/2$, with the underlying unframed pair of $\tau(S^1\times H)$ representing the nontrivial bordism class. 

\textbf{Step 5}: Theorem~\ref{thm:Lee_foam_evaluation} holds for $F=\tau(S^1\times H)$.

If $W\subset\R^3$ is an admissible web and $T_W$ is the foam given by the trace of $W$ under a $2\pi$-rotation in $\R^3$, then by choosing the rotation to be along the $z$-axis, we see the induced map 
\[CKhR_2^{univ}(T_W)\colon CKhR_2^{univ}(W)\to CKhR_2^{univ}(W)\]
is chain homotopic to the identity map. Therefore, we may replace the twisted embedding $\tau\colon S^1\times B^3\hookrightarrow\R^4$ by the untwisted embedding $i\colon S^1\times B^3\hookrightarrow\R^4$ without affecting the Lee foam evaluation. But $i(S^1\times H)$ is a null-cobordant pair, so the proof is complete.
\end{proof}

\subsubsection{A local statement}
\label{sec:gl2_webs_local}
Before extending the functor \eqref{eq:CKhR_2_webs_S3} of the previous section to singular foams, we establish notation for a local version of functoriality in the spirit of Bar-Natan's notion of canopolis \cite{bar2005khovanov}, see \cite[Section~2.2]{ehrig2018functoriality} and \cite{queffelec2021khovanov}. We also prove a lemma that will be useful. This section only plays a minor role in proving Theorem~\ref{thm:gl2_webs_functorial}.

Let $S$ denote an oriented surface, and let $\epsilon$ denote a collection of oriented points $p\subset\partial S$, each with a label $1$ or $2$. We now consider the graded $\Z[E_1,E_2]$-linear additive category $\Foams{S,\epsilon}$ of $\gl_2$ foams in the thickened surface. (A version with $\epsilon=\emptyset$ and $E_1=E_2=0$ was defined in \cite[Definition~3.1]{queffelec2021khovanov}.) The category $\Foams{S,\epsilon}$ has as objects $\gl_2$ webs $W\subset S$ with $\partial W=\epsilon$, as well as formal grading shifts and direct sums thereof. The morphisms in $\Foams{S,\epsilon}$ are (matrices of) $\Z[E_1,E_2]$-linear combinations of dotted $\gl_2$ foams in $I\times S$ rel $I \times \epsilon$ between such webs, up to isotopy rel boundary and appropriate skein relations\footnote{Specifically, the relations \cite[(3.3)-(3.12)]{queffelec2022gl2} together with equivariant versions of sphere and neck-cutting relations, see e.g. \cite[Definition~2.6]{beliakova2023functoriality}. All these relation arise as local relations from (an equivariant analog of) Blanchet's abstract foam evaluation \cite{blanchet2010oriented}, as explained in detail in \cite[Section~2]{ehrig2018functoriality}.}.

For tangled webs in the thickening of $S$ we similarly define the category $\TanWeb{I\times S,\epsilon}$ with objects admissible $\gl_2$ webs $W\subset I\times S$ with $\partial W = \{1/2\}\times \epsilon$, and morphisms dotted $\gl_2$ foams in $I\times I \times S$ rel $I\times \{1/2\} \times \epsilon$ between such webs; these webs and foams are defined analogously to Section~\ref{sec:gl2_top}, with webs admissible if their projections onto $S$ are generic. (A version with $\epsilon=\emptyset$ was defined in \cite[Definition~4.6]{queffelec2021khovanov}.)

The local invariant we consider is a functor of the form
\begin{equation} 
 \label{eq:onsurface}
 \TanWeb{I\times S,\epsilon}\xrightarrow{\BNfunc{\cdot}}\Kb(\Foams{S,\epsilon}),
\end{equation}
where $\Kb(\Foams{S,\epsilon})$ is the bounded homotopy category of cochain complexes in $\Foams{S,\epsilon}$. The existence of a functor \eqref{eq:onsurface} that categorifies the evaluation of tangled webs in $\gl_2$ skein theory was proven for tangles (i.e. purely $1$-labeled webs) in \cite[Theorem~1.1]{queffelec2021khovanov}, conjectured in full in \cite[Conjecture~4.8]{queffelec2021khovanov} and proven in \cite{queffelec2022gl2}.\smallskip

With this notation in place, we describe the central idea of the local lemma we will need. Suppose we have a closed web $W\subset B^3$ within some larger web diagram $W'$, and suppose that we would like to slide $W$ either under or over some other strand in $W'$. We would like to prove that such a move induces ``the identity map on both $CKhR_2^{univ}(W)$ and the other strand.'' To interpret this idea properly, we first note that for any planar $\gl_2$ web $W\in\Foams{D^2,\emptyset}$, the universal construction implies that we can neck-cut the identity foam morphism $id_W$ and write it as a finite sum
\begin{align*}
\centering
\includegraphics[width=0.65\linewidth]{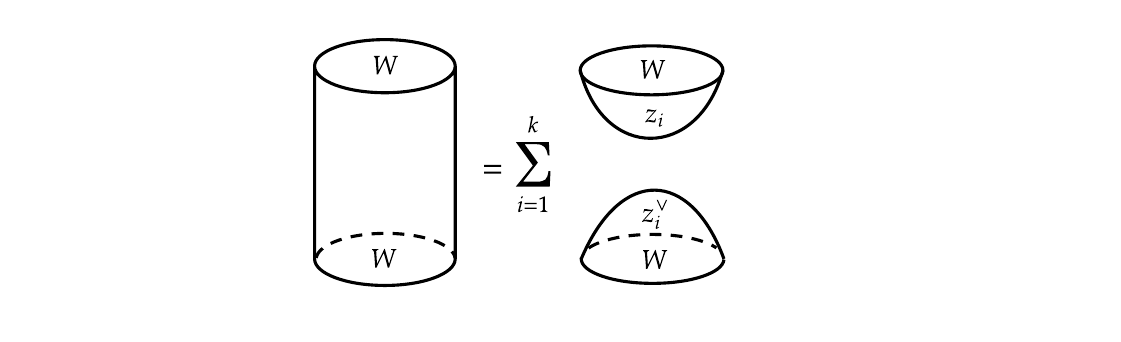},
\end{align*}
where $z_i$'s are foams representing a basis of $KhR_2^{univ}(W)$, and $z_i^\vee$'s are foams representing a basis of $KhR_2^{univ}(\bar W)$ dual to the basis representing by the $z_i$'s. If we now let $W_L$ (resp. $W_R$) denote the (planar) web in $\Foams{D^2,\epsilon}$ consisting of $W$ sitting to the left (resp. right) of a single strand through the disk $D^2$ (with endpoint data $\epsilon$), we can define the ``identity'' map between such webs by the ``shift'' map $W_L\to W_R$ in $\Foams{D^2,\epsilon}$ given by
\begin{align*}
\centering
\includegraphics[width=0.7\linewidth]{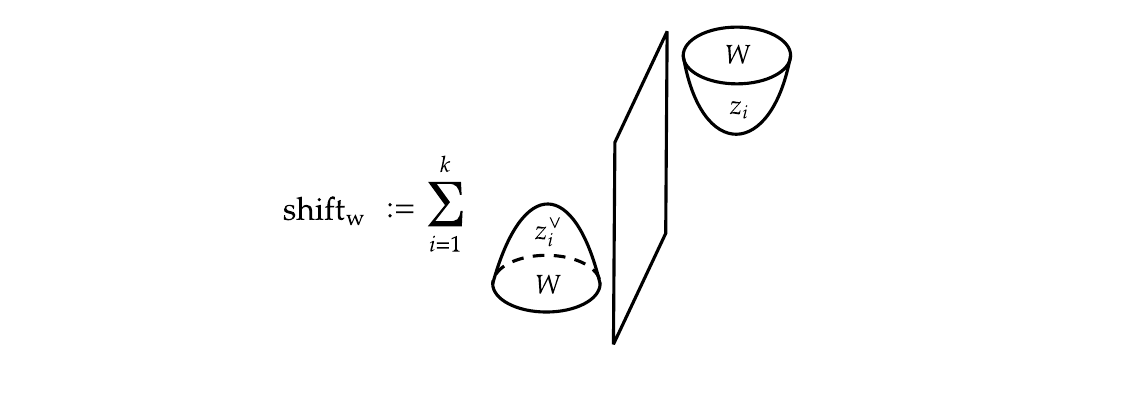}.
\end{align*}
More generally, if $W$ is any admissible closed web in $B^3=I\times D^2$, we let $W_L,W_R\in\TanWeb{I\times D^2,\epsilon}$ be defined similarly and then define the shift map $\BNfunc{W_L}\xrightarrow{\iota_{sh}}\BNfunc{W_R}$ by applying the shift described above termwise between the two complexes.

\begin{Lem}\label{lem:slide_=_shift}
Let $W$ be an admissible closed $\gl_2$ web in $B^3=I\times D^2$, and let $W_L$ (resp. $W_R$) be the web in $\TanWeb{I\times D^2,\epsilon}$ formed by placing $W$ to the left (resp. to the right) of an arbitrarily oriented and labeled through-strand in $\{1/2\}\times D^2$ (with endpoint data $\epsilon$). Let $W_L\xrightarrow{\eta_{un}}W_R$ (resp. $W_L\xrightarrow{\eta_{ov}}W_R$) denote the foam in $\TanWeb{I\times D^2,\epsilon}$ tracing out the isotopy of passing $W$ under (resp. over) the through-strand. Then the three maps $\iota_{sh},\BNfunc{\eta_{un}},\BNfunc{\eta_{ov}}$ are equal morphisms in $\Kb(\Foams{D^2,\epsilon})$.
\end{Lem}

\begin{Rmk}
One could reprove the sweep-around move required for the $\R^3$ to $S^3$ functoriality upgrade using the equality of $\BNfunc{\eta_{un}}$ and $\BNfunc{\eta_{ov}}$ in Lemma~\ref{lem:slide_=_shift}, similarly as in \cite[Corollary~5.7]{chen2025flip}. Nevertheless, we proceeded as in Section~\ref{sec:gl2_webs_functorial_S3}, hoping that Theorem~\ref{thm:Lee_foam_evaluation} or its proof might be of independent interest.
\end{Rmk}

\begin{proof}
We only prove $\BNfunc{\eta_{un}}=\iota_{sh}$, as the other half is analogous. We remark that, in the case when the middle strand is $1$-labeled and the web $W$ is a ($1$-labeled) link, the equality follows from the argument of the sweep-around move in \cite[Theorem~1.1]{morrison2022invariants}. Following their idea, we proceed as follows:

\textbf{Special case}: $W$ is planar.

The universal $\gl_2$ homology of $W$ is generated by foams capping it off in $B^3$. Birthing a foam on the left of the strand and sliding the boundary web under the strand to the right is isotopic to birthing the foam on the right. By Queffelec's functoriality \cite{queffelec2022gl2}, this proves the desired equality.

\textbf{The general case}:

The map $W_L\xrightarrow{\eta_{un}}W_R$ can be decomposed into a sequence of elementary moves
\[ W_L=W_0 \xrightarrow{\eta_1} W_1 \xrightarrow{\eta_2} \cdots \xrightarrow{\eta_n} W_n =W_R,\]
where each $\eta_i$ is either a Reidemeister II move, a fork slide, or a Reidemeister III move involving three strands having non-alternating orientations. Letting $N$ denote the number of crossings in $W$, each complex $\BNfunc{W_t}$ can be viewed as a twisted complex with terms indexed by $\delta\in\{0,1\}^N$ via the cube of resolutions for $W$, and we have termwise maps
\[\BNfunc{W_L}_\delta=\BNfunc{W_0}_\delta \xrightarrow{\BNfunc{\eta_1}_\delta} \BNfunc{W_1}_\delta \xrightarrow{\BNfunc{\eta_2}_\delta} \cdots \xrightarrow{\BNfunc{\eta_n}_\delta} \BNfunc{W_n}_\delta = \BNfunc{W_R}_\delta\]
induced by the corresponding isotopies on the resolutions. If $W_{t-1}\xrightarrow{\eta_t}W_t$ is a Reidemeister II move or a fork slide, $\BNfunc{\eta_t} =\bigoplus_\delta \BNfunc{\eta_t}_\delta$ with no cross terms. If $\eta_t$ is a Reidemeister III move involving only $1$-labeled strands (with non-alternating orientations), \cite{morrison2022invariants} showed that $\BNfunc{\eta_t}$ can be chosen carefully in its chain homotopy class to ensure that the cross terms $\BNfunc{\eta_t}-\bigoplus_\delta\BNfunc{\eta_t}_\delta$ strictly increase the \emph{internal homological degree}, defined as the homological degree of the index $\delta$. Since the entire composition $\BNfunc{W_0}\xrightarrow{\BNfunc{\eta_{un}}}\BNfunc{W_1}$ preserves the (internal=total) homological degree, the contributions from these cross terms vanish in the full composition $\BNfunc{\eta_{un}}$.

It remains to consider $\BNfunc{\eta_t}$ for Reidemeister III moves involving $2$-labeled strands (with non-alternating orientations). One can check, case by case, using \cite[Appendix~A]{queffelec2022gl2}, that in these cases $\BNfunc{\eta_t}=\bigoplus_\delta \BNfunc{\eta_t}_\delta$ again with no cross terms, so that $\BNfunc{\eta_{un}}=\bigoplus_\delta \BNfunc{\eta_{un}}_\delta$, and the desired statement follows from the planar case. We outline an alternative effortless proof. Note that the endomorphism space between the local source and target of a Reidemeister III induced map with $2$-labeled strands involved in the quantum grading $0$ is isomorphic to $\Z$. Thus the local map must agree with the local termwise Reidemeister maps up to sign; combining this with the rest of the argument thus far, we see $\BNfunc{\eta_{un}}=\pm\iota_{sh}$. To fix the sign, we pass to Lee homology. We add in a $1$-labeled framed crossingless unknot as needed to force $W$ to have writhe zero, so that there exists a $\gl_2$ foam that caps off $W$ in $B^4$, compatibly idempotent-colored in the sense explained in the proof of Theorem~\ref{thm:Lee_foam_evaluation}. Then the functoriality of \eqref{eq:onsurface} \cite{queffelec2022gl2} shows that birthing $W$ on the left of the middle strand and sliding it to the right induces the same map on Lee homology as birthing $W$ on the right of the strand. Since both maps are nonzero, this proves $\BNfunc{\eta_{un}}=\iota_{sh}$.
\end{proof}

\begin{Rmk}
In fact, by taking advantage of Bar-Natan's canopolis formalism, one can view $(\cdot)_L$ (resp. $(\cdot)_R$) as a functor $\TanWeb{I\times D^2,\emptyset}\to \TanWeb{I\times D^2,\epsilon}$ which plugs objects and morphisms into a local thickened disc to the left (resp. to the right) of a through-strand in $\{1/2\}\times D^2$ with endpoint data $\epsilon$. In this language, the functoriality of \eqref{eq:onsurface} ensures that our maps in Lemma~\ref{lem:slide_=_shift} induce natural transformations (indeed, natural isomorphisms) $\iota_{sh},\BNfunc{\eta_{un}},\BNfunc{\eta_{ov}}$ between the functors $$\BNfunc{(\cdot)_L},\BNfunc{(\cdot)_R}\colon\TanWeb{I\times D^2,\emptyset}\to\Kb(\Foams{D^2,\epsilon}).$$ Lemma~\ref{lem:slide_=_shift} implies that these natural isomorphisms are in fact equal.
\end{Rmk}

\subsubsection{Extension to singular \texorpdfstring{$\gl_2$}{gl2} foams}\label{sec:gl2_webs_functorial_singular}
We now extend the functor \eqref{eq:CKhR_2_webs_S3} defined in Section~\ref{sec:gl2_webs_functorial_S3} to singular foams, obtain the claimed functor \eqref{eq:CKhR_2_webs_S3_sing} in Theorem~\ref{thm:gl2_webs_functorial}, and prove the extra assertions in Theorem~\ref{thm:gl2_webs_functorial}.

We first describe the link of singularity $L$ for each singularity model in singular $\gl_2$ foams, fix admissible representatives of them, and fix explicit cocycles $z(L)\in CKhR_2^{univ}(L)$, which will be useful in the construction. Below, we follow the notation in \cite{queffelec2022gl2}. In particular, $a_3,a_4\in\{\pm1\}$ are free sign variables for the universal Khovanov-Rozansky $\gl_2$ homology theory, which affects only the sign assignments to morphisms.
\begin{enumerate}[(i)]
\item A positive transverse double point between $1$- and $2$-labeled faces:
\begin{figure}[H]
\centering
\includegraphics[width=0.8\linewidth]{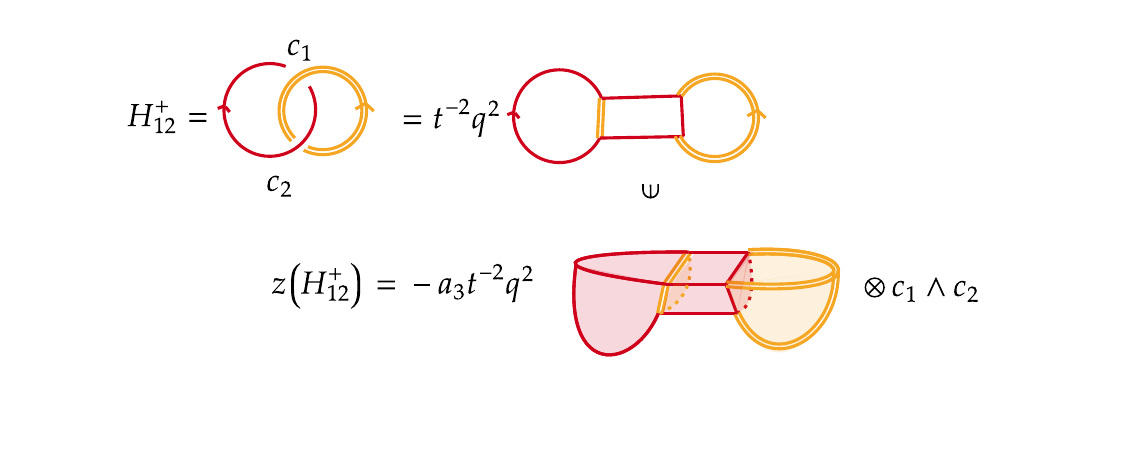}.
\end{figure}
\item A negative transverse double point between $1$- and $2$-labeled faces:
\begin{figure}[H]
\centering
\includegraphics[width=0.8\linewidth]{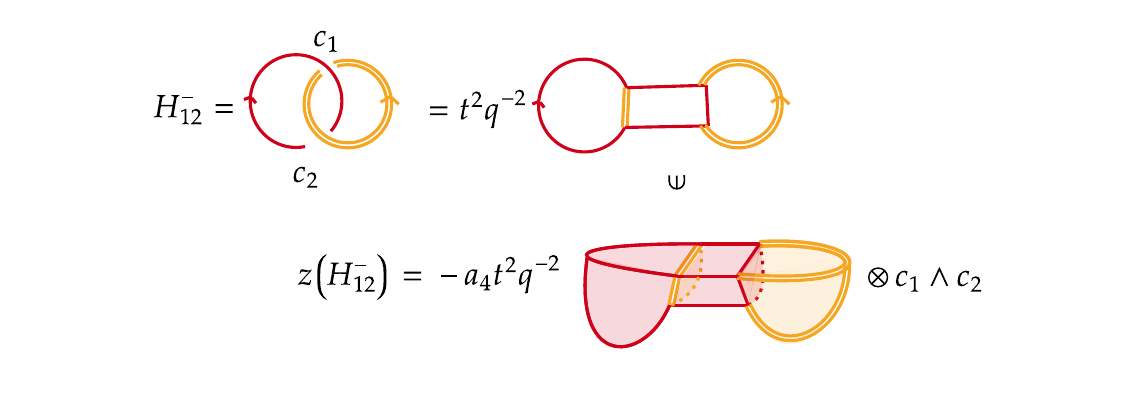}.
\end{figure}
\item A positive transverse double point between $2$-labeled faces:
\begin{figure}[H]
\centering
\includegraphics[width=0.8\linewidth]{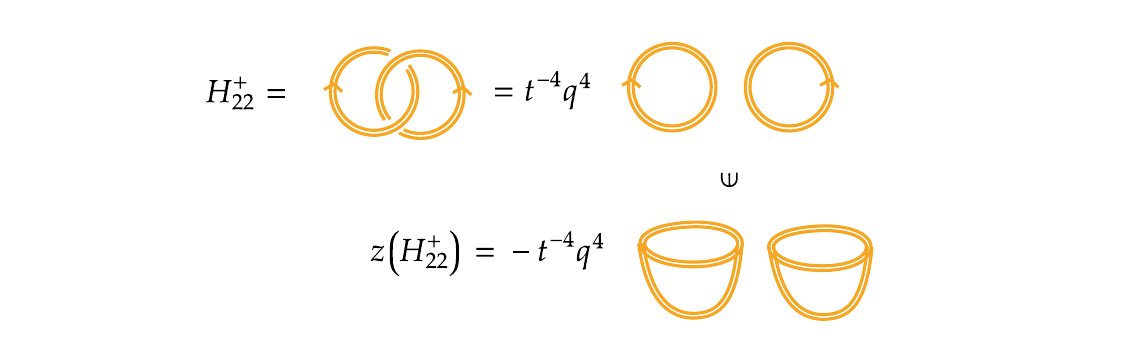}.
\end{figure}
\item A negative transverse double point between $2$-labeled faces:
\begin{figure}[H]
\centering
\includegraphics[width=0.8\linewidth]{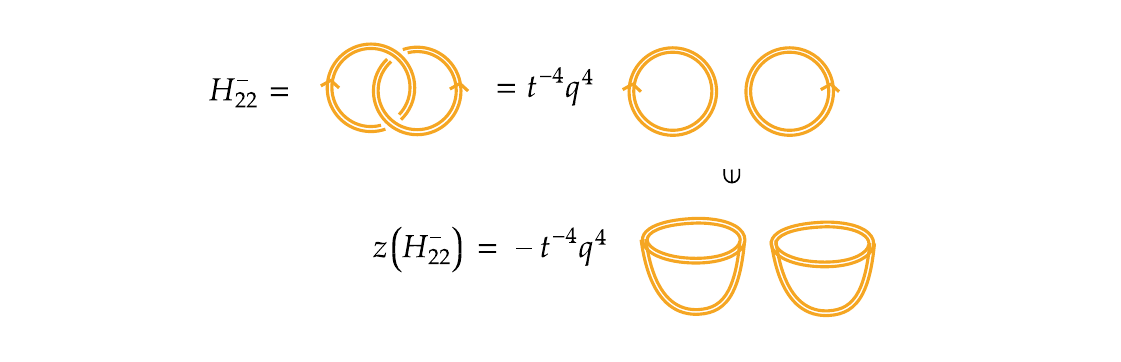}.
\end{figure}
\item An $n$-labeled framing point on a $1$-labeled face:
\begin{figure}[H]
\centering
\includegraphics[width=0.8\linewidth]{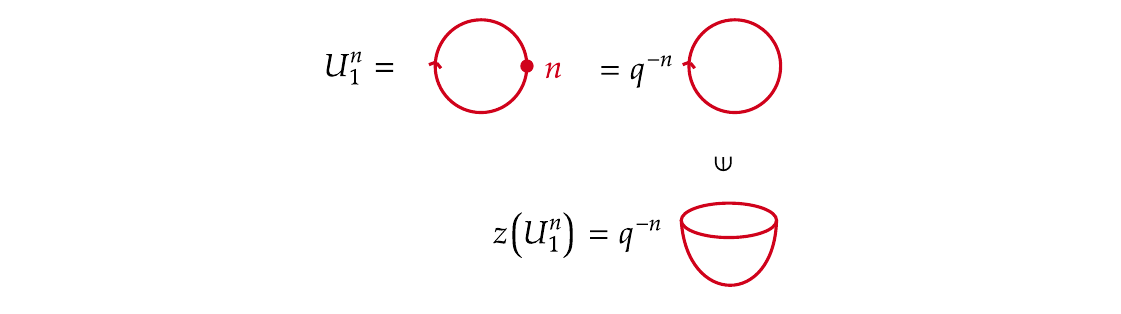}.
\end{figure}
\item An $n$-labeled framing point on a $2$-labeled face:
\begin{figure}[H]
\centering
\includegraphics[width=0.8\linewidth]{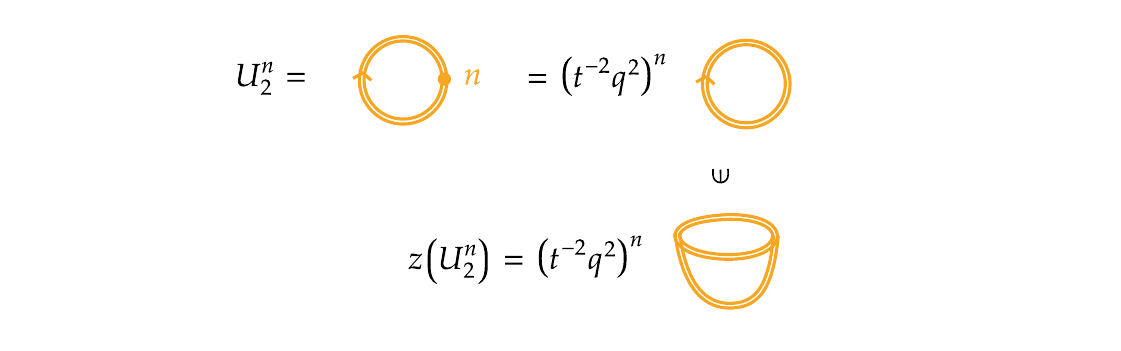}.
\end{figure}
\end{enumerate}

Now, for a singular $\gl_2$ foam $F\subset I\times S^3$ between admissible $\gl_2$ webs $W_0,W_1\subset S^3$, we delete small $4$-balls $B_1,\cdots,B_k$ from $I\times S^3$ around singular points of $F$, and choose disjoint framed paths in the exterior of $F\cup\sqcup_{i=1}^kB_i$ to tube each $\partial B_i$ to $\{0\}\times S^3$ so that the links of singularities land disjointly near $\{0\}\times\infty\in\{0\}\times S^3$ as admissible links of one of the standard models (i)-(vi) above, denoted $L_1,\cdots,L_k$. We have built a nonsingular dotted $\gl_2$ foam $F^\circ\colon W_0\sqcup(\sqcup_{i=1}^kL_i)\to W_1$ in $I\times S^3$, which induces a map $$CKhR_2^{univ}(F^\circ)\colon CKhR_2^{univ}(W_0)\otimes(\otimes_{i=1}^kCKhR_2^{univ}(L_i))\to CKhR_2^{univ}(W_1).$$ Evaluating at cocycles $z(L_1),\cdots,z(L_k)$ at all but the first tensorial factor in the source, we obtain a map $$CKhR_2^{univ}(F):=CKhR_2^{univ}(F^\circ)(-\otimes(\otimes_{i=1}^kz(L_i)))\colon CKhR_2^{univ}(W_0)\to CKhR_2^{univ}(W_1).$$ By the same argument as in \cite[Theorem~5.2]{morrison2022invariants}, $CKhR_2^{univ}(F)$ up to chain homotopy is independent of the choices of the ordering of singular points and of framed paths, and is functorial under composition of singular foams (\cite{morrison2022invariants} only argued this on the level of homology, but the relevant facts, namely the triviality of the sweep-around move and the $\pi_1(SO(3))$-action, both hold on the chain level up to homotopy).

It remains to prove the two extra items in Theorem~\ref{thm:gl2_webs_functorial}. Since Theorem~\ref{thm:gl2_webs_functorial}(1) holds for nonsingular dotted $\gl_2$ foams and respects compositions, it suffices to check it when $F$ has a single singular point of type (i)-(vi) and is a product elsewhere. By our construction, in addition to the contribution from $\chi(u(F))$, the bigrading shift of a singular point of type (i)-(vi) is given by $t^{-2}q^2$, $t^2q^{-2}$, $t^{-4}q^4$, $t^4q^{-4}$, $q^{-n}$, $(t^{-2}q^2)^n$, respectively, each of which is consistent with Theorem~\ref{thm:gl2_webs_functorial}(1), proving the statement.

Before proving Theorem~\ref{thm:gl2_webs_functorial}(2), we reinterpret the assignment $CKhR_2^{univ}(F)$ for singular foams $F$ on the level of diagrams. If $F$ has a single type (i) singularity and is a product elsewhere, then up to precomposing and postcomposing with isotopies, we may assume that $F$ takes a standard form, which is a negative-to-positive crossing change between a $1$-labeled edge and a $2$-labeled edge as shown in Figure~\ref{fig:H_12+_semi_std}.

The induced map $CKhR_2^{univ}(F)$ is represented by the movie which births the element $z(H_{12}^+)$ in the complex for the Hopf link $H_{12}^+$ near the point at $\infty$, drags this Hopf link near the relevant crossing (which maintains the element $z(H_{12}^+)$ via Lemma~\ref{lem:slide_=_shift}), and then performs the saddles and Reidemeister moves shown below. 
\begin{figure}[H]
\centering
\includegraphics[width=0.7\linewidth]{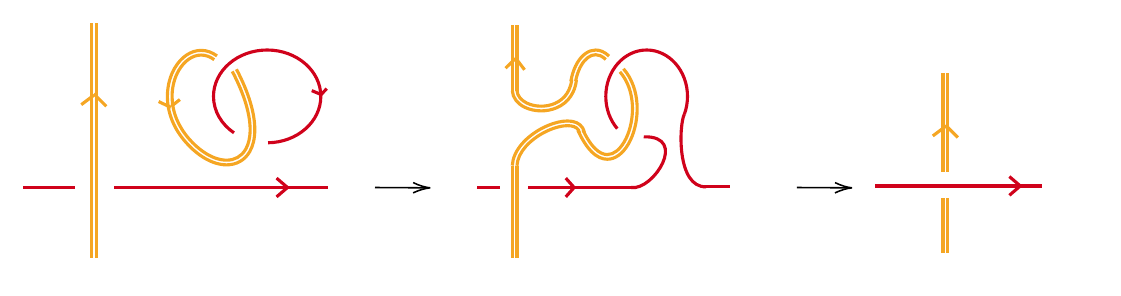}
\end{figure}
On the level of the resolution cube, this becomes
\begin{figure}[H]
\centering
\includegraphics[width=0.8\linewidth]{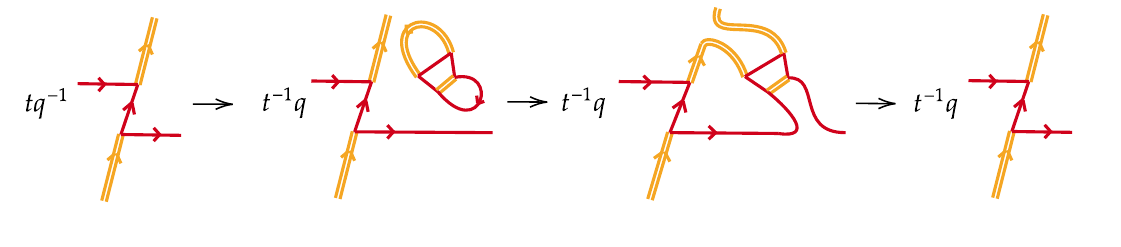},
\end{figure}
where the first map is the birth of $z(H_{12}^+)$, the second map is induced by two saddles, and the third map is a Reidemeister II induced map, given explicitly by the downward arrow in \cite[(A.15)]{queffelec2022gl2} (note that our diagrams correspond to the mirror of those in \cite{queffelec2022gl2}, hence we are using his (A.15) instead of (A.16))\footnote{The foam $F$ depicted in \cite[(A.15)]{queffelec2022gl2} and some other foams therein have some color inconsistency, which the readers may ignore.}. One can check that the composition of the underlying foams is isotopic to the identity foam, and the total shift and coefficient the composition carries is $t^{-2}q^2$. Therefore, the induced map of a standard type (i) singularity as shown in Figure~\ref{fig:H_12+_semi_std} is the degree shift by $t^{-2}q^2$.

In an analogous manner, one can compute the induced map of other type singularities explicitly in terms of movies on diagrams. We collect the results as follows.

\begin{Lem}\label{lem:singular_maps}
The induced map of a $\gl_2$ foam with a single singularity, represented as a chosen movie of diagrams, is given on the level of resolutions by the maps determined from Table~\ref{tab:singular_maps}.\qed
\end{Lem}
\begin{table}
\centering
\begin{tabular}{|c|c|c|}
\hline
Type of singularity & Diagram & Induced map on resolution \\\hline
(i) & 
\includegraphics[scale=0.5]{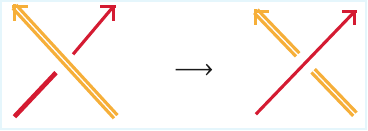}
\raisebox{10pt}
& $t^{-2}q^2$ \\[2pt]\hline
(ii) & 
\includegraphics[scale=0.5]{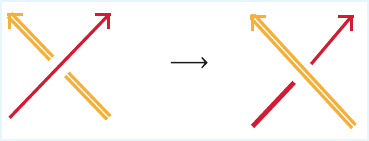}
\raisebox{10pt}
& $t^2q^{-2}$ \\\hline
(iii) & 
\includegraphics[scale=0.5]{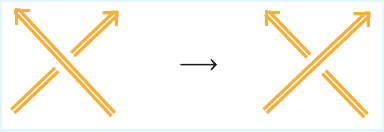}
\raisebox{10pt}
& $-t^{-4}q^4$ \\\hline
(iv) & 
\includegraphics[scale=0.5]{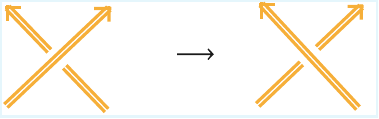}
\raisebox{10pt}
& $-t^4q^{-4}$ \\\hline
(v) & 
\includegraphics[scale=0.5]{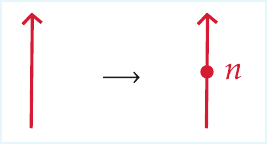}
\raisebox{10pt}
& $q^{-n}$ \\\hline
(vi) & 
\includegraphics[scale=0.5]{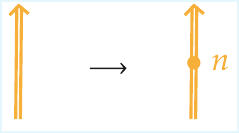}
\raisebox{10pt}
& $(t^{-2}q^2)^n$ \\\hline
(i) & 
\includegraphics[scale=0.5]{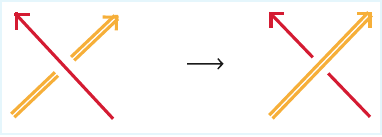}
\raisebox{10pt}
& $a_3a_4t^{-2}q^2$ \\\hline
(ii) & 
\includegraphics[scale=0.5]{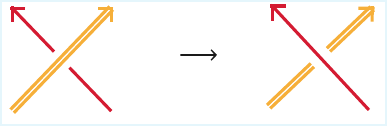}
\raisebox{10pt}
& $a_3a_4t^2q^{-2}$ \\\hline
\end{tabular}\vspace{5pt}
\caption{Induced maps by a singularity in a singular $\gl_2$ foam.}
\label{tab:singular_maps}
\end{table}

We will not use the last two rows of Table~\ref{tab:singular_maps}, but they are included for readers' convenience. It would be reasonable to impose the additional constraint $a_3=a_4$ in \cite{queffelec2022gl2} to remove the extra sign twists in the descriptions.

\begin{figure}
\centering
\includegraphics[width=0.7\linewidth]{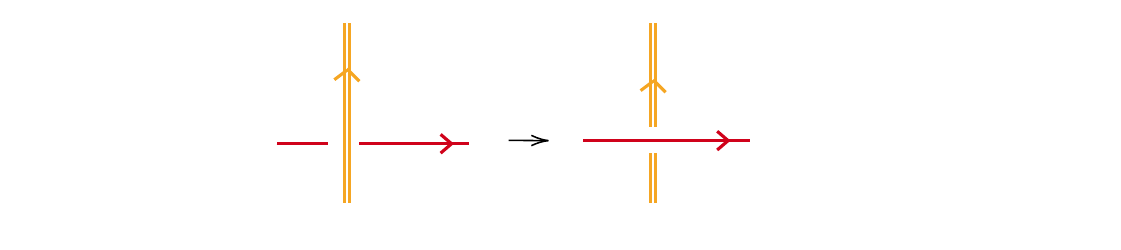}
\caption{The movie of a standard type (i) singularity.}
\label{fig:H_12+_semi_std}
\end{figure}

We are now ready to prove Theorem~\ref{thm:gl2_webs_functorial}(2). We first show the following topological lemma.
\begin{Lem}\label{lem:reembed_thick_moves}
If two singular $\gl_2$ foams $F,F'\colon W_0\to W_1$ in $I\times S^3$ have $u(F)=u(F')$, along which the germs of all data in $F$ and $F'$ agree, then $F$ is related to $F'$ by a sequence of the following moves:
\begin{enumerate}
\item Creation of a local unknotted $2$-labeled framed oriented $2$-sphere disjoint from $F$, or its inverse.
\item Tubing $2$-labeled faces along a framed arc ending on $F$ with interior disjoint from $F$, or its inverse.
\item Trading a local $\pm$-self-intersection (see Figure~\ref{fig:pos_self_int}) of a $2$-labeled face with a $\pm2$-labeled framing point, or its inverse.
\begin{figure}
\centering
\includegraphics[width=0.75\linewidth]{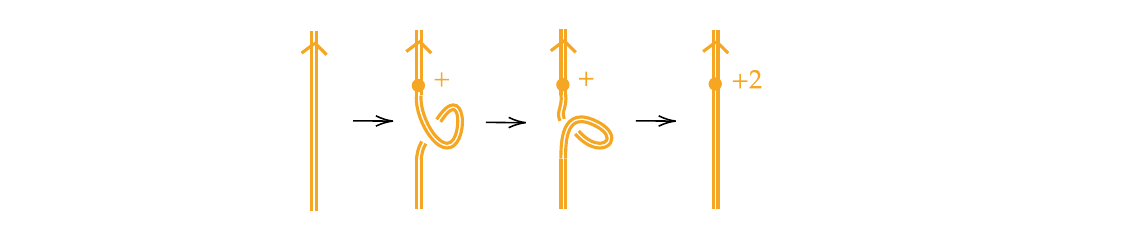}
\caption{A movie of a local positive self-intersection of a $2$-labeled face}
\label{fig:pos_self_int}
\end{figure}
\item A finger/Whitney move between $1$- and $2$-labeled faces.
\item A finger/Whitney move between $2$-labeled faces.
\item A fork version of the finger/Whitney move through a $2$-labeled face, as shown in Figure~\ref{fig:fork_Whitney}.
\item A framing change in the interior of $2$-labeled faces.
\item Collision of framing points on $2$-labeled faces in a weight-preserving way, or its inverse.
\begin{figure}
\centering
\includegraphics[width=0.85\linewidth]{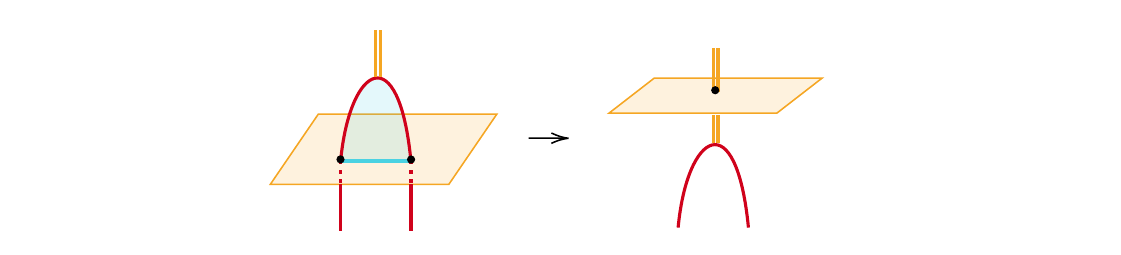}
\caption{A time slice of a fork-Whitney move between a $2$-labeled face and a neighborhood of a point on a seam. The shaded blue region indicates a Whitney disk. Any compatible orientation is allowed.}
\label{fig:fork_Whitney}
\end{figure}
\end{enumerate}
\end{Lem}
\begin{proof}
Let $F_1,F_2$ denote the closures of the union of $1,2$-labeled faces in $F$, respectively. We divide into two cases.

\textbf{Case 1}: $F$ has empty seam.

Without loss of generality, assume $F_2'\ne\emptyset$. Further assume $F_2'$ is connected by applying some moves (2). The immersed surface $(\{0\}\times F_2)\cup(I\times\partial F_2)\cup(\{1\}\times F_2')\subset\partial(I\times(I\times S^3))$ bounds an immersed $3$-manifold $W\looparrowright I\times(I\times S^3)$. By general position (and further modification of $W$ for the third item below if necessary), we may assume:
\begin{itemize}
\item $W$ has only transverse double point singularities along arcs and circles, where arcs may end either on $\partial W\subset\partial(I\times(I\times S^3))$ or on Whitney umbrella singularities in $int(W)$ \cite{whitney1944singularities}.
\item $W$ and $I\times F_1$ are in general position.
\item The projection of $W$ to the first $I$ coordinate is Morse without index $3$ critical points, and with all critical points away from $I\times F_1$.
\item The projection of $W\cap(I\times F_1)$ to the first $I$ coordinate is Morse.
\item The projection of the interior of the double point locus of $W$ to the first $I$ coordinate is Morse.
\item Local neighborhoods of Whitney umbrella singularities on $W$ are in general position with respect to the Morse function; thus, they appear in the time movie as local $\pm$-self-intersection creations/annihilations.
\end{itemize}
By going up the first $I$ coordinate, we see that there exists some singular $\gl_2$ foam $F''\colon W_0\to W_1$ that differs from $F'$ only by some framing points on $2$-labeled faces, so that $F$ and $F''$ are related by a sequence of moves (1)(2)(3)(4)(5)(7). The self-intersections of $F''$ and $F'$ are equal, since they are both determined by the common boundary data $W_0,W_1$; therefore, they are further related by some moves (8), as desired.

\textbf{Case 2}: $F$ has nonempty seams.

Using move (2), we assume $F_2$ (and $F_2'$) to be connected. If $p\in F_2$ is a self-intersection point, pick a generic path $\gamma$ on $F_2$ connecting $p$ to a point on a seam. One can then tube the other sheet of $F_2$ at $p$ along $\gamma$ and use move (6) to remove the intersection point $p$. Thus, we may assume $F_2$ (and $F_2'$) to have no self-intersections.

Let $s=\partial F_2=\partial F_2'$, and $\nu(s)$ be a tubular neighborhood of $s$. The relative homology classes represented by $F_2$ and $F_2'$ in $I\times S^3\backslash\nu(s)$ rel the common boundary differ by some meridian spheres of seams of $F$. If $s_0$ is a seam of $F$, using move (3), we may create a local self-intersection on $F_2$ near $s_0$. Then we may slide the self-intersection off the seam $s_0$ as in the previous paragraph, changing the relative homology class of $F_2$ by a meridian sphere of $s_0$. Hence, we may arrange so that $F_2$ and $F_2'$ are homologous rel the common boundary in $I\times S^3\backslash\nu(s)$.

Since $F_2$ and $F_2'$ are embedded and homologous rel boundary, there is an embedded $3$-manifold $W\subset I\times(I\times S^3)$ cobounding $\{0\}\times F_2$ and $\{1\}\times F_2'$, which agrees with $I\times F_2$ in $I\times\nu(s)$. Now the same Morse theory argument as in Case 1 shows that one may change $F_2$ to $F_2'$ by a sequence of moves (1)(2)(4)(7)(8).
\end{proof}

Hence, to prove Theorem~\ref{thm:gl2_webs_functorial}(2), it suffices to show that $CKhR_2^{univ}(F)$ changes sign under moves (1)(2) and is invariant under moves (3)-(8) described in Lemma~\ref{lem:reembed_thick_moves}.

Let $F$ and $F'$ be related by one of the moves (1)-(8).

(1)(2): $F$ and $F'$ are related by a local skein relation \cite[(3.2)]{queffelec2022gl2} which introduces a sign change.

(7): The induced maps by $F$ and $F'$ are equal by the same proof as in Claim 3 of the proof of Theorem~\ref{thm:Lee_foam_evaluation}.

(8): This follows from the description of induced maps for type (vi) singularities in Lemma~\ref{lem:singular_maps}.

(4): If $F'$ is obtained from $F$ by a finger move, then locally we can represent $F$ by the constant movie and $F'$ by a movie that does a crossing change shown in the first row of Table~\ref{tab:singular_maps}, followed by the inverse change shown in the second row of Table~\ref{tab:singular_maps}. Lemma~\ref{lem:singular_maps} implies that $F'$ induces the same map as $F$.

(5): This is similar to (4), where we use the third and fourth rows of Table~\ref{tab:singular_maps} instead of the first two.

(3): For positive local self-intersections, we need to run through the moves in Figure~\ref{fig:pos_self_int} and show that the induced map of the composition agrees with that of a type (vi) singularity with $n=2$. By the description of moves \cite[(A.3)(A.4)]{queffelec2022gl2} and the third row of Table~\ref{tab:singular_maps}, the composition induces the degree shift by $t^{-4}q^4$, which agrees with the description in the sixth row of Table~\ref{tab:singular_maps} for $n=2$. The calculation for negative local self-intersections is similar; alternatively, one may decompose a negative local self-intersection annihilation as a positive local self-intersection creation followed by a Whitney move.

(6): If $F'$ is obtained from $F$ by a fork-finger move, then locally we can represent $F$ by the constant movie and $F'$ by the movie
\begin{figure}[H]
\centering
\includegraphics[width=0.75\linewidth]{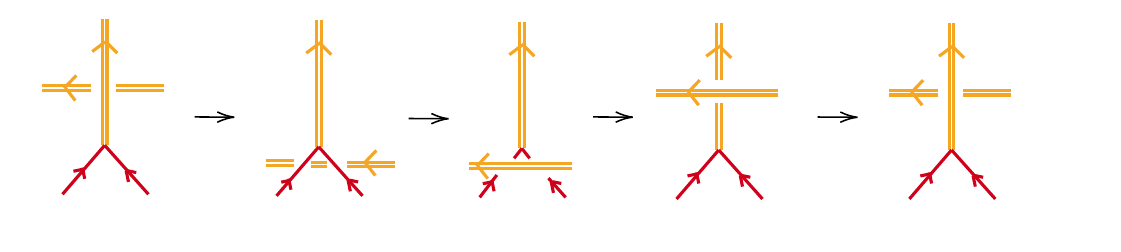}
\end{figure}
or its reverse. The reverse of every move in the movie induces the inverse map of the corresponding forward move. Hence it suffices to check that the forward movie induces the identity map, or more conveniently, that the composition of the first two maps is equal to the composition of the inverses of the last two maps. This follows from moves \cite[(A.57)(A.58)]{queffelec2022gl2} and the second and fourth rows of Table~\ref{tab:singular_maps}.

The proof of Theorem~\ref{thm:gl2_webs_functorial} is complete.\qed\smallskip

Before moving on, we note that the proof of Theorem~\ref{thm:gl2_webs_functorial} can be used to provide a singular version of the functor \eqref{eq:onsurface} with domain $\TanWebsing{I\times S,\epsilon}$, namely:
\begin{equation}\label{eq:local singular BNfunc}
\TanWebsing{I\times S,\epsilon}\xrightarrow{\BNfunc{\cdot}} \HChb(\Foams{S,\epsilon}).
\end{equation}

\subsubsection{Monoidal \texorpdfstring{$2$}{2}-categories and braidings}
\label{ssec:braidedcat}

In Section~\ref{sec:gl2_webs_local}, we regarded $\Foams{S,\epsilon},\TanWeb{I\times S,\epsilon}$ as $1$-categories, with boundary conditions fixed. In the special case $S=I^2$ one may instead leave the boundary data free (with certain restrictions) and proceed as follows.

\begin{enumerate}
\item We let $\Foams{I^2,-}$ denote the monoidal $2$-category defined as follows. Objects are labeled sign sequences $\sigma \in (\{+,-\}\times\{1,2\})^n$ for various $n\geq 0$, with monoidal structure given by concatenation. We often think of $\sigma$ alternatively as a set of evenly spaced labeled oriented points in $I$. The morphism category from $\sigma_0$ to $\sigma_1$ is precisely $\Foams{I^2,\epsilon(\sigma_0,\sigma_1)}$, where $\epsilon(\sigma_0,\sigma_1)=(-\sigma_0\times\{0\})\sqcup(\sigma_1\times\{1\})$. Thus $1$-morphisms are planar webs from $\sigma_0$ to $\sigma_1$, and $2$-morphisms are foams between them. The monoidal structure is given by placing webs and foams side-by-side and rescaling. Every $1$-morphism in the endomorphism category of the distinguished object $\emptyset$ (no boundary points), i.e. every closed web, is isomorphic to a direct sum of grading shifts of $\mathrm{id}_{\emptyset}$, i.e. the empty web. Since $\mathrm{End}(\mathrm{id}_{\emptyset})$ is the ground ring $\Z[E_1,E_2]$, the representable functor $\operatorname{Hom}(\mathrm{id}_{\emptyset}, -)$ witnesses an equivalence with the category of finitely generated graded free $\Z[E_1,E_2]$-modules which is braided (in fact, symmetric) monoidal.
\item We let $\TanWebsing{I\times I^2,-}$ denote the braided monoidal $2$-category with the same objects as $\Foams{I^2,-}$, whose morphism category from $\sigma_0$ to $\sigma_1$ is $\TanWebsing{I\times I^2,\epsilon(\sigma_0,\sigma_1)}$. The braiding $1$-morphisms use standard crossings in $I\times I^2$. 
\item We let $\HChb(\mathbf{Foams}_{I^2,-})$ denote the monoidal $2$-category whose objects are the same as in $\Foams{I^2,-}$, $1$-morphisms are chain complexes over $\mathbf{Foams}_{I^2,-}$, and $2$-morphisms are equivalence classes of homogeneous closed morphisms between such chain complexes, with components given by $2$-morphisms in $\mathbf{Foams}_{I^2,-}$, considered up to homogeneous exact morphisms. It is expected, but not yet proven, that this monoidal $2$-category admits a braiding\footnote{In light of the homological algebra involved, it may be more natural to model this braided monoidal $2$-category as a truncation of an $\mathbb{E}_2$-monoidal $(\infty,2)$-category, i.e. as a $\gl_2$ version of \cite{liu2024braided,stroppel2024braiding}. We refer to these articles for an in-depth discussion of the necessary higher algebra, which then accommodates all higher movie moves.}, such that the functors \eqref{eq:local singular BNfunc} assemble into a braided monoidal $2$-functor
\[
\TanWebsing{I\times I^2,-} \xrightarrow{\BNfunc{\cdot}} \HChb(\Foams{I^2,-}).
\]
Since endomorphisms of $\emptyset$ in $\TanWebsing{I\times I^2,-}$ can be identified with $\TanWebsing{\R^3}$, composing with a representable functor recovers the singular version
\[CKhR_2^{univ}\colon
\TanWebsing{\R^3}\to\HChb(\Z[E_1,E_2])^\Z
\]
of Queffelec's functor \eqref{eq:CKhR_2_webs_R3}.
 
\item As in \cite[Section~6]{morrison2022invariants}, one can construct a braided monoidal $2$-category without further higher-algebraic complications by a mixture between (2) and (3): objects and $1$-morphisms are as in $\TanWebsing{I\times I^2,-}$, but spaces of $2$-morphisms are computed inside $\HChb(\mathbf{Foams}_{I^2,-})$ after applying the functor \eqref{eq:onsurface}. The axioms of a braided monoidal $2$-category then follow from the functoriality of \eqref{eq:onsurface}. We invite the readers to keep this braided $2$-category in mind for the following discussion of sylleptic centers.
 \end{enumerate}

\begin{Rmk}
\label{rem:blockdecomposition}
Note that all (higher) categories here decompose by $\Z$-valued \emph{weight} computed on objects as signed sum of all labels.
\end{Rmk}

\subsubsection{Graded sylleptic considerations}
\label{sec:sylleptic}

Recall that in a braided monoidal $2$-category, every pair of objects $A,B$ admits a braiding $1$-morphism
\[ A \boxtimes B \xrightarrow{R_{A,B}} B \boxtimes A\]
which is invertible up to $2$-isomorphisms and equipped with higher coherence $2$-morphisms that express categorified analogs of the naturality and hexagon axioms of braided monoidal $1$-categories. A braided $1$-category is symmetric if the double braiding of any two objects is the identity. For braided $2$-categories the situation is more subtle. A coherent trivialization of the double braiding:
\[ \nu_{A,B}\colon (A \boxtimes B \xrightarrow{R_{A,B}} B \boxtimes A \xrightarrow{R_{B,A}} A \boxtimes B) \xrightarrow{\cong} (A \boxtimes B \xrightarrow{\mathrm{id}} A \boxtimes B)\] 
is called a \emph{syllepsis}. Coherence of the trivializations requires, amongst others, naturality in both arguments. Note that a syllepsis can also be considered as a coherent identification
\[
(A \boxtimes B \xrightarrow{R_{A,B}} B \boxtimes A) 
\xrightarrow{\cong} 
(A \boxtimes B \xrightarrow{R_{B,A}^{-1}} B \boxtimes A)
\]
between positive and negative (inverse) braiding $1$-morphisms.

Given a braided monoidal $2$-category, one can consider its \emph{sylleptic center} \cite[Section~5.1]{crans1998generalized}\footnote{Named \emph{$2$-center} there.}, which consists of objects $A$ equipped with a coherent trivialization of the double braiding with any other object $B$. True to the naming, the sylleptic center is then naturally a sylleptic monoidal $2$-category itself \cite[Theorem~5.1]{crans1998generalized}. 

For the putative braided $2$-category $\HChb(\Foams{I^2,-})$ from Section~\ref{ssec:braidedcat} we expect that all purely $2$-labeled objects can be interpreted as objects in a \emph{$\Z$-crossed} analog of the sylleptic center. Note that we have a natural $\Z$-action on $1$-morphisms (i.e. complexes of planar webs) with the generator $1\in\Z$ acting by the grading shift autoequivalence $t^2q^{-2}$. We observe that Section~\ref{sec:gl2_webs_functorial_singular} (with any choice of signs $a_3$ and $a_4$) provides for every purely $2$-labeled object $A$ and every other object $B$ a coherent identification
\[
(A \boxtimes B \xrightarrow{R_{A,B}} B \boxtimes A) 
\xrightarrow{\cong} 
(t^2q^{-2})^{|A|\cdot|B|/2}(A \boxtimes B \xrightarrow{R_{B,A}^{-1}} B \boxtimes A)
\]
of the braiding of $A$ and $B$ with a grading shift of the inverse braiding, see Table~\ref{tab:singular_maps}.

The ($\Z$-crossed) sylleptic center is the natural home of objects whose identity $1$-morphisms admit a coherent system of \emph{unbelting} $2$-isomorphisms. The bottom projector $P_{\sigma,0}^\vee$ to be constructed in Section~\ref{sec:gl2_projectors} can be interpreted as projection onto the full sub-$2$-category on purely $2$-labeled objects and thus, possibly, into the $\Z$-crossed sylleptic center.

\subsection{\texorpdfstring{$\gl_2$}{gl2} Rozansky projectors}\label{sec:gl2_projectors}
In this section, we follow a recipe of Hogancamp in \cite[Section~5.2]{hogancamp2020constructing} to construct Rozansky projectors in the universal $\gl_2$ webs and foams setting. We also give sketch proofs of properties of the Rozansky projectors stated in Proposition~\ref{prop:projector_properties} in this setup. We follow the terminologies in \cite{hogancamp2020constructing}, with the caveat that we are applying the dual construction of \cite{hogancamp2020constructing} (see the comment at the beginning of Section~1.2 in \cite{hogancamp2020constructing}).

Let $\sigma$ be an object in $\Foams{I^2,-}$, namely a labeled signed sequence $(\{+,-\}\times\{1,2\})^n$ for some $n\ge0$. Let $u(\sigma)$ denote the underlying (unoriented) $1$-labeled points determined by $\sigma$, and let $\ell:=\#u(\sigma)$. We build a Rozansky projector $P_{\sigma,0}^\vee\in\mathbf K^+(\mathbf{Foams}_{I^2,-}(\sigma,\sigma))$ which, upon forgetting the thick edges, orientations, and setting $E_1=E_2=0$, recovers the Rozansky projector $P_{\ell,0}^\vee$ that appeared in Section~\ref{sec:projectors}. Here, $\mathbf K^+(\mathbf{Foams}_{I^2,-}(\sigma,\sigma))$ denotes the bounded below homotopy category of cochain complexes in $\mathbf{Foams}_{I^2,-}(\sigma,\sigma)=\mathbf{Foams}_{I^2,\epsilon(\sigma,\sigma)}$.

Let $\mathcal B_{u(\sigma)}$ denote the finite set of crossingless matchings of $u(\sigma)\times\{1\}\subset I^2$ up to isotopy rel boundary\footnote{Note that for $\ell$ odd, $\mathcal B_{u(\sigma)}=\emptyset$, and the following construction will produce $P_{\sigma,0}^\vee=0$.}. For each $\delta\in\mathcal B_{u(\sigma)}$, we pick a web $W_\delta\in\mathbf{Foams}_{I^2,-}(\tau,\sigma)$ for some (necessarily $2$-labeled) object $\tau$ so that forgetting the $2$-labeled edges and orientations gives $u(W_\delta)=\delta$. We also pick a $2$-morphism $\eta_{\delta}\colon 1_\sigma\to q^{\ell/2}W_\delta\otimes W_\delta^t$ which recovers the cobordism $1_\ell\to q^{\ell/2}\delta\otimes\delta^t$ given by $\ell/2$ saddles upon forgetting $2$-labeled faces and orientations. Here, $1_\sigma$ is the identity $1$-morphism at $\sigma$, and $(\cdot)^t$ denotes vertical reflection composed with orientation reversal.

Let $C:=q^{\ell/2}(\oplus_{\delta\in\mathcal B_{u(\sigma)}}W_\delta\otimes W_\delta^t)$ be an object in the monoidal category $\mathcal A=\mathcal A_\sigma:=\mathbf{Foams}_{I^2,-}(\sigma,\sigma)$, equipped with the unit map
\[\eta_C=\oplus_{\delta\in\mathcal B_{u(\sigma)}}\eta_\delta\colon1_\sigma\to C.\]
Then, an object in $\mathcal A$ is (left, or equivalently right) $C$-injective in the dual sense of \cite[Definition~2.5]{hogancamp2020constructing} if and only if it factors through a purely $2$-labeled object.

Let $P_C^\vee\in\mathbf K^+(\mathcal A)$ be defined by applying the cobar construction dual to \cite[(3.1)]{hogancamp2020constructing} to the unital object $C$ in $\mathcal A$, which comes with a unit map $\iota_C\colon1_\sigma\to P_C^\vee$. By Hogancamp \cite[Theorem~3.12]{hogancamp2020constructing}, $(P_C^\vee,\iota_C)$ is an idempotent algebra in $\mathbf K^+(\mathcal A)$ characterized up to homotopy by
\begin{enumerate}[(a)]
\item $P_C^\vee$ is a complex of $C$-injective objects;
\item $\mathrm{id}_C\otimes\iota_C\colon C\to C\otimes P_C^\vee$ and/or $\iota_C\otimes\mathrm{id}_C\colon C\to P_C^\vee\otimes C$ is a homotopy equivalence.
\end{enumerate}
Moreover, if $P_C^\vee,P_C'^\vee$ are two unital idempotent algebras in $\mathbf K^+(\mathcal A)$ satisfying (a) and (b), then there is a unique homotopy equivalence $P_C^\vee\simeq P_C'^\vee$ up to homotopy that interwines with the unit maps up to homotopy.

We may further deloop all loops formed by $1$-labeled edges in the components of $P_C^\vee$ to obtain a preferred model of the Rozansky projector, denoted $(P_{\sigma,0}^\vee,\iota_\sigma)$. Properties (1) and (2) in Proposition~\ref{prop:projector_properties} in their webs and foams versions thus follow from the construction. When $\sigma$ is purely $2$-labeled, we may choose $C=1_\sigma$, hence (after a further homotopy) $P_{\sigma,0}^\vee=1_\sigma$ with the identity unit map. This shows the analog corresponding to Proposition~\ref{prop:projector_properties}(3). The analog of Proposition~\ref{prop:projector_properties}(4) follows from \cite[Remark~3.4]{hogancamp2020constructing}.

Let $\sigma^-$ denote the orientation-reversal of the flip of $\sigma$. The $\pi$-rotation of $P_{\sigma,0}^\vee$ is a complex in $\mathbf K^+(\mathcal A_{\sigma^-})$ satisfying the characterizing properties of $P_{\sigma^-,0}^\vee$. Hence, there is a canonical homotopy equivalence verifying the analog of Proposition~\ref{prop:projector_properties}(8).

To prove the analogs of Proposition~\ref{prop:projector_properties}(5)(6)(7) in our setup, we observe the following property of the Rozansky projectors. Let $\Foams{I^2,-}^{(1.5)}$ denote the collection of full subcategories of the hom-categories in $\Foams{I^2,-}$ on all $1$-morphisms that factor through purely $2$-labeled objects. We think of $\mathbf{Foams}_{I^2,-}^{(1.5)}$ as an ideal of $\mathbf{Foams}_{I^2,-}$ with respect to the horizontal composition. Thus, by construction, $P_{\sigma,0}^\vee\in\mathbf K^+(\mathcal A^{(1.5)})\subset\mathbf K^+(\mathcal A)$, where $\mathcal A^{(1.5)}=\mathcal A_\sigma^{(1.5)}:=\Foams{I^2,-}^{(1.5)}(\sigma,\sigma)$. By \cite[Theorem~3.12]{hogancamp2020constructing}, if $X\in\mathbf K^+(\mathbf{Foams}_{I^2,-}^{(1.5)}(\sigma',\sigma))$ for some $\sigma'$, then $\iota_\sigma\otimes\mathrm{id}_X\colon X\xrightarrow{\simeq}P_{\sigma,0}^\vee\otimes X$ is a homotopy equivalence. Similarly, if $X\in\mathbf K^+(\mathbf{Foams}_{I^2,-}^{(1.5)}(\sigma, \sigma'))$ for some $\sigma'$, then $\iota_\sigma\otimes\mathrm{id}_X\colon X\xrightarrow{\simeq}P_{\sigma,0}^\vee\otimes X$ is a homotopy equivalence. Morally, one should think of $P_{\sigma,0}^\vee$ as projecting $\mathbf K^+(\mathcal A)$ onto $\mathbf K^+(\mathcal A^{(1.5)})$. The analog of Proposition~\ref{prop:projector_properties}(5) and the first part of (7) directly follow from these properties. The second part of (7) follows by bending up the lower right half of the diagrams. Finally, (6) follows by bending down the two sides of the over/understrand, since the complex in the source is then termwise, hence overall, homotopy equivalent to a complex in some $\mathbf K^+(\mathbf{Foams}_{I^2,-}^{(1.5)}(\sigma',\sigma))$.

\subsection{The sign fixes}\label{sec:gl2_sign}
In this section we use the webs and foams formalism to resolve all sign ambiguities present in the main body of the paper. Throughout this section we consider webs and foams versions of various earlier diagrams. In particular, in any previous diagram involving $P_{\ell,0}^\vee$, we note that the orientations of the $\ell$ strands determine a (purely $1$-labeled) sign sequence $\sigma\in\Foams{I^2,-}$, and we replace such $P_{\ell,0}^\vee$ (boxes in the diagrams) with $P_{\sigma,0}^\vee$ from Section~\ref{sec:gl2_projectors}.

\subsubsection{Sliding belts down}\label{sec:sign_slide}
In Section~\ref{sec:SZ}, when deriving the isomorphism \eqref{eq:SZ} by breaking it into a sequence of isomorphisms, the isomorphism on row \eqref{eq:SZ_slide} by ``sliding off'' the belts was only well-defined up to sign. In the webs and foams formalism, the ``sliding-off'' maps are termwise given by singular foams that drag the belts off, hitting the $2$-labeled strands in the middle level transversely---these can be interpreted as components of the syllepsis in the sense of Section~\ref{sec:sylleptic}. By Theorem~\ref{thm:gl2_webs_functorial}, these termwise maps fit into a ``sliding-off'' isomorphism supplying \eqref{eq:SZ_slide}.

\subsubsection{Regions \texorpdfstring{$B$}{B} to \texorpdfstring{$H$}{H} and \texorpdfstring{$R_1$}{R1}}\label{sec:sign_region_B_R1}
We fix the sign in the proof of the commutativity of the lower triangle in region $R_1$ in Section~\ref{sec:handleslides}. It suffices to fix the sign termwise. On each term of the twisted complex, the two composite cobordisms agree up to re-embedding the interior of $2$-labeled faces, hence the commutativity follows from Theorem~\ref{thm:gl2_webs_functorial}.\medskip

We fix the sign in the proof of the commutativity of region $B$ in Section~\ref{sec:concrete_finger}, and the fixes for regions $C,D,E,F,G,H$ are analogous. This is a consequence of Lemma~\ref{lem:sweep_around_across} below.

\begin{Lem}[Enhanced sweep-around move]\label{lem:sweep_around_across}
The singular $\gl_2$ foam given by the movie in Figure~\ref{fig:sweep_around_across} induces the identity chain map up to homotopy on the universal $\gl_2$ tangle invariant.
\end{Lem}

\begin{figure}
\centering
\includegraphics[width=0.8\linewidth]{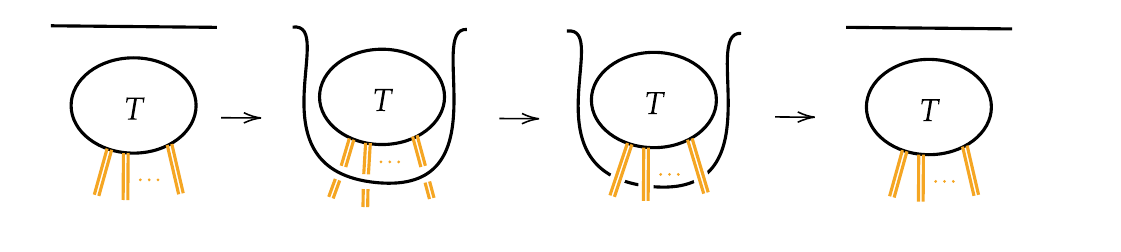}
\caption{The movie of a ``sweep-around-across'' singular foam, where $T$ is an arbitrary tangled web, the black strand has an arbitrary label, and all strands shown have arbitrary orientations.}
\label{fig:sweep_around_across}
\end{figure}

\begin{proof}
Say there are $k$ incoming $2$-labeled strands connected to $T$, and hence $k$ outgoing ones. When $k=0$, the statement follows from applying Lemma~\ref{lem:slide_=_shift} twice. In general, use $k$ saddles to pair up the incoming and outgoing strands, exploiting functoriality and reduce to the case $k=0$.
\end{proof}

\subsubsection{The barbell move}\label{sec:sign_barbell}
We fix the sign $c=1$ near the end of the proof of Lemma~\ref{lem:barbell_las}.

In the webs and foams formalism, when fixing the constant $c$, instead of capping the $(m_+,m_-)$ (resp. $(n_+,n_-)$ strands off by dotted annuli in each row of Figure~\ref{fig:barbell_check}, we need to perform a combination of dotted annular caps and zips, illustrated on $(+,+,+,-)$ (write for short ambiguously still as $(3,1)$) strands as
\begin{figure}[H]
\centering
\includegraphics[width=0.75\linewidth]{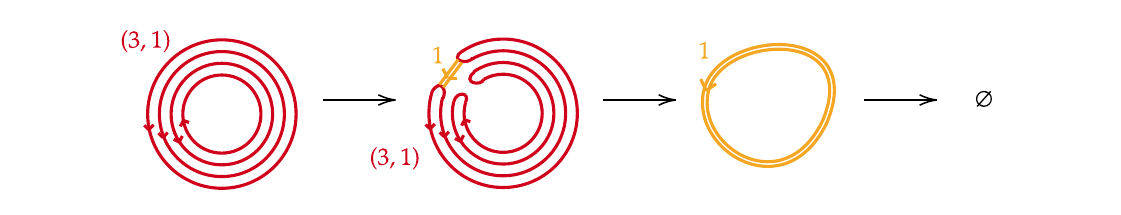},
\end{figure}
where the first map is a zip and a saddle, the second map is two dots, a cap, and a zip-cap, and the third map is a cap. Here, cyclic orderings on seams and placements of dots before zip-caps are fixed once and for all, for each sequence of $+,-$ of even length. As this capping procedure is more complicated than before, we present the relevant diagram chasing in Figure~\ref{fig:barbell_sign}. Here, only a half of each term is drawn, where the other half is understood as obtained by switching $m$ and $n$, $m'$ and $n'$, and reversing the orientations on the $m,m'$ strands. All squares and the triangle commute by locality and Theorem~\ref{thm:gl2_webs_functorial}. The two downward maps to the front middle term in the bottom are equal by Lemma~\ref{lem:sweep_around_across}. We need to show that the composite maps from $\emptyset$ to the rightmost term are equal in $KhR_2$, under either the upward or the downward coevaluation map composed with any other path of maps or their inverses (if invertible). For this purpose, we stay in the terms without projectors in Figure~\ref{fig:barbell_sign}, from which the equality of the two compositions follows from Theorem~\ref{thm:gl2_webs_functorial}.

\begin{figure}
\centering
\includegraphics[width=0.8\linewidth]{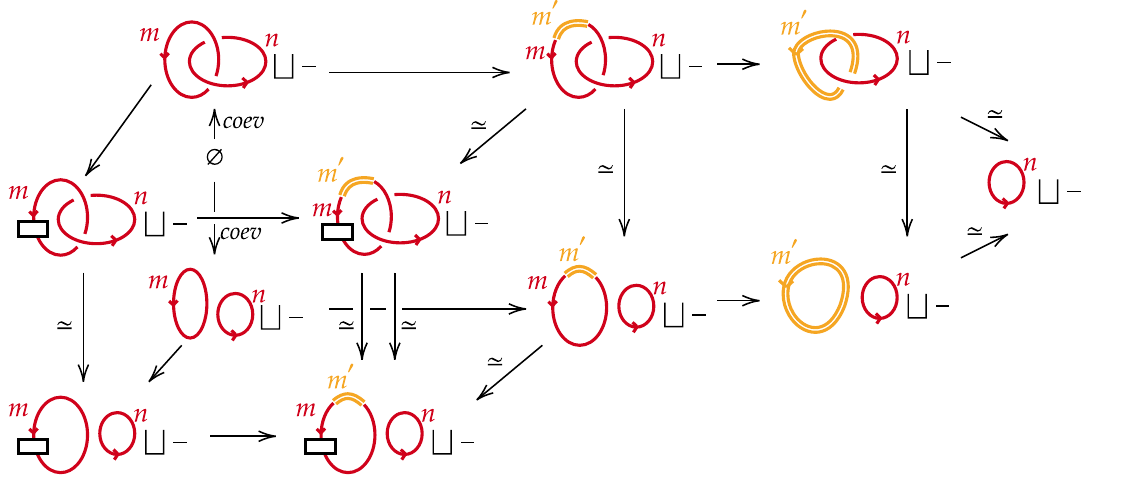}
\caption{Diagram chase for fixing $c=1$ in the proof of Lemma~\ref{lem:barbell_las}.}
\label{fig:barbell_sign}
\end{figure}

\subsubsection{The Gluck twist}\label{sec:sign_gluck}
We fix the sign ambiguities in Section~\ref{sec:las_gluck} that arose when defining various element $1$'s in the homology of twisted belt links, as well as when showing some compatibilities of these element $1$'s in Lemma~\ref{lem:twist_belt_link}.

Let $T\subset S^1\times S^2$ be a standard positive/negative twisted belt link and $U\subset S^1\times S^2$ be the corresponding standard belt link obtained by untwisting. In the webs and foams formalism, the isomorphism $\widetilde{KhR}_2^+(T)\cong\widetilde{KhR}_2^+(U)$ is obtained termwise by pushing the $\pm1$ twist above the Rozansky projector region to the vertical $2$-labeled edges in the middle of the Rozansky projector region by simplifying Reidemeister I,II moves and fork twist moves (in the sense of \cite{queffelec2022gl2}), post-composed with the singular $\gl_2$ foam that undoes the $\pm1$ twist on the $2$-labeled strands by crossing changes together with a $\mp1$ framing change on each strand. This sign fix consequently fixes the signs of $1_\pm\in\widetilde{KhR}_2^+(T)$, $1_\pm\in\mathcal S_0^2(D^2\times S^2;T)$, and $1\in\widetilde{KhR}_2^+(T(n,n)_{n_+,n_-})$.\medskip

It remains to fix the termwise signs when proving the commutativities of regions $TB$, $TC$, $TD$, $TE$, $TG$, $TX$. For regions $TB$, $TC$, $TD$, $TE$, it suffices to check that the composition
\begin{figure}[H]
\centering
\includegraphics[width=0.85\linewidth]{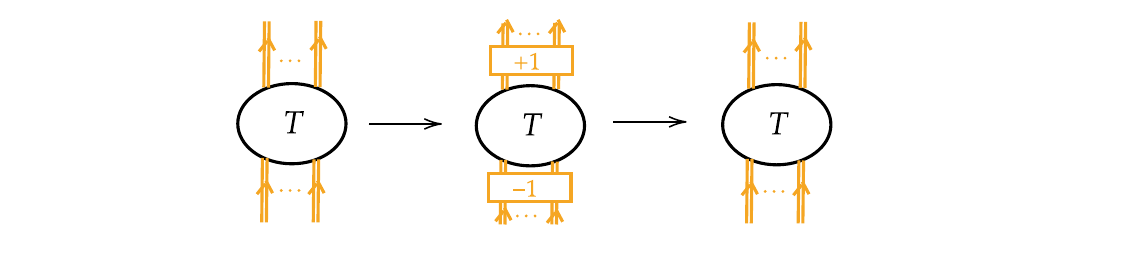}
\end{figure}
induces the identity map on homology (rather than its negative), where $T$ is any tangle, the first map is a twist around $T$, the second map is an unlinking of the $\pm1$ twists together with framing changes on $2$-labeled edges. It suffices to check this after we close up the $2$-labeled strands. In the closed up diagram, the second map can be replaced by the map that cancels the twists in the outer part without changing its induced map, thanks to Theorem~\ref{thm:gl2_webs_functorial}. The composition is now equal to the rotation by $2\pi$ around the core of the solid torus that the closed up diagrams live in. Since this rotation extends to a $2\pi$ rotation in $S^3$, it induces the identity map. By a re-embedding of $2$-labeled faces, the same argument applies to fix the sign for region $TG$. The sign fix for region $TX$ is easier and we omit the proof.

\printbibliography

\end{document}